\documentclass[11pt]{article}

\usepackage[margin=1.1in]{geometry}

\usepackage{listings}
\usepackage{amsmath,empheq}
\usepackage{amssymb}
\usepackage{amsthm}
\usepackage{enumerate}
\usepackage{graphicx}
\usepackage{color}
\usepackage{cases}
\usepackage{caption}
\usepackage{subfigure}\usepackage{caption}
\usepackage{algorithm}
\usepackage{algorithmicx}
\usepackage{algpseudocode}
\usepackage{url}
\usepackage[dvipsnames]{xcolor}
\usepackage{mathabx}

\usepackage{kbordermatrix}
\usepackage{latexsym}
\usepackage{mathrsfs,amsbsy}
\usepackage{tikz}
\usetikzlibrary{arrows,decorations.pathmorphing,backgrounds,positioning,fit,matrix}

\def\cA{{\cal A}}
\def\cB{{\cal B}}
\def\cE{{\cal E}}
\def\cI{{\cal I}}
\def\cJ{{\cal J}}
\def\cN{{\cal N}}
\def\cU{{\cal U}}
\def\cV{{\cal V}}
\def\cR{{\cal R}}

\def\bone{\mathbf{1}}

\newcommand{\ti}{\mbox{\tt i}}

\def\sss{\scriptscriptstyle}

\def\wtd{\widetilde}
\def\what{\widehat}

\def\scrL{\mathscr{L}}
\def\scrT{\mathscr{T}}

\def\bbC{\mathbb{C}}
\def\bbR{\mathbb{R}}

\DeclareMathOperator{\cut}{cut}
\DeclareMathOperator{\Ncut}{Ncut}

\DeclareMathOperator{\diag}{diag}
\DeclareMathOperator{\eig}{eig}
\DeclareMathOperator{\err}{err}

\DeclareMathOperator{\rank}{rank}
\DeclareMathOperator{\subspan}{span}
\DeclareMathOperator{\trans}{trans}
\DeclareMathOperator{\vol}{vol}

\DeclareMathOperator{\LGopt}{LGopt}
\DeclareMathOperator{\QEPmin}{QEPmin}

\DeclareMathOperator{\HH}{H}
\DeclareMathOperator{\T}{T}

\DeclareMathOperator{\NRes}{NRes}

\floatname{algorithm}{Algorithm}

\newtheorem{theorem}{Theorem}[section]
\newtheorem{proposition}{Proposition}[section]
\newtheorem{lemma}{Lemma}[section]
\newtheorem{corollary}{Corollary}[section]

\theoremstyle{definition}
\newtheorem{definition}{Definition}[section]
\newtheorem{remark}{Remark}[section]
\newtheorem{example}{Example}[section]

\newcommand{\ignore}[1]{}
\numberwithin{equation}{section}

\title{
Linear Constrained Rayleigh Quotient Optimization: \\
Theory and Algorithms}
\author{Yunshen Zhou \and Zhaojun Bai \and Ren-Cang Li}
\date{\today}

\begin{document}
\maketitle

\begin{abstract}
%\marginpar{\tiny to be rewritten}
We consider the following constrained Rayleigh quotient optimization problem (CRQopt)
$$
\min_{x\in \mathbb{R}^n} x^{\T}Ax\,\,\mbox{subject to}\,\, x^{\T}x=1\,\mbox{and}\,C^{\T}x=b,
$$
where $A$ is an $n\times n$ real symmetric matrix and $C$ is an $n\times m$ real matrix. 
      Usually, $m\ll n$.
The problem is also known as the constrained eigenvalue problem in the literature 
because it becomes an eigenvalue problem
if the linear constraint $C^{\T}x=b$ is removed. We start by equivalently transforming  CRQopt into
an optimization problem, called LGopt, of minimizing the Lagrangian multiplier of CRQopt, and then
an problem, called QEPmin, of finding the smallest eigenvalue of a quadratic eigenvalue problem.
Although such equivalences has been discussed 
      in the literature, it appears to be the first time that these equivalences are
rigorously justified.  Then we propose to numerically solve LGopt and QEPmin by
the Krylov subspace projection method via the Lanczos process. The basic idea, as the Lanczos method
for the symmetric eigenvalue problem,
is to first reduce LGopt and QEPmin by projecting them onto Krylov subspaces to yield problems of the same types but
of much smaller sizes, and then solve the reduced problems by some direct methods, which is
either a secular equation solver (in the case of LGopt) or an eigensolver (in the case of QEPmin).
The resulting algorithm is called the Lanczos algorithm.
We perform convergence analysis for the proposed method and obtain error bounds.
The sharpness of the error bound is demonstrated by artificial examples, although in applications the method often converges
much faster than the bounds suggest.
Finally, we apply the Lanczos algorithm to semi-supervised learning in the context of 
      constrained clustering. 
\end{abstract}

\newpage
\tableofcontents

%\bigskip
%\hrule 

%\newpage
\section{Introduction} \label{sec-intro}
%\marginpar{\tiny to be rewritten}
In this paper, we are concerned with the following {\em linear constrained Rayleigh
quotient\/} (CRQ) optimization:
\begin{subequations}\label{eq:CRQopt}
\begin{empheq}[left={\mbox{CRQopt:}\quad}\empheqlbrace]{alignat=2}
\min ~ & v^{\T}Av,    \label{eq:CRQopt-0} \\
\text{s.t.}~ & v^{\T} v = 1, \label{eq:CRQopt-1} \\
~ & C^{\T} v = b, \label{eq:CRQopt-2}
\end{empheq}
\end{subequations}
where $A\in \bbR^{n\times n}$ is symmetric, i.e., $A=A^{\T}$,
$C\in \bbR^{ n\times m}$ has full column rank, and
$b\in\bbR^m$. Necessarily $m< n$ but often $m\ll n$.
We are particularly interested in the case where $A$ is large
and sparse and  $b\neq 0$.

CRQopt \eqref{eq:CRQopt} is  also  known
as {\em the constrained eigenvalue problem},
a term coined in \cite{gagv:1989} in 1989. However, it had appeared 
in the literature much earlier than that \cite{golu:1973}. 
In that sense, it is a classical problem.
However, past studies are fragmented with some claims, although often true, 
not rigorously justified
or needed conditions to hold. In this paper, 
our goal is to provide a thorough investigation into this classical problem,
including rigorous justifications of statements previously taken for granted in the literature and
addressing the theoretical subtleties that were not paid attention to. We also present a quantitative convergence
analysis for the Krylov type subspace projection method,  
which we will also call the Lanczos algorithm, for solving 
large scale CRQopt~\eqref{eq:CRQopt}.

%\subsection{Related works}\label{sec-intro-motivation}
\paragraph{Related works.}\label{sec-intro-motivation}
CRQopt \eqref{eq:CRQopt} has found a wide range of applications, such as
ridge regression \cite{drap:1963,gohw:1979},
trust-region subproblem \cite{moso:1983,rewo:1997},
constrained least square problem \cite{gand:1981},
spectral image segmentation \cite{erof:2011,shmj:2000},
transductive learning \cite{joac:2003}, and
community detection \cite{newm:2013}.

The first systematic study of CRQopt~\eqref{eq:CRQopt} perhaps belongs to
Gander, Golub and von Matt \cite{gagv:1989}.
Using the full QR and eigen-decompositions,
 they first reformulated CRQopt~\eqref{eq:CRQopt} as an optimization problem of
finding the minimal Lagrangian multiplier via solving a secular equation 
(in a way that is different from
our secular equation solver in Appendix~\ref{sec:secularEq}).
Alternatively, they also turned
CRQopt~\eqref{eq:CRQopt} into an optimization problem of finding
the smallest real eigenvalue of a quadratic eigenvalue problem (QEP).
However, the equivalence between the QEP optimization and the Lagrangian multiplier
problem was not rigorously justified there.

Numerical algorithms proposed in \cite{gagv:1989} are not suitable 
for large scale CRQopt~\eqref{eq:CRQopt} because they
requires a full eigen-decomposition of $A$ as a dense matrix.
Later in \cite{gozz:2000}, Golub, Zhang and Zha considered large and sparse
CRQopt~\eqref{eq:CRQopt} but only with the homogeneous constraint, i.e, $b=0$.
In this special case, CRQopt~\eqref{eq:CRQopt} is equivalent to computing the smallest
eigenvalue of $A$ restricted to the null space of $C^{\T}$.
An inner-outer iterative Lanczos method was proposed to solve
the homogeneous CRQopt~\eqref{eq:CRQopt}.
In \cite{xuls:2009}, Xu, Li and Schuurmans proposed a projected power method
for solving CRQopt~\eqref{eq:CRQopt}.
The projected power method is an
iterative method only involving matrix-vector products, and thus
it is suitable for large and sparse CRQopt~\eqref{eq:CRQopt}. However, its convergence
is linear at best and often too slow.
In \cite{erof:2011}, Eriksson, Olsson and Kahl
reformulated CRQopt~\eqref{eq:CRQopt}
into an eigenvalue optimization problem 
(see Appendix B for details)
An algorithm based on the line search was used to find the optimal solution.
This algorithm is suitable for CRQopt \eqref{eq:CRQopt} with a large and sparse 
matrix $A$, but
it is too costly because the smallest eigenvalue has to be computed multiple times
during each line search action.

%Let
%$$
%L=\begin{bmatrix}
%A &0\\0& 0
%\end{bmatrix}\in\bbR^{(n+1)\times(n+1)},\,\,
%M=\begin{bmatrix}
%I_n &0\\0& 0
%\end{bmatrix}\in\bbR^{(n+1)\times(n+1)},\,\,
%\widehat{C}=[C^{\T} -\sqrt{n}b]\in\bbR^{m\times(n+1)},
%$$
%and let
%$N\in\bbR^{(n+1)\times(n-m+1)}$ be a matrix whose columns form a basis of null space of $\widehat{C}$.
%Set
%$$
%L_{\widehat{C}}=N^{\T}LN,\,\,
%M_{\widehat{C}}=N^{\T}MN,\,\,
%E_{\widehat{C}}=N^{\T}\begin{bmatrix}
%-\frac 1{n+1}I_n & 0\\0 & 1
%\end{bmatrix}N.
%$$
%It is proved that solving CRQopt \eqref{eq:CRQopt} is equivalent to solving
%$$
%\max_t\lambda_{\min}(L_{\widehat{C}}+tE_{\widehat{C}},M_{\widehat{C}}),
%$$
%where
%$\lambda_{\min}(L_{\widehat{C}}+tE_{\widehat{C}},M_{\widehat{C}})$ is the smallest eigenvalue of the generalized eigenvalue problem $(L_{\widehat{C}}+tE_{\widehat{C}})x=\lambda M_{\widehat{C}}x$,

%demands a huge
%amount of calculation.

%Quantitative convergence analyses for those methods have not been done yet.

%\subsection{Our contributions}\label{sec-intro-contributions}
\paragraph{Contributions.}\label{sec-intro-contributions}
Our study in this paper on CRQopt~\eqref{eq:CRQopt} begins with
the standard approach of Lagrangian multipliers, as was taken in \cite{gagv:1989},
which leads to an optimization problem of minimizing the Lagrangian multiplier of CRQopt,
called LGopt (Section~\ref{ssec:LGopt}), and then
an optimization problem of finding the smallest real eigenvalue of a quadratic eigenvalue problem (QEP),
called QEPmin (Section~\ref{ssec:QEPmin}).
We summarize our major contributions as follows.
\begin{enumerate}
\item Although transforming CRQopt into LGopt and QEPmin is not really new, our formulations
      of LGopt and QEPmin
      set them up onto a natural path for use in Krylov subspace type projection methods that only requires matrix-vector products.
      Therefore, the formulations are suitable for large scale CRQopt.
      We rigorously proved the equivalences among the three problems while they were only loosely argued previously as, e.g., in
      \cite{gagv:1989}. As far as subtle technicalities are concerned, we prove that the leftmost eigenvalue in the complex plane is real, which has a significant implication when it comes to numerical computations.

%\item We show that the smallest real eigenvalue of the QEP is its leftmost eigenvalue in the complex plane.
%, which can be solved by Krylov subspace
%method.

\item We devise a Lanczos algorithm to solve the induced optimization problems: 
      LGopt and QEPmin. This algorithm is made possible, as we argued moments ago, by our different formulations from what in the literature.
      Along the way, we also propose an efficient numerical algorithm for
      the type of secular equations arising from solving each projected LGopt.

% the Lagrange equations and the QEP into the Krylov subspace and solved the projected Lagrange equations by solving a secular equation and solved the projected QEP by linearization. We call the algorithm ``Lanczos algorithms for constrained Rayleigh quotient optimization'' and numerical examples show the correctness of the algorithm.

\item We establish a quantitative convergence analysis for the Lanczos algorithm and
      obtain error bounds on approximations generated by the algorithm. These error bounds are in general  sharp in the worst case
      as  demonstrated by artificially designed numerical examples.
%to show that the error of the
%objective function value and the solution vector can be bounded by the reciprocal of
%a Chebyshev polynomial. The tightness of the bound
%can be shown by numerical examples.

\item We apply our algorithm to the large scale CRQopt from the constrained clustering that
      arises from the standard spectral algorithm with linear constraints to encode 
            prior knowledge labels. During our
      tests, we observed that our algorithm was $2$ to $23$  
            times faster than FAST-GE-2.0 \cite{jixb:2017} for
            constrained image segmentation, depending on given image data.

%into . We provide a way to encode the labels into linear
%constraints and apply our Lanczos algorithm of the CRQopt to solve
%constrained image segmentation problems.
%The encoding is theoretically consistent with the standard spectral
%clustering and numerically suitable for different kinds of images. Our Lanczos algorithm for CRQopt is $2$ to $23$  times faster than FAST-GE-2.0 \cite{jixb:2017},
%depending on the image data.

\end{enumerate}

%\subsection{Organization}\label{sec-intro-organiztion}
\paragraph{Organization.}\label{sec-intro-organiztion}
%The article is organized as follows.
In Section \ref{sec-crq-theory},  we investigate the theoretical aspect of CRQopt \eqref{eq:CRQopt} such as
the feasible set, the existence of a minimizer, and  transforming
CRQopt \eqref{eq:CRQopt} into two equivalent optimization problems with rigorous justifications.
A Krylov subspace projection approach for solving CRQopt \eqref{eq:CRQopt} via its equivalent optimization problems just mentioned are
detailed in Section \ref{sec-crq-alg} and its convergence analysis is given in Section~\ref{sec-crq-conv}.
Numerical examples that demonstrate the sharpness of our error bounds for convergence are presented in Section \ref{sec-crq-ex}.
Section \ref{sec-app} describes an application of our algorithms 
to the constrained image segmentation problem.
Concluding remarks are in Section \ref{sec-concluding}.
There are three appendices. Appendix~\ref{sec:secularEq} explains 
how to solve the secular equation arising from solving the reduced LGopt.
Appendix~\ref{sec:dual} proves the equivalence between CRQopt \eqref{eq:CRQopt} 
and an eigenvalue optimization problem proposed by 
Eriksson, Olsson and Kahl \cite{erof:2011}.
Appendix~\ref{sec-alg-software} documents CRQPACK, a software package
for an implementation of Lanczos algorithm and reproducing 
numerical experiments presented in this paper.

%\subsection{Notation}
\paragraph{Notation.}
Throughout the article, $\bbR$, $\bbR^n$ and
$\bbR^{m\times n}$ are set of real numbers, columns vectors
of dimension $n$, and $m\times n$ matrices, respectively. $\mathbb{C}$, $\mathbb{C}^n$ and  $\mathbb{C}^{m\times n}$
are set of complex numbers, columns vectors of dimension $n$, and $m\times n$ matrices, respectively.
We use MATLAB-like notation $X_{(i:j,k:l)}$ to denote the submatrix of $X$, consisting of the intersections of
rows $i$ to $j$ and columns $k$ to $l$, and when $i : j$ is replaced by $:$, it means all rows, similarly for columns.
For a vector $v\in\mathbb{C}$, $v_{(k)}$ refers the $k$th entry of $v$ and $v_{(i:j)}$
is the subvector of $v$ consisting of the $i$th to $j$th entries inclusive.
The $n \times n$ identity matrix is $I_n$ or simply $I$ if its size is clear from the context, and $e_j$ is the $j$th
column of an identity matrix whose size is determined by the context.
$\diag(c_1,c_2,\dots,c_n)$ is an  $n\times n$ diagonal matrix with diagonal elements $c_1,c_2,\dots,c_n$.
The imaginary unit is $\ti=\sqrt{-1}$.

For $X\in\mathbb{C}^{m\times n}$, $X^{\T}$, $\cR(X)$, $\cN(X)$ denote its transpose, range and null space, respectively.
For a real symmetric matrix $H$,  $\eig(H)$ stands for the set of all eigenvalues of $H$, and $\lambda_{\min}(H)$ and
$\lambda_{\max}(H)$ denote the smallest and largest  eigenvalue of $H$, respectively.
$\|\cdot\|_p$ $(1\le p\le\infty$) is the $\ell_p$-vector or $\ell_p$-operator norm, respectively, depending on the argument.
As a special case, $\|\cdot\|_2$ or $\|\cdot\|$ is either the Euclidean norm of vector or the spectral norm of a matrix.

%This is an example for a chapter, additional chapter can be added in the skeleton-article
%To generate the final document run latex, build and quick build commands on the skeleton-article file not this one.
%This is chapter 3, the default skeleton article expects 2 chapters

\section{Theory}\label{sec-crq-theory}

\subsection{Feasible set and solution existence}\label{sec-crq-problem-exist}
In CRQopt \eqref{eq:CRQopt},  we assumed $\rank(C)=m$. Let
\begin{equation}\label{eq:n0:defn}
n_0=(C^{\T})^{\dag}b,
\end{equation}
i.e., $n_0$ is the unique minimal norm solution
of $C^{\T}v=b$,  where $C^{\dag}$ is the Moore-Penrose
inverse of $C$. Because of the assumption $\rank(C)=m$, we have  \cite{bjor:1996,demm:1997,stsu:1990}
$$
C^{\dag}=(C^{\T}C)^{-1}C^{\T},\,\,(C^{\T})^{\dag} = (C^{\dag})^{\T} = C (C^{\T}C)^{-1}.
$$
The most important orthogonal projection throughout this article is
\begin{equation}\label{eq:theP-dfn}
P=I-CC^{\dag}
\end{equation}
which orthogonally projects any vector onto $\cN(C^{\T})$,
the null space of $C^{\T}$ \cite{stsu:1990}.
Any $v\in\bbR^n$ that satisfies $C^{\T}v=b$ can be orthogonally decomposed as
\begin{equation}\label{eq:v}
v = (I-P)v+Pv=n_0+Pv\in n_0+\cN(C^{\T}).
\end{equation}
Evidently $\|v\|^2=\|n_0\|^2+\|Pv\|^2$, which, together with
the unit length constraint \eqref{eq:CRQopt-1}, lead to
the following immediate conclusions about the solvability of
CRQopt~\eqref{eq:CRQopt}:
\begin{itemize}
\item If $\|n_0\|>1$,
then there is no unit vector $v$ satisfying $C^{\T}v=b$.
This is because for
any $v$ satisfying $C^{\T}v=b$ has norm no smaller than $\|n_0\|$.
Thus CRQopt~\eqref{eq:CRQopt} has no minimizer.

\item If $\|n_0\|=1$, then $v=n_0$ is the only unit vector that
satisfies $C^{\T}v=b$. Thus CRQopt~\eqref{eq:CRQopt} has a
unique minimizer $v = n_0$.
\item If $\|n_0\|<1$, then there are infinitely many feasible
vectors $v$ that satisfy $C^{\T}v=b$.
\end{itemize}
Therefore only the case $\|n_0\|<1$ needs further investigation.
Consequently, throughout the rest of the article,
we will assume $\|n_0\|<1$.

\subsection{Equivalent LGopt}\label{ssec:LGopt}

Using the orthogonal decomposition \eqref{eq:v}, we have
\begin{subequations}\label{eq:decomp-vAv}
\begin{align}
v^{\T}Av& =v^{\T}PAPv+2v^{\T}PAn_0+n_0^{\T}An_0, \label{eq:decomp-vAv:1}\\
v^{\T}v& =\|n_0\|^2+\|Pv\|^2. \label{eq:decomp-vAv:2}
\end{align}
\end{subequations}
Since $n_0^{\T}An_0$ and $\|n_0\|$ are constants,
CRQopt~\eqref{eq:CRQopt} is equivalent
to the following constrained quadratic minimization problem
\begin{subequations}\label{eq:CQopt}
\begin{empheq}[left={\mbox{CQopt:}\quad}\empheqlbrace]{alignat=2}
\min         ~& v^{\T}PAPv+2v^{\T}b_0,\label{eq:CQopt-0}\\
\mbox{s.t.}  ~& \|Pv\|=\gamma, \label{eq:CQopt-1}\\
             ~& v\in n_0+\cN(C^{\T}),\label{eq:CQopt-2}
\end{empheq}
\end{subequations}
where
\begin{equation}\label{eq:b0-gamma:defn}
b_0=PAn_0 \in \cN(C^{\T}),\quad \gamma:=\sqrt{1-\|n_0\|^2}>0.
\end{equation}
Necessarily, $0<\gamma<1$. However, in the rest of our development, unless we refer back to CRQopt~\eqref{eq:CRQopt},
$\gamma<1$ can be removed, i.e., $\gamma$ can be any positive number.

\begin{theorem}\label{thm:CRQopt=CQopt}
$v_{\ast}$ is a minimizer of {\rm CRQopt}~\eqref{eq:CRQopt}
if and only if $v_{\ast}$ is a minimizer of {\rm CQopt}~\eqref{eq:CQopt}.
\end{theorem}

One way to solve CQopt \eqref{eq:CQopt} is the method of the Lagrangian multipliers. It seeks
the stationary points of the Lagrangian function
\begin{equation}\label{eq:Lfun4CQopt}
\scrL(v,\lambda) = v^{\T}PAPv+2v^{\T}b_0-\lambda(v^{\T}Pv-\gamma^2).
\end{equation}
Differentiating $\scrL$ with respect to $v$ and $\lambda$, we get
\begin{subequations} \label{eq:CQopt-LagEq}
\begin{align}
(PA-\lambda I)Pv&=-b_0, \label{eq:CQopt-LagEq-1}\\
\|Pv\|&=\gamma. \label{eq:CQopt-LagEq-2}
\end{align}
\end{subequations}
Let $u=Pv\in\cN(C^{\T})$. Then $u=Pu$ and $v=n_0+u$.
The Lagrangian equations in \eqref{eq:CQopt-LagEq}
are equivalent to the following  equations:
\begin{subequations} \label{eq:CQopt-LagEq'}
\begin{align}
(PAP-\lambda I)u&=-b_0,  \label{eq:CQopt-LagEq'-1}\\
\|u\|&=\gamma,  \label{eq:CQopt-LagEq'-2}\\
u&\in\cN(C^{\T}).  \label{eq:CQopt-LagEq'-3}
\end{align}
\end{subequations}
In fact,
any solution $(\lambda,v)$ of \eqref{eq:CQopt-LagEq} gives rise to a solution
$(\lambda,u)$ with $u=Pv$ of \eqref{eq:CQopt-LagEq'}, and conversely
any solution $(\lambda,u)$ of \eqref{eq:CQopt-LagEq'} leads to
a solution $(\lambda,v)$ with $v=n_0+u$ of \eqref{eq:CQopt-LagEq}.

The system of equations  \eqref{eq:CQopt-LagEq'} has more than one solution pairs
$(\lambda, u)$.
We seek a pair $(\lambda,u)$ among them that minimizes
the objective function of \eqref{eq:CQopt} for  $v\in\bbR^n$. Note that
\begin{alignat}{2}
f(v) :&= v^{\T}PAPv+2v^{\T}b_0  \nonumber\\
      & =  v^{\T}PAPv+2v^{\T}P A n_0  \nonumber\\
      &\stackrel{\text{$u=Pv$}}{=} u^{\T}Au+2u^{\T} A n_0 \nonumber\\
      & \stackrel{\text{$u=Pu$}}{=}  u^{\T}PAPu+2u^{\T} PA n_0 \nonumber\\
      & =  u^{\T}PAPu+2u^{\T} b_0 \nonumber\\
      & = f(u), \label{eq:fun-f:dfn}
\end{alignat}
i.e., $f(v)=f(u)$ for $v\in\bbR^n$ and $u=Pv$.
Therefore minimizing $f(v)$ over $v\in\bbR^n$ is equivalent to 
minimizing $f(u)$ over $u\in\cN(C^{\T})$.
The following lemma compares the value of $f$ at different
solution pairs $(\lambda,u)$ of the system \eqref{eq:CQopt-LagEq'}.
The proof of the lemma is inspired by Gander \cite{gand:1981} 
on solving a least squares problem with a quadratic constraint, 

\begin{lemma}\label{lem:minlambda}
For two solution pairs $(\lambda_i,u_i)$ for $i=1,2$  of the Lagrangian system of equations \eqref{eq:CQopt-LagEq'},
$\lambda_1<\lambda_2$ if and only if  $f(u_1)< f(u_2)$.
\end{lemma}
\begin{proof}
The proof relies on the following three facts:
\begin{enumerate}
  \item For any solution pair $(\lambda,u)$ of \eqref{eq:CQopt-LagEq'}, we have
      \begin{equation}\label{eq:u->lambda:LGopt}
      \lambda u=PAPu+b_0
      \quad\Rightarrow\quad
\lambda=\frac{1}{u^{\T}u}{u^{\T}(PAPu+b_0)}
=\frac{1}{\gamma^2}{u^{\T}(PAPu+b_0)}.
      \end{equation}

  \item Given $(\lambda_i,u_i)$ for $i=1,2$, satisfying \eqref{eq:CQopt-LagEq'},
      we have
      \begin{alignat*}{2}
f(u_1)&=u_1^{\T}PAPu_1+2u^{\T}_1 b_0 \\
&\stackrel{\text{ \eqref{eq:CQopt-LagEq'-1}}}{=}
-b^{\T}_0 u_1+\lambda_1u_1^{\T}u_1+2u^{\T}_1 b_0 \\
&\stackrel{\text{ \eqref{eq:CQopt-LagEq'-2}}}{=}u^{\T}_1 b_0 +\lambda_1\gamma^2 \\
&\stackrel{\text{ \eqref{eq:CQopt-LagEq'-1}}}{=}-u_2^{\T}(PAP-\lambda_2I)u_1+\lambda_1\gamma^2.
\end{alignat*}
      Similarly, we have $f(u_2)=-u_1^{\T}(PAP-\lambda_1I)u_2+\lambda_2\gamma^2$. Therefore
      \begin{equation}\label{eq:difff}
      f(u_1)-f(u_2)=(\lambda_1-\lambda_2)(\gamma^2-u_1^{\T}u_2).
      \end{equation}
  \item For $u_i$ of norm $\gamma$,
by the Cauchy-Schwartz inequality, $u_1^{\T}u_2\le\|u_1\|\,\|u_2\|=\gamma^2$, and
$u_1^{\T}u_2=\|u_1\|\,\|u_2\|=\gamma^2$ if and only if $u_1=u_2$.
Hence if $u_1\ne u_2$, then $\gamma^2-u_1^{\T}u_2>0$.
\end{enumerate}
Now we are ready to prove the claim of the lemma.
If $\lambda_1<\lambda_2$, then $u_1\neq u_2$ otherwise \eqref{eq:u->lambda:LGopt}
would imply $\lambda_1=\lambda_2$, and thus $f(u_1)<f(u_2)$ by \eqref{eq:difff}. On the other hand,
if $f(u_1)<f(u_2)$, then $\gamma^2-u_1^{\T}u_2>0$ because $\gamma^2-u_1^{\T}u_2\ge 0$ always and it cannot be $0$
by \eqref{eq:difff},
and thus $\lambda_1-\lambda_2<0$ again by \eqref{eq:difff}.
\end{proof}

As a consequence of Lemma \ref{lem:minlambda},  we find that solving CQopt \eqref{eq:CQopt} is
equivalent  to
solving the smallest Lagrangian multiplier $\lambda$ of \eqref{eq:Lfun4CQopt}, i.e., those $\lambda$ that
satisfy \eqref{eq:CQopt-LagEq'}. Specifically, solving CQopt \eqref{eq:CQopt} is
equivalent  to solving
the  following  Lagrangian minimization problem:
% and the fact that $f(u)=f(v)$, $f(v)$ is strictly increasing by the Lagrangian multiplier of \eqref{eq:CQopt-LagEq'}. Therefore,
\begin{subequations}\label{eq:LGopt}
\begin{empheq}[left={\mbox{LGopt:}\quad}\empheqlbrace]{alignat=2}
\min  ~& \lambda  \label{eq:LGopt-0}  \\
\mbox{s.t.}  ~& (PAP-\lambda I)u=-b_0,\label{eq:LGopt-1}\\
~&\|u\|=\gamma,\label{eq:LGopt-2}\\
~&u\in\cN(C^{\T}).\label{eq:LGopt-3}
\end{empheq}
\end{subequations}

\begin{theorem}\label{thm:CQopt=LGopt}
If $v_{\ast}$ is a minimizer of {\rm CQopt}~\eqref{eq:CQopt}, then
$(\lambda_*,u_*)$ with
$$
u_*=Pv_*,\,\,\lambda_*=\frac{1}{\gamma^2} {u_*^{\T}(PAPu_*+b_0)}
$$
is a minimizer of {\rm LGopt}~\eqref{eq:LGopt}. Conversely
if $(\lambda_*,u_*)$ is a minimizer of {\rm LGopt}~\eqref{eq:LGopt}, then $v_\ast=n_0+u_{\ast}$ is a minimizer of {\rm CQopt}~\eqref{eq:CQopt}.
\end{theorem}

The case $b_0= P A n_0 = 0$,
which includes but is not equivalent to
the homogeneous CRQopt~\eqref{eq:CRQopt} (i.e., $b = 0$)
\cite{golu:1973,gozz:2000}, can be dealt with as follows.  Suppose $b_0=0$ and
let $\theta_1$ be the smallest eigenvalue of $PAP$. Keep in mind that
$PAP$ always has an eigenvalue $0$ with multiplicity $m$ associated with the subspace $\cN(C^{\T})^{\bot}=\cR(C)$, the column space of $C$.
There are the following two subcases:
\begin{itemize}
\item {\bf Subcase $\theta_1 \neq 0$:}
Then\footnote {This cannot happen if $A$ is positive semidefinite.} $\theta_1<0$.
      Let $z_1$ be a corresponding eigenvector of $PAP$. Then $z_1=PAPz_1/\theta_1\in\cN(C^{\T})$.
      So $(\theta_1,z_1)$ is a minimizer of {\rm LGopt}~\eqref{eq:LGopt} and therefore $z_1$ is a minimizer of CQopt~\eqref{eq:CQopt},
      which in turn implies that $v_{\ast} = n_0+ \gamma z_1/\|z_1\|$ is a minimizer of CRQopt \eqref{eq:CRQopt}.

\item {\bf Subcase $\theta_1 = 0$:}  If there exists a corresponding eigenvector $z_1\in\cN(C^{\T})$, i.e., $Pz_1 \neq 0$,
      then $(\theta_1,Pz_1)$ is a minimizer of {\rm LGopt}~\eqref{eq:LGopt} and therefore $Pz_1$ is a minimizer of CQopt~\eqref{eq:CQopt},
      which in turn implies that $v_{\ast} = n_0+ \gamma Pz_1/\|Pz_1\|$ is a minimizer of CRQopt \eqref{eq:CRQopt}.
      Otherwise there exists no corresponding eigenvector $z_1$ such that $Pz_1 \neq 0$.
      Let $\theta_2$ be the second smallest eigenvalue of $PAP$, which is nonzero, and $z_2$ a corresponding  eigenvector.
      Then $z_2=PAPz_2/\theta_2\in\cN(C^{\T})$, and
      $(\theta_2,z_2)$ is a minimizer of {\rm LGopt}~\eqref{eq:LGopt} and therefore $z_2$ is a minimizer of CQopt~\eqref{eq:CQopt},
      which in turn implies that $v_{\ast} = n_0+ \gamma z_2/\|z_2\|$ is a minimizer of CRQopt \eqref{eq:CRQopt}.
\end{itemize}
In view of such a quick resolution for the case $b_0=0$, in
the rest of this article, we will assume
\begin{equation}\label{eq:b0NOT0}
b_0=PAn_0\neq 0.
\end{equation}

\subsection{Equivalent QEPmin}\label{ssec:QEPmin}
Let $(\lambda,u)$ be a feasible pair  of LGopt \eqref{eq:LGopt}
and $\lambda\not\in\eig(PAP)$.
We can write $u=-(PAP-\lambda I)^{-1}b_0$, and then
\begin{equation}\label{eq:gamma-sq}
\gamma^2=u^{\T}u=b_0^{\T}(PAP-\lambda I)^{-2}b_0 = b_0^{\T}z,
\end{equation}
where $z=(PAP-\lambda I)^{-2}b_0$, or equivalently, $(PAP-\lambda I)^{2}z=b_0$.
%Since $\|n_0\|<1$, we have $\gamma>0$.
Therefore $b_0^{\T}z/\gamma^2=1$ by \eqref{eq:gamma-sq}, and thus
the pair $(\lambda, z)$ satisfies the
quadratic eigenvalue problem (QEP):
\begin{equation} \label{eq:QEP-deriv}
(PAP-\lambda I)^{2}z
=b_0
= b_0\cdot 1
= b_0 \left(b_0^{\T}z/\gamma^2\right)=\frac 1{\gamma^2}b_0b_0^{\T}z.
\end{equation}
We claim that any $z$ satisfying \eqref{eq:QEP-deriv} is
in $\cN(C^{\T})$.
To see this, we
expand $(PAP-\lambda I)^{2}z$
and extract $\lambda^2 z$ from $(PAP-\lambda I)^{2}z =b_0$ to get
\[
z=\frac{1}{\lambda^2}\left[-(PAP)^2z+2\lambda\cdot PAP z+b_0\right] \in\cN(C^{\T}),
\]
where we have used the assumption
$\lambda\not\in\eig(PAP)$ to conclude $\lambda\ne 0$,
and $b_0=PAn_0 \in \cN(C^{\T})$.
Therefore we have shown that
under the assumption that LGopt \eqref{eq:LGopt}
has no feasible pair $(\lambda,u)$ with $\lambda\in\eig(PAP)$,
any feasible pair $(\lambda,u)$ of LGopt \eqref{eq:LGopt}
satisfies QEP \eqref{eq:QEP-deriv} with $z\in\cN(C^{\T})$.

Next, we prove that any pair $(\lambda,z)$ satisfying
\begin{equation}\label{eq:QEPmin-(lambda,z)}
0\ne z\in\cN(C^{\T}),\,\,\lambda\not\in\eig(PAP)\,\,
\mbox{and QEP \eqref{eq:QEP-deriv}},
\end{equation}
leads to a feasible pair of the Lagrange equations \eqref{eq:LGopt}.
First we note that $b_0^{\T}z\ne 0$; otherwise we would have
$(PAP-\lambda I)^{2}z=0$ by \eqref{eq:QEP-deriv},
implying $z=0$ since $\lambda\not\in\eig(PAP)$, a contradiction.
Let $(\lambda,z)$ be a scalar-vector pair that satisfying \eqref{eq:QEPmin-(lambda,z)}.
Define $u:=-(PAP-\lambda I)^{-1}b_0$. Then $(PAP-\lambda I)u=-b_0$, i.e.,
\eqref{eq:LGopt-1} holds, and also
$$
\lambda u=PAPu+b_0
\quad\Rightarrow\quad
u=\frac 1{\lambda}(PAPu+b_0)\in\cN(C^{\T}),
$$
i.e., \eqref{eq:LGopt-3} holds. Without loss of generality, we may scale $z$ such that $b_0^{\T}z=\gamma^2$. It follows from \eqref{eq:QEP-deriv} that
$$
(PAP-\lambda I)^2z=b_0
\quad\Rightarrow\quad
z=(PAP-\lambda I)^{-2}b_0,
$$
implying
$$
1=\frac 1{\gamma^2}b_0^{\T}z=\frac 1{\gamma^2}b_0^{\T}(PAP-\lambda I)^{-2}b_0=\frac 1{\gamma^2}u^{\T}u
\quad\Rightarrow\quad
\|u\|=\gamma,
$$
i.e., \eqref{eq:LGopt-2} holds. Lemma~\ref{lm:LGopt2QEPmin-0} summarizes what we have just proved.

\begin{lemma}\label{lm:LGopt2QEPmin-0}
Suppose the constraints of {\rm LGopt~\eqref{eq:LGopt}} has no
feasible pair $(\lambda,u)$
with $\lambda\in\eig(PAP)$, and suppose that {\rm QEP}~\eqref{eq:QEP-deriv}
has no solution pair $(\lambda,z)$
with $0\ne z\in\cN(C^{\T})$ and {$\lambda\in\eig(PAP)$}.
Then any pair $(\lambda,u)$ satisfying the constraints of
{\rm LGopt}~\eqref{eq:LGopt}
gives rise to a pair $(\lambda,z)$ with $z=(PAP-\lambda I)^{-2}b_0$
that satisfies {\rm QEP}~\eqref{eq:QEP-deriv}.
Conversely, any pair $(\lambda,z)$ with $z\ne 0$ satisfying
{\rm QEP}~\eqref{eq:QEP-deriv} leads to a
pair $(\lambda,u)$ with $u:=-(PAP-\lambda I)^{-1}b_0$ that
satisfies the constraints of {\rm LGopt}~\eqref{eq:LGopt}.
\end{lemma}

%Together with what we just proved before this,
%the constraints of LGopt \eqref{eq:LGopt} on pair $(\lambda,u)$ are equivalent to \eqref{eq:QEP-deriv}
%on pair $(\lambda,z)$
%under the conditions given in \eqref{eq:QEPmin-(lambda,z)}.

As a corollary of Lemma~\ref{lm:LGopt2QEPmin-0},
we conclude that LGopt \eqref{eq:LGopt} is equivalent to
\begin{subequations}\label{eq:QEPmin}
\begin{empheq}[left={\mbox{QEPmin:}\quad}\empheqlbrace]{alignat=2}
\min ~& \lambda \label{eq:QEPmin-0}\\
\mbox{s.t.} %& $\lambda \text{ real}$, \\
~& (PAP-\lambda I)^{2}z=\gamma^{-2}b_0b_0^{\T}z,\label{eq:QEPmin-1}\\
~& \lambda\in\bbR, 0\ne z\in\cN(C^{\T}),\label{eq:QEPmin-2}
\end{empheq}
\end{subequations}
under the assumptions of Lemma~\ref{lm:LGopt2QEPmin-0}.
Soon  we  show that LGopt \eqref{eq:LGopt} and
QEPmin \eqref{eq:QEPmin} are still equivalent even without the assumptions.

%\marginpar{\tiny add a theorem that LGopt \eqref{eq:LGopt} and QEPmin \eqref{eq:QEPmin} are equivalent?}
%we just used to go from LGopt \eqref{eq:LGopt} to QEPmin \eqref{eq:QEPmin}.

We name the minimization problem \eqref{eq:QEPmin} QEPmin because
the constraint~\eqref{eq:QEPmin-1} is a
quadratic eigenvalue problem (QEP).
Although this QEP generally may have complex eigenvalues $\lambda$,
the ``$\min$'' in \eqref{eq:QEPmin-0}
implicitly restricts the consideration only to the real eigenvalues $\lambda$
of QEP \eqref{eq:QEPmin-1} in the context of QEPmin \eqref{eq:QEPmin}.
In this sense, there is no need to specify
$\lambda\in\bbR$ in \eqref{eq:QEPmin-2}, but we are doing it anyway
to emphasize the implication. This comment applies to two other
minimization problems pQEPmin \eqref{eq:pQEPmin} and rQEPmin \eqref{eq:rQEPmin} later that involve
a QEP as a constraint as well.

In the rest of this section, we prove the equivalence between LGopt \eqref{eq:LGopt}
and QEPmin \eqref{eq:QEPmin} without the assumptions of Lemma~\ref{lm:LGopt2QEPmin-0}.
The key idea is is to remove the null space
conditions $u,z\in\mathcal{N}(C^{\T})$ by projecting equations \eqref{eq:LGopt-1}, \eqref{eq:LGopt-2}
in LGopt and \eqref{eq:QEPmin-1} in QEPmin onto an appropriate subspace.

%The first step is to remove the null space
%conditions $u,z\in\mathcal{N}(C^{\T})$ by introducing two orthogonal matrices
%and projecting equations \eqref{eq:LGopt-1}, \eqref{eq:LGopt-2}
%in LGopt and \eqref{eq:QEPmin-1} in QEPmin onto the range of an orthogonal matrix.

\subsection{pLGopt}
Let $S=[S_1,\ S_2]\in\bbR^{n\times n}$ be an orthogonal matrix with
\begin{equation}\label{eq:s1s2}
\cR(S_1)=\cN(C^{\T}),\,\,\cR(S_2)=\cN(C^{\T})^{\perp}.
\end{equation}
Since $\rank(C)=m$, we know
$S_1\in\bbR^{n\times (n-m)}$ and $S_2\in\bbR^{n\times m}$.
It can be verified that the projection matrix
$P=I-CC^{\dag}$ in \eqref{eq:theP-dfn} can be written as
\begin{equation}\label{eq:theP:properties}
P=S_1S_1^{\T}=I-S_2S_2^{\T},
\end{equation}
and we have
\begin{equation}
PS_1=S_1,\,\,
PS_2=0.
\end{equation}

Set
\begin{equation}\label{eq:def:gH}
g_0=S_1^{\T}b_0,\,\, H=S_1^{\T}PAPS_1=S_1^{\T}AS_1\in\bbR^{(n-m)\times (n-m)},
\end{equation}
we have
\begin{subequations}\label{eq:decomp:PAP}
\begin{align}
S^{\T}PAPS&=\begin{bmatrix}
           S_1^{\T}PAPS_1& S_1^{\T}PAPS_2\\
           S_2^{\T}PAPS_1& S_2^{\T}PAPS_2
        \end{bmatrix}
       =\kbordermatrix{
         &\sss n-m & \sss m \\
       \sss n-m &    H & 0 \\
       \sss m   &  0 & 0        }, \label{eq:decomp:PAP-1}\\
S^{\T}b_0&=\begin{bmatrix}
           S_1^{\T}b_0\\
           S_2^{\T}b_0
        \end{bmatrix}=\kbordermatrix{
         & \\
       \sss n-m &    g_0 \\
       \sss m  &  0         }. \label{eq:decomp:PAP-2}
\end{align}
\end{subequations}
Immediately from the decomposition \eqref{eq:decomp:PAP-1}, we conclude the following lemma:

\begin{lemma}\label{lm:eig(PAP)}
The  eigenvalues of $PAP$ consist of those of $H$ and
$0$ with multiplicities $m$, i.e.,
$\eig(PAP)=\eig(H)\cup\{0,0,\ldots,0\}$. If $0\ne\lambda\in\eig(PAP)$,
then $\lambda\in\eig(H)$ and its associated eigenvector must be in $\cN(C^{\T})$.
The matrix $PAP$ has more than $m$ eigenvalues $0$
if and only if $H$ is singular. For each
eigenvalue $0$ of $PAP$ coming from $\eig(H)$, there is
an eigenvector $z$ of $PAP$ such that
$Pz\ne 0$ (in fact, $Pz$ is an eigenvector
for that particular eigenvalue $0$ as well).
\end{lemma}

To explicitly eliminate the constraint $u\in\cN(C^{\T})$ in LGopt \eqref{eq:LGopt}, we project LGopt \eqref{eq:LGopt}
onto  $\mathcal{R}(S_1)$ and introduce the following
projected minimization problem
\begin{subequations}\label{eq:pLGopt}
\begin{empheq}[left={\mbox{pLGopt:}\quad}\empheqlbrace]{alignat=2}
\min        ~& \lambda \label{eq:pLGopt-0} \\
\mbox{s.t.} ~&(H-\lambda I)y=-g_0,\label{eq:pLGopt-1}\\
            ~&\|y\|=\gamma.\label{eq:pLGopt-2}
\end{empheq}
\end{subequations}
The next theorem establishes the equivalence between LGopt~\eqref{eq:LGopt}
and pLGopt \eqref{eq:pLGopt}.

\begin{theorem}\label{thm:LGopt=pLGopt}
The pair $(\lambda_{\ast},y_{\ast})$ is a minimizer of {\rm pLGopt} \eqref{eq:pLGopt}
if and only if $(\lambda_{\ast},u_{\ast})$ with $u_{\ast}=S_1y_{\ast}$ is a  minimizer of
{\rm LGopt}~\eqref{eq:LGopt}.
\end{theorem}
\begin{proof}
We begin by showing the equivalence between
the constraints of LGopt \eqref{eq:LGopt} and those of pLGopt \eqref{eq:pLGopt}.
Note that any $0\ne u\in \cN(C^{\T})$ can be expressed by $u=S_1y$ for some $0\ne y\in\bbR^{n-m}$ and vice versa.
Making use of \eqref{eq:decomp:PAP}, we have
\begin{align}
S^{\T}[(PAP-\lambda I)u+b_0]
  &=S^{\T}(PAP-\lambda I)SS^{\T}u+S^{\T}b_0 \nonumber\\
%  &=[S_1\ S_2]^{\T}[(PAP-\lambda I)[S_1\ S_2][S_1\ S_2]^{\T}u+b_0] \nonumber\\
%  &=\begin{bmatrix}
%       S_1^{\T}(PAP-\lambda I)S_1& S_1^{\T}(PAP-\lambda I)S_2\\
%       S_2^{\T}(PAP-\lambda I)S_1& S_2^{\T}(PAP-\lambda I)S_2
%     \end{bmatrix}\begin{bmatrix}
%                     S_1^{\T}u\\
%                     S_2^{\T}u
%                   \end{bmatrix}+\begin{bmatrix}
%                                    S_1^{\T}b_0\\
%                                    S_2^{\T}b_0
%                                 \end{bmatrix} \nonumber\\
  &=\begin{bmatrix}
       H-\lambda I & 0\\
       0&-\lambda I
    \end{bmatrix}\begin{bmatrix}
                   y\\
                   0
                 \end{bmatrix}+\begin{bmatrix}
                                   g_0\\
                                   0
                               \end{bmatrix}, \label{eq:2lag}
\end{align}
and
\begin{equation}\label{eq:normeq}
u^{\T}u=y^{\T}S_1^{\T}S_1y=y^{\T}y.
\end{equation}
Now if $(\lambda,u)$ satisfies the constraints of LGopt \eqref{eq:LGopt},
then $S^{\T}[(PAP-\lambda I)u+b_0]=0$ because of \eqref{eq:LGopt-1},
$u=S_1y$ for some $y$ because of \eqref{eq:LGopt-3}, and $\|y\|=\gamma$ because of
\eqref{eq:LGopt-2} and \eqref{eq:normeq}. It follows from \eqref{eq:2lag} that $(H-\lambda I)y+g_0=0$.
Thus $(\lambda,y)$ satisfies the constraints of pLGopt~\eqref{eq:pLGopt}.

On the other hand, suppose $(\lambda,y)$ satisfies the constraints of pLGopt \eqref{eq:pLGopt}. Let $u=S_1y\in\cN(C^{\T})$.
Both \eqref{eq:2lag} and \eqref{eq:normeq} remain valid. Then
$S^{\T}[(PAP-\lambda I)u+b_0]=0$ which implies $(PAP-\lambda I)u+b_0=0$ because $S^{\T}$ is an orthogonal matrix.
Also $\|u\|=\gamma$ by \eqref{eq:normeq}. This completes the proof of that
$(\lambda,u)$ satisfies the constraints of LGopt \eqref{eq:LGopt}.

Therefore, LGopt \eqref{eq:LGopt} and pLGopt \eqref{eq:pLGopt} have the same optimal value $\lambda_{\ast}$.
More than that,  if $(\lambda_{\ast},u_{\ast})$ is a minimizer of LGopt \eqref{eq:LGopt}, then there exists
$y_{\ast}$ such that $u_{\ast}=S_1y_{\ast}$ and that
$(\lambda_{\ast},y_{\ast})$ is a minimizer of pLGopt \eqref{eq:pLGopt}, and vice versa.
\end{proof}

We note that for a modest-sized CRQopt~\eqref{eq:CRQopt}, say $n$ up to $2000$, we may as well perform the reduction to form
pLGopt~\eqref{eq:pLGopt} explicitly. Due to its modest size, pLGopt~\eqref{eq:pLGopt} can be solved as a dense matrix computational problem. The detail is buried later in the proof of Lemma~\ref{lm:pLGopt}.

\subsection{pQEPmin}\label{ssec:pQEPmin}
For the same purpose as we projected the Lagrange equations,
we introduce the following projected minimization problem
as the counterpart of QEPmin \eqref{eq:QEPmin}:
\begin{subequations}\label{eq:pQEPmin}
\begin{empheq}[left={\mbox{pQEPmin:}\quad}\empheqlbrace]{alignat=2}
\min        ~& \lambda  \label{eq:pQEPmin-0}\\
\mbox{s.t.} ~&(H-\lambda I)^2w=\gamma^{-2}g_0g_0^{\T}w,  \label{eq:pQEPmin-1} \\
            ~&\lambda\in\bbR, w\ne 0. \label{eq:pQEPmin-2}
\end{empheq}
\end{subequations}
The equation in \eqref{eq:pQEPmin-1} has an appearance of a QEP.
As stated, the optimal value of pQEPmin \eqref{eq:pQEPmin} is the smallest real eigenvalue of QEP \eqref{eq:pQEPmin-1}.
The next theorem establishes the equivalence between QEPmin~\eqref{eq:QEPmin} and pQEPmin \eqref{eq:pQEPmin}.
%In its generality, as a QEP, it may and can have
%complex eigenvalues $\lambda$. But as stated, the ``$\min$'' in \eqref{eq:pQEPmin-0} is taken over all
%real eigenvalues $\lambda$ of QEP \eqref{eq:pQEPmin-1}.

\begin{theorem}\label{thm:QEPmin=pQEPmin}
The pair $(\lambda_{\ast},w_{\ast})$ is a minimizer of {\rm pQEPmin} \eqref{eq:pQEPmin}
if and only if $(\lambda_{\ast},z_{\ast})$ with $z_{\ast}=S_1w_{\ast}$ is a minimizer of
{\rm QEPmin} \eqref{eq:QEPmin}.
\end{theorem}
\begin{proof}
Similarly, we begin by showing the equivalence between
the constraints of  QEPmin \eqref{eq:QEPmin} and those of pQEPmin \eqref{eq:pQEPmin}.
Keeping \eqref{eq:decomp:PAP} in mind,  we have for any $z=S_1w$
\begin{align}
S^{\T}&\left[(PAP-\lambda I)^2z-\gamma^{-2}\,{b_0b_0^{\T}}z\right] \nonumber\\
  &=S^{\T}(PAP-\lambda I)SS^{\T}(PAP-\lambda I)SS^{\T}z-\gamma^{-2}\,S^{\T}b_0b_0^{\T}SS^{\T}z \nonumber\\
&=\begin{bmatrix}
    (H-\lambda I)^2 & 0\\
    0&\lambda^2 I
  \end{bmatrix}\begin{bmatrix}
                  w\\
                  0
               \end{bmatrix}-\begin{bmatrix}
                                \gamma^{-2}\,{g_0 g_0^{\T}} & 0\\
                                0&0
                              \end{bmatrix}\begin{bmatrix}
                                              w\\
                                              0
                                            \end{bmatrix}. \label{eq:2qepeq}
\end{align}
Now if $(\lambda,z)$ satisfies the constraints of  QEPmin \eqref{eq:QEPmin}, then $0\ne z\in\cN(C^{\T})$
and thus $z=S_1w$ for some $0\ne w\in\bbR^{n-m}$. Therefore, by
\eqref{eq:2qepeq}, $(\lambda,w)$ satisfies  \eqref{eq:pQEPmin-1}.

On the other hand, suppose
$(\lambda,w)$ satisfies \eqref{eq:pQEPmin-1} and \eqref{eq:pQEPmin-2}. Let $z=S_1w\in\cN(C^{\T})$. 
Then $z\ne 0$ and by \eqref{eq:2qepeq},  $S^{\T}[(PAP-\lambda I)^2z-\gamma^{-2}{b_0b_0^{\T}}z]=0$.
Since $S^{\T}$ is orthogonal, we get \eqref{eq:QEPmin-1}. This proves that
$(\lambda,z)$ satisfies the constraints of  QEPmin \eqref{eq:QEPmin}.

Therefore, QEPmin \eqref{eq:QEPmin} and  pQEPmin \eqref{eq:pQEPmin} have the same optimal value $\lambda_{\ast}$.
More than that,  if $(\lambda_{\ast},z_{\ast})$ is a minimizer of QEPmin \eqref{eq:QEPmin}, then there exists
$w_{\ast}\ne 0$ such that $z_{\ast}=S_1w_{\ast}$ and that
$(\lambda_{\ast},w_{\ast})$ is a minimizer of pQEPmin \eqref{eq:pQEPmin}, and vice versa.
\end{proof}

%\subsubsection{pLGopt  and pQEPmin  are equivalent}\label{sec:pLGopt=pQEPmin}
%After the introduction of pLGopt \eqref{eq:pLGopt}  and
%pQEPmin \eqref{eq:pQEPmin}, we prove the equivalence between the two problems.
%Although, in leading to pLGopt \eqref{eq:pLGopt} and pQEPmin \eqref{eq:pQEPmin},
%the matrix $H$ and the vector $g_0$ are derived from reducing $A$, $C$,
%and $b$ in CRQopt \eqref{eq:CRQopt},
%the developments here does not require that.
\subsection{pLGopt  and pQEPmin  are equivalent}\label{sec:pLGopt=pQEPmin}
Although, in leading to pLGopt \eqref{eq:pLGopt} and pQEPmin \eqref{eq:pQEPmin}, the matrix $H$ and the vector $g_0$
are derived from reducing $A$, $C$, and $b$ in the original CRQopt \eqref{eq:CRQopt}, the developments
in this section does not require that.
Given this,  in the rest of this section, we consider general
pLGopt \eqref{eq:pLGopt} and pQEPmin \eqref{eq:pQEPmin} with\footnote{Unlike before, there is no need
          to assume $\gamma<1$. In addition, the size of square matrix $H$
          and vector $g_0$ can be arbitrary, not necessarily equal to $n-m$.}
$$
H\in\bbR^{\ell\times\ell},\,\,
H^{\T}=H,\,\,
0\ne g_0\in\bbR^{\ell},\,\,\mbox{and}\,\,
\gamma>0.
$$
To set up the stage for the rest of this subsection,
we let $H=Y\Theta Y^{\T}$ be the eigen-decomposition of $H$:
\begin{equation}\label{eq:eig-decomp:H}
H=Y\Theta Y^{\T}
\,\,\mbox{with}\,\,
\Theta=\diag(\theta_1,\theta_2,\ldots,\theta_{\ell}),\,\,
Y=[y_1,y_2,\ldots,y_{\ell}],\,\,
Y^{\T}Y=I_{\ell}.
\end{equation}
Without loss of generality, we arrange $\theta_i$ in the ascending order, i.e.,
$$
\theta_1=\theta_2=\cdots=\theta_d<\theta_{d+1}\le\cdots\le\theta_{\ell},
$$
so $\lambda_{\min}(H)=\theta_1$. Define the secular function
\begin{equation}\label{eq:sec-fun:H}
\chi(\lambda):=g_0^{\T}(H-\lambda I)^{-2}g_0-\gamma^2
    =(Y^{\T}g_0)^{\T}(\Theta-\lambda I)^{-2}(Y^{\T}g_0)-\gamma^2
    =\sum_{i=1}^{l}\frac{\xi_i^2}{(\lambda-\theta_i)^2}-\gamma^2,
\end{equation}
where  $\xi_i=g_0^{\T}y_i$ for $i=1,2,\cdots,n$, and let
\begin{equation}\label{eq:j04ci}
j_0=\min\{i\,:\, \xi_i\ne 0\}.
\end{equation}

\begin{lemma} \label{lm:pLGopt}
Let $(\lambda_{\ast},y_{\ast})$ be a minimizer of {\rm pLGopt} \eqref{eq:pLGopt}. The following statements hold.
\begin{enumerate}[{\rm (a)}]
  \item $\lambda_{\ast} \leq \lambda_{\min}(H)$.
  \item $\lambda_{\ast} = \lambda_{\min}(H)$ if and only if
$$
g_0\,\bot\,\,\cU\,\,\mbox{and}\,\,
\|(H-\lambda_{\min}(H) I)^{\dagger} g_0\|_2\le \gamma,
$$
where $\cU$ is the eigenspace of $H$ associated with its eigenvalue $\lambda_{\min}(H)$.
  \item If $g_0\notperp\cU$, then $\lambda_{\ast} < \lambda_{\min}(H)$ and $\lambda_{\ast}$ is the smallest root of the secular function $\chi(\lambda)$,
        and $y_{\ast}=-(H-\lambda_{\ast} I)^{-1}g_0$.
\end{enumerate}
\end{lemma}

\begin{proof}
The secular function $\chi(\lambda)$ in \eqref{eq:sec-fun:H} is continuous on $(-\infty,\theta_1)$
and $\lim\limits_{\lambda\rightarrow-\infty}\chi(\lambda)=-\gamma^2<0$.
Since
$$
\chi'(\lambda)=-2\sum_{i=1}^{\ell}\frac{\xi_i^2}{(\lambda-\theta_i)^3}>0\quad\mbox{for  $\lambda<\theta_1$},
$$
$\chi(\lambda)$ is strictly increasing in $(-\infty,\theta_1)$.
We have the following situations to deal with:
\begin{enumerate}[(1)]
\item If $g_0\notperp\cU$, then $\sum_{i=1}^d \xi_i^2>0$, i.e., $j_0\le d$, then $\lim\limits_{\lambda\rightarrow\theta_1^-}\chi(\lambda)=+\infty>0$.
   There exists a unique $\lambda_{\ast} \in(-\infty,\theta_1)$ such that $\chi(\lambda_{\ast})=0$.
   Let $y_{\ast}=-(H-\lambda_{\ast} I)^{-1}g_0$. We have
   $$
   (H-\lambda_{\ast} I)y_{\ast}=-g_0,\quad
   y_{\ast}^{\T}y_{\ast}=g_0^{\T}(H-\lambda_{\ast} I)^{-2}g_0=\chi(\lambda_{\ast})+\gamma^2=\gamma^2.
   $$
  Therefore, $(\lambda_{\ast},y_{\ast})$  satisfies
                 the constraints of pLGopt \eqref{eq:pLGopt}.

\item Suppose that $g_0\,\bot\,\cU$, then $\sum_{i=1}^d \xi_i^2=0$, i.e., $j_0>d$. Let
      $$
      w=-(H-\theta_1 I)^{\dag} g_0=-\sum_{i=d+1}^{\ell}\frac{\xi_i}{\theta_i-\theta_1}y_i.
      $$
      Then $(H-\theta_1 I)w=-g_0$ and $\lim\limits_{\lambda\rightarrow\theta_1^-}\chi(\lambda)=w^{\T}w-\gamma^2$. There
      are three subcases to consider.
      \begin{enumerate}[(i)]
      \item If $\|w\|>\gamma$, then there exists a unique $\lambda_{\ast} \in(-\infty,\theta_1)$ such that
                 $\chi(\lambda_{\ast})=0$. Moreover $(\lambda_{\ast},y_{\ast})$ with $y_{\ast}=-(H-\lambda_{\ast} I)^{-1}g_0$ satisfies
                 the constraints of pLGopt \eqref{eq:pLGopt}.
      \item If $\|w\|=\gamma$, then $(\lambda_{\ast},y_{\ast})$ with $\lambda_{\ast}=\theta_1$ and $y_{\ast}=w$ satisfies
                 the constraints of pLGopt \eqref{eq:pLGopt}.
      \item If $\|w\|<\gamma$, then  $(\lambda_{\ast},y_{\ast})$ with $\lambda_{\ast}=\theta_1$ and
                 $y_{\ast}=w+\sqrt{\gamma^2-\|w\|^2}\,y_1$ satisfies
                 the constraints of pLGopt \eqref{eq:pLGopt}.
      \end{enumerate}
\end{enumerate}
So far we have proved that    $(\lambda_{\ast},y_{\ast})$ satisfies the constraints of pLGopt \eqref{eq:pLGopt} for all situations. Now we prove $\lambda_\ast$ is the smallest Lagrange multiplier of pLGopt \eqref{eq:pLGopt}. Suppose there exists $\widehat{\lambda}<\lambda_\ast$ such that $(\widehat{\lambda},\widehat{y})$ satisfies the constraints of pLGopt \eqref{eq:pLGopt}, then $\widehat{\lambda}<\lambda_\ast\le\theta_1$, so $\widehat{\lambda}\notin\text{eig}(H)$. Therefore, in order to make $(\widehat{\lambda},\widehat{y})$ satisfies \eqref{eq:pLGopt-1}, we have  $\widehat{y}=-(H-\widehat{\lambda} I)^{-1}g_0$.  Note that  $\lim\limits_{\lambda\rightarrow\lambda_\ast^-}\chi(\lambda)\le 0$ for all cases and $\chi(\lambda)$ is strictly increasing in $(-\infty,\lambda_\ast)$, so  $\chi(\widehat{\lambda})=\widehat{y}^{\T}\widehat{y}-\gamma^2<0$, which is contradictory to \eqref{eq:pLGopt-2} that $\|\widehat{y}\|=\gamma$. Therefore, $\lambda_{\ast}$ is the smallest
      Lagrangian multiplier, and thus
      $(\lambda_{\ast},y_{\ast})$ is a minimizer of {\rm pLGopt} \eqref{eq:pLGopt}.

For all situations, the smallest Lagrangian multiplier $\lambda_{\ast}$ of pLGopt \eqref{eq:pLGopt}
satisfies $\lambda_{\ast} \leq \lambda_{\min}(H)$, as expected.
Also $\lambda_{\ast}=\theta_1$ can only happen in the subcase (ii) or (iii).
\end{proof}

Buried in the proof above is a viable numerical algorithm to solve pLGopt~\eqref{eq:pLGopt}, provided
$\lambda_*$ in the case (a) and the subcase (i) of the case (b) can be efficiently solved. In both cases,
it is the unique root of secular equation $\chi(\lambda)=0$ in $(-\infty,\theta_1)$ in which $\chi(\lambda)$
monotonically increasing. A default method is Newton's method which applies the tangent line approximation, since both
$\chi(\lambda)$ and its derivative $\chi'(\lambda)$ is rather straightforward to evaluate. However,
this secular equation $\chi(\lambda)=0$ has a special rational form. Previous ideas in
solving secular equations of similar types \cite{bund:1978,gagv:1989,li:1993f,zhsl:2017} can be adopted to devise a much fast method
than Newton's method. Details are presented
in Appendix~\ref{sec:secularEq}.

%In order to prove the equivalence between pLGopt \eqref{eq:pLGopt}
%and pQEPmin \eqref{eq:pQEPmin}, we first establish a lemma that given
%a solution which satisfies the constrains of pLGopt \eqref{eq:pLGopt},
%then there exists a corresponding solution which satisfies the
%constrains of pQEPmin \eqref{eq:pQEPmin}.

\begin{lemma}\label{lm:pLGopt2pQEPmin}
If $(\lambda,y)$ satisfies the constraints of {\rm pLGopt} \eqref{eq:pLGopt}, then there exists
a vector $w\in\bbR^{\ell}$ such that $(\lambda,w)$ satisfies the constraints of {\rm pQEPmin} \eqref{eq:pQEPmin}.
Specifically,
$$
w=\begin{cases}
  (H-\lambda I)^{-1}y, &\quad\mbox{if $\lambda\notin\eig(H)$}, \\
  \mbox{the corresponding eigenvector of $H$}, &\quad\mbox{if $\lambda\in\eig(H)$}.
  \end{cases}
$$
%    \begin{itemize}
%    \item[(a)] if $\lambda\in\eig(H)$, then $w$ is the corresponding eigenvector of $H$;
%    \item[(b)] if $\lambda\notin\eig(H)$, then
%          $w=(H-\lambda I)^{-1}y$.
%    \end{itemize}
In particular, the optimal value of {\rm pQEPmin} \eqref{eq:pQEPmin} is
less than or equal to the optimal value of {\rm pLGopt} \eqref{eq:pLGopt}.
\end{lemma}
\begin{proof}
There are two cases to consider.
\begin{itemize}
\item Case $\lambda\in\eig(H)$:
Let $w$ be an eigenvector
of $H$ corresponding to eigenvalue $\lambda$, i.e., $H w = \lambda w$.
By \eqref{eq:pLGopt-1}, $g_0=-(H-\lambda I)y$, and thus
\[
\gamma^{-2}g_0 g_0^{\T}w=-\gamma^{-2}g_0 y^{\T}(H-\lambda I)w=0.
\]
Evidently, $(H-\lambda I)^{2}w=0$.
Hence $(\lambda,w)$ satisfies \eqref{eq:pQEPmin-1}.

\item Case $\lambda\not\in\eig(H)$:
Let $w=(H-\lambda I)^{-1}y$. Using \eqref{eq:pLGopt-1}, we have
    \begin{align*}
    (H-\lambda I)^{2}w&=(H-\lambda I)y=-g_0, \\
    \gamma^{-2}g_0g_0^{\T}w&=\gamma^{-2}g_0g_0^{\T}(H-\lambda I)^{-1}y=-\gamma^{-2}g_0y^{\T}y=-g_0.
    \end{align*}
   Again $(\lambda,w)$ satisfies \eqref{eq:pQEPmin-1}.
\end{itemize}
This proves that $(\lambda,w)$ satisfies the constraints of pQEPmin \eqref{eq:pQEPmin}.
As a corollary, the optimal value of {\rm pQEPmin} \eqref{eq:pQEPmin} is
less than or equal to the optimal value of {\rm pLGopt} \eqref{eq:pLGopt}.
\end{proof}

%Now we apply Lemma \ref{lm:pLGopt2pQEPmin} together with Lemma \ref{lm:pLGopt} to prove a lemma which claims a stronger conclusion.
The next lemma claims a stronger conclusion than the last statement in the previous lemma.

\begin{lemma}\label{lm:pLGopt=pQEPmin}
The optimal value of {\rm pLGopt} \eqref{eq:pLGopt} is
 equal to the optimal value of {\rm pQEPmin} \eqref{eq:pQEPmin}.
\end{lemma}

\begin{proof}
Let $(\lambda_{\ast},y_{\ast})$ be a minimizer of {\rm pLGopt} \eqref{eq:pLGopt}, and let
$\what\lambda$ be the optimal value  of pQEPmin \eqref{eq:pQEPmin}.
By Lemma~\ref{lm:pLGopt2pQEPmin}, we have $\what\lambda\le\lambda_{\ast}$.
It suffices to show that $\widehat\lambda<\lambda_{\ast}$ cannot happen.
Assume, to the contrary, that $\widehat\lambda<\lambda_{\ast}$. By Lemma \ref{lm:pLGopt}, we have
 $\widehat\lambda<\lambda_{\min}(H)$.
In particular, $\widehat\lambda\notin \eig(H)$.
Let $(\widehat\lambda,\widehat w)$ be a minimizer of {\rm pQEPmin} \eqref{eq:pQEPmin}. By \eqref{eq:pQEPmin-1}, we have
$$
\frac 1{\gamma^2} (\what w^{\T}g_0)^2=\what w^{\T}\frac 1{\gamma^2} g_0g_0^{\T}\what w=\what w^{\T}(H-\what\lambda I)^2\what w>0,
$$
implying $g_0^{\T}\what w\neq 0$.
Let $\what y=-({\gamma^2}/{g_0^{\T}\what w})\,(H-\widehat{\lambda} I)\what w$, and observe that
\begin{subequations}\label{eq:lag2qep}
\begin{align}
(H-\widehat{\lambda} I)\what y&=-\frac{\gamma^2}{g_0^{\T}\what w}\cdot (H-\widehat{\lambda} I)^2\what w
             =-\frac{\gamma^2}{g_0^{\T}\what w}\cdot\gamma^{-2}g_0g_0^{\T}\what w=-g_0,\\
\what y^{\T}\what y&=\left(\frac{\gamma^2}{g_0^{\T}\what w}\right)^2\what w^{\T}(H-\widehat{\lambda} I)^2\what w
                    =\left(\frac{\gamma^2}{g_0^{\T}\what w}\right)^2\frac{\what w^{\T}g_0g_0^{\T}\what w}{\gamma^2}=\gamma^2,
\end{align}
\end{subequations}
i.e., $(\widehat{\lambda}, \what y)$
 satisfies the
constraints of pLGopt \eqref{eq:pLGopt}. This implies $\lambda_{\ast}\le \what\lambda$,
contradicting the assumption $\what\lambda<\lambda_{\ast}$.
Therefore, $\widehat\lambda=\lambda_{\ast}$, as expected.
\end{proof}

We are ready to establish the equivalence between pLGopt \eqref{eq:pLGopt} and  pQEPmin~\eqref{eq:pQEPmin}.

\begin{theorem}[\bf pLGopt~\eqref{eq:pLGopt}  and pQEPmin~\eqref{eq:pQEPmin} are equivalent]\label{thm:pLGopt=pQEPmin} ~
\begin{itemize}
\item[\rm (1)]
    Let $(\lambda_{\ast},y_{\ast})$ be a minimizer of
    {\rm pLGopt} \eqref{eq:pLGopt}. Then either $\lambda_{\ast}<\lambda_{\min}(H)$ or $\lambda_{\ast}=\lambda_{\min}(H)$,
    and there exists $w_{\ast}$ such that $(\lambda_{\ast},w_{\ast})$ is a minimizer of
    {\rm pQEPmin} \eqref{eq:pQEPmin}. Specifically,
    $$
    w_{\ast}=\begin{cases}
          (H-\lambda_{\ast} I)^{-1}y_{\ast}, &\quad\mbox{if $\lambda_{\ast}<\lambda_{\min}(H)$}, \\
          \mbox{the corresponding eigenvector of $H$}, &\quad\mbox{if $\lambda_{\ast}=\lambda_{\min}(H)$}.
        \end{cases}
    $$
%    \begin{itemize}
%    \item[(a)] if $\lambda_{\ast}\in\eig(H)$, then $w_{\ast}$ is the corresponding
%          eigenvector of $H$;
%    \item[(b)] if $\lambda_{\ast}\notin\eig(H)$, then
%          $w_{\ast}=(H-\lambda_{\ast} I)^{-1}y_{\ast}$.
%    \end{itemize}

\item[\rm (2)] Conversely, if $(\lambda_{\ast} ,w_{\ast})$ is a
    minimizer of {\rm pQEPmin} \eqref{eq:pQEPmin}, then there exists $y_{\ast}$
    such that $(\lambda_{\ast},y_{\ast})$ is a minimizer of
    {\rm pLGopt}  \eqref{eq:pLGopt}. Specifically,
    $$
    y_{\ast}=\begin{cases}
          -({\gamma^2}/{g_0^{\T}w_{\ast}})\,(H-\lambda_{\ast} I)w_{\ast}, &\quad\mbox{if $g_0^{\T}w_{\ast}\neq 0$}, \\
          x_{\ast} + \sqrt{\gamma^2-\|x_{\ast}\|^2}\, ({w_{\ast}}/{\|w_{\ast}\|}), &\quad\mbox{if $g_0^{\T}w_{\ast}=0$},
        \end{cases}
    $$
    where $x_{\ast}=-(H-\lambda_{\ast} I)^{\dag} g_0$ in the case $g_0^{\T}w_{\ast}=0$, and it is guaranteed that $\|x_{\ast}\|\le\gamma$.
%    \begin{itemize}
%    \item[(a)] if $g_0^{\T}w_{\ast}\neq 0$, then
%    $y_{\ast}=-({\gamma^2}/{g_0^{\T}w_{\ast}})\,(H-\lambda_{\ast} I)w_{\ast}$;
%
%    \item[(b)] if $g_0^{\T}w_{\ast}=0$, then
%    $$
%    y_{\ast} = x + \sqrt{\gamma^2-\|x\|^2} \frac{w_{\ast}}{\|w_{\ast}\|},
%    $$
%    where $x=-(H-\lambda_{\ast} I)^{\dag} g$ (and can be shown that $\|x\|\le\gamma$).
%    \end{itemize}
\end{itemize}
\end{theorem}
\begin{proof}
Item (1) is a consequence of Lemmas~\ref{lm:pLGopt2pQEPmin} and \ref{lm:pLGopt=pQEPmin}.

Consider item (2). Suppose $(\lambda_{\ast} ,w_{\ast})$ is a
minimizer of pQEPmin \eqref{eq:pQEPmin}. By Lemma~\ref{lm:pLGopt=pQEPmin}, it suffices to
show that  there exists $y_{\ast}$
    such that $(\lambda_{\ast},y_{\ast})$ satisfies the constraints of pLGopt  \eqref{eq:pLGopt}.
\begin{itemize}
\item Case $g_0^{\T}w_{\ast}\neq 0$:
     The equations in \eqref{eq:lag2qep} hold with substitutions
     $$
     \widehat{\lambda}\to \lambda_{\ast}, \quad
     \what y\to y_{\ast}=-({\gamma^2}/{g_0^{\T}w_{\ast}})\,(H-\lambda_{\ast} I)w_{\ast}.
     $$
     So $(\lambda_{\ast},y_{\ast})$ satisfies the constraints of pLGopt  \eqref{eq:pLGopt}.

\item Case $g_0^{\T}w_{\ast}=0$:
By \eqref{eq:pQEPmin-1}, we find that
     $(H-\lambda_{\ast}I)^2w_{\ast}=0$, implying $(H-\lambda_{\ast}I)w_{\ast}=0$ since
     $H-\lambda_{\ast}I$ is real symmetric. Hence $\lambda_{\ast}\in\eig(H)$ and $w_{\ast}$ is
     an associated eigenvector. Let
     $x_{\ast}$ be the minimum norm solution of $(H-\lambda_{\ast} I)x_{\ast}=-g_0$.
     %Since $\lambda_{\ast}\in\eig(H)$,     $x_{\ast}$ lives in the eigen-space associated with all eigenvalues of $H$ different from $\lambda_{\ast}$     and thus $x_{\ast}^{\T}w_{\ast}=0$.
     Note that we already know  $\lambda_\ast$ is the optimal value of pLGopt \eqref{eq:pLGopt}, which means there exists $y$ such that $(\lambda_\ast,y)$ satisfies \eqref{eq:pLGopt-1} and $\|y\|=\gamma$. On the other hand, $x$ is minimal norm solution of \eqref{eq:pLGopt-1}, so $\|x\|\le\|y\|=\gamma$. Then it can be verified that
     $(\lambda_{\ast},y_{\ast})$ with $y_{\ast} = x_{\ast} + \sqrt{\gamma^2-\|x_{\ast}\|^2}\, ({w_{\ast}}/{\|w_{\ast}\|})$
     satisfies the constraints of pLGopt  \eqref{eq:pLGopt}. %It remains to show $\|x_{\ast}\|\le\gamma$.
\end{itemize}
This proves that $(\lambda_{\ast},y_{\ast})$ satisfies the constraints of pLGopt  \eqref{eq:pLGopt}. In addition, by Lemma \ref{lm:pLGopt=pQEPmin}, $\lambda_\ast$ is the optimal value of pLGopt \eqref{eq:pLGopt}, which proves the result.
\end{proof}

%Therefore, we conclude the equivalence of the original
%pLGopt \eqref{eq:pLGopt} and pQEPmin \eqref{eq:pQEPmin}.
The following theorem is about the uniqueness of the solution
for pLGopt \eqref{eq:pLGopt}.

\begin{theorem}[\bf Uniqueness of the minimizer for pLGopt \eqref{eq:pLGopt}]\label{thm:pLGoptunique}
Let $(\lambda_\ast,w_\ast)$ be a  minimizer of {\rm pQEPmin} \eqref{eq:pQEPmin}.
\begin{itemize}
\item[\rm (1)] If $g_0^{\T}w_\ast\neq 0$ for all possible minimizers for {\rm pQEPmin} \eqref{eq:pQEPmin},
      then $\lambda_\ast<\lambda_{\min}(H)$ and the minimizer of {\rm pLGopt} \eqref{eq:pLGopt} is unique.
\item[\rm (2)] If there exists a minimizer for {\rm pQEPmin} \eqref{eq:pQEPmin} such that $g_0^{\T}w_\ast=0$, then
       $\lambda_\ast=\lambda_{\min}(H)$ and the minimizer of {\rm pLGopt} \eqref{eq:pLGopt} is unique if and only if $\|x_\ast\|=\gamma$, where $x_{\ast}=-(H-\lambda_{\ast} I)^{\dag} g_0$.
\end{itemize}
\end{theorem}

\begin{proof}
\begin{itemize}
\item[\rm (1)] First we prove $\lambda_\ast<\lambda_{\min}(H)$. Suppose it is not true, i.e., $\lambda_\ast=\lambda_{\min}(H)$, let $w_\ast$ be an eigenvector of $H$ corresponding with eigenvalue $\lambda_{\min}(H)$, then by Theorem \ref{thm:pLGopt=pQEPmin}, $(\lambda_\ast, w_\ast)$ is a minimizer of pQEPmin \eqref{eq:pQEPmin}. Since QEP \eqref{eq:pQEPmin-1} leads to $\gamma^{-2}g_0g_0^{\T}w_\ast=(H-\lambda_\ast I)^2w_\ast=0$ and $w_\ast\neq 0$, we have $g_0^{\T}w_\ast=0$, which is contradictory to our assumption that $g_0^{\T}w_\ast\neq 0$ for all possible minimizers $(\lambda_\ast,w_\ast)$ of pQEPmin \eqref{eq:pQEPmin}. Therefore, $\lambda_\ast<\lambda_{\min}(H)$.

In this case $(\lambda_\ast,x_\ast=-(H-\lambda_\ast I)^{-1}g_0)$ is the unique minimizer of pLGopt \eqref{eq:pLGopt} since the $H-\lambda_\ast I$ is nonsingular and $x_\ast$ is the unique solution of \eqref{eq:pLGopt-1}.
\item[\rm (2)] Making use of \eqref{eq:pQEPmin-1}, we have
$$
(H-\lambda_\ast I)^2w_\ast=\gamma^{-2}g_0g_0^{\T}w_\ast=0
\quad\Rightarrow\quad
(H-\lambda_\ast I)w_\ast=0
$$
because $H-\lambda_\ast I$ is real symmetric. Therefore $\lambda_\ast\in\eig(H)$, which yields $\lambda_{\ast}=\lambda_{\min}(H)$. Note that $x_\ast$ is unique and $w_\ast$ can be chosen arbitrarily in the eigenspance of $H$ corresponding with eigenvalue $\lambda_{\min}(H)$, so  $w_\ast$ is not unique. Therefore, $y_\ast= x_{\ast} + \sqrt{\gamma^2-\|x_{\ast}\|^2}\, ({w_{\ast}}/{\|w_{\ast}\|})$ is unique if and only if $\|x_\ast\|=\gamma$.

\end{itemize}
\end{proof}

\begin{remark}\label{rm:incomplete}
In \cite{gagv:1989}, the authors investigate
the relationship between the problems
\begin{align}
{\mbox{pLG:}\quad} & (H-\lambda I)y=-g_0,\,\, \|y\|=\gamma, \label{eq:pLG} \\
{\mbox{pQEP:}\quad} & (H-\lambda I)^2w=\gamma^{-2}g_0g_0^{\T}w,\,\, \lambda\in\bbR,\,\, w\ne 0. \label{eq:pQEP}
\end{align}
They differ from pLGopt and pQEPmin, respectively, just without taking the min over $\lambda$.
%
%the solutions of
%pLGopt and pQEPmin without the condition ``$\min\lambda$''.
%Define the Lagrange equations
%\begin{equation}\label{eq:pLG}
%{\mbox{pLG:}\quad} \begin{cases}
%&(H-\lambda I)y=-g_0,\\
%            &\|y\|=\gamma,
%\end{cases}
%\end{equation}
%and QEP
%\begin{equation}\label{eq:pQEP}
%{\mbox{pQEP:}\quad} \begin{cases}
%&(H-\lambda I)^2w=\gamma^{-2}g_0g_0^{\T}w,\\
%            &\lambda\in\bbR, w\ne 0.
%\end{cases}
%\end{equation}
%The authors proved
The following results were obtained there:
\begin{enumerate}
\item If $(\lambda,y)$ is a solution of pLG \eqref{eq:pLG}, then there exists $w$ such that $(\lambda,w)$ is a solution of pQEP \eqref{eq:pQEP}.
\item Suppose that $(\lambda,w)$ is a solution of pQEP \eqref{eq:pQEP}.
\begin{itemize}
\item If $\lambda\notin\text{eig}(H)$, then there exists $y$ such that $(\lambda,y)$ is a solution of pLG \eqref{eq:pLG}.
\item If $\lambda\in\text{eig}(H)$, then  there exists $y$ such that $(\lambda,y)$ is a solution of pLG \eqref{eq:pLG} if and only if $\|(H-\lambda I)^{\dag}g_0\|\le \gamma$.
\end{itemize}
\end{enumerate}
Consequently, these results provide no guarantee that for any solution $(\lambda,w)$ of
pQEP \eqref{eq:pQEP}, there exists a corresponding solution $(\lambda,y)$
of pLG \eqref{eq:pLG}. Nonetheless, the authors stated without any proof that for the
solution $(\lambda_\ast,w_\ast)$ of pQEP \eqref{eq:pQEP} with
$\lambda_\ast$ being the smallest eigenvalue of pQEP \eqref{eq:pQEP}, there does exist a solution
$(\lambda_\ast,y_\ast)$ of pLGopt \eqref{eq:pLGopt}, a conclusion that doesn't look like a straightforward
one to us. Because of that,
in Theorem \ref{thm:pLGopt=pQEPmin} we rigorously proved
that for any minimizer $(\lambda_\ast,w_\ast)$ of
pQEPmin \eqref{eq:pQEPmin}, there exists $y_\ast$
such that $(\lambda_\ast,y_\ast)$ is a minimizer of pLGopt \eqref{eq:pLGopt}.
\hfill $\Box$
\end{remark}

Next we will establish an important result in Theorem~\ref{thm:pQEPmin=lefteig} below that says the leftmost eigenvalue of
QEP \eqref{eq:pQEPmin-1} is real. We begin by establishing a close relationship in Lemma~\ref{lem:qep2secluar}
between the zeros of the secular function $\chi(\lambda)$ in \eqref{eq:sec-fun:H} and the eigenvalues
of QEP \eqref{eq:pQEPmin-1},
and then using  the relation to expose an eigenvalue distribution
property of QEP \eqref{eq:pQEPmin-1} in Lemmas~\ref{lm:noeig} and \ref{lm:existeig}, in preparing for proving
our main result in Theorem~\ref{thm:pQEPmin=lefteig}.

\begin{lemma}\label{lem:qep2secluar}
Suppose $\lambda\notin \eig(H)$,
$\lambda$ (possibly complex) is an eigenvalue of {\em QEP} \eqref{eq:pQEPmin-1} if and only if $\chi(\lambda)=0$, where $\chi(\lambda)$ is defined in \eqref{eq:sec-fun:H}.
\end{lemma}

\begin{proof}
Let $\chi(\lambda)=0$ and $\lambda\notin \eig(H)$.
Define $z=(H-\lambda I)^{-2}g_0$. Then we have $(H-\lambda I)^2z=g_0$
and
\[
g_0^{\T}z=\sum_{i=1}^{\ell}\frac{\xi_i^2}{(\theta_i-\lambda)^2}=\gamma^2
\,\,\mbox{and thus}\,\,
(H-\lambda I)^2z=g =\gamma^{-2} g g_0^{\T} z,
\]
i.e., $(\lambda,z)$ is an eigenpair of  QEP \eqref{eq:pQEPmin-1}.

On the other hand,
suppose $\lambda$ is an eigenvalue of QEP \eqref{eq:pQEPmin-1} and
$\lambda\notin \eig(H)$. Pre-multiply \eqref{eq:pQEPmin-1} by $g_0^{\T}(H-\lambda I)^{-2}$ to get
\begin{equation}\label{eq:bz}
g_0^{\T}z=\gamma^{-2} g_0^{\T}(H-\lambda I)^{-2}g_0g_0^{\T}z.
\end{equation}
We claim that $g_0^{\T} z \neq 0$. Otherwise, $(H-\lambda I)^2z=0$ by   \eqref{eq:pQEPmin-1},
which implies $(H-\lambda I)z=0$, i.e., $\lambda\in \eig(H)$, a contradiction.
So $g_0^{\T} z \neq 0$ and thus it follows from \eqref{eq:bz} that
\[
\gamma^{-2} g_0^{\T}(H-\lambda I)^{-2}g_0=1,
\]
i.e., $\lambda$ is a zero of $\chi(\lambda)$, as was to be shown.
\end{proof}

%Based on Lemma \ref{lem:qep2secluar}, we show a lemma that
%QEP~\eqref{eq:pQEPmin-1} has no eigenvalue in a certain area in complex plane.

\begin{lemma}\label{lm:noeig}
{\em QEP} \eqref{eq:pQEPmin-1} has no eigenvalue $\lambda = \alpha + \ti \beta$
with $\alpha<\theta_{j_0}$ and $\beta\neq 0$,
where $\alpha,\,\beta\in\bbR$, $\ti$ is the imaginary unit, and $j_0$ is defined in \eqref{eq:j04ci}.
\end{lemma}

\begin{proof}
Suppose, to the contrary, that QEP \eqref{eq:pQEPmin-1} has an eigenvalue $\lambda = \alpha + \ti \beta$ with
$\alpha<\theta_{j_0}$ and $\beta\neq0$.
Evidently $\lambda = \alpha+\ti \beta\notin \eig(H)$ because all eigenvalues of $H$ are real.
By Lemma~\ref{lem:qep2secluar}, $\alpha+\ti \beta$ must be a zero of
the secular function $\chi(\lambda)$ in \eqref{eq:sec-fun:H}, i.e.,
\begin{align*}
0=\chi(\alpha+\ti \beta)
&=\sum_{i=1}^{\ell} \frac {\xi_i^2 }
                     {(\alpha-\theta_i+\ti \beta )^2}-\gamma^2\\
&=\sum_{i=1}^{\ell} \frac {\xi_i^2 }
                     {(\alpha-\theta_i)^2-\beta^2+2\ti (\alpha-\theta_i)\beta}-\gamma^2\\
&=\sum_{i=1}^{\ell} \frac {\xi_i^2[(\alpha-\theta_i)^2-\beta^2-2\ti (\alpha-\theta_i)\beta] }
                     {[(\alpha-\theta_i)^2-\beta^2]^2+4\beta^2(\alpha-\theta_i)^2}-\gamma^2.
\end{align*}
In particular, the imaginary part of $\chi(\alpha+\ti \beta)$ is zero, i.e.,
\begin{equation}\label{eq:ims}
\sum_{i=1}^{\ell} \frac {-2(\alpha-\theta_i)\beta \xi_i^2 }
                     {[(\alpha-\theta_i)^2-\beta^2]^2+4\beta^2(\alpha-\theta_i)^2}
=\beta \left( \sum_{i=j_0}^{\ell} \frac {-2(\alpha-\theta_i) \xi_i^2 }
                                    {[(\alpha-\theta_i)^2-\beta^2]^2+4\beta^2(\alpha-\theta_i)^2} \right) =0.
\end{equation}
Since $\alpha<\theta_i$ for all $i\geq j_0$,
$\xi_{j_0}^2>0$ and $\xi_i^2\geq 0$ for all $i>j_0$, we know
$$
\sum_{i=j_0}^{\ell} \frac {-2(\alpha-\theta_i) \xi_i^2 }
                     {[(\alpha-\theta_i)^2-\beta^2]^2+4\beta^2(\alpha-\theta_i)^2}>0.
$$
Therefore, by \eqref{eq:ims},
we conclude $\beta=0$, a contradiction.
\end{proof}

%The next lemma shows the existence of eigenvalues of
%QEP~\eqref{eq:pQEPmin-1} in a specific interval in the real line.

\begin{lemma}\label{lm:existeig}
{\em QEP} \eqref{eq:pQEPmin-1} has an eigenvalue  $\wtd{\lambda}<\theta_{j_0}$ (necessarily $\wtd{\lambda}\in\bbR$), where $j_0$ is defined in \eqref{eq:j04ci}.
\end{lemma}

\begin{proof} There are two possible cases:
\begin{itemize}
\item {Case $\theta_{j_0}=\theta_1$:}
Without loss of generality, let $\xi_1\ne 0$. Since $\chi(\lambda)$ is continuous and strictly
increasing in $(-\infty,\theta_1)$, and
\[
\lim\limits_{\lambda\rightarrow  -\infty} \chi(\lambda)=-\gamma^2<0,\,\,
\lim_{\lambda\rightarrow \theta_1^-}\chi(\lambda)
      \geq \lim_{\lambda\rightarrow \theta_1^-} \frac {\xi_1^2}{(\lambda-\theta_1)^2}-\gamma^2 =+\infty>0,
\]
there exists a zero $\wtd{\lambda}\in(-\infty,\theta_1)$ of $\chi(\lambda)$.
Evidently $\wtd{\lambda}\notin \eig(H)$, and then
by Lemma \ref{lem:qep2secluar},
$\wtd{\lambda}$ must be an eigenvalue of QEP \eqref{eq:pQEPmin-1}.

\item {Case $\theta_{j_0}>\theta_1$:}
Let $\wtd{\lambda}=\theta_1$ and $z=y_1$.
We have $(H-\wtd{\lambda} I)^2z=(H-\wtd{\lambda} I)^2y_1=0$.
Furthermore, $g_0^{\T}z=g_0^{\T}y_1=\xi_1=0$.
Therefore $(\wtd{\lambda},z)$ satisfies \eqref{eq:pQEPmin-1}, impliying $\wtd{\lambda}$
is an eigenvalue of
QEP \eqref{eq:pQEPmin-1} and $\tilde{\lambda}=\theta_1<\theta_{j_0}$.
\end{itemize}
The proof is completed.
\end{proof}

With the three lemmas above, now we are ready to prove our main result on the leftmost eigenvalue of {\rm QEP} \eqref{eq:pQEPmin-1}.

\begin{theorem}\label{thm:pQEPmin=lefteig}
The leftmost eigenvalue, by which we mean the one with
the smallest real part,  of {\rm QEP} \eqref{eq:pQEPmin-1} is real.
As a consequence, the optimal value of {\rm pQEPmin} \eqref{eq:pQEPmin} $\lambda_\ast$
is the leftmost eigenvalue of {\rm QEP} \eqref{eq:pQEPmin-1}.
%, \Purple{which means any eigenvalue of \eqref{eq:pQEPmin-1} $\alpha+\ti\beta$ such that $\beta\neq 0$ satisfies $\alpha>\lambda_\ast$.}
\end{theorem}

\begin{proof}
Let $\lambda_{\ast} =\alpha_{\ast} +\ti\beta_{\ast} $ be the leftmost eigenvalue.
By Lemma \ref{lm:existeig}, QEP \eqref{eq:pQEPmin-1} has
a real eigenvalue $\wtd{\lambda}$ with $\wtd{\lambda} < \theta_{j_0}$. Hence
$\alpha_{\ast} \le\wtd{\lambda}<\theta_{j_0}$, which together with Lemma \ref{lm:noeig}
tell us that $\beta_{\ast} =0$ and thus $\lambda_{\ast}\in\bbR$.
%
%\Purple{Besides, by Lemma \ref{lm:noeig}, any complex eigenvalue $\alpha+\ti\beta$ where $\beta\neq 0$ satisfies $\alpha\geq \theta_{ j_0}$. Together with the fact that $\lambda_\ast=\alpha_\ast<\theta_{ j_0}$ tells us that $\alpha>\lambda_\ast$.}
\end{proof}

\begin{remark}
In \cite{sihg:2004}, the authors stated without proof
that the rightmost eigenvalue
of the QEP
\begin{equation}\label{eq:remarkqep}
((W+\lambda I)^2-\delta^{-2}hh^{\T})x=0
\end{equation}
is real and positive,
where $W$ is a real symmetric matrix,
$h$ is a vector, and $\delta>0$ is a scalar.
It was pointed out in \cite{lav:2007} that the rightmost eigenvalue of \eqref{eq:remarkqep} may not always be positive and the authors proved in \cite[Theorem 4.1]{lav:2007} that the largest real eigenvalue of \eqref{eq:remarkqep} is the rightmost eigenvalue. The authors applied a maximin principle for nonlinear eigenproblems for the proof.
In Theorem~\ref{thm:pQEPmin=lefteig} we have proved the leftmost eigenvalue $\lambda_\ast$ of  \eqref{eq:pQEPmin-1} is real, i.e., there is no complex eigenvalue
of QEP~\eqref{eq:pQEPmin-1} with real part equal to $\lambda_\ast$ and nonzero complex part.
This result cannot be obtained by the approach used in \cite{lav:2007}.
\hfill $\Box$
\end{remark}

\subsection{LGopt  and QEPmin  are equivalent}\label{sec:LGopt=QEPmin}
Theorem \ref{thm:pLGopt=pQEPmin} says that
pLGopt~\eqref{eq:pLGopt} and
pQEPmin~\eqref{eq:pQEPmin} are equivalent. Previously in Lemma~\ref{lm:LGopt2QEPmin-0}, we showed that
LGopt~\eqref{eq:LGopt} and QEPmin~\eqref{eq:QEPmin} are also equivalent under
the assumptions stated there. Our goal in this subsection is to have
the assumptions of Lemma~\ref{lm:LGopt2QEPmin-0} removed.

For convenience, we restate LGopt~\eqref{eq:LGopt} and QEPmin~\eqref{eq:QEPmin} as follows:
\begin{empheq}[left={\mbox{LGopt:}\quad}\empheqlbrace]{alignat=2}
\min  ~& \lambda \tag{\ref{eq:LGopt-0}} \\
\mbox{s.t.}  ~& (PAP-\lambda I)u=-b_0,\tag{\ref{eq:LGopt-1}}\\
~&\|u\|=\gamma,\tag{\ref{eq:LGopt-2}}\\
~&u\in\cN(C^{\T});\tag{\ref{eq:LGopt-3}}
\end{empheq}
\begin{empheq}[left={\mbox{QEPmin:}\quad}\empheqlbrace]{alignat=2}
\min ~& \lambda \tag{\ref{eq:QEPmin-0}} \\
\mbox{s.t.}
~& (PAP-\lambda I)^{2}z=\gamma^{-2}b_0b_0^{\T}z,\tag{\ref{eq:QEPmin-1}}\\
~& \lambda\in\bbR,~ 0\ne z\in\cN(C^{\T}).\tag{\ref{eq:QEPmin-2}}
\end{empheq}

Recall  $S_1$ and $S_2$ as  defined in \eqref{eq:s1s2} and $H$ and $g$
as defined in \eqref{eq:def:gH}. Before stating our main result in this subsection, we need two lemmas.
The first one is  about an eigen-relationship between
$PAP$ and $H$ and the second one is on the relationships among $PAP-\lambda I$, $H-\lambda I$,  $(PAP-\lambda I)^\dag$ and $(H-\lambda I)^\dag$.

\begin{lemma}\label{lem:eigpaph1}
$(\lambda, s)$ is an eigenpair of $H$ if and only if $(\lambda,S_1s)$ is an eigenpair of $PAP$ with
$S_1s\in\mathcal{N}(C^{\T})$.
%\Purple{where the corresponding eigenvector is in $\mathcal{N}(C^{\T})$}.
\end{lemma}

\begin{proof}
This is a consequence of the decomposition \eqref{eq:decomp:PAP-1}.
\end{proof}

\begin{lemma}\label{lem:2pinv}
For any $\lambda\in\bbR$, $(PAP-\lambda I)S_1=S_1(H-\lambda I)
$ and $(PAP-\lambda I)^{\dag} S_1=S_1(H-\lambda I)^{\dag}$.
\end{lemma}
\begin{proof}
Let $H=Y\Theta Y^{\T}$ be the eigen-decomposition of $H$, where $Y\in\bbR^{(n-m)\times (n-m)}$ is orthogonal
and $\Theta$ is a diagonal matrix. Then the eigen-decomposition of $PAP$ is given by
\begin{equation}\label{eq:papdiag}
PAP=[S_1\ S_2]\begin{bmatrix}
Y& 0\\ 0& I
\end{bmatrix}
\begin{bmatrix}
\Theta &0\\ 0&0
\end{bmatrix}
\begin{bmatrix}
Y^{\T}&0 \\ 0& I
\end{bmatrix}
[S_1\ S_2]^{\T}.
\end{equation}
Therefore
$
(PAP-\lambda I)S_1=S_1Y(\Theta-\lambda I)Y^{\T}=S_1(H-\lambda I).
$
On the other hand, for $\lambda\ne 0$,
$$
(PAP-\lambda I)^{\dag}
    =[S_1\ S_2]\begin{bmatrix}
                    Y&0 \\ 0& I
                \end{bmatrix}\begin{bmatrix}
                                (\Theta-\lambda I)^{\dag} &0\\
                                0&-\frac{1}{\lambda}I
                             \end{bmatrix}
                             \begin{bmatrix}
                                 Y^{\T}&0 \\ 0& I
                             \end{bmatrix}
                             [S_1\ S_2]^{\T},
$$
and for $\lambda=0$,
$$
(PAP)^{\dag}
    =[S_1\ S_2]\begin{bmatrix}
                    Y&0 \\ 0& I
                \end{bmatrix}\begin{bmatrix}
                                \Theta^{\dag} &0\\
                                0&0
                             \end{bmatrix}
                             \begin{bmatrix}
                                 Y^{\T}&0 \\ 0& I
                             \end{bmatrix}
                             [S_1\ S_2]^{\T}.
$$
Hence
$
(PAP-\lambda I)^{\dag} S_1
    =S_1Y(\Theta-\lambda I)^{\dag} Y^{\T}
    =S_1(H-\lambda I)^{\dag},
$
as was to be shown.
\end{proof}

Now we are ready to state the main result of the subsection.

\begin{theorem}[\bf LGopt~\eqref{eq:LGopt} and QEPmin~\eqref{eq:QEPmin}
are equivalent]\label{thm:LGopt=QEPmin} ~
\begin{enumerate}[\rm (1)]
\item
    Let $(\lambda_{\ast},u_{\ast})$ be a minimizer of  {\rm LGopt} \eqref{eq:LGopt}. Then
    there exists $z_{\ast}$ such that $(\lambda_{\ast},z_{\ast})$ is a minimizer of
    {\rm QEPmin} \eqref{eq:QEPmin}. Specifically,
    $$
    z_{\ast}=\begin{cases}
               (PAP-\lambda_{\ast} I)^{\dag}u_{\ast}, &\parbox[t]{9cm}{if $\lambda_{\ast}\notin\eig(PAP)$ or $\lambda_{\ast}\in\eig(PAP)$ but
          there is no corresponding eigenvector entirely in $\cN(C^{\T})$,} \\
               s, &\parbox[t]{9cm}{if $\lambda_{\ast}\in\eig(PAP)$ and there is a  corresponding eigenvector $s\in\cN(C^{\T})$.}
             \end{cases}
    $$
%    \begin{itemize}
%      \item[(a)] if $\lambda_{\ast}\notin\eig(PAP)$ or $\lambda_{\ast}\in\eig(PAP)$ but
%          there is no corresponding eigenvector $s\in\cN(C^{\T})$, then $z_{\ast}=(PAP-\lambda_{\ast} I)^{\dag}u_{\ast}$;
%          %In particular, if $\lambda_{\ast}\notin\eig(PAP)$, then $z_{\ast}=(PAP-\lambda_{\ast} I)^{-1}u$.
%      \item[(b)] if $\lambda_{\ast}\in\eig(PAP)$ and there is a  corresponding eigenvector $s\in\cN(C^{\T})$, then $z_{\ast}=s$.
%    \end{itemize}

\item
    Let $(\lambda_{\ast} ,z_{\ast})$ be a minimizer of {\rm QEPmin} \eqref{eq:QEPmin}. Then there exists $u_{\ast}\in\bbR^n$
    such that $(\lambda_{\ast} ,u_{\ast})$
    is a minimizer of {\rm LGopt} \eqref{eq:LGopt}. Specifically,
    $$
    u_{\ast}=\begin{cases}
               -({\gamma^2}/{b_0^{\T}z_{\ast}})(PAP-\lambda_{\ast} I)z_{\ast}, &\mbox{if $b_0^{\T}z_{\ast}\neq 0$,} \\
               x_{\ast} + \sqrt{\gamma^2-\|x_{\ast}\|^2} \,({z_{\ast}}/{\|z_{\ast}\|}), &\mbox{if $b_0^{\T}z_{\ast}=0$},
             \end{cases}
    $$
    where $x_{\ast}=-(PAP-\lambda_{\ast} I)^{\dag} b_0$ in the case  $b_0^{\T}z_{\ast}=0$ and it is guaranteed that $\|x_{\ast}\|\le\gamma$.
%    \begin{itemize}
%      \item[(a)] if $b_0^{\T}z_{\ast}\neq 0$, then $u_{\ast}=-\frac{\gamma^2}{b_0^{\T}z}(PAP-\lambda_{\ast} I)z$;
%
%      \item[(b)] If $b_0^{\T}z_{\ast}=0$,  then $\|x\|\le\gamma$ and
%      $$
%      u_{\ast}= x + \sqrt{\gamma^2-\|x\|^2} \,({z_{\ast}}/{\|z_{\ast}\|}),
%      $$
%      where $x=-(PAP-\lambda_{\ast} I)^{\dag} b_0$.
%    \end{itemize}
\end{enumerate}
\end{theorem}
\begin{proof}
We prove item (1) first.
By Theorem \ref{thm:LGopt=pLGopt}, $(\lambda_{\ast},y_{\ast})$ with $y_{\ast}=S_1^{\T}u_{\ast}$ is a minimizer of
pLGopt \eqref{eq:pLGopt}. We have two  cases to consider.
\begin{itemize}
\item[(a)] If $\lambda_{\ast}\notin\eig(PAP)$ or $\lambda_{\ast}\in\eig(PAP)$ but
          there is no corresponding eigenvector $s\in\cN(C^{\T})$, then $\lambda\notin\eig(H)$ by Lemma~\ref{lem:eigpaph1}.
Using Theorem \ref{thm:pLGopt=pQEPmin}, we conclude that
$(\lambda_{\ast},w_{\ast})$ with
$$
w_{\ast}=(H-\lambda_{\ast} I)^{-1}y_{\ast}=(H-\lambda_{\ast} I)^{\dag}y_{\ast}
$$
is a minimizer of pQEPmin~\eqref{eq:pQEPmin}.
Now use Theorem~\ref{thm:QEPmin=pQEPmin} to conclude that $(\lambda_{\ast},z_{\ast})$ with $z_{\ast}=S_1(H-\lambda_{\ast} I)^{\dag}y_{\ast}$
is a minimizer of QEPmin  \eqref{eq:QEPmin}.
By Lemma \ref{lem:2pinv},
$$
z_{\ast}=S_1(H-\lambda_{\ast} I)^{\dag}w_{\ast}=(PAP-\lambda_{\ast} I)^{\dag}S_1w_{\ast}=(PAP-\lambda_{\ast} I)^{\dag}u_{\ast}.
$$
%
%Specifically, when $\lambda_{\ast}\neq 0$, suppose $\lambda\in\eig(PAP)$ and $s$ is a corresponding eigenvector, then $s=(PAPs)/\lambda_{\ast}\in\cN(C^{\T})$, which is contradictory with the assumptions. Therefore, $\lambda\notin\eig(PAP)$ and $z_{\ast}=(PAP-\lambda_{\ast} I)^{-1}u$.

\item[(b)] Suppose that  $\lambda_{\ast}\in\eig(PAP)$ and there is a  corresponding eigenvector $s\in\cN(C^{\T})$.
Then  $s=S_1r$ for some $0\ne r\in\bbR^{n-m}$.
By Lemma \ref{lem:eigpaph1}, $r$ is an eigenvector of $H$ corresponding to the eigenvalue $\lambda_{\ast}$.
Use Theorem \ref{thm:pLGopt=pQEPmin} to conclude that $(\lambda_{\ast}, w_{\ast})$ with $w_{\ast}=r$ is a minimizer
of pQEPmin \eqref{eq:pQEPmin}, which in turn, by Theorem~\ref{thm:QEPmin=pQEPmin}, yields
that  $(\lambda_{\ast},z_{\ast})$ with $z_{\ast}=s=S_1r$ is a minimizer of QEPmin \eqref{eq:QEPmin}.
\end{itemize}
Next we consider item (2). By Theorem~\ref{thm:QEPmin=pQEPmin}, $(\lambda_{\ast},w_{\ast})$ with $w_{\ast}=S_1^{\T}z_{\ast}$
is a minimizer of
pQEPmin \eqref{eq:pQEPmin}. Since $b_0,\,z_{\ast}\in\cN(C^{\T})$, we have $z_{\ast}=S_1w_{\ast}$ and $b_0^{\T}z_{\ast}=g_0^{\T}S_1^{\T}S_1w_{\ast}=g_0^{\T}w_{\ast}$.
\begin{itemize}
\item Case $b_0^{\T}z_{\ast}\neq 0$:
We have $g_0^{\T}w_{\ast}\neq 0$. By Theorem \ref{thm:pLGopt=pQEPmin},
$(\lambda_{\ast}, y_{\ast})$ with $y_{\ast}=-({\gamma^2}/{g_0^{\T}w_{\ast}})\,(H-\lambda_{\ast} I)w_{\ast}$ solves pLGopt  \eqref{eq:pLGopt}.
By Theorem \ref{thm:LGopt=pLGopt}, $(\lambda_{\ast}, u_{\ast})$
with $u_{\ast}=-({\gamma^2}/{g_0^{\T}w_{\ast}})\,S_1(H-\lambda_{\ast} I)w_{\ast}$ solves
LGopt \eqref{eq:LGopt}. Furthermore, by Lemma \ref{lem:2pinv}, $(PAP-\lambda_{\ast} I)z_{\ast}=(PAP-\lambda_{\ast} I)S_1w_{\ast}=S_1(H-\lambda_{\ast} I)w_{\ast}$.
Therefore $u_{\ast}=-({\gamma^2}/{g_0^{\T}w_{\ast}})\,S_1(H-\lambda_{\ast} I)w_{\ast}
=-({\gamma^2}/{b_0^{\T}z})\,(PAP-\lambda_{\ast} I) z_{\ast}
$.

\item Case $b_0^{\T}z_{\ast}=0$:
We have $g_0^{\T}w_{\ast}=0$ and $z_{\ast}$ is an eigenvector of $PAP$ corresponding
to its eigenvalue $\lambda_{\ast}$. By Lemma \ref{lem:eigpaph1}, $y_{\ast}=S_1^{\T}z_{\ast}$ is an eigenvector of $H$ corresponding to
its eigenvalue $\lambda_{\ast}$.  Let $s=-(H-\lambda_{\ast} I)^{\dag} g$, according to Theorem \ref{thm:pLGopt=pQEPmin},
 $\|s\|\le\gamma$ and $(\lambda_{\ast}, w_{\ast})$ with $w_{\ast}=s+\sqrt{\gamma^2-\|s\|^2}\,(y_{\ast}/\|y_{\ast}\|)$ solves pLGopt \eqref{eq:pLGopt}.
By Theorem~\ref{thm:QEPmin=pQEPmin}, $(\lambda_{\ast},u_{\ast})$ with $u_{\ast}=S_1w_{\ast}$ is a minimizer of LGopt \eqref{eq:LGopt}.
Now set
$$
x_{\ast}=S_1s=-S_1(H-\lambda_{\ast} I)^{\dag} g=-(PAP-\lambda_{\ast} I)^{\dag} b_0,
$$
and thus
$$
u_{\ast}=S_1w_{\ast}=S_1s+\sqrt{\gamma^2-\|S_1s\|^2}\,\frac{S_1y_{\ast}}{\|S_1y_{\ast}\|}
      =x_{\ast} + \sqrt{\gamma^2-\|x_{\ast}\|^2}\, \frac{z_{\ast}}{\|z_{\ast}\|},
$$
as expected.
\end{itemize}
This completes the proof.
\end{proof}

We note that proving the equivalence between
LGopt \eqref{eq:LGopt} and QEPmin \eqref{eq:QEPmin} is
of theoretical interest.
The proof in \cite{gagv:1989} is incomplete since in Remark \ref{rm:incomplete} we mentioned that they did not prove that pLGopt \eqref{eq:pLGopt} and pQEPmin \eqref{eq:pQEPmin} are equivalent.
Here we provided a complete proof in Theorem \ref{thm:LGopt=QEPmin}.

Returning to the original CRQopt \eqref{eq:CRQopt}, we observe that
if $(\lambda_{\ast},u_{\ast})$ solves LGopt \eqref{eq:LGopt}, then $n_0+u_{\ast}$ solves CRQopt \eqref{eq:CRQopt}. Therefore
immediately we obtain the following theorem.

\begin{theorem}\label{thm:CRQopt=QEPmin}
Suppose $(\lambda_{\ast},z_{\ast})$ is a minimizer of
{\rm QEPmin} \eqref{eq:QEPmin}. Then
a minimizer $v_*$ of {\rm CRQopt \eqref{eq:CRQopt}} is given by
$$
v_*=\begin{cases}
        n_0-({\gamma^2}/{b_0^{\T}z_{\ast}})\,(PAP-\lambda_{\ast} I)z_{\ast}, &\quad\mbox{if $b_0^{\T}z_{\ast}\neq 0$}, \\
        n_0+x_{\ast} + \sqrt{\gamma^2-\|x_{\ast}\|^2}\, ({z_{\ast}}/{\|z_{\ast}\|}), &\quad\mbox{if $b_0^{\T}z_{\ast}=0$},
    \end{cases}
$$
where $x_{\ast}=-(PAP-\lambda_{\ast} I)^{\dag} b_0$ in the case of $b_0^{\T}z_{\ast}=0$ and it is guaranteed that $\|x_{\ast}\|\le\gamma$.
\end{theorem}

What the next theorem says is that solving QEPmin \eqref{eq:QEPmin}
is equivalent to calculating the leftmost eigenvalue of  QEP \eqref{eq:QEPmin-1}
among those having
eigenvectors\footnote{This does not exclude the possibility that they may have eigenvectors not in  $\cN(C^{\T})$.}
in $\cN(C^{\T})$. This result paves the way for the use of a Krylov subspace method to
calculate the minimizer of QEPmin \eqref{eq:QEPmin} in Section~\ref{sec-crq-alg} ahead.
%calculate
%This means our target eigenvalue is an extreme eigenvalue so
%that it is suitable to be solved by a Krylov subspace method, see

\begin{theorem}\label{thm:leftmosteiggen}
If $(\lambda_{\ast},z_{\ast})$ is a minimizer of {\rm QEPmin} \eqref{eq:QEPmin},
then $\lambda_{\ast}$ is the leftmost eigenvalue
of {\rm QEP} \eqref{eq:QEPmin-1} among those having eigenvectors in $\cN(C^{\T})$.
\end{theorem}

\begin{proof}
Following the argument in the proof of Theorem~\ref{thm:QEPmin=pQEPmin}, we find that
the set of eigenvalues of QEP \eqref{eq:QEPmin-1} that have eigenvectors in $\in\cN(C^{\T})$ and
the set of eigenvalues of QEP \eqref{eq:pQEPmin-1} are the same.
The conclusion is an immediate consequence of Theorems~\ref{thm:QEPmin=pQEPmin} and \ref{thm:pQEPmin=lefteig}.
\end{proof}

%--------------------------------------------------------------------------
%\newpage

\subsection{Summary}
Starting with CRQopt~\eqref{eq:CRQopt}, we have introduced five
equivalent optimization problems. Figure~\ref{fig:eqiv-opts}
summarizes the relationships of these problems.
The edge ``$\longleftrightarrow$'' in
Figure~\ref{fig:eqiv-opts}
connecting two optimization problems indicates that we have
an equivalent relationship in the previous subsections.
We note that CRQopt~\eqref{eq:CRQopt} and CQopt~\eqref{eq:CQopt}
share the same minimizers $v_{\ast}$, while correspondingly
the minimizer for LGopt~\eqref{eq:LGopt} is $u_{\ast}=Pv_{\ast}$.
Slightly more efforts are needed to describe
corresponding minimizers for other equivalent optimization problems
as shown in Figure~\ref{fig:eqiv-opts}.
The optimal values for the objective functions of
LGopt~\eqref{eq:LGopt}, pLGopt~\eqref{eq:pLGopt}, QEPmin~\eqref{eq:QEPmin},
and pQEPmin~\eqref{eq:pQEPmin}
are all the same. The proof of Theorem \ref{thm:LGopt=QEPmin}
relies on Theorems \ref{thm:LGopt=pLGopt}, \ref{thm:QEPmin=pQEPmin}, and \ref{thm:pLGopt=pQEPmin}.

\begin{figure}
\begin{center}
\begin{picture}(450,60)(-10,-35)
%\thinlines
%\put(-10,-40){\line(1,0){360}}
%\put(-10,-40){\line(0,1){140}}
%
\put(0,-5){\framebox[1.1\width]{\small CRQopt~\eqref{eq:CRQopt}}}
\put(78,5){\scriptsize Theorem \ref{thm:CRQopt=CQopt}}
\put(70,0){\vector(1,0){60}}
\put(130,-0){\vector(-1,0){60}}
\put(130,-5){\framebox[1.1\width]{\small CQopt~\eqref{eq:CQopt}}}
%
%\put(150,82){\vector(-1,-1){10}}
\put(158,14){\scriptsize Theorem \ref{thm:CQopt=LGopt}}
\put(190,0){\vector(1,1){28}}
\put(218,28){\vector(-1,-1){28}}
%\put(140,-2){\rotatebox{-45}{\small $\Longleftrightarrow$}}
%
\put(220,32){\framebox[1.1\width]{\small LGopt~\eqref{eq:LGopt}}}
\put(220,-35){\framebox[1.1\width]{\small pLGopt~\eqref{eq:pLGopt}}}
\put(300,32){\vector(1,0){65}}
\put(343,32){\vector(-1,0){60}}
\put(300,36){\scriptsize Theorem \ref{thm:LGopt=QEPmin}}

\put(300,-35){\vector(1,0){65}}
\put(348,-35){\vector(-1,0){60}}
\put(300,-30){\scriptsize Theorem \ref{thm:pLGopt=pQEPmin}}

\put(365,32){\framebox[1.1\width]{\small QEPmin~\eqref{eq:QEPmin}}}
\put(365,-35){\framebox[1.1\width]{\small pQEPmin~\eqref{eq:pQEPmin}}}
\put(240,13){\vector(0,-1){38}}
\put(240,-15){\vector(0,1){38}}
\put(245,-5){\scriptsize Theorem \ref{thm:LGopt=pLGopt}}

\put(380,13){\vector(0,-1){38}}
\put(380,-15){\vector(0,1){38}}
\put(385,-5){\scriptsize Theorem \ref{thm:QEPmin=pQEPmin}}
\end{picture}
\end{center}
\caption{Equivalence of optimization problems}
\label{fig:eqiv-opts}
\end{figure}
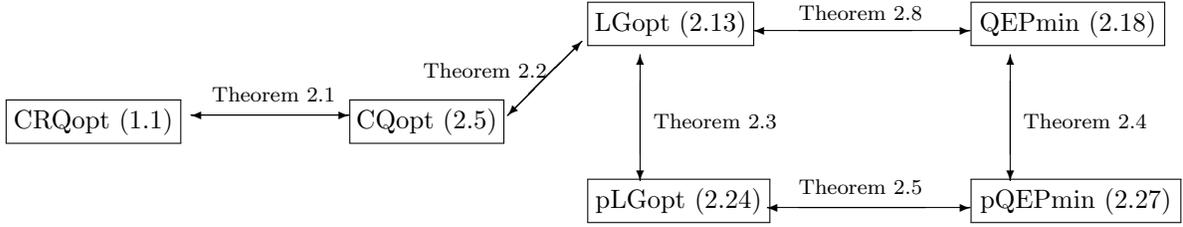

\subsection{Easy and hard cases}\label{sec-crq-theory-easyhard}

%\subsubsection{Definition}
Motivated by the treatments
of the trust-region subproblem \cite{moso:1983,zhsl:2017},
QEPmin \eqref{eq:QEPmin} can be classified into two categories:
the {\em easy\/} case and the {\em hard\/} case, defined as follows.

\begin{definition}\label{dfn:hard-easy}
QEPmin \eqref{eq:QEPmin} is in the {\em hard\/} case
if it has a minimizer $(\lambda_{\ast},z_{\ast})$ with $b_0^{\T}z_{\ast}=0$.
Otherwise, QEPmin \eqref{eq:QEPmin} is in the {\em easy\/} case.
Furthermore, any one of the equivalent optimization problems as shown
in Figure~\ref{fig:eqiv-opts}
is said to be in the {\em hard\/} or {\em easy\/} case
if the corresponding {\rm QEPmin\/} is.
\end{definition}

This notion of hardness and easiness exists has its historical reason in dealing with
the trust-region subproblem. The hard case is not really hard as its name suggests
when it comes to numerical computation. It is just a
degenerate and rare case that needs special attention.
The easy case is a generic one.
Consider the hard case, let $\cV$ be the maximal eigenspace of $PAP$ corresponding to eigenvalue $\lambda_{\ast}$,
then $b_0\,\bot\, \cV$ by Theorem~\ref{thm:hard-case}. This creates difficulties to our later Lanczos method  to solve
QEPmin \eqref{eq:QEPmin} in that the Krylov subspace
$\mathcal{K}_k(PAP,b_0)\subset \cV^{\perp}$ for any $k$. So in theory there is no vector in $\mathcal{K}_k(PAP,b_0)$
can approximate any eigenvector $z\in \cV$ well.

In Theorems~\ref{thm:hard-case} and \ref{thm:hard-case:2} below, we 
present a number of characterizations
about the {\em hard\/} case.

\begin{lemma}\label{rk:hard-case:pQEPmin}
{\rm QEPmin} \eqref{eq:QEPmin} is in the {\em hard\/} case if and only if
{\rm pQEPmin}~\eqref{eq:pQEPmin} has a minimizer $(\lambda_{\ast}, w_*)$ satisfying
$g_0^{\T}w_*=0$.
\end{lemma}

\begin{proof}
To see this, we let $(\lambda_{\ast},z_{\ast})$ be
a minimizer QEPmin \eqref{eq:QEPmin}
satisfying $b_0^{\T}z_{\ast}=0$. By Theorem~\ref{thm:QEPmin=pQEPmin}, we know that
$z_*$ and $w_*$ are related by $z_*=S_1w_*$. Since also $b_0=S_1g_0$,
$b_0^{\T}z_{\ast}=g_0^{\T}w_*$.
\end{proof}

\begin{theorem}\label{thm:hard-case}
Suppose that {\rm QEPmin} \eqref{eq:QEPmin} is in the {\em hard\/} case, and let $(\lambda_{\ast},z_{\ast})$
be a minimizer such that $b_0^{\T}z_{\ast}=0$. Then we have the following statements:
\begin{enumerate}[\rm (1)]
  \item $\lambda_{\ast}=\lambda_{\min}(H)$, the smallest eigenvalue of $H$;
  \item $g_0\,\bot\,\cU$, where $\cU$ is the eigenspace of $H$ associated with its eigenvalue $\lambda_{\min}(H)$;
  \item $b_0\,\bot\,\cV$, where $\cV$ is the eigenspace of $PAP$ associated with its eigenvalue $\lambda_{\min}(H)\in\eig(PAP)$.
\end{enumerate}
\end{theorem}

\begin{proof}
By Lemma~\ref{rk:hard-case:pQEPmin}, {\rm pQEPmin}~\eqref{eq:pQEPmin} has a minimizer $(\lambda_{\ast}, w_*)$ satisfying
$g_0^{\T}w_*=0$. Theorem~\ref{thm:pLGoptunique} immediately leads to
item (1).  Item (2) is a corollary of Lemma~\ref{lm:pLGopt}.

For item (3), it follows from Lemma~\ref{lm:eig(PAP)} that if $\lambda_{\min}(H)\ne 0$, then
$\cV=S_1\cU$.
Since $b_0=S_1g_0$ and $g_0\,\bot\,\cU$ by item (2), we conclude that
$b_0\,\bot\, S_1\cU$ .
If, however, $\lambda_{\min}(H)=0$, then $\cV=S_1\cU+\cR(S_2)$.
Since again $g_0\,\bot\,\cU$ by item (2) and also $b_0\,\bot\, \cR(S_2)$, we still have
$b_0\,\bot\,\cV$.
\end{proof}

\begin{theorem}\label{thm:hard-case:2}
{\rm QEPmin} \eqref{eq:QEPmin} is in the {\em hard\/} case if and only if
\begin{equation}\label{eq:hard-case:character}
g_0\,\bot\,\cU\,\,\mbox{and}\,\,
\|[H-\lambda_{\min}(H) I]^{\dagger} g_0\|_2\le \gamma,
\end{equation}
where $\cU$ is as defined in {\rm Theorem~\ref{thm:hard-case}}.
\end{theorem}

\begin{proof}
If QEPmin \eqref{eq:QEPmin} is in the {\em hard\/} case,
then
its optimal value (which is also the one of LGopt \eqref{eq:LGopt}) $\lambda_*=\lambda_{\min}(H)$. This can only happen when
\eqref{eq:hard-case:character} holds. On the other hand, if \eqref{eq:hard-case:character} holds, then $\lambda_*=\lambda_{\min}(H)$
by Lemma~\ref{lm:pLGopt}. By Theorem~\ref{thm:pLGopt=pQEPmin}, pQEPmin~\eqref{eq:pQEPmin}
has a minimizer $(\lambda_*,w_*)$, where $Hw_*=\lambda_* w_*$. Thus $g_0^{\T}w_*=0$ because $g_0\,\bot\,\cU$ and
$w_*\in\cU$. Hence QEPmin \eqref{eq:QEPmin} is in the {\em hard\/} case by Lemma~\ref{rk:hard-case:pQEPmin}.
\end{proof}

When QEPmin \eqref{eq:QEPmin} is in the easy case, the situation is much simpler to characterize.

\begin{theorem}\label{thm:unique}
{\rm CRQopt} \eqref{eq:CRQopt} has a unique minimizer when {\rm QEPmin} \eqref{eq:QEPmin} is in the easy case.
\end{theorem}

\begin{proof}
Suppose that QEPmin \eqref{eq:QEPmin} is in the easy case.
By Definition~\ref{dfn:hard-easy},
all minimizers $(\lambda_\ast,w_\ast)$ of pQEPmin \eqref{eq:pQEPmin}
satisfy $g_0^{\T}w_\ast\neq 0$. Theorem~\ref{thm:pLGoptunique} guarantees that
pLGopt \eqref{eq:pLGopt} has a unique minimizer.
Consequently,  the
minimizer of LGopt \eqref{eq:LGopt} is unique
by Theorem \ref{thm:LGopt=pLGopt} and so is
the minimizer of CRQopt \eqref{eq:CRQopt}.
\end{proof}

%Later we will introduce   So for any $k$, if we want to
%approximate the eigenvector $z_{\ast}$ in the Krylov subspace $\mathcal{K}_k(PAP,b_0)$,
%the approximation lies in $V^{\perp}$. However, the eigenvector $z_{\ast}$ satisfies
%, so it can never be approximated in the Krylov subspace $\mathcal{K}_k(PAP,b_0)$.

% then let $x=-(PAP-\lambda_{\ast} I)^{\dag} b_0$ and $v = x + \sqrt{\gamma^2-\|x\|^2} \frac{z}{\|z\|}+n_0$
% is a solution of CRQopt \eqref{eq:CRQopt}. Therefore,
% we ``can'' construct
%the solution directly from the eigenvalue and eigenvector, as treated in trust-region subprogram.

%\Blue{Claim: The hard case is a \Red{rare} case. In real applications it happens very rarely.}

%\subsubsection{Connection with trust-region subproblems}
We use the remaining part of this subsection to explain how
CRQopt \eqref{eq:CRQopt} and the well-known
trust-region subproblem (TRS) are related.

We have already proved in
Theorem \ref{thm:CRQopt=CQopt} that CRQopt \eqref{eq:CRQopt} is
equivalent to CQopt \eqref{eq:CQopt}.
Set $u=Pv$. Solving CQopt \eqref{eq:CQopt}
is equivalent to solving
\begin{subequations}\label{eq:CQtrust}
\begin{empheq}[left={\mbox{}}\empheqlbrace]{alignat=2}
\min         ~& u^{\T}PAPu+2u^{\T}b_0,\label{eq:CQtrust-0}\\
\mbox{s.t.}  ~& \|u\|=\gamma, \label{eq:CQtrust-1}\\
             ~& u\in \cN(C^{\T}).\label{eq:CQtrust-2}
\end{empheq}
\end{subequations}
Let $H$ and $g_0$ be defined in \eqref{eq:def:gH} and $S_1$ be defined in \eqref{eq:s1s2}.
Then $u$ is a minimizer of optimization
problem \eqref{eq:CQtrust} if and only if $y=S_1^{\T}u$ is a minimizer of
the following equality constrained optimization problem
\begin{subequations}\label{eq:pCQtrust}
\begin{empheq}[left={\mbox{}\quad}\empheqlbrace]{alignat=2}
\min         ~& y^{\T}Hy+2y^{\T}g_0,\label{eq:pCQtrust-0}\\
\mbox{s.t.}  ~& \|y\|=\gamma. \label{eq:pCQtrust-1}
\end{empheq}
\end{subequations}
The Lagrange equations for \eqref{eq:pCQtrust} is exactly
the same as pLGopt \eqref{eq:pLGopt}.
The problem \eqref{eq:pCQtrust} is similar to TRS
\begin{subequations}\label{eq:pCQtrust1}
\begin{empheq}[left={\mbox{}\quad}\empheqlbrace]{alignat=2}
\min         ~& y^{\T}Hy+2y^{\T}g_0,\label{eq:pCQtrust1-0}\\
\mbox{s.t.}  ~& \|y\|\le\gamma, \label{eq:pCQtrust1-1}
\end{empheq}
\end{subequations}
except that its constraint is an equality instead of an inequality.
When $H$ is not positive definite, solution of \eqref{eq:pCQtrust}
and TRS \eqref{eq:pCQtrust1} are  exactly the same. But when $H$ is positive definite,
we need to check whether $\|H^{-1}g_0\|<\gamma$. If so, $H^{-1}g_0$, instead of the minimizer of \eqref{eq:pCQtrust},
is the minimizer of TRS \eqref{eq:pCQtrust1}. If, however, $\|H^{-1}g_0\|\geq\gamma$,
then the minimizer of TRS \eqref{eq:pCQtrust1} is the same as that of \eqref{eq:pCQtrust}.

Lemma 2.1 in \cite{hage:2001} shows that $y$ is the \eqref{eq:pCQtrust1} of
\eqref{eq:pCQtrust}
if and only if there exists $\widehat{\lambda}\in\bbR$ such that $(\widehat{\lambda},y)$ satisfies the constraints of
pLGopt \eqref{eq:pLGopt} and $H-\widehat{\lambda} I$ is positive semi-definite.
According to Lemma \ref{lm:pLGopt}, the optimal value of pLGopt \eqref{eq:pLGopt}
satisfies {$\lambda_\ast\le\lambda_{\min}(H)$}, which indicates that $H-\lambda_\ast I$ is positive semi-definite.
Therefore, solving the equality constrained problem \eqref{eq:pCQtrust} is equivalent to solving pLGopt \eqref{eq:pLGopt}.

As we have mentioned, the terms ``{\em easy\/}'' and ``{\em hard\/}''
were adopted from the treatments of the trust-region
subproblem \cite{moso:1983,zhsl:2017}, where the term ``easy''
means the associated case is easy to explain, not
implying the case is easy to solve, however.
A more detailed connection with TRS \eqref{eq:pCQtrust1}
is as follows.
\begin{enumerate}
\item In the easy case of QEPmin \eqref{eq:QEPmin}, $b_0^{\T}z_{\ast}\neq 0$ for all mimimizers $(\lambda_\ast,z_\ast)$.
        By Theorem~\ref{thm:QEPmin=pQEPmin}, $z_{\ast}=S_1w_{\ast}$ for some $w_{\ast}\in\bbR^{n-m}$ and thus
        $g_0^{\T}w_{\ast}=b_0^{\T}S_1w_{\ast}=b_0^{\T}z_{\ast}\ne 0$.
By Theorem \ref{thm:pLGoptunique}, $\lambda_\ast<\lambda_{\min}(H)$, and thus $(\lambda_{\ast}, y_{\ast})$
with $y_{\ast}=(H-\lambda_{\ast} I)^{-1}g_0$ is the unique minimizer of
pLGopt \eqref{eq:pLGopt}. Hence $y_{\ast}$ is the unique minimizer
of \eqref{eq:pCQtrust}, {which is related to the easy case of TRS \eqref{eq:pCQtrust1}}.

         \item In the hard case of QEPmin \eqref{eq:QEPmin},
        there exists a minimizer $(\lambda_{\ast},z_{\ast})$
such that $b_0^{\T}z_{\ast}= 0$.
        Again by Theorem~\ref{thm:QEPmin=pQEPmin}, $z_{\ast}=S_1w_{\ast}$ for some $w_{\ast}\in\bbR^{n-m}$ and
        $g_0^{\T}w_{\ast}= 0$.
        %By \eqref{eq:pQEPmin-1}, $(H-\lambda_{\ast} I)^2w_{\ast}=0$,
        %so  $\lambda_{\ast}\in\eig(H)$ and $w_{\ast}$ is the corresponding eigenvector.
        By Theorem~\ref{thm:pLGopt=pQEPmin}, a minimizer
of pLGopt \eqref{eq:pLGopt} is given by
        $$
        (\lambda_{\ast}, y_{\ast})
        \quad\mbox{with}\,\,  y_{\ast}= x_{\ast} + \sqrt{\gamma^2-\|x_{\ast}\|^2}\, \frac{w_{\ast}}{\|w_{\ast}\|},
        $$
        where $x_{\ast}=-(H-\lambda_{\ast} I)^{\dag} g_0$ and it is guaranteed that $\|x_{\ast}\|\le\gamma$.
        Therefore, in general a minimizer of  \eqref{eq:pCQtrust} can be
        expressed by $y_{\ast} = x_{\ast} + \sqrt{\gamma^2-\|x_{\ast}\|^2}\, ({w_{\ast}}/{\|w_{\ast}\|})$,
        {which is related to hard case of TRS \eqref{eq:pCQtrust1}}.

\end{enumerate}
It is known that the generalized Lanczos method does not work for
TRS \eqref{eq:pCQtrust1} in the hard case \cite[Theorem 4.6]{zhsl:2017}.
A restarting strategy was proposed
to overcome the difficulty, but it was commented that the strategy
computationally is very expensive
for large scale problems~\cite[Theorem 5.8]{golr:1999}.

In the next section,
we present that the Lanczos algorithms
for CRQopt~\eqref{eq:CRQopt}, which
resemble the generalized Lanczos
method for TRS and are suitable for handling the easy case.
However, with some additional effort, the hard case can be detected.
In the rest of this article, we mostly focus only on the easy case.

\section{Lanczos algorithm}\label{sec-crq-alg}
%Armed with the theoretical results we have developed in Section~\ref{sec-crq-theory},
%we will present in this section
%Krylov subspace methods to solve LGopt \eqref{eq:LGopt} and
%QEPmin \eqref{eq:QEPmin}.
%By Theorem~\ref{thm:CRQopt=QEPmin}, in turn,
%we solve the original CRQopt \eqref{eq:CRQopt}.

As was shown in Section~\ref{sec-crq-theory}, solving CRQopt \eqref{eq:CRQopt} is
equivalent to solving LGopt \eqref{eq:LGopt} or QEPmin \eqref{eq:QEPmin}.
In this section we present algorithms to solve CRQopt \eqref{eq:CRQopt}
by solving LGopt \eqref{eq:LGopt} and QEPmin \eqref{eq:QEPmin}.
We first review the Lanczos procedure in section \ref{sec-prelim-lan},
then we apply the procedure to reduce
LGopt \eqref{eq:LGopt} and QEPmin \eqref{eq:QEPmin}, and finally solve
the reduced LGopt and QEPmin to yield approximations to the minimizer of the original CRQopt \eqref{eq:CRQopt}.  Besides,
we prove the finite step stopping
property of the proposed algorithms and comment on how to detect
the hard case.

\subsection{Lanczos process}\label{sec-prelim-lan}
We review the standard symmetric Lanczos process \cite{demm:1997,govl:2013,parl:1998,saad:1992}.
 Given a real symmetric matrix $M\in\bbR^{n\times n}$ and
a starting vector $r_0\in\bbR^n$,
the Lanczos process partially computes
the decomposition $MQ=QT$, where $T\in\bbR^{n\times n}$
is symmetric and tridiagonal, $Q\in\bbR^{n\times n}$
is orthogonal and the first column of $Q$ is parallel to $r_0$.

Specifically, let $Q=[q_1,q_2,\ldots,q_n]$ and denote by
$\alpha_i$ for $1\le i\le n$
 the diagonal entries of $T$,
and by $\beta_i$ for $2\le i\le n$
the sub-diagonal and super-diagonal entries of $T$.
The Lanczos process goes as follows: set
$q_1 = r_0/\|r_0\|$, and equate the first column of both sides of the equation $MQ=QT$ to get
\begin{equation}\label{eq:Lanczos-step-1}
Mq_1=q_1\alpha_1+q_2\beta_2.
\end{equation}
Pre-multiply both sides of the equation \eqref{eq:Lanczos-step-1} by $q_1^{\T}$ to get $\alpha_1=q_1^{\T}Mq_1$, and then let
$$
\widehat q_2=Mq_1-q_1\alpha_1,\,\,
\beta_2=\|\widehat q_2\|.
$$
Now if $\beta_2>0$, we set $q_2=\widehat q_2/\beta_2$; otherwise the process breaks down. In general for $j\ge 2$, equating the $j$th
column of both sides of the equation $MQ=QT$ leads to
\begin{equation}\label{eq:Lanczos-step-j}
Mq_j=q_{j-1}\beta_j+q_j\alpha_j+q_{j+1}\beta_{j+1}.
\end{equation}
Up to this point, $q_i$ for $1\le i\le j$, $\alpha_i$ for $1\le i\le j-1$, and $\beta_i$ for $2\le i\le j$ have already
been determined. Pre-multiply both sides of the equation \eqref{eq:Lanczos-step-j} by $q_j^{\T}$ to get $\alpha_j=q_j^{\T}Mq_j$, and then let
$$
\widehat q_{j+1}=Mq_j-q_{j-1}\beta_j-q_j\alpha_j,\,\,
\beta_{j+1}=\|\widehat q_{j+1}\|.
$$
Now if $\beta_{j+1}>0$, we set $q_{j+1}=\widehat q_{j+1}/\beta_{j+1}$; otherwise the process breaks down. The process can be compactly
expressed
by\footnote {We sacrifice
             slightly mathematical rigor in writing \eqref{eq:Lanczos-compact} in exchange for simplicity and convenience, since
             $q_{k+1}$ cannot be determined unless also $\beta_{k+1}>0$.}
\begin{equation}\label{eq:Lanczos-compact}
MQ_k=Q_kT_k+\beta_{k+1} q_{k+1}e_k^{\T}
\end{equation}
assuming the process encounters no breakdown for the first $k$ steps, i.e., no $\beta_i=0$ for $2\le i\le k$, where
$$
Q_k=[q_1,q_2,\ldots,q_k],\,\,
T_k=Q_k^{\T} MQ_k=\left[\begin{array}{ccccc}
      \alpha_1& \beta_2 & &  & \\
      \beta_2 & \alpha_2 & \beta_3 &  &  \\
      & \ddots & \ddots & \ddots & \\
      &  & \beta_{k-1}& \alpha_{k-1} & \beta_k \\
      &  &  & \beta_k & \alpha_k
      \end{array}\right].
$$
Furthermore, the column space $\cR(Q_k)$ is the same as the $k$th Krylov subspace
$$
\mathcal{K}_k(M,r_0):=\subspan(r_0, Mr_0,\cdots,M^{k-1} r_0).
$$
In the case of a breakdown with $\beta_{k+1}=0$,
$MQ_k=Q_kT_k$ and $\cR(Q_k)$ is an invariant subspace of $M$.

\subsection{Solving LGopt  }\label{sec-crq-alg-projlag}
In this subsection, we first use \eqref{eq:Lanczos-compact}  obtained by the Lanczos process with $M=PAP$
to reduce LGopt~\eqref{eq:LGopt}, and then
solve the reduced LGopt via an approach based on a secular equation solver.

\subsubsection{Dimensional reduction of LGopt  }
For the dimensional reduction of LGopt \eqref{eq:LGopt},
we restate the Lagrange equations \eqref{eq:LGopt-1}
and \eqref{eq:LGopt-1}  here
\begin{equation}\label{eq:LGoptre}
(PAP-\lambda I)u=-b_0,\quad \|u\|=\gamma,\quad Pu=u,
\end{equation}
where we include the constraint $Pu=u$ since
we are only interested in those vectors $u\in\cN(C^{\T})$.

Apply the Lanczos process with
$M=PAP$ and the starting vector $r_0=b_0$ to get \eqref{eq:Lanczos-compact} with $M=PAP$. It then follows that
for any scalar $\lambda$
\[
Q_k^{\T}(PAP-\lambda I)Q_k=T_k-\lambda I
\quad \mbox{and} \quad
Q_k^{\T}b_0=\|b_0\|e_1.
\]
Consequently, we arrive at the reduced  LGopt \eqref{eq:LGopt}
\begin{subequations}\label{eq:rLGopt}
\begin{empheq}[left={\mbox{rLGopt:}\quad}\empheqlbrace]{alignat=2}
\min ~& \lambda  \label{eq:rLGopt-0}\\
\mbox{s.t.} ~&(T_k -\lambda I) x =-\|b_0\|e_1, \label{eq:rLGopt-1} \\
~& \|x\| =\gamma.  \label{eq:rLGopt-2}
\end{empheq}
\end{subequations}

A couple of comments are for the efficiency of the Lanczos process with
$M=PAP$. In the process, we have to calculate matrix-vector products $Mx=P(A(Pq_j))$ efficiently.
For that purpose, it suffices for us to
be able to calculate the product $Pc$ efficiently for any given $c\in\bbR^n$.
In fact
$$
Pc=q_j-CC^{\dag}c=q_j-Cy,
$$
where $y=C^{\dag}c$ is the minimum-norm solution
of the least squares problem
\begin{equation}\label{eq:lsp}
y=\arg\min_{z\in\bbR^m}\|Cz-c\|_2,
\end{equation}
which can be computed by using
the QR decomposition of $C \in\bbR^{n\times m}$
or an iterative method such as LSQR \cite{fosa:2011,pasa:1982,saun:1995}.
Another cost-saving observation due to \cite{gozz:2000}
is that for the matrix-vector product
$Mq_j=P(A(Pq_j))$, the first application of $P$ in $Pq_j$
can be skipped due to the fact that
if the initial vector $b_0\in\cN(C^{\T})$, then
$Pq_j=q_j$ for all $1\le j\le k+1$.

We end this subsection by pointing out rLGopt \eqref{eq:rLGopt}
cannot fall into the hard case. The same phenomenon happens to the tridiagonal TRS generated by the generalized Lanczos method \cite[Theorem 5.3]{golr:1999} as well.
Let the eigen-decomposition of $T_k$ be
\begin{equation}\label{eq:EigD4Tk}
T_k=Y\Theta Y^{\T},\,\, Y^{\T}Y=I_k,\,\,\Theta=\diag(\vartheta_1,\vartheta_2,\ldots,\vartheta_k),
\end{equation}
where we suppress the dependency of $Y$, $\Theta$, and $\vartheta_j$ on $k$ for notational
convenience. Further, we arrange $\vartheta_j$ in nondecreasing order, i.e.,
$\vartheta_1\le\vartheta_2\le\cdots\le\vartheta_k$ and
$Y=[y_1,y_2,\cdots,y_k]$.
Let $\mu^{(k)}$ be the optimal value of rLGopt \eqref{eq:rLGopt}.

\begin{theorem}\label{thm:nondegenerate0}
Suppose that $\beta_{j}\neq 0$ for $j=2,3,\ldots, k$ in
the Lanczos process.  Then $\mu^{(k)}<\vartheta_1 \equiv \lambda_{\min}(T_k)$, and
{\rm rLGopt} cannot fall into the hard case.
\end{theorem}

\begin{proof}
It is well-known that the first components of
all eigenvectors $y_i$ of irreducible $T_k$
are nonzero \cite[p.140]{parl:1998}. In particular,
$e_1^{\T}y_1\neq 0$. Lemma \ref{lm:pLGopt} immediately leads to
$\mu^{(k)}<\vartheta_1$.

Since $\mu^{(k)}<\lambda_{\min}(T_k)$ by Theorem \ref{thm:hard-case}(1),
we conclude that rLGopt cannot fall into the hard case.
\end{proof}

\subsubsection{Solving rLGopt }

Now we explain how to solve rLGopt \eqref{eq:rLGopt}. Suppose that $\beta_{j}\neq 0$ for $j=2,3,\ldots, k$, and let
the eigen-decomposition of $T_k$ be given by \eqref{eq:EigD4Tk}.

\begin{theorem}\label{thm:nondegenerate}
The optimal value $\mu^{(k)}$ of {\rm rLGopt} \eqref{eq:rLGopt}
is the smallest root of the secular
function
\begin{equation}\label{eq:seculartk}
\what\chi(\lambda)=\|b_0\|^2e_1^{\T}(T_k-\lambda I)^{-2}e_1-\gamma^2= \sum_{i=1}^k\frac{\zeta_i^2}{(\lambda-\vartheta_i)^2}-\gamma^2,
\end{equation}
where $\zeta_i=\|b_0\|e_1^{\T}y_i$ for $i=1,2,\cdots,k$. Furthermore,
\begin{equation}  \label{eq:rQEPminsol}
(\mu^{(k)},x^{(k)}) =
(\mu^{(k)},\, -\|b_0\|(T_k-\mu^{(k)} I)^{-1}e_1)
\end{equation}
is a  minimizer of {\rm rLGopt} \eqref{eq:rLGopt}.
\end{theorem}

\begin{proof}
{\rm rLGopt} \eqref{eq:rLGopt} takes the same form as pLGopt~\eqref{eq:pLGopt}. By Theorem~\ref{thm:nondegenerate0},
$\mu^{(k)}<\lambda_{\min}(T_k)$. The conclusions of the lemma are now consequences of Lemma~\ref{lm:pLGopt}.
\end{proof}

Theorem~\ref{thm:nondegenerate} naturally leads to a method for solving rLGopt \eqref{eq:rLGopt}
through calculating the smallest root of the secular
function $\what\chi(\lambda)$. Algorithm~\ref{eq:rLGopt} outlines the method,
based on an efficient secular equation solver in Appendix~\ref{sec:secularEq}.

%\Red{
%We have proved the optimal value of rLGopt \eqref{eq:rLGopt} is the
%smallest root of secular function $\widehat{\chi}(\lambda)$.
%Therefore, we can compute
%a full eigendecomposition of $T_k$ and solve the smallest zero of the
%secular function  $\widehat{\chi}(\lambda)$.
%Let $\mu^{(k)}$ be the smallest zero, then
%\begin{equation}  \label{eq:rQEPminsol}
%(\mu^{(k)},x^{(k)}) =
%(\mu^{(k)},\, -\|b_0\|(T_k-\mu^{(k)} I)^{-1}e_1)
%\end{equation}
%is a  solution of rLGopt \eqref{eq:rLGopt}.
%We show an algorithm to solve the zero of secular function in
%Algorithm \ref{alg:LGopt}. Details for the algorithm is
%in Appendix~\ref{sec:secularEq}. Note that let $j_0$ be defined
%in \eqref{eq:j0} in Appendix~\ref{sec:secularEq}, then we have
%shown in the proof of Theorem \ref{thm:nondegenerate} that $j_0=1$ since $e_1^{\T}y_1\neq 0$.
%Algorithm 1 is an algorithm  for solving rLGopt \eqref{eq:rLGopt}. In the algorithm Lines 1-2 are for construction of the secular equation and Lines 3-16 are for solving the secular equation.}

\begin{algorithm}
\caption{Solving rLGopt \eqref{eq:rLGopt}}
\label{alg:LGopt}
\begin{algorithmic}[1]
\Require $T_k\in\bbR^{k\times k}$, $\|b_0\|$, $\gamma>0$, and tolerance $\epsilon$;
\Ensure $(\mu^{(k)}, x^{(k)})$, approximate minimizer of rLGopt~\eqref{eq:rLGopt};
            \State Compute the eigenvalues $\theta_1\le\theta_2\le\cdots\le\theta_k$ of $T_k$ and the corresponding eigenvectors $y_1,\cdots,y_k$;
            \State  $\xi_i\gets \|b_0\|e_1^{\T}y_i$ for $i= 1,2,\cdots,k$;
            \State  $\delta_0\gets \frac{1}{\gamma}\sqrt{\sum_{i= 1}^k\xi_i^2}$, $\alpha^{(0)}\gets \theta_1-\delta_0$, $\beta^{(0)}\gets \theta_1$ and
              $\eta\gets \gamma^2-\sum_{i= 2}^k \frac{\xi_i^2}{([\theta_1-\delta_0]-\theta_i)^2}$;
               \State {\bf if $\eta>0$ then $\lambda^{(0)}\gets \theta_1-|\xi_1|/\sqrt{\eta}$ else $\lambda^{(0)}\gets \theta_1-\delta_0/2$};
            \For {$j = 0,1,2,\ldots$}
            \State  $\chi\gets \sum_{i= 1}^k\frac{\xi_i^2}{(\lambda^{(j)}-\theta_i)^2}-\gamma^2$;
            \State {\bf if $\chi>0$ then $\alpha^{(j+1)}\gets \alpha^{(j)}$, $\beta^{(j+1)}\gets \lambda^{(j)}$ else $\alpha^{(j+1)}\gets \lambda^{(j)}$, $\beta^{(j+1)}\gets \beta^{(j)}$};
            \State  $a\gets (\lambda^{(j)}-\theta_1)^3\sum_{i= 1}^n\frac{\xi_i^2}{(\lambda^{(j)}-\theta_i)^3}$, $b\gets (\lambda^{(j)}-\theta_1)\sum_{i= 1}^n\frac{\xi_i^2}{(\lambda^{(j)}-\theta_i)^3}-\chi$;
            \If{$b>0$}
            \State  $\lambda_1\gets \theta_1-\sqrt{a/b}$;
            \State {\bf if $\lambda_1\in(\alpha^{(j+1)},\beta^{(j+1)})$ then $\lambda^{(j+1)}\gets \lambda_1$ else $\lambda^{(j+1)}\gets (\alpha^{(j+1)}+\beta^{(j+1)})/2$};
            \Else
            \State  $\lambda^{(j+1)}\gets (\alpha^{(j+1)}+\beta^{(j+1)})/2$;
            \EndIf
            \State {\bf if $|\lambda^{(j+1)}-\lambda^{(j)}|<\epsilon$ then stop};
            \EndFor
\State {\bf return}
$(\mu^{(k)},x^{(k)}) = (\lambda^{(j+1)}, -(T_k-\mu^{(k)} I)^{-1}\|b_0\|e_1)$
as a solution of rLGopt~\eqref{eq:rLGopt}.
%        \EndFunction
\end{algorithmic}
\end{algorithm}

%Note that we have shown in Theorem \ref{thm:nondegenerate}
%that $e_1^{\T}y_1\neq 0$,
%where $y_1$ be the eigenvector corresponding to the smallest real
%eigenvalue of $T_k$.
Although Theorem \ref{thm:nondegenerate} assures us that the hard case cannot happen for rLGopt \eqref{eq:rLGopt},
cases where $|e_1^{\T}y_1|$ is very tiny are possible. Such a nearly hard case has to be treated with care,
a subject of further future study.

\begin{remark}\label{rmk:gltrs}
Let us discuss the relationship between
solving rLGopt \eqref{eq:rLGopt} and
solving TRS by a generalized Lanczos (GLTRS) method proposed in \cite{golr:1999}.
GLTRS projects a similar  problem to \eqref{eq:CQtrust-0} and \eqref{eq:CQtrust-1} by
a Krylov subspace to yield a small-size problem.
Ignoring \eqref{eq:CQtrust-2} for the moment, we run  the Lanczos process with $M=PAP$
and the starting vector be $r_0=b_0$ to generate the orthonormal basis matrix
$Q_k$ and the tridiagonal matrix $T_k$. Since $b_0\in\cN(C^{\T})$, it can be verified that
$\cR(Q_k)\subset\cN(C^{\T})$, which means that \eqref{eq:CQtrust-2} is automatically taken care of.
Project  \eqref{eq:CQtrust-0} and \eqref{eq:CQtrust-1} onto
the column space of $Q_k$ and we arrive at the following equality
constrained optimization problem:
\begin{subequations}\label{eq:rCQtrust}
\begin{empheq}[left=\empheqlbrace]{alignat=2}
\min         ~& x^{\T}T_kx+2x^{\T}\|g_0\|e_1,\label{eq:rCQtrust-0}\\
\mbox{s.t.}  ~& \|x\|=\gamma. \label{eq:rCQtrust-1}
\end{empheq}
\end{subequations}
Problem \eqref{eq:rCQtrust} is similar to the tridiagonal TRS generated
by GLTRS except that the constraint here is equality instead of inequality.
Solving \eqref{eq:rCQtrust} by the method of the Lagrangian multipliers leads to   exactly
rLGopt \eqref{eq:rLGopt}.
\hfill $\Box$
\end{remark}

\subsubsection{Solving LGopt }
After computing $(\mu^{(k)},x^{(k)})$, the minimizer
of rLGopt \eqref{eq:rLGopt}, we deduce an approximate minimizer of LGopt~\eqref{eq:LGopt}:
\begin{equation}\label{eq:approx-sol:LGopt}
(\mu^{(k)},u^{(k)}) = (\mu^{(k)}, Q_kx^{(k)})
\end{equation}
%which in turn leads to an approximate minimizer
%$n_0+u^{(k)}$  of CRQopt~\eqref{eq:CRQopt}.
It can be verified that
\begin{equation}\label{eq:k-LGopt}
\|u^{(k)}\|=\|x^{(k)}\|=\gamma,\,\,
u^{(k)}\in\mathcal{R}(Q_k)\subset\cN(C^{\T}).
\end{equation}
That is the pair in \eqref{eq:approx-sol:LGopt} satisfies
the constraints \eqref{eq:LGopt-2} and \eqref{eq:LGopt-3}.

The accuracy of this approximate minimizer
$(\mu^{(k)},u^{(k)})$ can be measured by the residual vector
\begin{equation}\label{eq:res-LGopt}
r_k^{\LGopt}=(PAP-\mu^{(k)} I)u^{(k)}+b_0.
\end{equation}
For simplicity, we may assume that $(\mu^{(k)},x^{(k)})$ satisfies
the constraint of rLGopt~\eqref{eq:rLGopt} exactly,
in particular $(T_k -\mu^{(k)} I) x^{(k)} =-\|b_0\|e_1$,
since it is reasonable to assume that the error in
 $(\mu^{(k)},u^{(k)})$ as an approximate minimizer
of LGopt~\eqref{eq:LGopt} is much larger than the error in $(\mu^{(k)},x^{(k)})$ as the computed minimizer
of rLGopt \eqref{eq:rLGopt}.
Subsequently, we have the following expression for the
residual vector $r_k^{\LGopt}$, similar to the one on
the generalized Lanczos method for TRS \cite{golr:1999}.

\begin{proposition}\label{prop:res-LGopt}
Suppose that the approximate minimizer $(\mu^{(k)},x^{(k)})$
of {\rm rLGopt}~\eqref{eq:rLGopt} satisfies
the constraints of {\rm rLGopt}~\eqref{eq:rLGopt} exactly. We have
\begin{equation}\label{eq:res-LGopt:1}
r_k^{\LGopt} = \beta_{k+1} q_{k+1}e_k^{\T}x^{(k)}.
\end{equation}
\end{proposition}
\begin{proof}
We have by \eqref{eq:Lanczos-compact}
\begin{align*}
r_k^{\LGopt}&=(PAP-\mu^{(k)} I)Q_kx^{(k)}+b_0  \\
     &=[Q_k(T_k-\mu^{(k)} I)+\beta_{k+1} q_{k+1}e_k^{\T}]x^{(k)}+b_0  \\
     &=-Q_k\|b_0\|e_1+\beta_{k+1} q_{k+1}e_k^{\T}x^{(k)}+b_0  \\
     &=\beta_{k+1} q_{k+1}e_k^{\T}x^{(k)},
\end{align*}
as was to be shown.
\end{proof}

In deciding if $r_k^{\LGopt}$ is sufficiently small,
a sensible way is to check some kind of
normalized residual. In view of  \eqref{eq:res-LGopt},
a reasonable one is
\begin{equation}\label{eq:NRes-LGopt}
\NRes_k^{\LGopt}:=\frac {\|r_k^{\LGopt}\|}{(\|A\|+|\mu^{(k)}|)\|x^{(k)}\|+\|b_0\|}
     =\frac {|\beta_{k+1}|\, |e_k^{\T}x^{(k)}|}{(\|A\|+|\mu^{(k)}|)\|x^{(k)}\|+\|b_0\|}=:\delta_k^{\LGopt}.
\end{equation}
The Lanczos process is stopped
if $\delta_k^{\LGopt}\le\epsilon$, a prescribed tolerance.
In summary, the Lanczos algorithm for
solving LGopt~\eqref{eq:LGopt} is given in Algorithm~\ref{alg:solveLGopt}.

\begin{algorithm}
\caption{Solving LGopt \eqref{eq:LGopt}}
\label{alg:solveLGopt}
\begin{algorithmic}[1]
\Require
$A\in\bbR^{n\times n}$, $C\in\bbR^{n\times m}$, $b_0\in\bbR^n$, $\gamma>0$,  and tolerance $\epsilon$;
\Ensure $(\mu^{(k)}, u^{(k)})$,  approximate minimizer of LGopt~\eqref{eq:LGopt};
\State $\beta_1 \gets \|b_0\|$;
\State {\bf if $\beta_1=0$ then stop};
\State $q_1 \gets r_0/\beta_1$, $q_0 \gets 0$;
\For {$k=1,2,\ldots$}
\State $\hat q\gets Aq_k$, $\hat q\gets P\hat q$, $\hat q\gets \hat q-\beta_k q_{k-1}$;
\State $\alpha_k\gets q_k^{\T} \hat q$, $\hat q\gets \hat q-\alpha_k q_k$, $\beta_{k+1} \gets \|\hat q\|$;

\State compute the minimizer $(\mu^{(k)},x^{(k)})$ of
rLGopt~\eqref{eq:rLGopt} by Algorithm \ref{alg:LGopt};

\State {\bf if $\delta_k^{\LGopt}\le\epsilon$ then stop};

\State $q_{k+1} \gets \hat q/\beta_{k+1}$;

\EndFor
\State $Q_k=[q_1,q_2,\cdots,q_k]$;
\State {\bf return} $(\mu^{(k)}, u^{(k)})$ with $u^{(k)}=Q_kx^{(k)}$
as an approximate minimizer of LGopt~\eqref{eq:LGopt}.
\end{algorithmic}
\end{algorithm}

\subsection{Solving QEPmin }\label{sec-crq-alg-projqep}
In this section, we propose our Lanczos algorithm for the numerical solution of  QEPmin~\eqref{eq:QEPmin}.
It follows the same idea as the previous subsection. First, we reduce
QEPmin~\eqref{eq:QEPmin}   to a smaller problem by projection, and then solve
the reduced QEPmin by an eigensolver.
%There are several advantages to solving CRQopt~\eqref{eq:CRQopt}
%via solving QEPmin~\eqref{eq:QEPmin}.
One immediate advantage of doing so is the availability of
mature eigensolvers for use to solve the underlying QEP.
Independently, QEPmin \eqref{eq:QEPmin} is of interest of its own, e.g., it plays a role in
solving the total least square problems \cite{lav:2007,sihg:2004}.

\subsubsection{Dimensional reduction of QEPmin}
The Lanczos process is natural as a method to solve QEP \eqref{eq:QEPmin-1} for its leftmost eigenvalue
and the corresponding eigenvector.
For convenience, we restate QEP \eqref{eq:QEPmin-1} here:
\begin{equation}\label{eq:QEP2proj}
(PAP-\lambda I)^{2}z=\gamma^{-2}b_0b_0^{\T}z, \quad Pz=z.
\end{equation}
Note that we have added the constraint $Pz=z$ since
we are only interested in those eigenvectors $z\in\cN(C^{\T})$.

Now we discuss how to perform the
dimensional reduction of the QEP \eqref{eq:QEP2proj}
via the projection onto the Krylov subspace generated
by the Lanczos process described in Section~\ref{sec-prelim-lan}.
Let $Q_k$ be the orthogonal matrix and $T_k$ be
the tridiagonal matrix generated by
$k$ steps of the Lanczos process with the matrix
$M=PAP$ and the starting vector $b_0$. We will again have \eqref{eq:Lanczos-compact}, i.e.,
\begin{equation}\label{eq:b0b0:proj}
PAPQ_k=Q_kT_k+\beta_{k+1} q_{k+1}e_k^{\T}
\quad \mbox{and} \quad
Q_k^{\T}b_0b_0^{\T}Q_k=\|b_0\|^2 e_1e_1^{\T}.
\end{equation}
By a straightforward calculation, we have
\begin{align}
(PAP-\lambda I)^{2}Q_k
  &=(PAP-\lambda I)\big[Q_k(T_k-\lambda I)+\beta_{k+1} q_{k+1}e_k^{\T}\big]  \nonumber\\
  &=\big[Q_k(T_k-\lambda I)+\beta_{k+1} q_{k+1}e_k^{\T}\big](T_k-\lambda I)+(PAP-\lambda I)\beta_{k+1} q_{k+1}e_k^{\T}  \nonumber\\
  &=Q_k(T_k-\lambda I)^2+\beta_{k+1}
q_{k+1}e_k^{\T}(T_k-\lambda I)+\beta_{k+1}(PAP-\lambda I) q_{k+1}e_k^{\T}
     \label{eq:(PAP)sq-proj}
\end{align}
and
\begin{align}
Q_k^{\T}(PAP-\lambda I)^{2}Q_k
  &=(T_k-\lambda I)^2+0+\beta_{k+1}Q_k^{\T}(PAP-\lambda I) q_{k+1}e_k^{\T} \nonumber \\
  &=(T_k-\lambda I)^2+\beta_{k+1}\big[Q_k(T_k-\lambda I)+\beta_{k+1} q_{k+1}e_k^{\T}\big]^{\T} q_{k+1}e_k^{\T} \nonumber\\
  &=(T_k-\lambda I)^2+\beta_{k+1}^2e_ke_k^{\T}. \label{eq:LHS-proj}
\end{align}
By \eqref{eq:b0b0:proj} and \eqref{eq:LHS-proj},
naturally one would like to take the reduced QEP \eqref{eq:QEP2proj} to be
\begin{equation}\label{eq:QEPproj'ed0}
\big[(T_k-\lambda I)^2+\beta_{k+1}^2e_ke_k^{\T}\big]w={\gamma^{-2}}{\|b_0\|^2}e_1e_1^{\T} w.
\end{equation}
Unfortunately, this reduced QEP may not have any real eigenvalue,
not to mention that the leftmost eigenvalue is guaranteed to be real, as demonstrated by
Example~\ref{ex:leftmost} below.
To overcome it, we propose to drop
the term $\beta_{k+1}^2e_ke_k^{\T}$
in \eqref{eq:LHS-proj} and use the following
reduced QEP
\begin{equation}\label{eq:QEPproj'ed}
(T_k-\lambda I)^2w={\gamma^{-2}}{\|b_0\|^2}e_1e_1^{\T} w.
\end{equation}
Since it has the same form as the QEP in pQEPmin~\eqref{eq:pQEPmin-1},
the leftmost eigenvalue of the reduced QEP \eqref{eq:QEPproj'ed}
is guaranteed to be real by Theorem~\ref{thm:pQEPmin=lefteig}.
%Besides this, the impact of dropping the last term in \eqref{eq:LHS-proj} on approximation accuracy is likely negligible.
%This is because when an eigenpair $(\lambda,w)$ of \eqref{eq:QEPproj'ed0}
%or \eqref{eq:QEPproj'ed} converges to one of QEP \eqref{eq:QEP2proj}, $\beta_{k+1}e_k^{\T}w$ is tiny, making the \Blue{norm of the} dropping term
%$\beta_{k+1}^2e_ke_k^{\T}w$ negligible.

It can be seen that the corresponding reduced QEPmin \eqref{eq:QEPmin} to  QEP~\eqref{eq:QEPproj'ed} is given by
\begin{subequations}\label{eq:rQEPmin}
\begin{empheq}[left={\mbox{rQEPmin:}\quad}\empheqlbrace]{alignat=2}
\min ~& \lambda  \label{eq:rQEPmin-0}\\
\mbox{s.t.} ~&(T_k-\lambda I)^2w={\gamma^{-2}}{\|b_0\|^2}e_1e_1^{\T}w, \label{eq:rQEPmin-1} \\
~& \lambda\in\bbR, w\ne 0.  \label{eq:rQEPmin-2}
\end{empheq}
\end{subequations}
We note that the Lanczos process of $PAP$ on $b_0$
is the same as, upon a linear transformation by $S_1^{\T}$,
that of $H$ on $g_0$ in pQEPmin~\eqref{eq:pQEPmin}. Therefore,
rQEPmin \eqref{eq:rQEPmin} can be viewed as a reduced-form
of pQEPmin~\eqref{eq:pQEPmin}. %In addition, by Theorem \ref{thm:pLGopt=pQEPmin}, we can immediately conclude that rQEPmin \eqref{eq:rQEPmin} is equivalent to rLGopt \eqref{eq:rLGopt}.

\begin{example}\label{ex:leftmost}
Let $A=\diag(1,2,3,4,5)$,
$C=[0.65,1,0.68,1.13,-0.23]^{\T}$
and $b= [1]$. The  eigenvalues of QEP~\eqref{eq:QEPmin-1} and \eqref{eq:QEPmin-2} in QEPmin, computed by MATLAB, are
$$
{\tt 0.8333},\,\,
{\tt 1.6493},\,\,
{\tt 2.0000},\,\,
{\tt 2.9916 \pm 0.2369i},\,\,
{\tt 3.8786},\,\,
{\tt 4.8236},\,\,
{\tt 5.1196}.
$$
We see the leftmost eigenvalue {\tt 0.8333}$\in\bbR$.
Apply the Lanczos process with $k=2$ leads to a $2\times 2$
QEP \eqref{eq:QEPproj'ed0} whose eigenvalues are computed to be
%with two pairs of complex conjugate eigenvalues:
$$
{\tt 1.8124 \pm 0.4172i},\,\,{\tt 3.3714 \pm 0.2547i},
$$
both are genuine complex numbers!
In contrast, the eigenvalues of QEP~\eqref{eq:QEPproj'ed}
are
$$
{\tt 1.1429},\,\,
{\tt 2.2661},\,\,
{\tt 2.8915},\,\,
{\tt 4.0672},
$$
all of which are real.
\end{example}

\subsubsection{Solving rQEPmin}
To solve rQEPmin \eqref{eq:QEPproj'ed},  we first linearize it into a linear eigenvalue problem (LEP).
The reader is referred to  \cite[Chapter 1]{golr:1982} for many different
ways to linearize a general polynomial eigenvalue problem. Our rQEPmin \eqref{eq:QEPproj'ed} takes
a rather particular form, and we use similar ideas but slightly  different linearization.
Specifically, we let $y=(T_k-\lambda I)w$ and $s=\begin{bmatrix}y\\ w\end{bmatrix}$. Then
QEP \eqref{eq:rQEPmin-1}  can be converted to
the following LEP:
\begin{equation}\label{eq:QEPproj'ed-lin}
\begin{bmatrix}
T_k&-{\gamma^{-2}}{\|b_0\|^2}e_1e_1^{\T}\\-I&T_k\end{bmatrix}s=\lambda s.
\end{equation}
At this point, one can use a standard eigensolver %such as MATLAB's {\tt eig} 
to find the leftmost real eigenvalue $\mu^{(k)}$
of LEP \eqref{eq:QEPproj'ed-lin} and
its corresponding eigenvector
$s^{(k)}=\begin{bmatrix}y^{(k)}\\ w^{(k)}\end{bmatrix}$.
Subsequently, an approximate optimizer of rQEPmin \eqref{eq:rQEPmin}
is given by $(\mu^{(k)}, w^{(k)})$.

\subsubsection{Solving QEPmin}
The minimizer $(\mu^{(k)}, w^{(k)})$
of rQEPmin \eqref{eq:rQEPmin} yields an approximate minimizer of
QEPmin~\eqref{eq:QEPmin} as
\begin{equation}\label{eq:approx-eigenpair:QEPmin}
(\mu^{(k)},z^{(k)} )
= (\mu^{(k)}, Q_k w^{(k)}).
\end{equation}
The accuracy of this pair $(\mu^{(k)},z^{(k)} )$ as an approximate minimizer
can be measured by the norm of the following
the residual vector
\begin{equation}\label{eq:rk}
r_k^{\QEPmin} =\left(PAP - \mu^{(k)} I\right)^2 z^{(k)} - \gamma^{-2}b_0b_0^{\T}z^{(k)}.
\end{equation}
The following proposition shows that this residual vector
can be efficiently obtained during computation.

\begin{proposition}\label{prop:res-QEPmin}
Suppose that $(\mu^{(k)},w^{(k)})$ is an exact minimizer of
{\rm rQEPmin} \eqref{eq:rQEPmin} and $y^{(k)}=(T_k-\mu^{(k)} I)w^{(k)}$. Then
\begin{equation}\label{eq:res-QEPmin}
r_k^{\QEPmin} = \beta_{k+1} q_{k+1}e_k^{\T}y^{(k)}+\beta_{k+1} (PAP-\mu^{(k)} I) q_{k+1}e_k^{\T}w^{(k)}.
\end{equation}
\end{proposition}
\begin{proof}
Keeping \eqref{eq:(PAP)sq-proj} in mind, we find that
\begin{align*}
r_k^{\QEPmin}&=\left(PAP - \mu^{(k)} I\right)^2 Q_k w^{(k)} - \gamma^{-2}b_0b_0^{\T}Q_k w^{(k)} \\
  &\stackrel{{\text{ \eqref{eq:(PAP)sq-proj}}}}{=}Q_k(T_k-\mu^{(k)} I)^2w^{(k)}+\beta_{k+1} q_{k+1}e_k^{\T}(T_k-\mu^{(k)} I)w^{(k)}\\
  &\qquad+\beta_{k+1}({PAP}-\mu^{(k)} I) q_{k+1}e_k^{\T}w^{(k)}  - Q_k\frac {\|b_0\|^2}{\gamma^2}e_1e_1^{\T}w^{(k)} \\
  &\stackrel{{\text{ \eqref{eq:rQEPmin-1}}}}{=}\beta_{k+1} q_{k+1}e_k^{\T}(T_k-\mu^{(k)} I)w^{(k)}+\beta_{k+1}({PAP}-\mu^{(k)} I) q_{k+1}e_k^{\T}w^{(k)} \\
  &=\beta_{k+1} q_{k+1}e_k^{\T}y^{(k)}+\beta_{k+1}({PAP}-\mu^{(k)} I) q_{k+1}e_k^{\T}w^{(k)},
\end{align*}
as expected.
\end{proof}

We note that if the $(k+1)$st step are carried out in the
Lanczos process \eqref{eq:Lanczos-compact},
then the term $(PAP-\mu^{(k)} I) q_{k+1}$
in \eqref{eq:res-QEPmin} can be expressed as a
linear combination of $q_k$, $q_{k+1}$, and $q_{k+2}$.
We propose to use the following normalized residual norm as a stopping criterion for the
Lanczos process:
\begin{subequations}\label{eq:NRes-QEPmin}
\begin{align}
\NRes_k^{\QEPmin} & :=
\frac {\|r_k^{\QEPmin}\|}
{[(\|A\|+|\mu^{(k)}|)^2+\gamma^{-2}\|b_0\|^2\big]\|w^{(k)}\|_2} \label{eq:NRes-QEPmin-1} \\
& \le \frac{|\beta_{k+1}|\big[ |e_k^{\T}y^{(k)}|+
(\|A\|+|\mu^{(k)}|)\,|e_k^{\T}w^{(k)}|\big]}
{[(\|A\|+|\mu^{(k)}|)^2+\gamma^{-2}\|b_0\|^2\big]\|w^{(k)}\|_2}
=:\delta_k^{\QEPmin}.    \label{eq:NRes-QEPmin-2}
\end{align}
\end{subequations}
The Lanczos algorithm for solving QEPmin~\eqref{eq:QEPmin} is summarized in Algorithm~\ref{alg:solveQEP}.

\begin{algorithm}
\caption{Solving QEPmin~\eqref{eq:QEPmin}}
\label{alg:solveQEP}
\begin{algorithmic}[1]
\Require $A\in\bbR^{n\times n}$, $C\in\bbR^{n\times m}$, $b_0\in\bbR^n$, $\gamma>0$, and tolerance $\epsilon$;
\Ensure $(\mu^{(k)}, z^{(k)})$, approximate minimizer of QEPmin~\eqref{eq:QEPmin}
            \State $\beta_1 \gets \|b_0\|$;
            \State {\bf if $\beta_1=0$ then stop};
            \State $q_1 \gets r_0/\beta_1$, $q_0 \gets 0$;
            \For {$k= 1,2,\ldots$}
              \State $\hat q\gets Aq_k$, $\hat q\gets P\hat q$, $\hat q\gets \hat q-\beta_k q_{k-1}$;
                \State $\alpha_k\gets q_k^{\T} \hat q$, $\hat q\gets \hat q-\alpha_k q_k$, $\beta_{k+1} \gets \|\hat q\|$;

                \State compute the leftmost eigenpair  $(\mu^{(k)},s)$ of LEP~\eqref{eq:QEPproj'ed-lin};
               \State $y^{(k)}\gets s_{(1:k)}$, $w^{(k)}\gets s_{(k+1:2k)}$;
               \State {\bf if $\delta_k^{\QEPmin}\le\epsilon$ then stop};
               \State $q_{k+1} \gets \hat q/\beta_{k+1}$;
            \EndFor
            \State $Q_k=[q_1,q_2,\cdots,q_k]$;
            \State $z^{(k)} = Q_k w^{(k)}$ and $u^{(k)}=-\frac{\gamma^2}{\|b_0\|e_1^{\T}w^{(k)}}Q_ky^{(k)}$;
            \State \Return $(\mu^{(k)}, z^{(k)})$ as an approximated minimizer of QEPmin \eqref{eq:QEPmin}
                   and, as a by-product, $(\mu^{(k)}, u^{(k)})$ as an approximated minimizer of LGopt \eqref{eq:LGopt}.
%        \EndFunction
\end{algorithmic}
\end{algorithm}

%\subsubsection{Back to LGopt}
%We note that by the relationship in \Red{Theorem ...},
%as a by-product, an approximate optimizer of
%LGopt \eqref{eq:LGopt} is given by
%\begin{equation} \label{eq:ubyQEPmin}
%(\mu^{(k)}, u^{(k)})
%= \left(\mu^{(k)},\,\, -\frac{\gamma^2}{\|b_0\|e_1^{\T}w^{(k)}}Q_ky^{(k)}\right).
%\end{equation}
It remains to explain why $(\mu^{(k)}, u^{(k)})$ at Line 14 of  Algorithm \ref{alg:solveQEP} is an approximated minimizer of LGopt \eqref{eq:LGopt}.
Let $(\mu^{(k)},\begin{bmatrix}y^{(k)}\\ w^{(k)}\end{bmatrix})$
be the leftmost eigenpair of LEP~\eqref{eq:QEPproj'ed-lin}.
By Theorem~\ref{thm:nondegenerate}, $\mu^{(k)}\notin\text{eig}(T_k)$, and
so  $(T_k-\mu^{(k)}I)^2w^{(k)}\neq 0$ and $e_1^{\T}w^{(k)}\neq 0$.
Through a straightforward application of Theorem \ref{thm:pLGopt=pQEPmin} to
rLGopt \eqref{eq:rLGopt} and rQEPmin~\eqref{eq:rQEPmin},
we find that $(\mu^{(k)},x^{(k)})$ is the minimizer
of rLGopt \eqref{eq:rLGopt} where
\begin{equation}\label{eq:rQEPmin2rLGopt}
x^{(k)}=-\frac{\gamma^2}{\|b_0\|e_1^{\T}w^{(k)}}(T_k-\mu^{(k)}I)w^{(k)}=-\frac{\gamma^2}{\|b_0\|e_1^{\T}w^{(k)}}y^{(k)}.
\end{equation}
Therefore, as a by-product, an approximate minimizer of
LGopt \eqref{eq:LGopt} is given by
\begin{equation} \label{eq:ubyQEPmin}
(\mu^{(k)}, u^{(k)})
= \left(\mu^{(k)},\,\, -\frac{\gamma^2}{\|b_0\|e_1^{\T}w^{(k)}}Q_ky^{(k)}\right).
\end{equation}

\subsection{Lanczos algorithm for CRQopt}\label{sec:alg4CRQopt}

Having obtained approximate minimizers of LGopt~\eqref{eq:LGopt}
and QEPmin~\eqref{eq:QEPmin}, by Theorem~\ref{thm:CQopt=LGopt} we can recover
an approximate minimizer
of CRQopt \eqref{eq:CRQopt} as
\begin{equation}\label{eq:approx-CRQopt}
v^{(k)}=n_0 + u^{(k)}.
\end{equation}
where
$u^{(k)}$ is given
by \eqref{eq:approx-sol:LGopt} if via solving LGopt~\eqref{eq:LGopt} or
by \eqref{eq:ubyQEPmin} if via solving QEPmin~\eqref{eq:QEPmin}.
The overall algorithm called {\em the Lanczos Method},
is outlined in Algorithm~\ref{alg:overall}.

\begin{algorithm}
\caption{Solving CRQopt \eqref{eq:CRQopt}} \label{alg:overall}
\begin{algorithmic}[1]
\Require $A\in \bbR^{ n\times n}$, $C\in \bbR^{ n\times m}$
with full column rank, $b\in\bbR^m$,
tolerance  $\epsilon$;

\Ensure  approximate minimizer $v$ of CRQopt~\eqref{eq:CRQopt};

\State $n_0\gets (C^{\T})^{\dag}b$ (by, e.g., the QR decomposition of $C$);

\State {\bf if $\|n_0\|>1$ then output {\em no solution}};

\State {\bf if $\|n_0\|=1$ then $v \gets n_0$ and output $v$};

\If {$\|n_0\|<1$}
\State $\gamma\gets \sqrt{1-\|n_0\|^2}$,
       $q\gets An_0$, $b_0\gets (I-CC^{\dag})q$;

\State compute an approximate solution of LGopt \eqref{eq:LGopt} $(\mu^{(k)}, u^{(k)})$ by
 Algorithm~\ref{alg:solveLGopt} or \ref{alg:solveQEP}

\State {\bf return} $v^{(k)}=n_0+u^{(k)}$, approximate minimizer of CRQopt~\eqref{eq:CRQopt};
            \EndIf
\end{algorithmic}
\end{algorithm}

\subsubsection{Finite step stopping property}
As in many Lanczos type methods for
numerical linear algebra problems \cite{demm:1997,govl:2013,parl:1998,saad:1992},
Algorithm~\ref{alg:overall} also enjoys a finite-step-stopping property
in the exact arithmetic, i.e., it
will deliver an exact solution in at most $n$ steps.
It is an excellent theoretic property but of little or no
practical significance for large scale problems.  We often expect
that the Lanczos process would stop much sooner before
the $n$th step for otherwise the method would be deemed too
expensive to be practical.

We will show the property using LGopt~\eqref{eq:LGopt} as an example,
which, for convenience, is restated here.
\begin{empheq}[left={\mbox{LGopt:}\quad}\empheqlbrace]{alignat=2}
\min  ~& \lambda \tag{\ref{eq:LGopt-0}} \\
\mbox{s.t.}  ~& (PAP-\lambda I)u=-b_0,\tag{\ref{eq:LGopt-1}}\\
~&\|u\|=\gamma,\tag{\ref{eq:LGopt-2}}\\
~&u\in\cN(C^{\T}).\tag{\ref{eq:LGopt-3}}
\end{empheq}
Let $(\lambda_\ast,u_\ast)$ be the minimizer of LGopt \eqref{eq:LGopt}
and $k_{\max}$ be the smallest $k$ such that $\beta_{k+1}=0$
in the Lancozs process, namely the Lanczos process breaks down at step $k=k_{\max}$.
We will prove that $\mu^{(k_{\max})}=\lambda_\ast$ and $u^{(k_{\max})}=u_\ast$.

We have already shown in \eqref{eq:k-LGopt} that
the second and third constraints of LGopt~\eqref{eq:LGopt}
are satisfied by $u^{(k_{\max})}$.
Besides, since $\beta_{k_{\max}+1}=0$, $r_{k_{\max}}^{\LGopt}=0$
by Proposition~\ref{prop:res-LGopt}, i.e.,
the first constraint of LGopt \eqref{eq:LGopt} holds.
It remains to show that $\mu^{(k_{\max})}=\lambda_\ast$.

\begin{lemma}\label{thm:ldkminroot}
 $\mu^{(k_{\max})}$ is the smallest root of
 \begin{equation}\label{eq:chitielde}
 {\widetilde{\chi}(\lambda)}:=g^{\T}[(H-\lambda I)^{\dag}]^2g^{\T}-\gamma^2.
 \end{equation}
In addition, if {\rm LGopt} \eqref{eq:LGopt} is in the easy case,
then $\mu^{(k_{\max})}=\lambda_{\ast}$, where $(\lambda_{\ast},z_\ast)$ is  the minimizer of {\rm LGopt} \eqref{eq:LGopt}.
\end{lemma}

\begin{proof}
Let  $\vartheta_1\le\vartheta_2\le\cdots\le\vartheta_{k_{\max}}$ be the eigenvalues of $T_{k_{\max}}$ and let $y_1,y_2,\cdots,y_{k_{\max}}$ be the corresponding
orthonormal eigenvectors.
Expand $\|b_0\|e_1=\sum_{i=1}^{k_{\max}}\zeta_iy_i$ and define the secular function
\begin{equation}\label{ex:explicitseculart}
{\widehat{\chi}(\lambda)=\|b_0\|^2e_1^{\T}(T_{k_{\max}}-\lambda I)^{-2}e_1-\gamma^2= \sum_{i=1}^{k_{\max}} \frac{\zeta_i^2 }{(\lambda-\vartheta_i )^2}-\gamma^2.}
\end{equation}
By Theorem \ref{thm:nondegenerate}, $\mu^{(k_{\max})}<\vartheta_1$.
Apply Lemma \ref{lem:qep2secluar} with $H=T_{k_{\max}}$ and $g=\|b_0\|e_1$ to conclude that
$\mu^{(k_{\max})}$ is a root of the secular function \eqref{ex:explicitseculart}.
Since ${\widehat{\chi}}(\lambda)$ is strictly increasing in $(-\infty,\mu^{(k_{\max})})$, $\mu^{(k_{\max})}$ is the smallest root of ${\widehat{\chi}}(\lambda)$.

Expand $Q_{k_{\max}}$ to form an the orthogonal matrix $\widehat{Q}:=[Q_{k_{\max}},\ Q_{\perp}]\in\bbR^{n\times n}$
and let $T=\what Q^{\T}PAP\what Q$. Since the column space of $Q_{k_{\max}}$ is an invariant subspace of $PAP$, we have
$$
T=\begin{bmatrix}
T_{k_{\max}}& \\
 & T_{\perp}
\end{bmatrix}.
$$
Let $S=[S_1, S_2]$ be defined in \eqref{eq:s1s2}, and let
$H=S_1^{\T}PAPS_1$ and $g_0=S_1^{\T}b_0$. For any $\lambda<\vartheta_1$, we have
\begin{equation*}
\begin{split}
{\widehat{\chi}}(\lambda) &=\|b_0\|e_1^{\T}[(T_{k_{\max}}-\lambda I)^{-1}]^2\|b_0\|e_1-\gamma^2\\
&=\|b_0\|e_1^{\T}[(T-\lambda I)^{\dag}]^2\|b_0\|e_1-\gamma^2\\
&=b_0^{\T}\widehat{Q}\widehat{Q}^{\T}[(PAP-\lambda I)^{\dag}]^2\widehat{Q}\widehat{Q}^{\T}b_0-\gamma^2\\
&= b_0^{\T}[(PAP-\lambda I)^{\dag}]^2b_0-\gamma^2 \\
&= b_0^{\T}SS^{\T}[(PAP-\lambda I)^{\dag}]^2SS^{\T}b_0-\gamma^2 \\
&=[g_0^{\T}\ 0]\begin{bmatrix}
[(H-\lambda I)^{\dag}]^2 &0\\
0&[(-\lambda I)^{\dag}]^2
\end{bmatrix}[g_0^{\T}\ 0]^{\T}-\gamma^2\\
&=g_0^{\T}[(H-\lambda I)^{\dag}]^2g_0-\gamma^2=:{\widetilde{\chi}(\lambda)}.\\
\end{split}
\end{equation*}
Therefore, ${\widetilde{\chi}(\lambda)}=0$ and  ${\widetilde{\chi}(\lambda)}<0$ for $\lambda<\mu^{(k_{\max})}$, implying $\mu^{(k_{\max})}$
is the smallest root of ${\widetilde{\chi}(\lambda)}$.

%\marginpar{\tiny changes in red text}
On the other hand,
by the definition of the easy case, $b_0^{\T}z_{\ast}\neq 0$ for all possible minimizers $(\lambda_\ast,z_\ast)$ of QEPmin \eqref{eq:QEPmin}.
Theorem~\ref{thm:QEPmin=pQEPmin} says that $z_{\ast}=S_1w_{\ast}$ for some $w_{\ast}\in\bbR^{n-m}$ and thus
$g^{\T}w_{\ast}=b_0^{\T}S_1w_{\ast}=b_0^{\T}z_{\ast}\ne 0$. By Theorem \ref{thm:pLGoptunique}, $\lambda_{\ast}<\lambda_{\min}(H)$.
Therefore, it is related to case (1) or subcase (i) in case (2) of the proof in Lemma \ref{lm:pLGopt}, for which $\lambda_{\ast}$ is the smallest root of $\widetilde{\chi}(\lambda)$, and thus
$\lambda_{\ast}=\mu^{(k_{\max})}$.
\end{proof}

%\begin{remark}\label{rm:hardcase}
%%We have shown in Lemma \ref{thm:ldkminroot} that $\mu^{(k_{\max})}$ is the smallest root of ${\widetilde{\chi}(\lambda)}=g^{\T}[(H-\lambda I)^{\dag}]^2g^{\T}-\gamma^2$.
%%It is true for both easy and hard case of LGopt \eqref{eq:LGopt}.
%Denote by $(\lambda_{\ast},z_\ast)$  the minimizer of {\rm LGopt} \eqref{eq:LGopt}.
%In the easy case, the smallest root of ${\widetilde{\chi}(\lambda)}$ is $\lambda_\ast$ and $\lambda_\ast<\lambda_{\min}(H)$,
%while in the hard case, $\lambda_\ast=\lambda_{\min}(H)$ and the smallest root of ${\widetilde{\chi}(\lambda)}$ is greater than or equal to $\lambda_{\min}(H)$.
%Since $\mu^{(k)}$ converges to $\mu^{(k_{\max})}$, eventually whether  $\mu^{(k)}<\lambda_{\min}(H)$
%provide a reasonably good test to see if it is the easy case. We discuss this issue in more details later in section \ref{sec-crq-alg-hard}.
%\end{remark}

Theorem~\ref{thm:unique} guarantees that the minimizer of CRQopt \eqref{eq:CRQopt} is unique if
QEPmin \eqref{eq:QEPmin} is in the easy case.  We also have established a finite step stopping property
for Algorithm~\ref{alg:overall} as detailed in the following theorem, since $k_{\max}\le n$.

\begin{corollary}\label{cor:nondekmax}
Suppose {\em QEPmin \eqref{eq:QEPmin}} is in the easy case, and
let $(\mu^{(k)},w^{(k)})$ be the minimizer of
{\rm rQEPmin} \eqref{eq:rQEPmin}. Define
$u^{(k)}$ as in \eqref{eq:approx-sol:LGopt} and $k_{\max}$ is the smallest $k$ such that $\beta_{k+1}=0$.
Then $\left(\mu^{(k_{\max})},u^{(k_{\max})}\right)$
solves {\rm LGopt} \eqref{eq:LGopt},
and $v^{(k_{\max})}=u^{(k_{\max})}+n_0$
is the unique  minimizer of {\em CRQopt \eqref{eq:CRQopt}}.
\end{corollary}

\subsubsection{Hard case}\label{sec-crq-alg-hard}
The hard case is characterized by Theorem~\ref{thm:hard-case:2} and we translate  $g_0\,\bot\,\cU$ into $b_0\,\bot\,\cV$,
where $\cV$ is the eigenspace of $PAP$ associated with its eigenvalue $\lambda_{\min}(H)$. For this reason,
$\mathcal{K}_k(PAP,b_0)$ will contain no eigen-information of $PAP$ associated with $\lambda_{\min}(H)$. Nonetheless, rLGopt~\eqref{eq:rLGopt} and
rQEPmin~\eqref{eq:rQEPmin}  can be still formed and solved to yield approximations
to the original CRQopt~\eqref{eq:CRQopt} with suitable stoping criteria satisfied. But the approximations will be utterly
wrong if it is indeed in the hard case. Hence in practice it is important to detect when the hard case occurs.

Denote by $(\lambda_{\ast},z_\ast)$  the minimizer of {\rm LGopt} \eqref{eq:LGopt}.
In the easy case, the smallest root of ${\widetilde{\chi}(\lambda)}$ is $\lambda_\ast$ and $\lambda_\ast<\lambda_{\min}(H)$,
while in the hard case, $\lambda_\ast=\lambda_{\min}(H)$ and the smallest root of ${\widetilde{\chi}(\lambda)}$ defined in \eqref{eq:chitielde} is greater than or equal to $\lambda_{\min}(H)$.
Since $\mu^{(k)}$ converges to $\mu^{(k_{\max})}$, eventually whether  $\mu^{(k)}<\lambda_{\min}(H)$
provide a reasonably good test to see if it is the easy case.
Therefore, we propose to detect hard case as follows:

\begin{enumerate}
  \item Solve rLGopt~\eqref{eq:rLGopt} or rQEPmin~\eqref{eq:rQEPmin}.
  \item Run the Lanczos process with $M=PAP$ with $r_0=P c$, where $c\in\bbR^n$ is random to compute
        $\lambda_{\min}(H)$ of $PAP$ and its associated eigenvector $\tilde z$;
  \item Check if the optimal value of rLGopt \eqref{eq:rLGopt} or rQEPmin \eqref{eq:rQEPmin} is greater than or equal to $\lambda_{\min}(H)$ within a prescribed
        accuracy.
  \item If it is, then QEPmin~\eqref{eq:QEPmin} is in the hard case; Compute an approximation $\tilde x$
       of $x_{\ast}=-(PAP-\lambda_{\ast} I)^{\dag} b_0$ as follows:
        $$
        \tilde y=\arg\min_{y\in\bbR^k}\left\|\begin{bmatrix}
                                              T_k \\
                                              \beta_{k+1}e_k^{\T}
                                            \end{bmatrix}y+\|b_0\|e_1\right\|,\,\,
        \tilde x=Q_k\tilde y.
        $$
        Finally an approximate minimizer of LGopt~\eqref{eq:LGopt} is given by
        $\tilde x + \sqrt{\gamma^2-\|\tilde x\|^2} \,({\tilde z}/{\|\tilde z\|})$.
\end{enumerate}
A remark is in order for item 2 above. Because of the randomness in $c$, with probability $1$, $r_0=Pc$ will
have a significant component in $S_1\cU$, where $\cU$ is as defined in Theorem~\ref{thm:hard-case}. Thus
$\lambda_{\min}(H)$ will get computed.

\section{Convergence analysis of the Lanczos algorithm}
\label{sec-crq-conv}
In this section, we present a convergence analysis of
the Lanczos algorithm (Algorithm~\ref{alg:overall})
for solving  CRQopt \eqref{eq:CRQopt} in the easy case.
Let $h(v)=v^{\T}Av$ be the objective function of CRQopt \eqref{eq:CRQopt}, $v_\ast$ be the unique solution of CRQopt \eqref{eq:CRQopt} and $(\lambda_\ast,u_\ast)$ be the solution of LGopt \eqref{eq:LGopt}. Our main results are
upper bounds on the
errors $h(v^{(k)})-h(v_\ast)$, $\|v^{(k)}-v_\ast\|$ and $|\mu^{(k)}-\lambda_\ast|$,
where $v^{(k)}$, defined in \eqref{eq:approx-CRQopt}, is the $k$th approximation by Algorithm~\ref{alg:overall} and $(\mu^{(k)},x^{(k)})$ is the solution of rLGopt \eqref{eq:rLGopt}.
Our technique is analogous to that in \cite{zhsl:2017}.

We start by establishing an optimality property of
$v^{(k)}$, as an approximation of $v_{\ast}$, that  minimizes $h(v)$
over $n_0+\mathcal{K}_k(PAP,b_0)$.

\begin{theorem}\label{thm:uopt}
Let $v^{(k)}$ be defined in \eqref{eq:approx-CRQopt}. Then it holds that
\begin{equation}\label{eq:vhtr}
h(v^{(k)})=\min_{v\in n_0+\mathcal{K}_k(PAP,b_0),\|v\|=1} h(v).
\end{equation}
\end{theorem}

\begin{proof}
%Since CRQopt \eqref{eq:CRQopt} is equivalent to CQopt \eqref{eq:CQopt}. Let $u=v-n_0$, we have $f(u)=f(v)$. Let $u^{(k)}$ be defined in \eqref{eq:approx-LGopt}, then it is sufficient to show that
%\begin{equation}\label{eq:uhtr}
%f(u^{(k)})=\min_{s\in \mathcal{K}_k(PAP,b_0),\|s\|=\gamma} f(s).
%\end{equation}
%Since $\mu^{(k)}$ is the smallest real eigenvalue of QEP \eqref{eq:rQEPmin-1} and $w^{(k)}$ is the corresponding eigenvector.
%By Theorem \ref{thm:nondegenerate}, $\|b_0\|e_1^{\T}w^{(k)}\neq 0$.
%Apply Theorem \ref{thm:pLGopt=pQEPmin} with $H=T_k$ and $g=\|b_0\|e_1$ to conclude that
%the pair in \eqref{eq:approx-rLGopt} solves rLGopt~\eqref{eq:rLGopt}.
Recall that $(\mu^{(k)},x^{(k)})$ solves rLGopt~\eqref{eq:rLGopt}.
Consider the optimization problem
\begin{subequations}\label{eq:rCQopt}
\begin{empheq}[left=\empheqlbrace]{alignat=2}
\min ~ & \ell(x):=x^{\T} T_k x+2\|b_0\|e_1^{\T}x,\\
\mbox{s.t.} ~ & \|x\| =\gamma.
\end{empheq}
\end{subequations}

By the theory of Lagrangian multipliers, we find the Lagrangian equations for \eqref{eq:rCQopt} are
\begin{equation} \label{eq:rCQopt-LGeqs}
(T_k -\lambda I)x =-\|b_0\|e_1,\quad
 \|x\| =\gamma.
\end{equation}

Following the same argument as we did to prove Lemma \ref{lem:minlambda},
we can reach the same conclusion that
$\ell(x)$ is strictly increasing with respect to $\lambda$ in the solution pair $(\lambda,x)$ of
\eqref{eq:rCQopt-LGeqs}. Therefore, in order to minimize $\ell(x)$,
we need to find the smallest Lagrangian multiplier satisfying \eqref{eq:rCQopt-LGeqs}. Hence, solving \eqref{eq:rCQopt} is
equivalent to solving rLGopt \eqref{eq:rLGopt} for which $(\mu^{(k)},x^{(k)})$ is
a minimizer and thus $x^{(k)}$ solves \eqref{eq:rCQopt},  where $x^{(k)}$ is defined in \eqref{eq:rQEPmin2rLGopt}.

By  definition, $u^{(k)}=Q_kx^{(k)}$ and $v^{(k)}=u^{(k)}+n_0$.
For any $v\in n_0+\mathcal{K}_k(PAP,b_0)$ with $\|v\|=1$, let
\begin{equation}\label{eq:u=v-n0}
u=v-n_0\in\mathcal{K}_k(PAP,b_0)\subset\cN(C^{\T}).
\end{equation}
Hence $Pu=u$, $\|u\|=\gamma$, and $u=Q_k \widetilde{u}$ for some $\wtd u\in\bbR^k$. We have
$v=u+n_0=Pu+n_0$ and
\begin{align*}
h(v)& %= (u+n_0)^{\T}A(u+n_0)
=(Pu+n_0)^{\T}A(Pu+n_0) \\
&=u^{\T} PAP u+2b_0^{\T} u+n_0^{\T}An_0 \\
&=\widetilde{u}^{\T} Q_k^{\T} PAP Q_k \widetilde{u}+2b_0^{\T} Q_k \widetilde{u}+n_0^{\T}An_0 \\
&=\widetilde{u}^{\T} T_k \widetilde{u}+2\|b_0\|e_1^{\T} \widetilde{u}+n_0^{\T}An_0 \\
&\geq [x^{(k)}]^{\T} T_k x^{(k)}+2\|b_0\|e_1^{\T} x^{(k)}+n_0^{\T}An_0
         \qquad \mbox{(since $x^{(k)}$ solves \eqref{eq:rCQopt})} \\
&=[x^{(k)}]^{\T} Q_k^{\T} PAP Q_kx^{(k)}+2b_0^{\T} Q_k x^{(k)}+n_0^{\T}An_0 \\
&=[u^{(k)}]^{\T} PAP u^{(k)}+2b_0^{\T} u^{(k)}+n_0^{\T}An_0 \\
&=(u^{(k)}+n_0)^{\T}A(u^{(k)}+n_0)\\
&=h(v^{(k)}).
\end{align*}
Since $v\in n_0+\mathcal{K}_k(PAP,b_0)$ with $\|v\|=1$ but otherwise is arbitrary, \eqref{eq:vhtr} holds.
\end{proof}

%\Brown{We have shown in Remark \ref{rmk:gltrs} that our method is related to GLTRS \cite{golr:1999}.
%In \cite{zhsl:2017}  the authors performed a detailed convergence analysis for
%the  GLTRS. Here we adopt their ideas to perform a convergence analysis for our method.}

%
%Let $f(u)=u^{\T}PAPu+2u^{\T}b_0$ for $u\in\cN(C^{\T})$ as in \eqref{eq:fun-f:dfn}. It follows from Theorem \ref{thm:uopt} and
%\begin{equation}\label{eq:h(v)=f(u)+sth}
%h(v)=h(u+n_0)=f(u)+n_0^{\T}An_0
%\end{equation}
%that $u^{(k)}$ satisfies
%\begin{equation*}
%f(u^{(k)})=\min_{u\in \mathcal{K}_k(PAP,b_0),\,\|u\|=\gamma} f(u).
%\end{equation*}

Recall that  $H$ and $g_0$ are defined in \eqref{eq:def:gH} and $S_1$, $S_2$  in \eqref{eq:s1s2}.
Let
$\theta_{\min}$ and $\theta_{\max}$ be the smallest and
the largest eigenvalue of $H$, respectively,  $v_{\ast}$ be the minimizer of CRQopt \eqref{eq:CRQopt}, and
$\lambda_{\ast}$ be the optimal objective value of  LGopt \eqref{eq:LGopt}. Then
$$
(\lambda_{\ast},u_{\ast})\quad\mbox{with}\,\, u_{\ast}=Pv_{\ast}=v_{\ast}-n_0
$$
is a minimizer of LGopt \eqref{eq:LGopt}. Set
$$
\kappa\equiv\kappa(H-\lambda_{\ast} I):=\frac{\theta_{\max}-\lambda_{\ast}}{\theta_{\min}-\lambda_{\ast}}.
$$
To estimate $h(v^{(k)})-h(v_\ast)$, $\|v^{(k)}-v_\ast\|$ and $|\mu^{(k)}-\lambda_\ast|$, we first establish
a lemma that provides a way to bound $h(v^{(k)})-h(v_\ast)$, $\|v^{(k)}-v_\ast\|$ and $|\mu^{(k)}-\lambda_\ast|$ in terms of any
nonzero $v\in n_0+\mathcal{K}_k(PAP,b_0)$.
\begin{lemma}\label{lm:LM4cvg}
For any nonzero $v\in n_0+\mathcal{K}_k(PAP,b_0)$, we have
\begin{subequations}\label{eq:rLGopt-UB}
\begin{align}
0\le h(v^{(k)})-h(v_{\ast})&\le 4\|H-\lambda_{\ast} I\|_2\cdot \|v-v_{\ast}\|_2^2, \label{eq:upperboundf} \\
\|v^{(k)}-v_{\ast}\|&\le 2\sqrt{\kappa}\, \|v-v_{\ast}\|_2,  \label{eq:upperboundv} \\
|\mu^{(k)}-\lambda_{\ast}|&\le\frac 1{\gamma^2}\left[ 4\|H-\lambda_{\ast} I\|_2\cdot \|v-v_{\ast}\|_2^2+2\sqrt{\kappa}\,\|b_0\|_2\cdot \|v-v_{\ast}\|_2\right].
                              \label{eq:upperbound-lambda}
\end{align}
\end{subequations}
%where $\kappa\equiv\kappa(H-\lambda_{\ast} I)=\frac{\theta_{\max}-\lambda_{\ast}}{\theta_{\min}-\lambda_{\ast}}$.
\end{lemma}

\begin{proof}
For $v\in n_0+\mathcal{K}_k(PAP,b_0)$, let
\begin{equation}\label{eq:LM4cvg:pf-1}
u=v-n_0\in\mathcal{K}_k(PAP,b_0),\,\,
\wtd u=\gamma u/\|u\|,\,\,
\wtd v=n_0+\wtd u\in n_0+\mathcal{K}_k(PAP,b_0).
\end{equation}
First, we have $|\|u\|-\gamma|=|\|u\|-\|u_{\ast}\||\le\|u-u_{\ast}\|=\|v-v_{\ast}\|$, which leads to
\begin{equation}\label{eq:hatvdelta}
\left|1-\frac{\gamma}{\|u\|}\right|\le\frac{\|v-v_{\ast}\|}{\|u\|}.
\end{equation}
Let  $r=\wtd v-v_{\ast}$. We  have
\begin{align}
\|r\|=\|v_{\ast}-\wtd v\|
   &\le \|v_{\ast}-v\|+\|v-\wtd v\| \nonumber\\
   &\le \|v_{\ast}-v\|+\|u-\wtd u\| \nonumber\\
   &=\|v_{\ast}-v\|+\left\|u-\frac{\gamma u}{\|u\|}\right\| \nonumber\\
   &=\|v_{\ast}-v\|+\|u\|\times\left|1-\frac{\gamma}{\|u\|}\right| \nonumber\\
   &\le 2\|v_{\ast}-v\|, \label{eq:LM4cvg:pf-2}
\end{align}
where we have used \eqref{eq:hatvdelta} to infer  the last inequality.

The first inequality in \eqref{eq:upperboundf} holds because
$$
h(v^{(k)})=\min_{v\in n_0+\mathcal{K}_k(PAP,b_0),\,\|v\|=1} h(v)\geq \min_{v\in n_0+\cN(C^{\T}),\,\|v\|=1} h(v)=h(v_{\ast}).
$$
Let $f(u)=u^{\T}Au+2u^{\T}b_0$, it can be verified that
$h(v)=h(u+n_0)=f(u)+n_0^{\T}An_0$. Therefore,
\begin{equation}\label{eq:LM4cvg:pf-3}
\wtd u-u_{\ast}=\wtd v-v_{\ast}=r,\,\,
h(\wtd v)-h(v_{\ast})=f(\wtd u)-f(u_{\ast}).
\end{equation}
Set $s=S_1^{\T}r$. It follows from  $r\in\cN(C^{\T})$ that $r=S_1s$ and $\|s\|=\|r\|$.
Noting that $\wtd v$ satisfies the constraint of CRQopt~\eqref{eq:CRQopt} and that $\wtd u=u_*+r$, we have
\begin{alignat}{2}
0\le h(v^{(k)})-h(v_{\ast})
&\le h(\wtd v)-h(v_{\ast}) \\
&\stackrel{{\text{ \eqref{eq:LM4cvg:pf-3}}}}{=} f(\wtd u)-f(u_{\ast})=f(u_*+r)-f(u_{\ast})  \nonumber\\
&=r^{\T} PAP r+2r^{\T}(PAPu_*+b_0) \nonumber\\
&=r^{\T} PAP r+2\lambda_{\ast} r^{\T} u_{\ast} \label{eq:bovndf''}\\
&=r^{\T} (PAP-\lambda_{\ast} I)r \label{eq:bovndf'}\\
&= s^{\T}S_1^{\T}(PAP-\lambda_{\ast} I)S_1s \nonumber\\
&=s^{\T} (H-\lambda_{\ast} I)s \nonumber\\
&\le \|H-\lambda_{\ast} I\|\|s\|^2=\|H-\lambda_{\ast} I\|\|r\|^2 \nonumber\\
&\stackrel{{\text{\eqref{eq:LM4cvg:pf-2}}}}{\le} 4\|H-\lambda_{\ast} I\|\|v_{\ast}-v\|^2, \label{eq:bovndf}
\end{alignat}
yielding the second inequality in \eqref{eq:upperboundf},
where we have used $(PAP-\lambda_* I)u_*=-b_0$ to get \eqref{eq:bovndf''} and
$$
\|r\|^2+2r^{\T}u_{\ast}=\|u_{\ast}+r\|^2-\|u_{\ast}\|^2=\|\wtd u\|^2-\|u_{\ast}\|^2=0
$$
to obtain $2r^{\T}u_{\ast}=-r^{\T}r$ and then \eqref{eq:bovndf'}.

Next we prove \eqref{eq:upperboundv}. Define
$$
\widetilde{f}(u):=f(u)-\lambda_*u^{\T}u=u^{\T}(PAP-\lambda_{\ast} I) u+2u^{\T}b_0.
$$
Noticing
$(PAP-\lambda_{\ast} I)u_{\ast}+b_0=0$ by \eqref{eq:LGopt-1}, let ${u}^{(k)}={v}^{(k)}-n_0$,
we have
$$
\widetilde{f}(u^{(k)})=\widetilde{f}(u_{\ast})+(u^{(k)}-u_{\ast})^{\T}(PAP-\lambda_{\ast} I)(u^{(k)}-u_{\ast}).
$$
Therefore
$$
\widetilde{f}(u^{(k)})-\widetilde{f}(u_{\ast})
       \geq (\theta_{\min}-\lambda_{\ast})\|u^{(k)}-u_{\ast}\|^2
        =(\theta_{\min}-\lambda_{\ast})\|v^{(k)}-v_{\ast}\|^2.
$$
On the other hand,
$$
\widetilde{f}(u^{(k)})-\widetilde{f}(u_{\ast})
           =[f(u^{(k)})+\lambda_{\ast} \|u^{(k)}\|^2]-[f(u_{\ast})+\lambda_{\ast} \|u_{\ast}\|^2]
           =f(u^{(k)})-f(u_{\ast})=h(v^{(k)})-h(v_{\ast}),
$$
yielding
\begin{equation}\label{eq:LM4cvg:pf-5}
(\theta_{\max}-\lambda_{\ast})\|v^{(k)}-v_{\ast}\|^2
      \le  h(v^{(k)})-h(v_{\ast})\le 4\|H-\lambda_{\ast} I\| \|v-v_*\|^2,
\end{equation}
which leads to \eqref{eq:upperboundv}.

To prove \eqref{eq:upperbound-lambda}, we pre-multiply $(PAP-\lambda_* I)u_*=-b_0$ by $u_*^{\T}$ and use
$u_*^{\T}u_*=\gamma^2$ to get
\begin{equation}\label{eq:LM4cvg:pf-6}
\gamma^2\lambda_*=u_*^{\T}PAPu_*+u_*^{\T}b_0=v_*^{\T}PAPv_*+v_*^{\T}b_0,
\end{equation}
since $Pv_*=u_*$ and $Pb_0=b_0$.
By \eqref{eq:decomp-vAv:1}, we have $h(v_*)=v_*^{\T}PAPv_*+2v_*^{\T}b_0+n_0^{\T}An_0$ and thus
$$
\gamma^2\lambda_*=h(v_*)-v_*^{\T}b_0-n_0^{\T}An_0.
$$
On the other hand, it follows from rLGopt \eqref{eq:rLGopt} that $[x^{(k)}]^{\T}T_kx^{(k)}+\|b_0\|_2[x^{(k)}]^{\T}e_1=\gamma^2\mu^{(k)}$. Plug in
$$
T_k=Q_k^{\T}PAPQ_k,\,\,
u^{(k)}=Q_kx^{(k)},\,\,
Q_k^{\T}b_0=\|b_0\|_2e_1,\,\,
v^{(k)}=u^{(k)}+n_0
$$
to get
\begin{equation}\label{eq:LM4cvg:pf-7}
\gamma^2\mu^{(k)}=h(u^{(k)})-[u^{(k)}]^{\T}b_0=h(v^{(k)})-[v^{(k)}]^{\T}b_0-n_0^{\T}An_0.
\end{equation}
It follows from \eqref{eq:LM4cvg:pf-6} and \eqref{eq:LM4cvg:pf-7} that
 \begin{equation}\label{eq:cauchy}
 \left|\mu^{(k)}-\lambda_\ast\right|=\frac{1}{\gamma^2}\left|h(v^{(k)})-h(v_\ast)-b_0^{\T}(v^{(k)}-v_\ast)\right| \le \frac{1}{\gamma^2}\left[|h(v^{(k)})-h(v_\ast)|+\|b_0\|_2\|v^{(k)}-v_\ast\|_2\right],
 \end{equation}
which combined with \eqref{eq:upperboundf} and \eqref{eq:upperboundv} yield \eqref{eq:upperbound-lambda}.
\end{proof}

The inequalities in \eqref{eq:rLGopt-UB} hold for any $v\in n_0+\mathcal{K}_k(PAP,b_0)$ which, in general can be expressed
as
$$
v=n_0+\phi_{k-1}(PAP)b_0,
$$
where $\phi_{k-1}(\,\cdot\,)$ is a polynomial of degree $k-1$.
By judicially picking
certain $\phi_{k-1}$, meaningful upper bounds on $h(v^{(k)})-h(v_\ast)$, $\|v^{(k)}-v_\ast\|$ and $|\mu^{(k)}-\lambda_\ast|$
are readily obtained. These upper bounds expose the convergence behavior of $v^{(k)}$.
The next theorem contains our main results of the section.

\begin{theorem}\label{thm:nondeer}
Suppose {\rm CRQopt} \eqref{eq:CRQopt} is in the easy case,  and let $v_\ast$ be its minimizer.
Let  $(\lambda_\ast,u_\ast)$ be the minimizer of the corresponding
{\rm LGopt} \eqref{eq:LGopt},  and, for its corresponding
{\rm pLGopt} \eqref{eq:pLGopt}, let $\theta_{\min}$ and $\theta_{\max}$ be the smallest
and largest eigenvalue of $H$, respectively, and set
$$
\kappa=\kappa(H-\lambda_{\ast} I):=\frac{\theta_{\max}-\lambda_{\ast}}{\theta_{\min}-\lambda_{\ast}}.
$$
Then the following statements hold:
\begin{enumerate}[{\rm (a)}]
\item The sequence $\lbrace h(v^{(k)}) \rbrace$ is nonincreasing;
\item For $k\le k_{\max}$, the smallest $k$ such that $\beta_{k+1}=0$,
      \begin{subequations}\label{eq:rLGopt-UBs}
      \begin{align}
      0\le h(v^{(k)})-h(v_{\ast})
       &\le 16\gamma^2\|H-\lambda_{\ast} I\|_2\,\left[\Gamma_{\kappa}^k+\Gamma_{\kappa}^{-k}\right]^{-2}, \label{eq:rLGopt-UBs-1}\\
      \|v^{(k)}-v_{\ast}\|_2
        &\le 4\gamma\sqrt{\kappa}\left[\Gamma_{\kappa}^k+\Gamma_{\kappa}^{-k}\right]^{-1}, \label{eq:rLGopt-UBs-2} \\
      |\mu^{(k)}-\lambda_{\ast}|&\le 16\|H-\lambda_{\ast} I\|_2\,\left[\Gamma_{\kappa}^k+\Gamma_{\kappa}^{-k}\right]^{-2}
                      +\frac 4{\gamma}\|b_0\|_2\sqrt{\kappa}\left[\Gamma_{\kappa}^k+\Gamma_{\kappa}^{-k}\right]^{-1}, \label{eq:rLGopt-UBs-3}
      \end{align}
      \end{subequations}
      where %$\kappa=\kappa(H-\lambda_{\ast} I)=\frac{\theta_{\max}-\lambda_{\ast}}{\theta_{\min}-\lambda_{\ast}}$, and
      \begin{equation}\label{eq:Gamma}
      \Gamma_{\kappa}:=\frac {\sqrt{\kappa}+1}{\sqrt{\kappa}-1}.
      \end{equation}
\end{enumerate}
\end{theorem}

\begin{proof}
 Item (a) holds because for any $0 \le k \le k_{\max}$,
$$
h(v^{(k)})=\min_{v\in n_0+\mathcal{K}_k(PAP,b_0),\,\|v\|=1} h(v)
                  \geq\min_{v\in n_0+\mathcal{K}_{k+1}(PAP,b_0),\,\|v\|=1} h(v)=h({v}^{(k+1)}).
$$
Before we prove item (b), we note that $(\lambda_{\ast},S_1^{\T}v_{\ast})$ solves pLGopt~\eqref{eq:pLGopt}.
In particular, since pLGopt~\eqref{eq:pLGopt} is in the easy case,
\begin{equation}\label{eq:nondeer:pf-1}
S_1^{\T}v_{\ast}=-(H-\lambda_{\ast} I)^{-1}g_0.
\end{equation}
Consider now $v\in n_0+\mathcal{K}_k(PAP,b_0)$.
Then
%$\widetilde{u}:=v-n_0\in \mathcal{K}_k(PAP,b_0)$.
%
%We know $S_1^{\T}n_0=0$.
%Let $g=S_1^{\T}b_0$.
%
%We have
$S_1^{\T} v\in\mathcal{K}_k(H,g_0)=\mathcal{K}_k(H-\lambda_{\ast} I,g_0)$.
Therefore by \eqref{eq:nondeer:pf-1}
\begin{align}
S_1^{\T}v-S_1^{\T}v_{\ast} &=\phi_{k-1}(H-\lambda_{\ast} I)g+(H-\lambda_{\ast} I)^{-1}g_0 \nonumber\\
&=[\phi_{k-1}(H-\lambda_{\ast} I)\,(H-\lambda_{\ast} I)+I](H-\lambda_{\ast} I)^{-1}g_0 \nonumber\\
&=-\psi_k(H-\lambda_{\ast} I)S_1^{\T}v_{\ast}, \label{eq:nondeer:pf-2}
\end{align}
where $\phi_{k-1}$ is a polynomial of degree $k-1$, and $\psi_k(t)=1+t\phi_{k-1}(t)$, a polynomial of degree $k$, that satisfies
$\psi_k(0)=1$.
Note that $\psi_k(0)=1$ but otherwise $\psi_k$ is an arbitrary polynomial of degree $k$, offering the freedom that we will
take advantage of in a moment.

Given that $v_{\ast}$ solves CRQopt~\eqref{eq:CRQopt}, we have
%$\|Pv_{\ast}\|=\gamma$. Note
%$$
%Pv_{\ast}=Pv_{\ast}-Pn_0=P(v_{\ast}-n_0)=S_1S_1^{\T}(v_{\ast}-n_0)
%$$
%to get
$$
\gamma=\|Pv_{\ast}\|=\|S_1S_1^{\T}v_{\ast}\|=\|S_1^{\T}v_{\ast}\|.
$$
Thus
\begin{align}
\min_{v\in n_0+\mathcal{K}_k(PAP,b_0)} \|v-v_{\ast}\|
     &=\min_{v\in n_0+\mathcal{K}_k(PAP,b_0)} \|S_1^{\T}v-S_1^{\T}v_{\ast}\| \qquad(\mbox{use \eqref{eq:nondeer:pf-2}})\nonumber\\
&\le \gamma\min_{\psi_k(0)=1}\|\psi_k(H-\lambda_{\ast} I)\| \nonumber\\
&\le \gamma\min_{\psi_k(0)=1}\max_{1\le i\le n-m}|\psi_k(\theta_i-\lambda_{\ast})| \label{eq:nondeer:pf-2a}\\
&\le \gamma\min_{\psi_k(0)=1}\max_{t\in[\theta_{\min}-\lambda_{\ast},\theta_{\max}-\lambda_{\ast}]}|\psi_k(t)|. \label{eq:nondeer:pf-3}
%&=\frac{\gamma}{p_{k+1}(\eta)},
\end{align}
The inequality \eqref{eq:nondeer:pf-3} holds for any polynomial $\psi_k$ of degree $k$ such that $\psi_k(0)=1$. For the purpose
of establishing upper bounds, we will pick one that is defined through the $k$th Chebyshev polynomial of the first kind:
\begin{subequations}\label{eq:Cheb1st}
\begin{alignat}{2}
\scrT_k(t)
 &=\cos(k\arccos t)&&\mbox{for $|t|\le 1$}, \label{eq:Cheb-sml}\\
 &=\frac 12\left[\left(t+\sqrt{t^2-1}\right)^k
       +\left(t+\sqrt{t^2-1}\right)^{-k}\right]&\quad&\mbox{for $|t|\ge 1$.} \label{eq:Cheb-big}
\end{alignat}
\end{subequations}
Specifically, we take
\begin{equation}\label{eq:psi-1}
\psi_k(t)=\scrT_k \left(\frac {2t-(\alpha+\beta)}{\beta-\alpha}\right)\left/\scrT_k\left(\frac {-(\alpha+\beta)}{\beta-\alpha}\right)\right.,
\end{equation}
where
$\alpha=\theta_{\min}-\lambda_{\ast}$ and $\beta=\theta_{\max}-\lambda_{\ast}$. Evidently, $\psi_k(0)=1$, and
for $t\in[\theta_{\min}-\lambda_{\ast},\theta_{\max}-\lambda_{\ast}]=[\alpha,\beta]$,
we have
$$
|2t-(\alpha+\beta)|=\left||t+\lambda_\ast-\theta_{\min}|-|t+\lambda_\ast-\theta_{\max}|\right|\le |\theta_{\max}-\theta_{\min}|=\beta-\alpha.
$$
Therefore,
${[2t-(\alpha+\beta)]}/{(\beta-\alpha)}\in[-1,1],
$
and thus for $t\in[\alpha,\beta]$ \cite{li:2006c08}
\begin{equation}\label{eq:nondeer:pf-5}
|\psi_k(t)|\le\left|\scrT_k \left(\frac {-(\alpha+\beta)}{\beta-\alpha}\right)\right|^{-1}
             =\left|\scrT_k\left(\frac {\kappa+1}{\kappa-1}\right)\right|^{-1}
             =2\left[\Gamma_{\kappa}^k+\Gamma_{\kappa}^{-k}\right]^{-1}.
\end{equation}
Minimize the right-most quantities in \eqref{eq:rLGopt-UB} over $v\in n_0+\mathcal{K}_k(PAP,b_0)$,
utilize \eqref{eq:nondeer:pf-3} and \eqref{eq:nondeer:pf-5} to get the inequalities in \eqref{eq:rLGopt-UBs}.
\end{proof}

We end this section with remarks regarding the results in Theorem \ref{thm:nondeer}.

\begin{remark}\label{rm:bound1}
 The rate of convergence for our Lanczos algorithm depends on $\kappa$.
 Recall that $\kappa=\frac{\theta_{\max}-\lambda_{\ast}}{\theta_{\min}-\lambda_{\ast}}$.
 When $\lambda_{\ast}$ is far away from $\theta_{\min}$, we may regard that CRQopt \eqref{eq:CRQopt} is far from hard case.
 In this case, $\kappa$ moves towards $1$, and we expect faster convergence of our Lanczos algorithm.
 However, when CRQopt \eqref{eq:CRQopt} is near hard case, i.e., $\theta_{\min}\approx\lambda_{\ast}$, $\kappa$ is large,
 and Theorem \ref{thm:nondeer} suggests slow convergence. These conclusions derived  from Theorem~\ref{thm:nondeer} are consistent
with the numerical observations in \cite{hage:2001}
that ``a Lanczos type process seems to be very effective when the problem
is far from the hard case''. We provide an example in example \ref{ex:sharp} later to illustrate
the relationship between the rate of convergence and $\kappa$.
\hfill $\Box$
      \end{remark}

\begin{remark}\label{rm:bound2}
For most examples, the bounds suggested in \eqref{eq:rLGopt-UBs-1} and \eqref{eq:rLGopt-UBs-2} are sharp.
%We provide such an example in Example \ref{ex:sharp} later.
However, there are some cases where the bounds suggested in \eqref{eq:rLGopt-UBs-1} and \eqref{eq:rLGopt-UBs-2} are pessimistic. This occurs for near-hard situations where $\lambda_\ast\approx
 \theta_{\min}$ and sometimes the Lanczos method can still enjoy fast convergence, even though
the bounds in \eqref{eq:rLGopt-UBs-1} and \eqref{eq:rLGopt-UBs-2} do not suggest so.
One of such situations is when
$$
\kappa_+:= \frac{\theta_{\max}-\lambda_{\ast}}{\theta_2-\lambda_{\ast}}
$$
is small, even though $\theta_{\min}\approx\lambda_{\ast}$
and thus $\kappa$ is huge, where $\theta_2$ is the second smallest eigenvalue of $H$. This suggests that
the bounds by \eqref{eq:rLGopt-UBs-1} and \eqref{eq:rLGopt-UBs-2} have room for improvement. In fact, instead of
\eqref{eq:psi-1}, we may choose
\begin{equation}\label{eq:psi-t}
\psi_k(t)=\frac {t-\alpha}{-\alpha} \cdot
\scrT_{k-1}\left(\frac {2t-(\alpha_++\beta)}{\beta-\alpha_+}\right)\left/
\scrT_{k-1}\left(\frac {-(\alpha_++\beta)}{\beta-\alpha_+}\right)\right.,
\end{equation}
where $\alpha$ and $\beta$ are as before, and $\alpha_+=\theta_2-\lambda_{\ast}$. Evidently, again $\psi_k(0)=1$, but now
$\psi_k(\theta_1-\lambda_\ast)=0$. We have
\begin{equation}\label{eq:bound2}
\begin{aligned}
\max_{1\le i\le n-m}|\psi_k(\theta_i-\lambda_{\ast})|
   &=\max_{2\le i\le n-m}|\psi_k(\theta_i-\lambda_{\ast})|
   \le\max_{t\in[\alpha_+,\beta]}|\psi_k(t)|
   \\
  &\le \max_{t\in[\alpha_+,\beta]}\left|\frac{t-\alpha}{-\alpha}\right|\cdot
 2 \left[\Gamma_{\kappa_+}^{(k-1)}+\Gamma_{\kappa_+}^{-(k-1)}\right]^{-1} \\
& =\frac{2(\theta_{\max}-\theta_{\min})}{\theta_{\min}-\lambda_\ast}\left[\Gamma_{\kappa_+}^{(k-1)}+\Gamma_{\kappa_+}^{-(k-1)}\right]^{-1}.
\end{aligned}
\end{equation}
It combined with \eqref{eq:nondeer:pf-2a} will lead to   bounds
\begin{subequations}\label{eq:bound21}
\begin{align}
 h(v^{(k)})-h(v_{\ast})&\le \frac {16\gamma^2\|H-\lambda_{\ast} I\|_2(\theta_{\max}-\theta_{\min})}{(\theta_{\min}-\lambda_\ast)}\,\left[\Gamma_{\kappa_+}^{(k-1)}+\Gamma_{\kappa_+}^{-(k-1)}\right]^{-2}, \label{eq:bound211} \\
\|v^{(k)}-v_{\ast}\|_2 &\le 4\gamma\sqrt{\kappa}\frac{\theta_{\max}-\theta_{\min}}{\theta_{\min}-\lambda_\ast}\left[\Gamma_{\kappa_+}^{(k-1)}+\Gamma_{\kappa_+}^{-(k-1)}\right]^{-1}, \label{eq:bound212} \\
|\mu^{(k)}-\lambda_{\ast}|&\le\frac{\theta_{\max}-\theta_{\min}}{\theta_{\min}-\lambda_\ast}\left[ 16\|H-\lambda_{\ast} I\|_2\,\left[\Gamma_{\kappa_+}^{(k-1)}+\Gamma_{\kappa_+}^{-(k-1)}\right]^{-2}\right.\nonumber \\
    &\hspace{3cm}\left.+\frac{4}{\gamma}\|b_0\|_2\sqrt{\kappa}\left[\Gamma_{\kappa_+}^{(k-1)}+\Gamma_{\kappa_+}^{-(k-1)}\right]^{-1}\right].
                              \label{eq:bound213}
\end{align}
\end{subequations}
which can be much sharper than the ones by \eqref{eq:rLGopt-UBs-1} and \eqref{eq:rLGopt-UBs-2} and they will be sharper if
$\theta_{\min}\approx\lambda_{\ast}$ and there is a reasonably gap between $\theta_{\min}$ and $\theta_2$.
We show such an example later in example \ref{ex:nearhard}.
\hfill $\Box$
\end{remark}

\begin{remark}\label{rm:pessimistic}
In our numerical experiments, we observed that
the bound \eqref{eq:rLGopt-UBs-3} often decays much slower than $|\mu^{(k)}-\lambda_\ast|$.
Recall that in obtaining \eqref{eq:rLGopt-UBs-3},  we used
\begin{equation}\label{eq:pessimistic}
\left|b_0^{\T}(v^{(k)}-v_\ast)\right|\le\|b_0\|\left\|v^{(k)}-v_\ast\right\|
\end{equation}
in \eqref{eq:cauchy}.
It turns out that $\|b_0\|\left\|v^{(k)}-v_\ast\right\|$ decays much slower than $\left|b_0^{\T}(v^{(k)}-v_\ast)\right|$,
as evidenced by our numerical tests.
While at this point we don't know how to estimate  $\left|b_0^{\T}(v^{(k)}-v_\ast)\right|$
much more than accurately than via the inequality \eqref{eq:pessimistic}, we offer a plausible explanation as follows.
Let $u^{(k)}=v^{(k)}-n_0$ and $u_\ast=v_\ast-n_0$. Since $u_\ast^{\T} u_\ast=[u^{(k)}]^{\T} u^{(k)}=\gamma^2$, we have
\begin{equation}\label{eq:uuer}
\begin{aligned}
\left|u_\ast^{\T}(v^{(k)}-v_\ast)\right|&=\left|u_\ast^{\T} u^{(k)}-u_\ast^{\T} u_\ast\right|=\frac12\left|2u_\ast^{\T} u^{(k)}-u_\ast^{\T} u_\ast- [u^{(k)}]^{\T} u^{(k)}\right|\\
&=\frac12\left\| u^{(k)}-u_\ast\right\|_2^2=\frac12\left\| v^{(k)}-v_\ast\right\|_2^2.
\end{aligned}
\end{equation}
By \eqref{eq:rLGopt-UBs-2}, $\left\| v^{(k)}-v_\ast\right\|_2^2$ is of
order $O\left(\left[\Gamma_{\kappa}^k+\Gamma_{\kappa}^{-k}\right]^{-2}\right)$,
and thus $\left|u_\ast^{\T}(v^{(k)}-v_\ast)\right|$ is also of order
$O\left(\left[\Gamma_{\kappa}^k+\Gamma_{\kappa}^{-k}\right]^{-2}\right)$
as \eqref{eq:uuer} suggests. Let $\theta_1\le\theta_2\le\cdots\le\theta_{n-m}$ be the eigenvalues of $PAP$ restricted to
the subspace $\cR(P)$,
$y_1,y_2,\cdots,y_{n-m}$ be the corresponding orthonormal eigenvectors in $\cR(P)$,  $u_\ast=\sum_{i=1}^{n-m} \xi_iy_i$,
and $v^{(k)}-v_\ast=u^{(k)}-u_\ast=\sum_{i=1}^{n-m} \epsilon_iy_i$. Then we have
$$
\left|u_\ast^{\T}(v^{(k)}-v_\ast)\right|=\left|\sum_{i=1}^{n-m}  \xi_i\epsilon_i\right|.
$$
On the other hand $b_0=-(PAP-\lambda_\ast I)u_\ast=-\sum_{i=1}^{n-m} (\theta_i-\lambda_\ast )\xi_iy_i$
and thus
$$
\left|b_0^{\T}(v^{(k)}-v_\ast)\right|=\left|\sum_{i=1}^n (\theta_i-\lambda_\ast)\xi_i\epsilon_i\right|.
$$
Note that sequence $\{\theta_i-\lambda_\ast\}$ is positive and  increasing for the easy case
and sequence $\{\xi_iy_i\}$ oscillates for most cases  in practice. Therefore,
when $\kappa(PAP-\lambda_\ast I)=\frac{\theta_{n-m}-\lambda_\ast}{\theta_1-\lambda_\ast}$
is modest, i.e., the difference between $\theta_i-\lambda_\ast$ for different $i$ is modest,
we expect that the difference between
$\left|b_0^{\T}(v^{(k)}-v_\ast)\right|=\left|\sum_{i=1}^{n-m} (\theta_i-\lambda_\ast)\xi_i\epsilon_i\right|$
and $\left|u_\ast^{\T}(v^{(k)}-v_\ast)\right|=\left|\sum_{i=1}^{n-m}  \xi_i\epsilon_i\right|$ is small.
Therefore, the convergence rate of $\left|b_0^{\T}(v^{(k)}-v_\ast)\right|$ can be
similar to the convergence rate of $\left|u_\ast^{\T}(v^{(k)}-v_\ast)\right|$,
which is $O\left(\left[\Gamma_{\kappa}^k+\Gamma_{\kappa}^{-k}\right]^{-2}\right)$.
Plausibly, we have explained why the bound \eqref{eq:rLGopt-UBs-3} decays
much slower than the actual  $|\mu^{(k)}-\lambda_\ast|$.
%However, this does not mean we can improve the bound of $\left|b_0^{\T}(v^{(k)}-v_\ast)\right|$ to the order $O\left(\left[\Gamma_{\kappa}^k+\Gamma_{\kappa}^{-k}\right]^{-2}\right)$. Suppose it is true, than there exists a constant number $C$ such that $\left|b_0^{\T}(v^{(k)}-v_\ast)\right|\le C\left|u_\ast^{\T}(v^{(k)}-v_\ast)\right|$, which means
%\begin{equation}\label{eq:inequalityfalse}
%\left|\sum_{i=1}^n (\theta_i-\lambda_\ast)\xi_i\epsilon_i\right|\le C\left|\sum_{i=1}^n  \xi_i\epsilon_i\right|.
%\end{equation}
% However, inequality \eqref{eq:inequalityfalse} is not true for the  general case.
 \hfill $\Box$
\end{remark}

\section{Numerical examples -- sharpness of error bounds}\label{sec-crq-ex}
In this section, we demonstrate the sharpness of our convergence error bounds 
in Theorem~\ref{thm:nondeer} for the Lanczos algorithm (Algorithm~\ref{alg:overall}) 
for solving  CRQopt \eqref{eq:CRQopt}. For that purpose, we first test examples 
that are hard for the Lanczos algorithm. The basic idea
is similar to that in \cite{li:2010}. Also shown are the history of the normalized residual
$\NRes_k^{\QEPmin}$ and its upper bound $\delta_k^{\QEPmin}$ in \eqref{eq:NRes-QEPmin-2}.
All numerical examples were carried out in MATLAB.
% and run on a machine with Intel Core i5 CPU@1.4GHz and 4GB memory.

%\newpage
\subsection{Construction of difficult CRQopt problems}\label{sec-crq-ex-matrix}

The convergence analysis of the Lanczos algorithm
(Algorithm~\ref{alg:overall}) for solving CRQopt \eqref{eq:CRQopt}
presented in  Theorem \ref{thm:nondeer} indicates
that the convergence behavior is determined by
the spectral distribution of the matrix $H$ defined in
pLGopt \eqref{eq:pLGopt} and the optimal value $\lambda_\ast$ of LGopt \eqref{eq:LGopt}.
Therefore, we construct matrices $A$, $C$ and vector $b$ through
constructing matrices $H$ and $g_0$ of
pLGopt \eqref{eq:pLGopt}.% that are difficult for the Lanczos algorithm.%, i.e, difficult ones on which the Lanczos algorithm behaves as predicted by Theorem \ref{thm:nondeer}.

\paragraph{$H$ and $g_0$.}
It is not a secret that
approximations by the Lanczos procedure converge most slowly when the eigenvalues of these matrices
distribute like the zeros or the extreme nodes of Chebyshev polynomials
of the first kind \cite{li:2006c08,li:2007a08,li:2010,zhsl:2017}.
In what follows,
we describe one set of test matrix-vector pair $(H,g_0)$
using the extreme nodes of Chebyshev polynomials
of the first kind.
% to construct $H$, and with a properly chosen vector $g_0$.
%\Red{This set of test matrices is considered to be
%the difficult test cases in the sense that the eigenvalues
%are clustered at the ends of the spectrum.}

The $\ell$th Chebyshev polynomials of the first kind $\scrT_{\ell}(t)$
has $\ell+1$ extreme points in $[-1,1]$, defined by
\begin{equation}\label{eq:ChebENodes}
\tau_{j\ell}=\cos\vartheta_{j\ell},
\quad \mbox{with} \quad
\vartheta_{jl}=\frac {j}{\ell}\pi
\quad \mbox{for} \quad j = 0, 1, \ldots, \ell.
\end{equation}
At these extreme points, $|\scrT_{\ell}(\tau_{j\ell})|=1$.
Given scalars $\alpha$ and $\beta$ such that $\alpha<\beta$, set
\begin{equation}\label{eq:omega-tau}
\omega=\frac {\beta-\alpha}2, \quad
\tau=-\frac {\alpha+\beta}{\beta-\alpha}.
\end{equation}
The so-called {\em the $\ell$th translated Chebyshev extreme nodes\/}
on $[\alpha,\beta]$ are given by \cite{li:2006c08,li:2007a08}
\begin{equation}\label{eq:transChebENodes}
\tau_{j\ell}^{\trans}=\omega (\tau_{j\ell}-\tau)
\quad \mbox{for} \quad j = 0, 1, \ldots, \ell.
\end{equation}
It can be verified that $\tau^{\trans}_{0\ell}=\beta$
and $\tau^{\trans}_{\ell\ell}=\alpha$.

Given integers $n$ and $m$ with $m<n$,
and the interval $[\alpha,\beta]$, we take
\begin{equation}\label{eq:Hg0-eg}
H=\diag\left(\tau_{0\,n-m-1}^{\trans},\tau_{1\,n-m-1}^{\trans},\dots,\tau_{n-m-1\,n-m-1}^{\trans} \right).
\end{equation}
Now we construct $g_0=[g_1,g_2,\cdots,g_{n-m}]^{\T}\in\bbR^{n-m}$. Recall that the eigenvector of $H$ corresponding
to the smallest eigenvalue is $e_{n-m}$.
In order to make pLGopt \eqref{eq:pLGopt} in the easy case,
we need to make $g_0$ not perpendicular to that eigenvector $e_{n-m}$, i.e., $g_{n-m} \neq 0$.
As a simple choice, we take
\iffalse
Recall that the convergence rate of the Lanczos algorithm
in Theorem~\ref{thm:nondeer} depends on
$$
\kappa=\kappa(H-\lambda_\ast I)
=\frac{\tau_{0\,n-m-1}^{\trans}-\lambda_\ast}
{\tau_{n-m-1\,n-m-1}^{\trans}-\lambda_\ast},$$
where $\lambda_\ast$ is the optimal value of LGopt \eqref{eq:LGopt}. Specifically, since LGopt \eqref{eq:LGopt} is in the easy case, $\lambda_\ast$ is the smallest root of the secular equation
\begin{equation}\label{eq:secularg0}
\chi(\lambda)=\sum_{i=1}^{n-m}\frac{g_i^2}{(\lambda-\tau_{i\,n-m-1}^{\trans})^2}-\gamma^2.
\end{equation}
Therefore, in order to construct an example
where $\kappa$ is small, $g_0$ should be chosen such that
the gap between the smallest eigenvalue of $H$, which is
$\tau_{n-m-1\,n-m-1}^{\trans} = \alpha$, and the smallest root
$\lambda_{\ast}$ of the secular equation \eqref{eq:secularg0}
is not small. Thus, given the matrix $H$ from \eqref{eq:Hg0-eg},
a good choice of $g_0$ is
the vector of ones, i.e.,
\fi
\begin{equation} \label{eq:g0}
g_0=[1,1,\cdots,1]^{\T}\in\bbR^{n-m}.
\end{equation}
\iffalse
For example, if we set $n = 1100$, $m = 100$, $\gamma = \sqrt{0.19}$,
$\alpha=1$ and $\beta=100$,
by $(H, g_0)$ defined in \eqref{eq:Hg0-eg} and \eqref{eq:g0},
we have $\lambda_\ast=-42.6007$ and $\kappa=3.2706$. In this example, $\kappa$ is small and we expect fast convergence for our Lanczos algorithm for CRQopt \eqref{eq:CRQopt}.

On the other hand, if $\beta=1000$ and other parameters are the same as the previous example,
then we have $\lambda_\ast=-18.2629$ and $\kappa=52.8613$. This  $\kappa$ is slightly larger than the  case where $\beta=100$ and we expect slower convergence for our Lanczos algorithm for CRQopt \eqref{eq:CRQopt}.
\fi

\paragraph{$A$, $C$ and $b$.}
With $H$ and $g_0$ set, we construct matrices $A$, $C$ and vector $b$ in the following way:
\begin{enumerate}
%\item Pick $n$, $m<n$ (usually $m\ll n$),
\item Pick $0<\zeta<1$, and $a\in\bbR^m$ with $\|a\|=1/\zeta$;
  \item Pick a random $C\in\bbR^{n\times m}$ and compute its QR decomposition
  \begin{equation}\label{eq:qrc}
  C=\kbordermatrix{ & \sss m & \sss n-m \\
                          & S_2 & S_1}\times\kbordermatrix{ & \sss m \\
                                       \sss m & R \\
                                       \sss n-m & 0}\equiv S_2R.
  \end{equation}
\item Let  $b=\zeta^2 R^{\T}a$.
\item Take $A_{12}=g_0a^{\T}$, $A_{22}=\eta I_m$
      with $\eta=(g_0^{\T}H^{-1}g_0)/\zeta^2$.

\item Set $A=S\begin{bmatrix}
        H & A_{12} \\
        A_{12}^{\T} & A_{22}
      \end{bmatrix}S^{\T}$, where $S=[S_1,S_2]$.
\end{enumerate}
Note that by the construction, the matrix $A$ is positive semidefinite when $H$ is positive definite.
This is because the Schur complement  of  $H$ in the matrix
$\begin{bmatrix}
        H & A_{12} \\
        A_{12}^{\T} & A_{22}
  \end{bmatrix}$:
\begin{align*}
A_{22}-A_{12}^{\T}H^{-1}A_{12}
& =A_{22}-ag_0^{\T}H^{-1}g_0a^{\T}=A_{22}-(g_0^{\T}H^{-1}g_0)aa^{\T} \\
& = \eta I -(g_0^{\T}H^{-1}g_0)aa^{\T}
= (g_0^{\T}H^{-1}g_0)(\|a\|^2 I - aa^{\T})
\end{align*}
is positive semidefinite since $H$ is positive definite and $g_0^{\T}H^{-1}g_0>0$.
%due to the fact that $A_{22}=\eta I_m$ with
%$$
%\eta=\frac {g_0^{\T}H^{-1}g_0}{\zeta^2}= (g_0^{\T}H^{-1}g_0)\|a\|^2.
%$$
%Therefore,  Therefore, $A$ is positive semidefinite.

\paragraph{Verification.}
Now we verify that CRQopt \eqref{eq:CRQopt} with $A$, $C$, $b$
constructed from the process above will yield
pLGopt \eqref{eq:pLGopt} with matrices $H$ and $g_0$ and scalar $\gamma=\sqrt{1-\zeta^2}$, as desired.

Recall the definitions in  \eqref{eq:def:gH}:
\begin{equation}\label{eq:def:gH2}
g_0=S_1^{\T}b_0,\,\, H=S_1^{\T}PAPS_1=S_1^{\T}AS_1\in\bbR^{(n-m)\times (n-m)}.
\end{equation}
By the construction of $A$,
$S_1^{\T}AS_1=H$, which is consistent with $H$ defined in \eqref{eq:def:gH2}.
Further recall that
$P$ is a projection matrix onto $\mathcal{N}(C^{\T})$ and the columns of $S_1$ form an orthonormal basis of $\mathcal{N}(C^{\T})$.
So $P=S_1S_1^{\T}$.
In addition, by the QR factorization \eqref{eq:qrc}, $(C^{\T})^\dag=S_2R^{-\T}$, and so $n_0=(C^{\T})^\dag b=S_2R^{-\T} b$.
By the definition of matrix $A$, $S_1^{\T}A S_2=A_{12}$,
we have
\begin{equation}\label{eq:g0verify}
S_1^{\T}b_0=S_1^{\T}PAn_0=S_1^{\T}S_1S_1^{\T}AS_2R^{-\T} b=S_1^{\T}A S_2R^{-\T} b=\zeta^2 A_{12}a=\zeta^2g_0a^{\T}a=g_0.
\end{equation}
which is consistent with $g_0$ defined in \eqref{eq:def:gH2}.
Finally,
$$\gamma=\sqrt{1-\|n_0\|^2}=\sqrt{1-\|S_2R^{-\T} b\|^2}=\sqrt{1-\|R^{-\T} b\|^2}=\sqrt{1-\|\zeta^2 a\|^2}=\sqrt{1-\zeta^2 }.$$

\subsection{Numerical results}
For testing purpose, we compute a solution $v_{\ast}$
%\marginpar{\tiny how $\lambda_\ast$ is computed?}
by the direct method in \cite{gagv:1989}
as a reference (exact) solution; otherwise it is generally unknown. We also
compute
$\kappa = \frac{\lambda_{\max}(H)-\lambda_\ast}{\lambda_{\min}(H)-\lambda_\ast}$
to examine our error bounds in Theorem~\ref{thm:nondeer}.

The Lanczos algorithm (Algorithm~\ref{alg:overall}) is applied
to solve CRQopt \eqref{eq:CRQopt} via QEPmin \eqref{eq:QEPmin} and via LGopt \eqref{eq:LGopt}.
For each computed $v^{(k)}$,
the $k$th iteration, we compute relative errors
\[
\err_1 = \frac{|(v^{(k)})^{\T}Av^{(k)}-v_\ast^{\T}Av_\ast|}{|v_\ast^{\T}Av_\ast|},
\quad
\err_2 = \|v^{(k)}-v_\ast\|,
\quad \mbox{and} \quad
\err_3=\frac{|\mu^{(k)}-\lambda_\ast|}{|\lambda_\ast|}.
\]
Since $\|v_\ast\|=1$, the absolute error
$\err_2$ is also  relative.
The stoping criterion for solving QEPmin \eqref{eq:QEPmin} is either
$\delta_k^{\QEPmin}< 10^{-15}$ or the  number of Lanczos steps reaches ${\tt maxit}=200$,
where $\delta_k^{\QEPmin}$ is defined in \eqref{eq:NRes-QEPmin}.
The stoping criterion for solving LGopt \eqref{eq:LGopt} is either
$\NRes_k^{\LGopt}<10^{-15}$ or the  number of Lanczos steps reaches ${\tt maxit}=200$.
%
%As a safeguard, the maximum number of Lanczos steps is set at ${\tt maxit}=200$.
%We also solve CRQopt \eqref{eq:CRQopt} via LGopt \eqref{eq:LGopt}
%and choose the tolerance for the residual of Lagrange equations defined
%in \eqref{eq:NRes-LGopt} that $\NRes_k^{\LGopt}<10^{-15}$ and make the maximum number of iterations $maxit=200$ to ensure we have run enough Lanczos steps.

\begin{example} \label{sec-crq-ex-correct}
In this example, we test the correctness and convergence behavior of the Lanczos algorithm to solve CRQopt \eqref{eq:CRQopt}.
Let $n=1100$, $m=100$, $\alpha=1$, $\beta=100$ or $1000$, and construct $H$   as in \eqref{eq:Hg0-eg} and $g_0$  as in \eqref{eq:g0}. For $(A, C, b)$, let
$\zeta=0.9$ and $a$ be random vector normalized to have norm $1/\zeta$ and then the rest follows Subsection \ref{sec-crq-ex-matrix}
in  constructing $A$, $C$ and $b$.

\begin{figure}
\begin{center}
\includegraphics[width=0.45\textwidth]{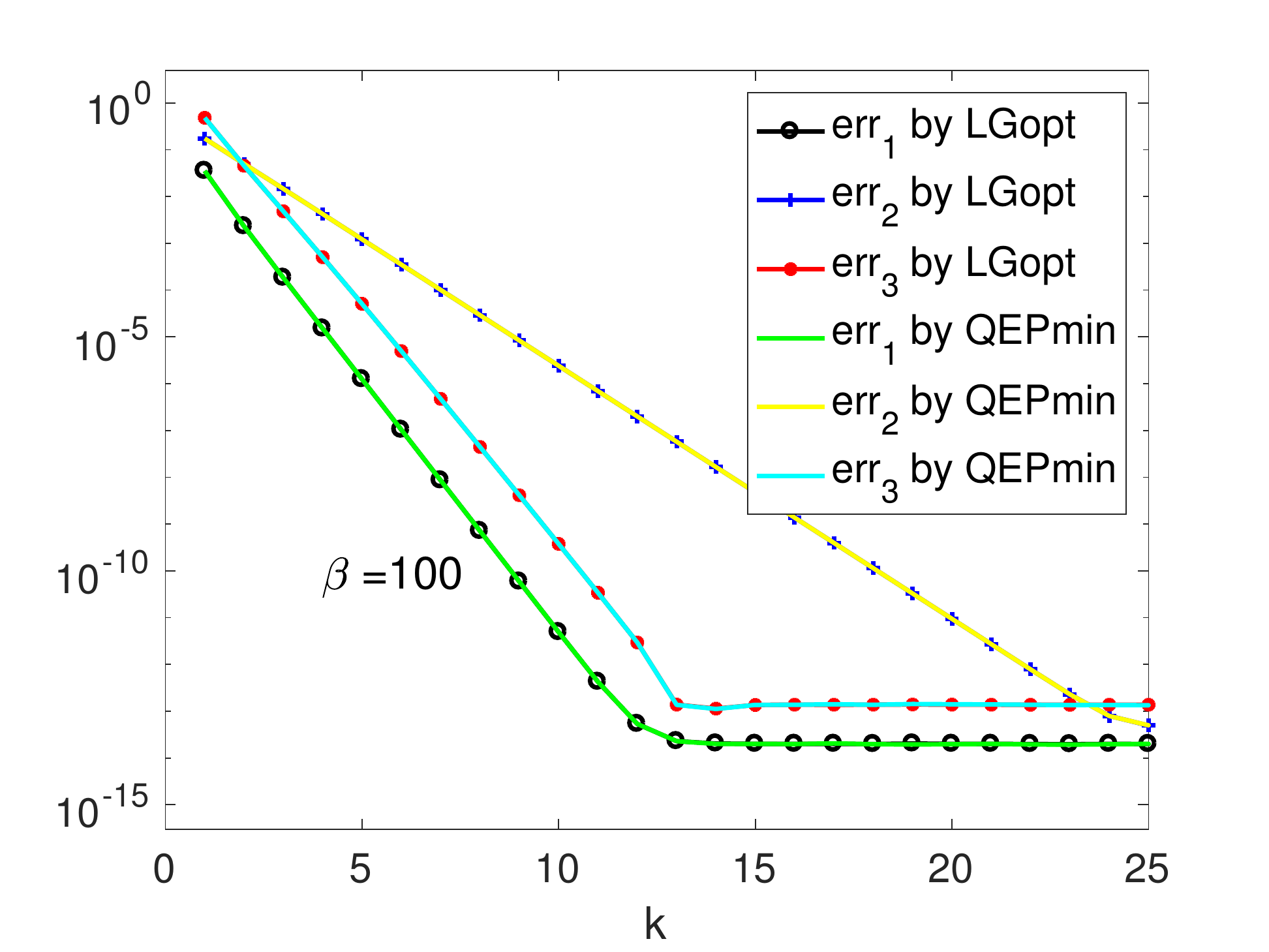} \quad
\includegraphics[width=0.45\textwidth]{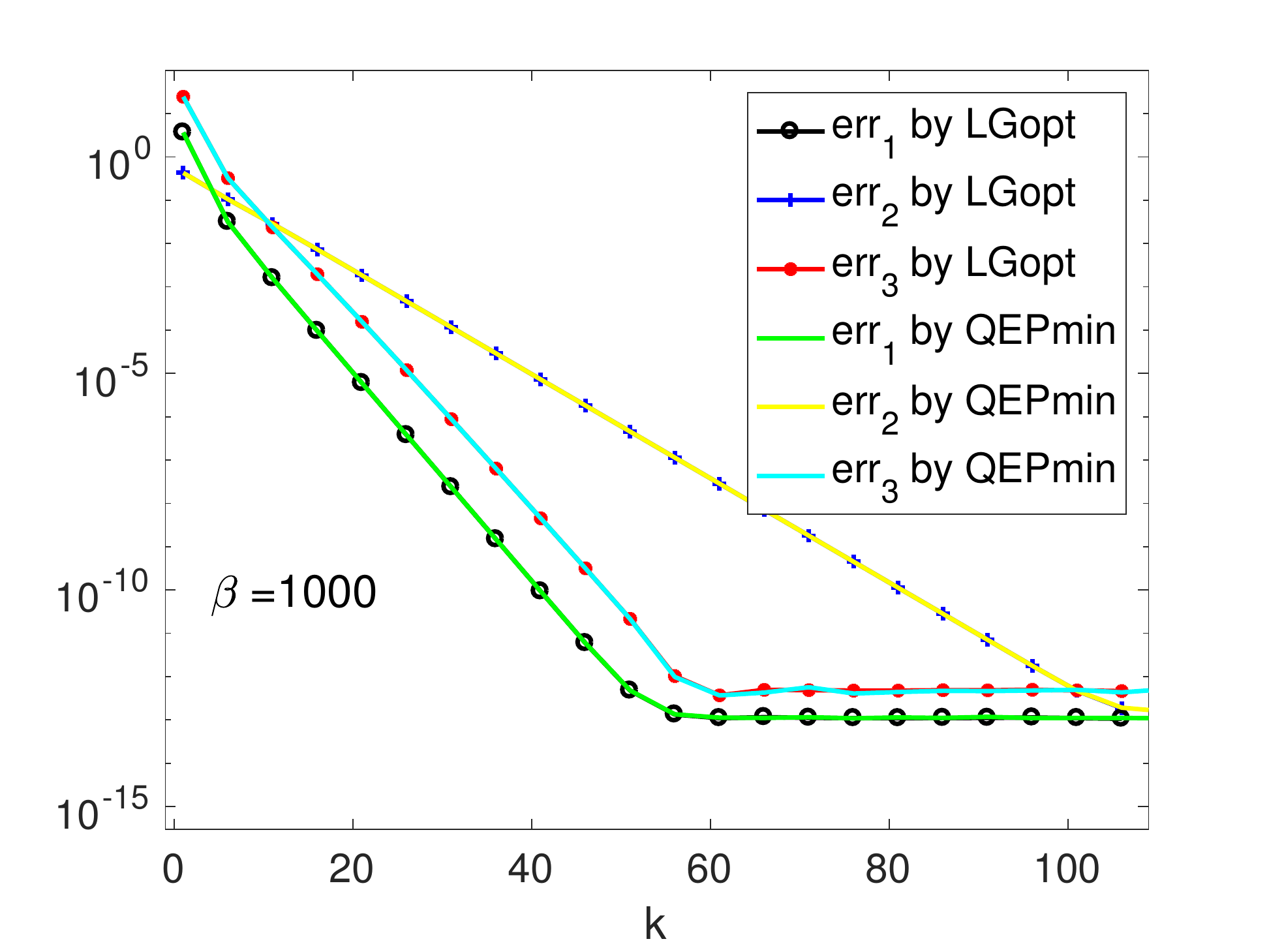}
\end{center}
\caption{Example~\ref{sec-crq-ex-correct}: history of
$\err_1$, $\err_2$
and  $\err_3$ for the cases where $\beta=100$ (left) and $\beta=1000$ (right).
}
\label{fig:correct}
\end{figure}

The convergence histories for  $\err_1$, $\err_2$
and  $\err_3$ are plotted in Figure \ref{fig:correct}. It can be seen that all
converge to the machine precision. %$0$.
Also  $\err_1$, $\err_2$ and $\err_3$  are the same, respectively, at every iteration whether
CRQopt \eqref{eq:CRQopt} is solved  via QEPmin \eqref{eq:QEPmin} or LGopt \eqref{eq:LGopt}, which is
consistent with our theory that solving rLGopt \eqref{eq:rLGopt} is equivalent
to solving rQEPmin \eqref{eq:rQEPmin}.
\end{example}

\begin{example} \label{ex:sharp}
We illustrate the sharpness of the
error bounds \eqref{eq:rLGopt-UBs} in Theorem~\ref{thm:nondeer} and the relationship
between the convergence rate of our Lanczos algorithm and $\kappa$.

\begin{figure}
\begin{center}
\includegraphics[width=0.45\textwidth]{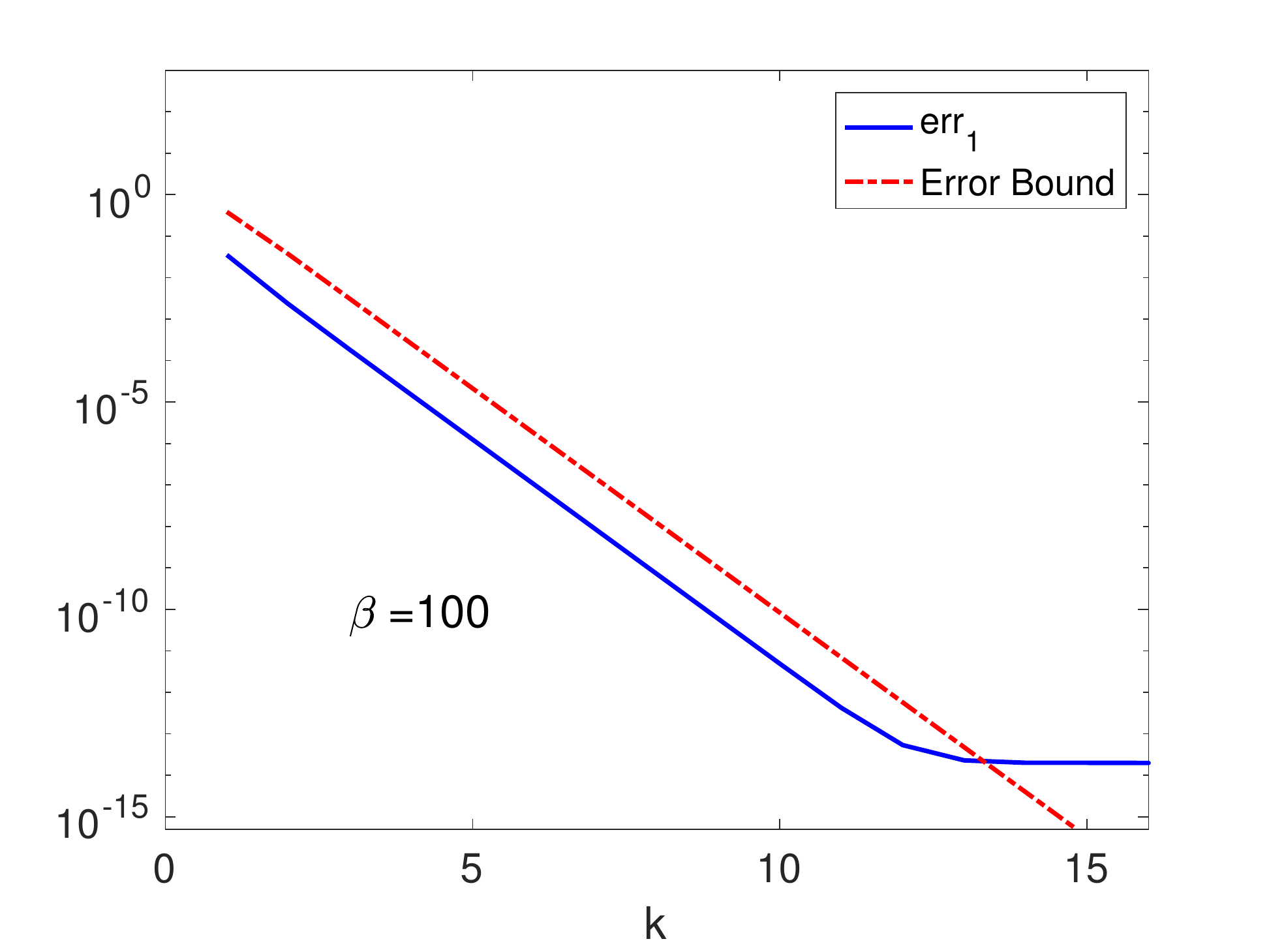}\quad\,
\includegraphics[width=0.45\textwidth]{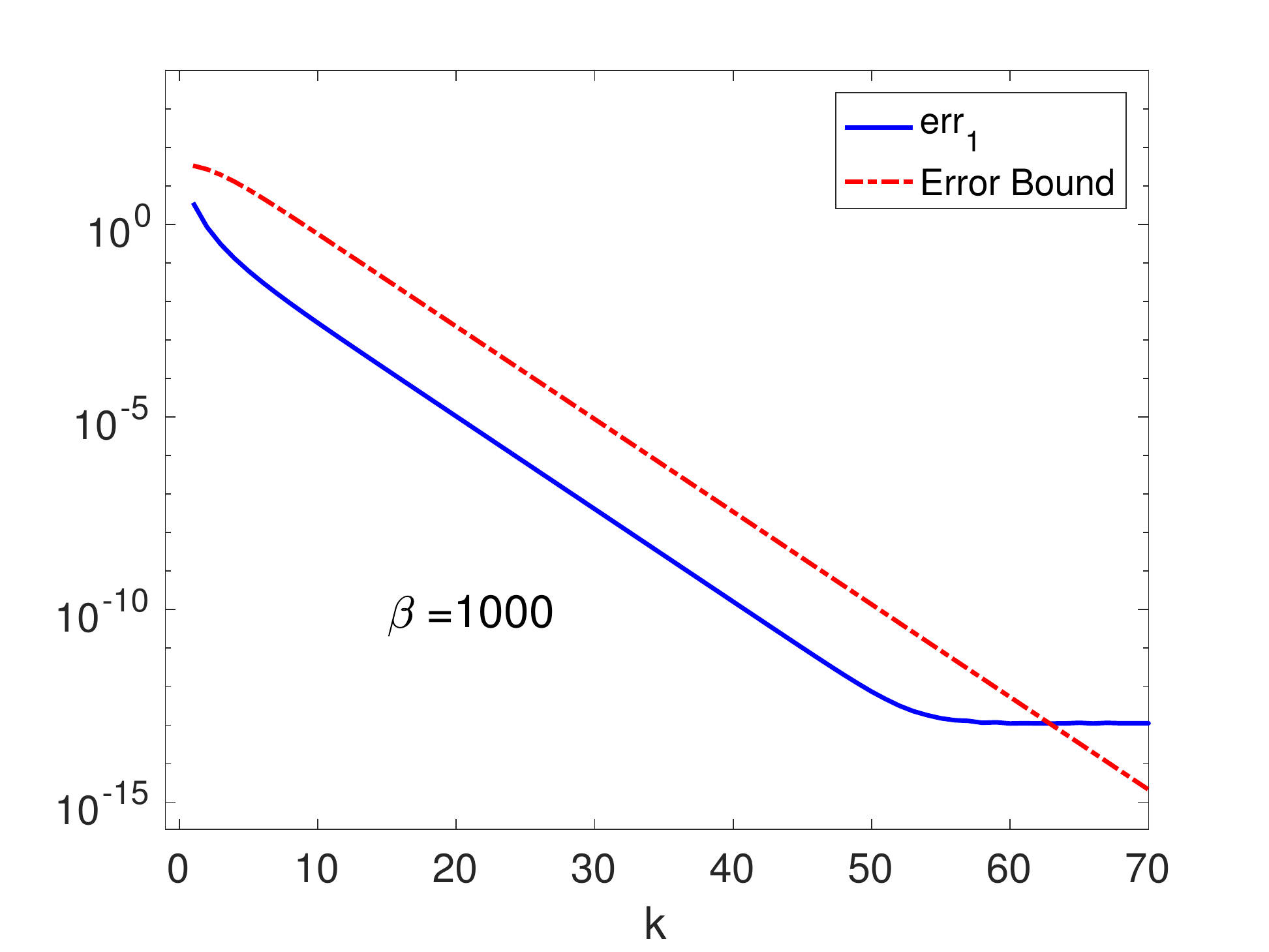} \\
\includegraphics[width=0.45\textwidth]{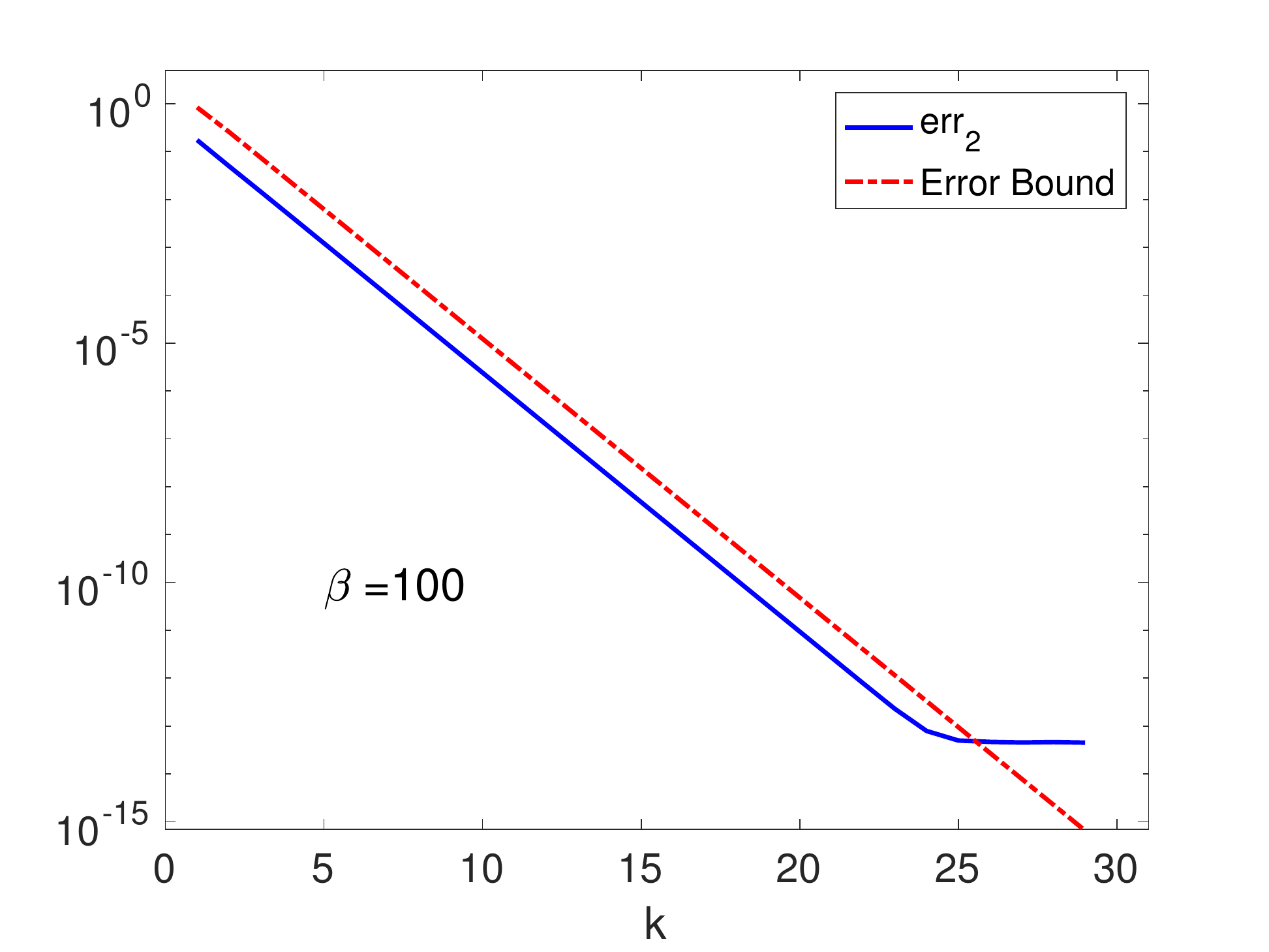}\quad\,
\includegraphics[width=0.45\textwidth]{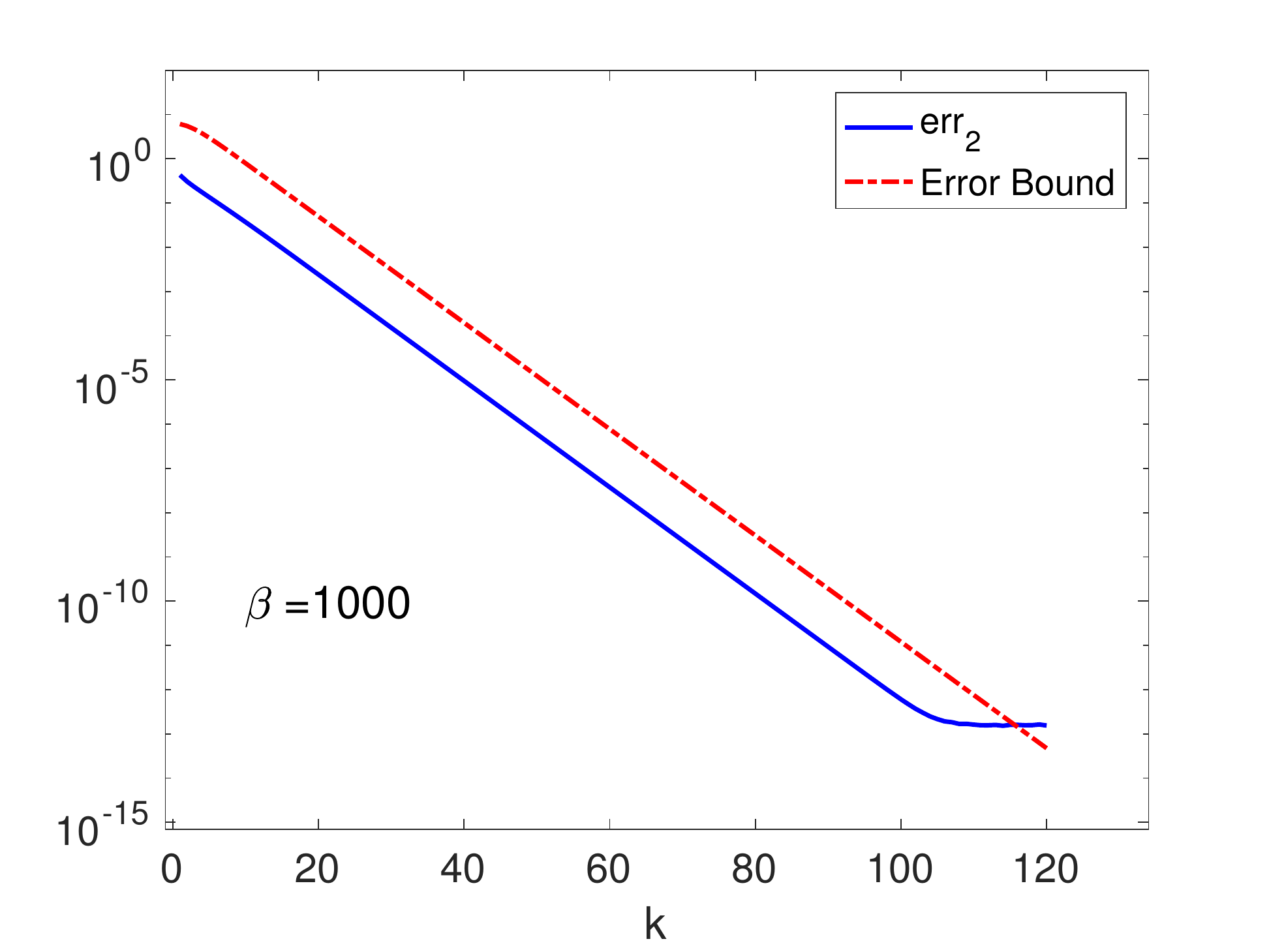}\\
\includegraphics[width=0.45\textwidth]{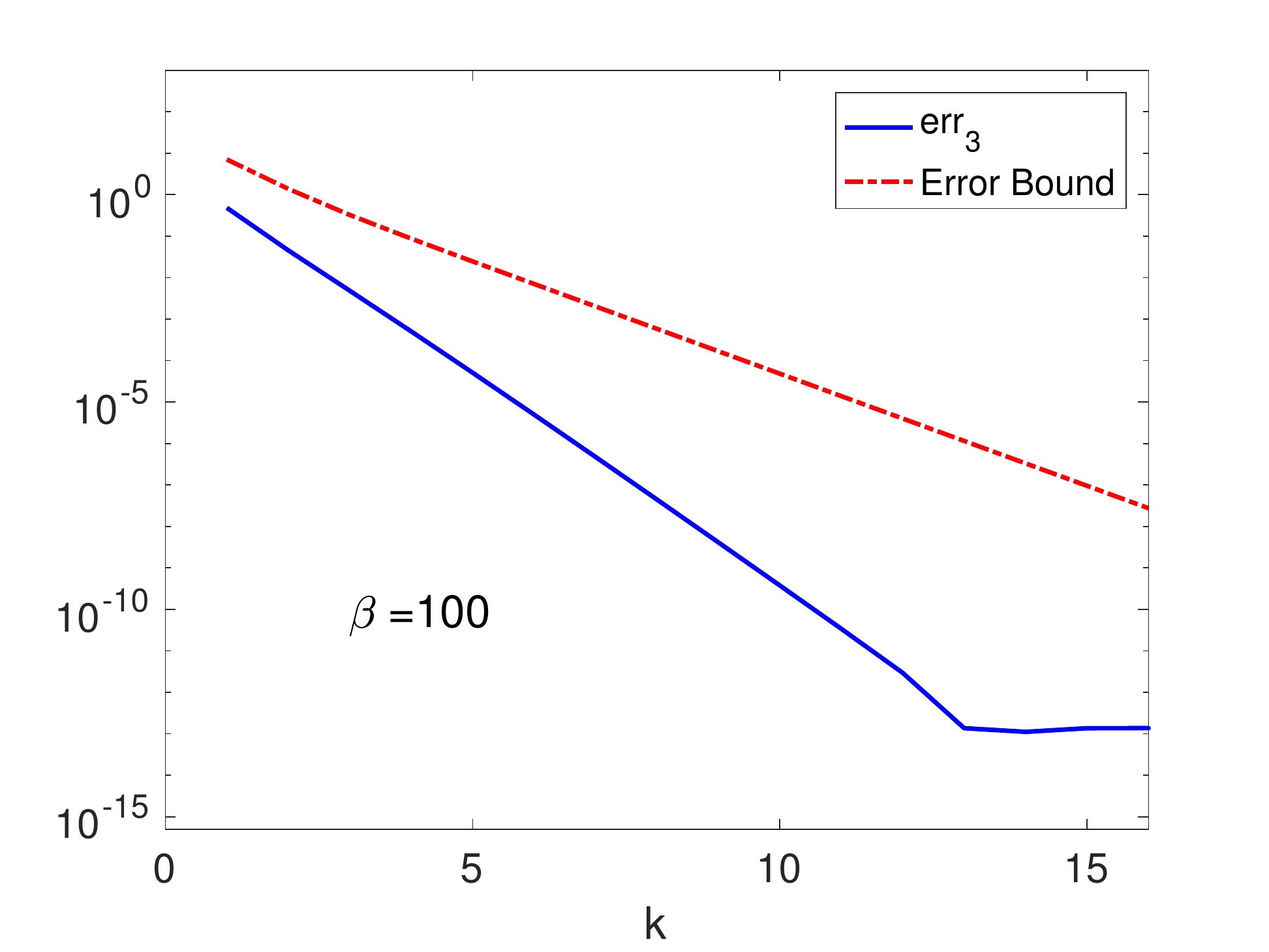}\quad\,
\includegraphics[width=0.45\textwidth]{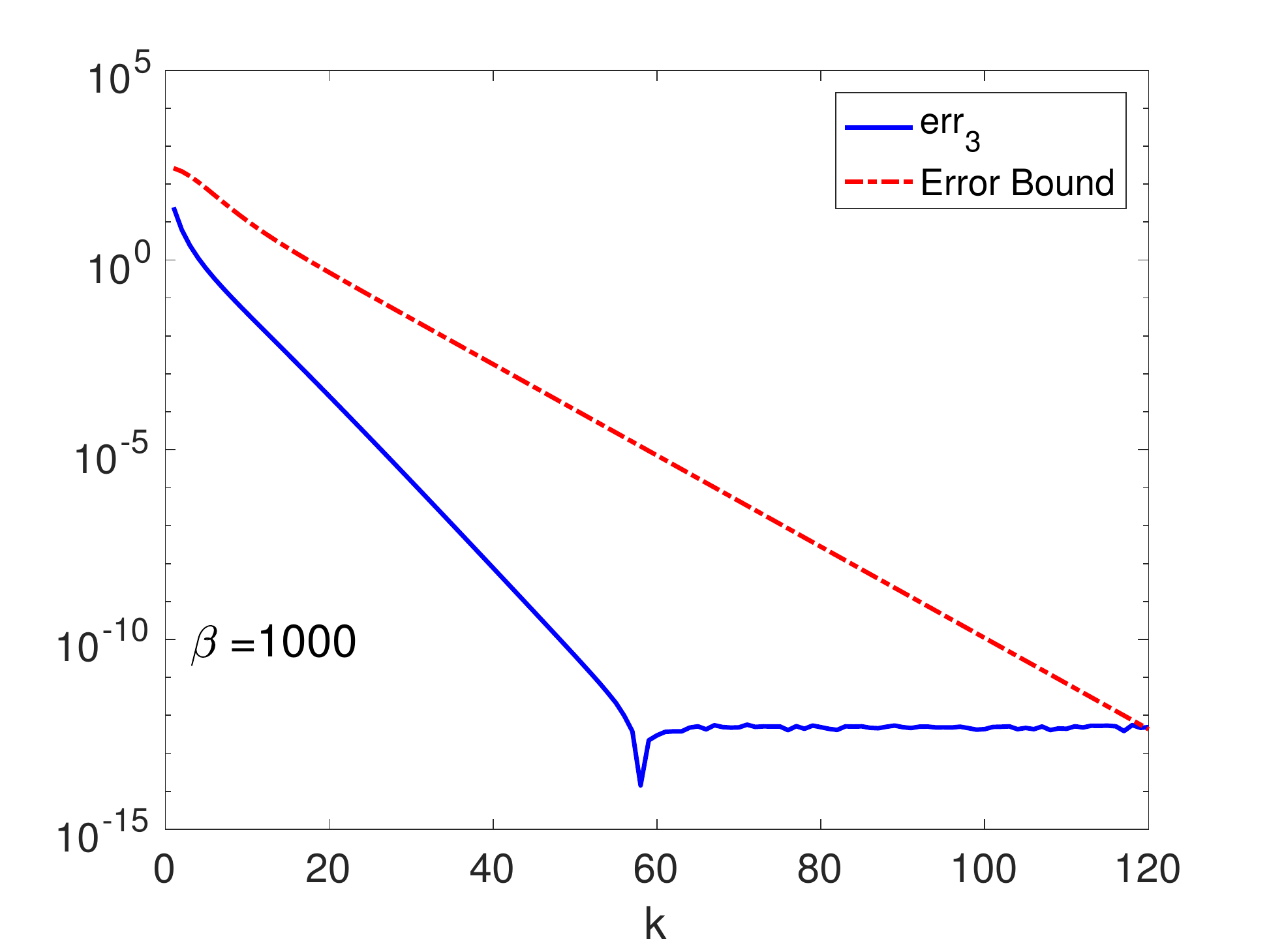}
\end{center}
\caption{Example~\ref{ex:sharp}: histories for $\err_1$ (first row), $\err_2$ (second row), $\err_3$ (third row) and their upper bounds for $\beta=100$ (left column) and $\beta=1000$ (right column).}
\label{fig:sharp}
\end{figure}

The same test matrices as in Example~\ref{sec-crq-ex-correct},
with $\beta = 100$ and  $1000$ are used.
We solve CRQopt \eqref{eq:CRQopt} by solving QEPmin \eqref{eq:QEPmin}
and choose the same parameters as in Example~\ref{sec-crq-ex-correct}. For $\alpha=1$ and $\beta=100$,
We calculate
$$
(\lambda_\ast,\kappa)
  =\begin{cases}
    (-42.6007,3.2706), &\quad\mbox{for $(\alpha,\beta)=(1,100)$}; \\
    (-18.2629,52.8613), &\quad\mbox{for $(\alpha,\beta)=(1,1000)$}.
   \end{cases}
$$
Judging from the corresponding $\kappa$, we expect our Lanczos algorithm will converge faster for the case $\beta=100$ than the case
$\beta=1000$. We plot in Figure \ref{fig:sharp} the convergence histories for
\begin{align*}
\err_1\,\,&\mbox{and its upper bound
       $\frac {16\gamma^2\|H-\lambda_{\ast} I\|}{v_{\ast}^{\T}Av_{\ast}}\,\left[\Gamma_{\kappa}^k+\Gamma_{\kappa}^{-k}\right]^{-2}$ by \eqref{eq:rLGopt-UBs-1}}, \\
\err_2\,\,&\mbox{and its upper bound $4\gamma\sqrt{\kappa}\left[\Gamma_{\kappa}^k+\Gamma_{\kappa}^{-k}\right]^{-1}$ by \eqref{eq:rLGopt-UBs-2}},\\
\err_3\,\,&\mbox{and its upper bound $\frac{16}{|\lambda_\ast|}\|H-\lambda_{\ast} I\|\,\left[\Gamma_{\kappa}^k+\Gamma_{\kappa}^{-k}\right]^{-2}+\frac{4}{\gamma|\lambda_\ast|}\sqrt{\kappa}\left[\Gamma_{\kappa}^k+\Gamma_{\kappa}^{-k}\right]^{-1}$ by \eqref{eq:rLGopt-UBs-3}}.
\end{align*}
%We point out that in practice, $v_{\ast}$ and $\kappa$ are unknown,
%but here we make $v_{\ast}$ and $\lambda_\ast$ available by solving
%CRQopt \eqref{eq:CRQopt}, applying the direct method \cite{gagv:1989},
%and computing
%$\kappa = \frac{\lambda_{\max}(H)-\lambda_\ast}{\lambda_{\min}(H)-\lambda_\ast}$.
The bounds for $\err_1$ and $\err_2$ by \eqref{eq:rLGopt-UBs-1} and \eqref{eq:rLGopt-UBs-2}
for both $\beta=100$ and $\beta=1000$ appear sharp.  However, the bound for $\err_3$ by \eqref{eq:rLGopt-UBs-3}
is pessimistic. In the plots, $\err_3$ goes to $0$ at about a similar  rate of $\err_1$, but the bounds by
\eqref{eq:rLGopt-UBs-2} and \eqref{eq:rLGopt-UBs-3} for $\err_3$ progress at the same rate as the bound by \eqref{eq:rLGopt-UBs-1}
for $\err_2$.
%, which is $O\left(\left[\Gamma_{\kappa}^k+\Gamma_{\kappa}^{-k}\right]^{-1}\right)$, rather than the order of $\err_1$, which is $O\left(\left[\Gamma_{\kappa}^k+\Gamma_{\kappa}^{-k}\right]^{-2}\right)$.
We unsuccessfully tried to establish a better bound for $\err_3$ to reflect what we just witnessed, but we offered
a plausible explanation  in Remark \ref{rm:pessimistic}.

As expected, $\err_1$, $\err_2$ and $\err_3$ go to $0$ faster for the case $\beta=100$ than the case  $\beta=1000$.
It is consistent with our convergence results in Theorem \ref{thm:nondeer} that our Lanczos algorithm for CRQopt \eqref{eq:CRQopt}
converges faster when $\kappa$ is smaller.
\end{example}

\begin{example}\label{ex:nearhard}
We consider an example  where the error bounds in Theorem \ref{thm:nondeer} are pessimistic, while those by \eqref{eq:bound21}
can correctly reveal the speed of convergence. This occurs
when CRQopt  is a ``{\em nearly hard case}'', i.e.,
where the optimal value of the corresponding
pLGopt \eqref{eq:pLGopt} $\lambda_\ast \approx \lambda_{\min}(H)$.
Specifically, we choose $n=1100$, $m=100$, $\zeta=0.9$, $a$  a random
vector with the norm $1/\zeta$, and
$$
H=\diag(\tau_{0\,n-m-2}^{\trans},\tau_{1\,n-m-2}^{\trans},\dots,\tau_{n-m-2\,n-m-2}^{\trans},1)
$$
with   $(\alpha,\beta)=(2,1000)$ in \eqref{eq:omega-tau} and \eqref{eq:transChebENodes}, and
$$g_0=\left[e^{\eta},e^{2\eta},\cdots,e^{(n-m)\eta}\right]^{\T}$$
where $\eta=-5\times 10^{-3}$.
%We construct this $g_0$ because we want to construct a near hard case such that the components of $g_0$ on the eigenspaces corresponding with small eigenvalues of $H$ are small.
In this case, $\lambda_{\min}(H)=1$ and $\lambda_\ast=0.9845$, so $\lambda_{\min}(H)\approx\lambda_\ast$ and
thus it is a nearly hard case. It is computed that%
\begin{figure}
\begin{center}
\includegraphics[width=0.31\textwidth]{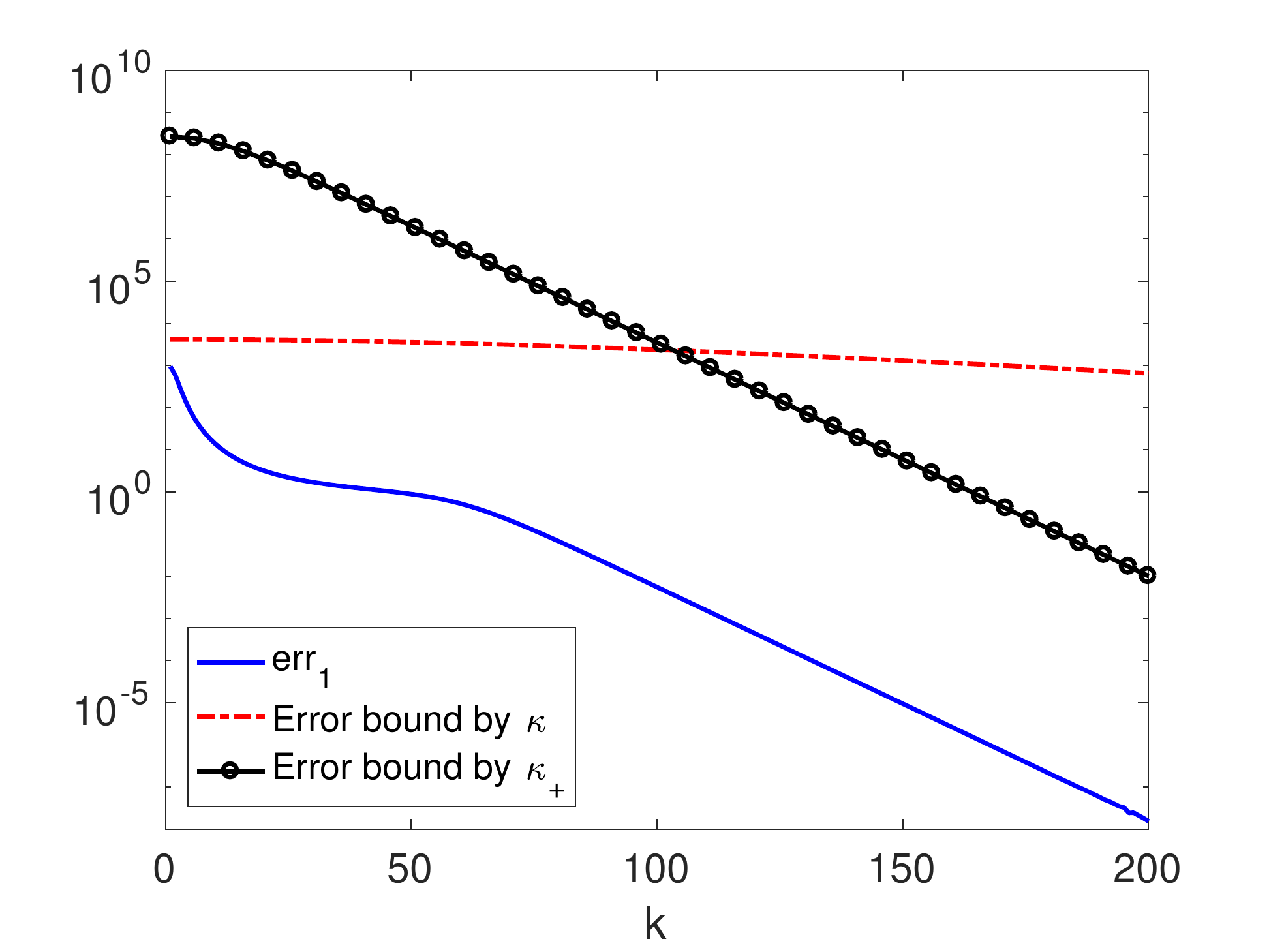}\quad\,
\includegraphics[width=0.31\textwidth]{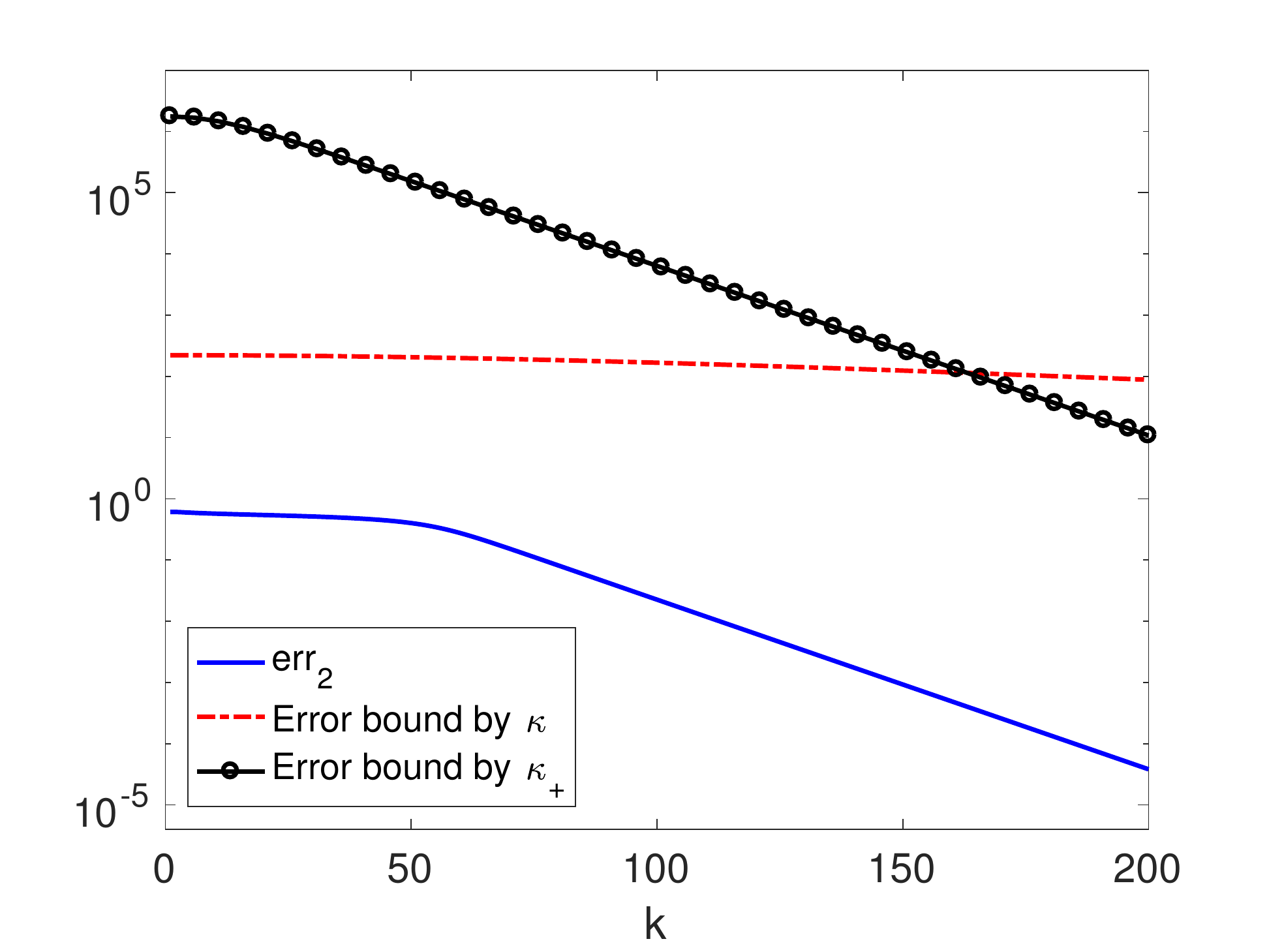}
\includegraphics[width=0.31\textwidth]{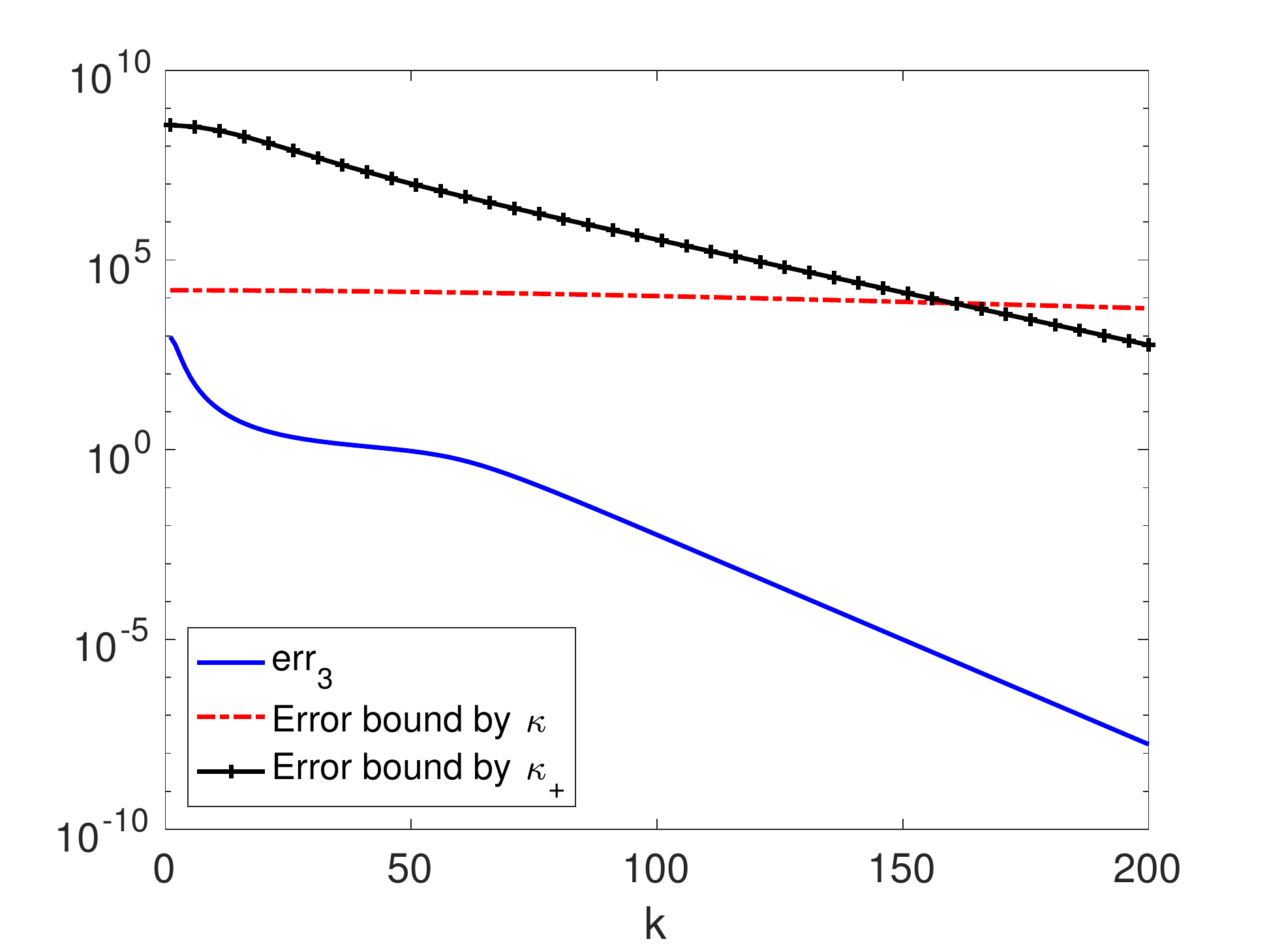}
\end{center}
\caption{Example~\ref{ex:nearhard}:  histories of $\err_1$, $\err_2$, $err_3$ and their upper bounds. ``Error bound by $\kappa$'' and ``Error bound by $\kappa_+$'' means upper bounds in \eqref{eq:rLGopt-UBs} and \eqref{eq:bound21}, respectively.}
\label{fig:notsharp}
\end{figure}
$$
\kappa = \frac{\lambda_{\max}(H)-\lambda_\ast}
              {\lambda_{\min}(H)-\lambda_\ast}
 \approx 6.4466\times 10^4
$$
which is big.  We solve the associated CRQopt \eqref{eq:CRQopt} via QEPmin \eqref{eq:QEPmin}.
%We choose the tolerance for the bound of the residual of QEP
%$\delta_k^{\QEPmin}< 10^{-15}$, where $\delta_k^{\QEPmin}$ is defined in \eqref{eq:NRes-QEPmin}, and we choose the maximum number of Lanczos iterations $maxit=200$.
 In Figure \ref{fig:notsharp}, we plot the convergence history:
\begin{align*}
\err_1&\mbox{, its upper bounds
       $\frac {16\gamma^2\|H-\lambda_{\ast} I\|}{v_{\ast}^{\T}Av_{\ast}}\,\left[\Gamma_{\kappa}^k+\Gamma_{\kappa}^{-k}\right]^{-2}$
                by \eqref{eq:rLGopt-UBs-1}, and}\\
       &\mbox{$\frac{16\gamma^2\|H-\lambda_{\ast} I\|(\theta_{\max}-\theta_{\min})}{(\theta_{\min}-\lambda_\ast){v_{\ast}^{\T}Av_{\ast}}}\,\left[\Gamma_{\kappa_+}^{(k-1)}+\Gamma_{\kappa_+}^{-(k-1)}\right]^{-2}$ by \eqref{eq:bound211}}, \\
\err_2&\mbox{, its upper bounds $4\gamma\sqrt{\kappa}\left[\Gamma_{\kappa}^k+\Gamma_{\kappa}^{-k}\right]^{-1}$ by \eqref{eq:rLGopt-UBs-2}, and}\\
&\mbox{$4\gamma\sqrt{\kappa}\frac{\theta_{\max}-\theta_{\min}}{\theta_{\min}-\lambda_\ast}\left[\Gamma_{\kappa_+}^{(k-1)}+\Gamma_{\kappa_+}^{-(k-1)}\right]^{-1}$ by \eqref{eq:bound212}},\\
\err_3&\mbox{, its upper bounds $\frac{16}{|\lambda_\ast|}\|H-\lambda_{\ast} I\|\,\left[\Gamma_{\kappa}^k+\Gamma_{\kappa}^{-k}\right]^{-2}+\frac{4\|b_0\|}{\gamma|\lambda_\ast|}\sqrt{\kappa}\left[\Gamma_{\kappa}^k+\Gamma_{\kappa}^{-k}\right]^{-1}$ by \eqref{eq:rLGopt-UBs-3}, and}\\
     &\mbox{$\frac{\theta_{\max}-\theta_{\min}}{|\lambda_\ast|(\theta_{\min}-\lambda_\ast)}\left[ 16\|H-\lambda_{\ast} I\|\,\left[\Gamma_{\kappa_+}^{(k-1)}+\Gamma_{\kappa_+}^{-(k-1)}\right]^{-2}+\frac{4}{\gamma}\|b_0\|\sqrt{\kappa}\left[\Gamma_{\kappa_+}^{(k-1)}+\Gamma_{\kappa_+}^{-(k-1)}\right]^{-1}\right]$}\\&\mbox{ by \eqref{eq:bound213}}.
\end{align*}
It can be observed that
The error bounds by Theorem \ref{thm:nondeer} decay much slower than $\err_1$, $\err_2$ and $\err_3$ in this ``near hard case''.
This is an example for which  $\kappa$ is large but $\kappa_+$ is small:
$$
\kappa_+:= \frac{\theta_{\max}-\lambda_{\ast}}{\theta_2-\lambda_{\ast}}\approx  983.7702,
$$
As commented in Remark \ref{rm:bound2}, sharper bounds like ones by \eqref{eq:bound21} should be used.
They are also included in Figure \ref{fig:notsharp}. We can see that the bounds \eqref{eq:bound21} correctly reflect
the speed of convergence, but they are bigger than the corresponding errors by several order of magnitudes.
% that when, the convergence rate for our Lanczos algorithm can be faster than Theorem \ref{thm:nondeer} suggests. In this example,
% which is smaller than $\kappa$. Bounds \eqref{eq:bound21} are represented by $\kappa_+$ and in this case they start from larger value but suggest faster convergence rate than bounds \eqref{eq:rLGopt-UBs} in Theorem \ref{thm:nondeer}.
\end{example}

\begin{example} \label{sec-crq-ex-res}
In this example, we test the effectiveness of the residual bound $\delta_k^{\QEPmin}$ in \eqref{eq:NRes-QEPmin}.
We use the same test problem as in Example \ref{sec-crq-ex-correct} for both  $\beta=100$  and $\beta=1000$.
%For the residual bound $\delta_k^{\QEPmin}$, 
We
run our Lanczos algorithm for QEPmin \eqref{eq:QEPmin} and record the residual $\NRes_k^{\QEPmin}$
and its bound  $\delta_k^{\QEPmin}$ defined in \eqref{eq:NRes-QEPmin} for every Lanczos step.
They are plotted in Figure \ref{fig:QEPres}. We observe that
both $\NRes_k^{\QEPmin}$ and
$\delta_k^{\QEPmin}$
in \eqref{eq:NRes-QEPmin}
converge to $0$ at the same rate, suggesting $\delta_k^{\QEPmin}$ is an very effective upper bound of
the residual $\NRes_k^{\QEPmin}$.
\begin{figure}
\begin{center}
\includegraphics[width=0.45\textwidth]{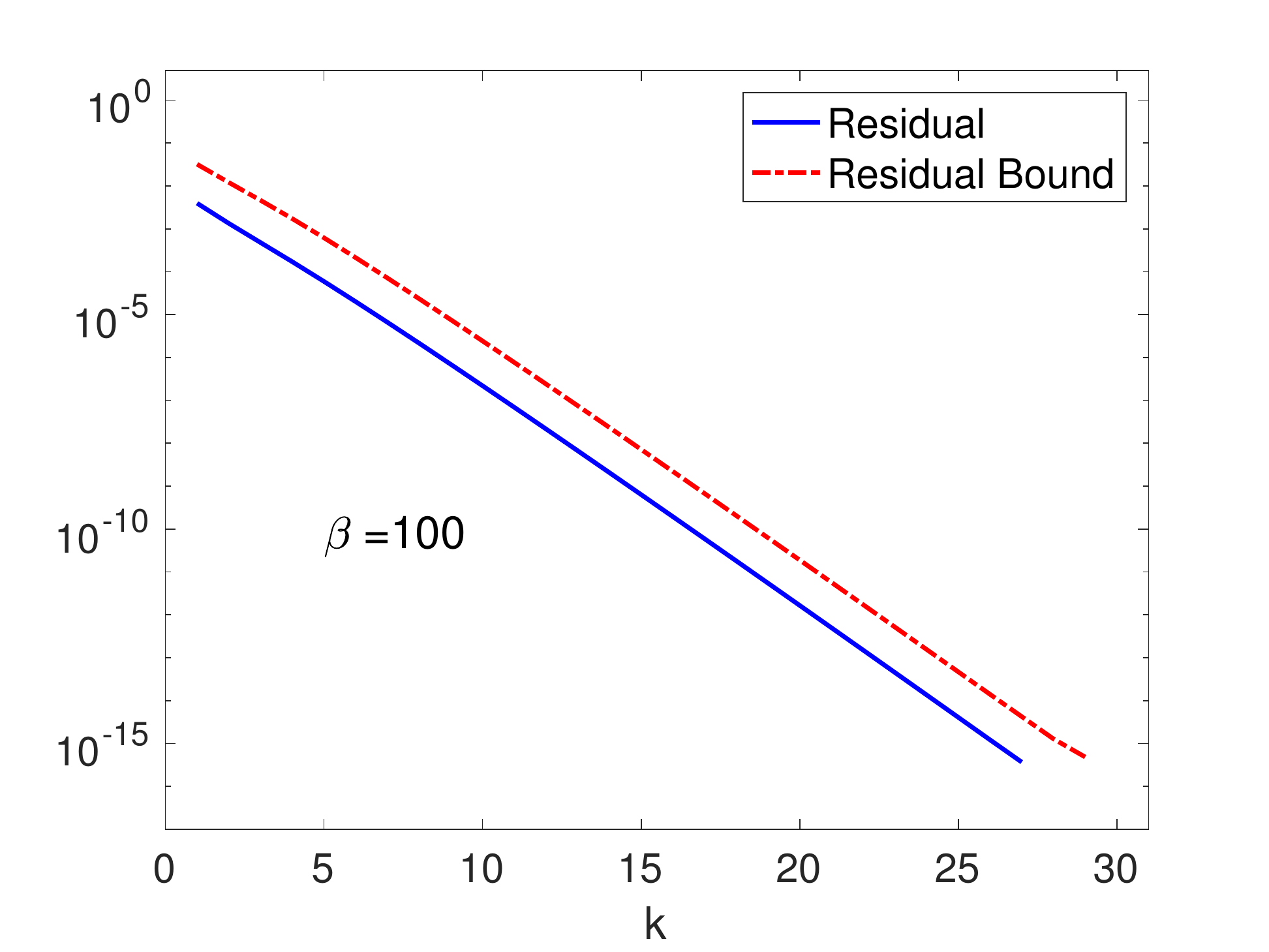}\quad\,
\includegraphics[width=0.45\textwidth]{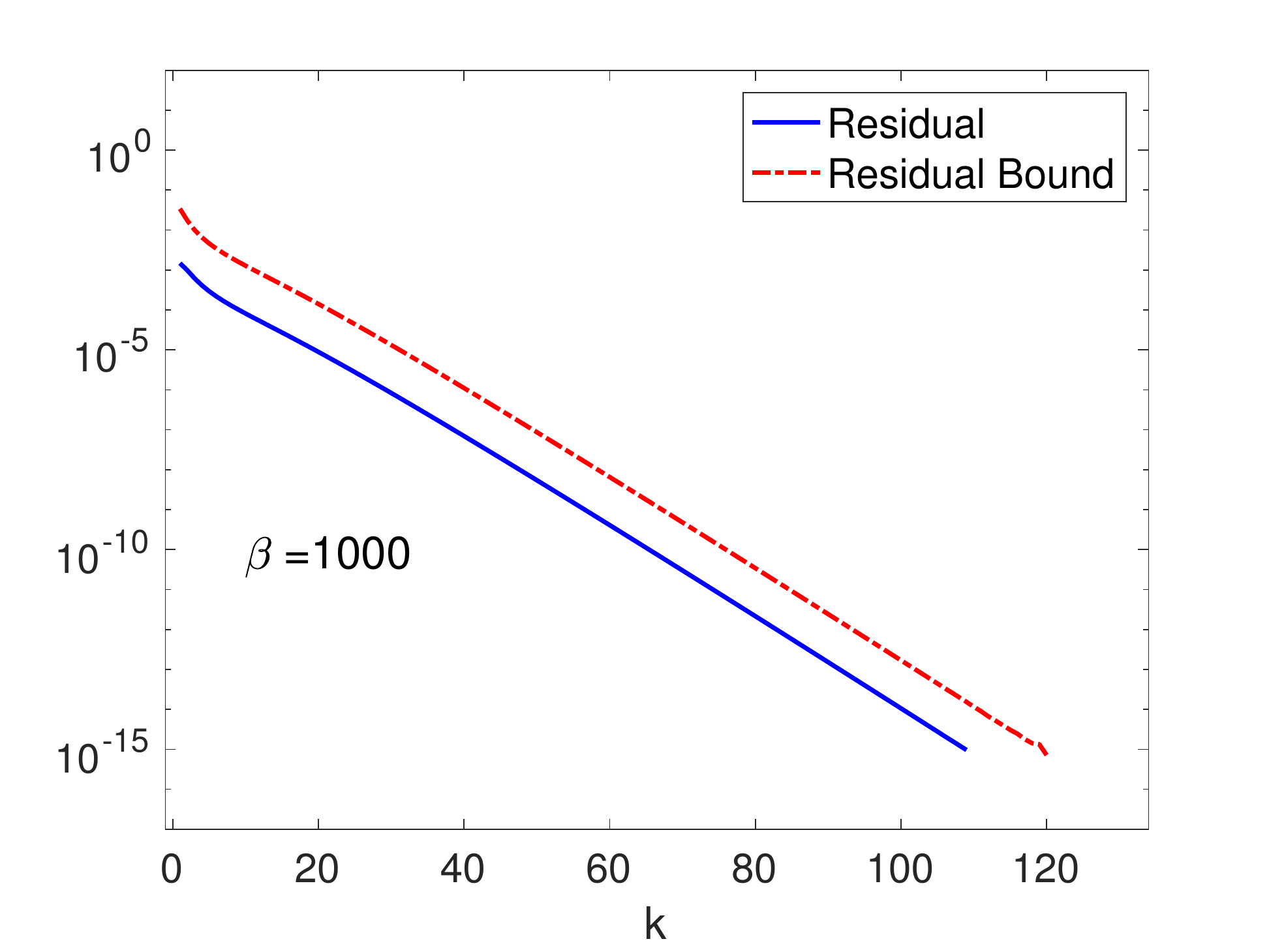}
\end{center}
\caption{Example~\ref{sec-crq-ex-res}: relative residual of QEP $\NRes_k^{\QEPmin}$  and the bound of the relative  residual $\delta_k^{\QEPmin}$ for the case where $\beta=100$ (left) and $\beta=1000$ (right).}
\label{fig:QEPres}
\end{figure}
\end{example}

%\section{Summary}\label{sec-summary}
%\Blue{According to our theory, solving CRQopt \eqref{eq:CRQopt} is equivalent to solving LGopt \eqref{eq:LGopt} and QEPmin \eqref{eq:QEPmin}, and Lanczos algorithm is suitable for solving LGopt \eqref{eq:LGopt} and QEPmin \eqref{eq:QEPmin}. We give a convergence analysis of the Lanczos algorithm. Numerical examples show the correctness of the Lanczos algorithm and the sharpness of the bound.
%}

%-----------------------------------------------------------------------

%\newpage
\section{Application to the constrained clustering}\label{sec-app}
In this section, we
use semi-supervised learning for clustering
as an application of CRQopt~\eqref{eq:CRQopt}.
We first discuss unconstrained clustering
in Section \ref{sec-app-unconstrain} and
then discuss a new model for constrained clustering
in Section \ref{sec-app-constrain}.
Numerical experiments are
shown in Sections \ref{sec-app-setting} and \ref{sec-app-result}.

\subsection{Unconstrained clustering}\label{sec-app-unconstrain}
Clustering is an important technique for data analysis and
is widely used in machine learning \cite[Chapter 14.5.3]{frht:2001},
bioinformatics \cite{pem:2005}, social science \cite{miss:2007} and
image analysis \cite{shmj:2000}. Clustering uses some similarity metric
to group data into different categories. In this section,
we discuss the normalized cut,  a spectral clustering method that
are popular for image segmentation \cite{shmj:2000,vou:2007}.

%Normalized cut \cite{shmj:2000} is defined on an undirected graph.
Given an undirected graph $G=(\cV,\cE)$ whose edge weights are represented by an affinity matrix $W=[w_{ij}]$, we define the  {\em cut\/} of a
partition on its vertices $\cV$ into
two disjoint sets $\cA$ and $\cB$, i.e., $\cA\cup \cB=\cV$, $\cA\cap \cB=\emptyset$ as
\begin{equation}\label{eq:cut}
\cut(\cA,\cB)=\sum_{i\in\cA, j\in\cB} w_{ij}.
\end{equation}
Intuitively one would minimize the $\cut$ to achieve an
optimal bipartition of the graph $G$, but it often results in a partition $(\cA,\cB)$ with one of them
containing only a few isolated vortices in the graph while the other containing the rest.
Such a bipartition is not balanced and not useful in practice.
To avoid such an unnatural
bias that leads to small sets of isolated  vortices,
the following {\em normalized cut\/} \cite{shmj:2000}
is introduced: 
\begin{equation}\label{eq:ncut}
\Ncut(\cA,\cB)=\frac{\cut(\cA,\cB)}{\vol(\cA)}
           +\frac{\cut(\cA,\cB)}{\vol(\cB)},
\end{equation}
where
$$
\vol(\cA)=\sum_{i\in\cA, j\in \cV}w_{ij}
\quad \mbox{and} \quad
\vol(\cB)=\sum_{i\in \cB, j\in \cV}w_{ij}.
$$
It turns out that minimizing $\Ncut(\cA,\cB)$ usually yields a more balanced bipartition.
Let
$$
c_+=\sqrt{\frac{\vol(\cB)}{\vol(\cA)\cdot
 \vol(\cV)}}
\quad \mbox{and} \quad
c_-=-\sqrt{\frac{\vol(\cA)}{\vol(\cB)
    \cdot \vol(\cV)}},
$$
and $x\in\bbR^n$ ($n=|\cV|$, the cardinality of $\cV$) be the 
indicator vector for bipartition $(\cA,\cB)$, i.e.,
\begin{equation}\label{eq:labelcut}
x_{(i)}=\left\{
\begin{aligned}
c_+,\quad i\in \cA,\\
c_-,\quad i\in \cB,
\end{aligned}
\right.
\end{equation}
and $D$ be a diagonal matrix with the row sums of $W$ on the diagonal, i.e., $D=\diag(W\bone)$.
Then it can be verified that
$$
\Ncut(\cA,\cB) = x^{\T}(D-W)x, \quad
x^{\T}Dx=1, \quad
(Dx)^{\T}{\bf 1}=0, 
$$
where ${\bf 1}$ is a vector of ones. Therefore in order to minimize $\Ncut(\cA,\cB)$, we will
solve
the following combinatorial optimization problem
\begin{subequations}\label{eq:ncut1}
\begin{empheq}[left={\empheqlbrace}]{alignat=2}
\min~& x^{\T}(D-W)x,\\
\mbox{s.t.}~& x_{(i)}\in \{c_+,\, c_-\},\\
& x^{\T}Dx=1,\\
& (Dx)^{\T} {\bf 1}=0.
\end{empheq}
\end{subequations}
However, the problem \eqref{eq:ncut1} is a discrete optimization problem and known to be NP-complete.
A common practice to make it numerical feasible is to relax $x$ to
a real vector and solve instead the
following optimization problem
\begin{subequations}\label{eq:ncut2}
\begin{empheq}[left={\empheqlbrace}]{alignat=2}
\min~& x^{\T}(D-W)x,\\
\mbox{s.t.}~& x^{\T}Dx=1,\\
& (Dx)^{\T} {\bf 1}=0, \\
& x \in \bbR^n.
\end{empheq}
\end{subequations}
Under the assumption that $D$ is positive definite,
by the Courant-Fisher variational principle \cite[Sec 8.1.1]{govl:2013},
solving \eqref{eq:ncut2} is equivalent to finding the eigenvector
$x$ corresponding to the second smallest eigenvalue of
the generalized symmetric definite eigenproblem
$$(D-W)x=\lambda Dx.$$

Note that the setting here is  different from the one
in \cite{shmj:2000}, where the indicator vector
$x_{(i)}\in\{1,-b\}$ and
$b=\frac{\vol(\cA)}{\vol(\cB)}$. Instead of minimizing a quotient of two quadratic functions in \cite{shmj:2000}, we use the constraint that $x^{\T}Dx=1$.
The model \eqref{eq:ncut1} is similar to the one
in  \cite[section 5.1]{vou:2007}, where they use the number of
vertices in the sets $\cA$ and $\cB$ instead of the volumes.
The model \eqref{eq:ncut1} is derived in a similar way to
the derivation in \cite[section 5.1]{vou:2007}.

\subsection{Constrained clustering}\label{sec-app-constrain}
When partial grouping information is known in advance,
we can use  partial grouping information to set up different models for better clustering.
These models are known as constrained clustering.
Existing methods for constrained spectral clustering includes
implicitly incorporating the constraints into
Laplacians \cite{chc:2015,jixb:2017} and imposing the constraints
in linear forms \cite{erof:2011,xuls:2009,yus:2004} or
bilinear forms \cite{waqd:2014}.

We encode the partial grouping information into a linear constraint, which
can be either homogeneous \cite{yus:2004}
or nonhomogeneous \cite{erof:2011,xuls:2009}.
In \cite{erof:2011}, the authors set up a model
where the objective function is the quotient of two quadratic
functions and used hard coding for the known associations of pixels
to specific classes in terms of linear constraints.
In \cite{xuls:2009}, the authors used a model for which
the objective function is quadratic and encoded
known labels by linear constraints.
%but they did not write the constraints in exact mathematical formulas.
%\Blue{We propose a model such that the objective function
%is a quadratic function
%and hard encode some pixels should belong to a specific class in a
%linear constraint, which is theoretically consistent with the settings of unconstrained clustering and numerically it's stable.
%}
This  is an approach that we  take to set up the model.

Let $\mathcal{I}=\{i_1,\cdots,i_{\ell}\}$ be the index set for which we have
the prior information such as $\cI\subseteq \cA$. According
to \eqref{eq:labelcut},  we  set $x_{(i)}=c_+$ for $i\in \mathcal{I}$.
Similarly, let $\mathcal{J}=\{j_1,\cdots,j_k\}$ be the index set
for which we have the prior information that $\cJ\subseteq\cB$, and
we set $x_{(j)}=c_-$ for $j\in \mathcal{J}$.
This leads to the following discrete constrained normalized cut problem
\begin{subequations}\label{eq:ncutcon1}
\begin{empheq}[left={\empheqlbrace}]{alignat=2}
\min~& x^{\T}(D-W)x,\\
\mbox{s.t.}~& x_{(i)}\in\left\{c_+,\,c_-\right\},\\
& x^{\T}Dx=1,\\
& (Dx)^{\T} {\bf 1}=0,\\
& x_{(i)}=c_+ \quad\mbox{for}\quad i\in \mathcal{I},\\
& x_{(i)}=c_- \quad\mbox{for}\quad i\in \mathcal{J}.
\end{empheq}
\end{subequations}
%\Blue{
However, there are two imminent issues associated with the model \eqref{eq:ncutcon1}:
\begin{enumerate}
  \item the combinatorial optimization \eqref{eq:ncutcon1} is NP-hard;
  \item the model is incomplete because to calculate $c_+$ and $c_-$
we need to know $\vol(\cA)$ and $\vol(\cB)$, which are unknown before the clustering.
\end{enumerate}
Common workarounds, which we use, are as follows.
For the first issue, we relax the model \eqref{eq:ncutcon1}
by allowing $x$ to be a real vector, i.e., $x \in \mathbb{R}^n$. 
For the second issue, we use
${\frac{\vol(\mathcal{J})}{\vol(\mathcal{I})}}$
as an estimate of  ${\frac{\vol(\cB)}{\vol(\cA)}}$ to get
%}
\[
c_+\approx\widehat{c}_+=\sqrt{\frac{\vol(\mathcal{J})}
                 {\vol(\mathcal{I})\cdot\vol(\cV)}},\,\,
c_-\approx\widehat{c}_-=-\sqrt{\frac{\vol(\mathcal{I})}
            {\vol(\mathcal{J})\cdot\vol(\cV)}}.
\]
By these relaxation, we reach a computational feasible model:
\begin{subequations}\label{eq:ncutcon2}
\begin{empheq}[left={\empheqlbrace}]{alignat=2}
\min~& x^{\T}(D-W)x,\\
\mbox{s.t.}~& x^{\T}Dx=1,\\
& (Dx)^{\T} {\bf 1}=0,\\
& x_{(i)}=\widehat{c}_+, \quad i\in \mathcal{I},\\
& x_{(i)}=\widehat{c}_-. \quad i\in \mathcal{J},
\end{empheq}
\end{subequations}
The last three equations are linear constraints and can be collectively written as
 a linear system of equations:
$$N^{\T}x=b.$$
Let $v=D^{1/2}x$, and define
\[
A=D^{-1/2}(D-W)D^{-1/2}
\quad \mbox{and} \quad
C=D^{1/2}N.
\]
Then the optimization problem \eqref{eq:ncutcon2}
is turned into CRQopt \eqref{eq:CRQopt} with matrices $A$, $C$ and $b$ just defined.

\subsection{Numerical results}\label{sec-app-result}

\paragraph{Experimental setting.}\label{sec-app-setting}

%The number of pixels $n$,  parameters $\kappa$ and $r$ and size of linear constraints $m$ for each image cut problems are shown in Table \ref{Table:parameter}.

For a grayscale image, we can construct
a weighted graph $G=(\cV,\cE)$ by taking each pixel
as a node and connecting each pair $(i,j)$ of
pixel
$i$ and $j$ by an edge with a weight given by
\begin{equation}\label{eq:weight}
w_{ij}=e^{-\frac{\|F(i)-F(j)\|_2^2}{\delta_F}}\times
\begin{cases}
1 &\quad  \mbox{if }\|X(i)-X(j)\|_\infty<r,\\
0 &\quad  \mbox{otherwise},
\end{cases}
\end{equation}
where $\delta_F$ and $r$ are chosen parameters,
$F$ is the brightness value and $X$ is the location of a
pixel \cite{shmj:2000}.\footnote {In a 2-D image, 
pixel $i$ may naturally be represented by $(i_x,i_y)$ where $i_x$ and $i_y$ are two integers.}
In our experiment, we take
$$
\delta_F=\delta \max_{i,j}\|F(i)-F(j)\|_2^2
$$
for some parameter $\delta$ to be specified.

The definition of weight in \eqref{eq:weight} ensures that every pixel is connected with an edge to at most $(2r+1)^2$
other pixels.
As shown
in Table \ref{Table:parameter}, in our experiments, $r$ is taken either $5$ or $10$, and thus
the weight matrix $W$ is sparse, which in turn makes the matrix $A$
in CRQopt \eqref{eq:CRQopt} sparse, too. Note that for the example Crab, the contrast between the upper right of the object and the background is not significant. Therefore, we choose $r$ to be twice as much as other examples to ensure the weight matrix correctly reflect the connectivity of the graph. In addition, in our experiments $\delta$ is around $0.1$, to be consistent with the statement in \cite{shmj:2000} that
``$\delta_F$ is typically set to 10 to 20 percent of the total
range of the feature distance function''. Besides, size  $m$
of linear constraints is relatively
small compared with the number of pixels $n$, yielding CRQopt \eqref{eq:CRQopt} with $m\ll n$.

\begin{table}[H]
\caption{The number of pixels $n$,
parameters $\delta$ and $r$ and size $m$ of linear constraints.
      }
\label{Table:parameter}
\centerline{
\begin{tabular}[b]{|@{\hspace{5pt}}c@{\hspace{5pt}}|c|c|c|c|c|} \hline
Image  & Number of pixels $n$  & $\delta$  &  $r$ &
 $m$ \\\hline
Flower  &  30,000  &  0.1 &  5&  24\\\hline
Road  &  50,268  &  0.1 &  5&  46\\\hline
%Shell  &  59,700  &  0.1 &  5&  26\\\hline
% Child  &  120,000 &  44 &  0.1 &  5\\\hline
Crab  &  143,000  &  0.1 &  10&  32\\\hline
Camel  &  240,057  &  0.08 &  5&  24\\\hline
Dog  &  395,520  &  0.1 &  5&  33\\\hline
Face1  &  562,500  &  0.1 &  5&  31\\\hline
Face2  &  922,560  &  0.1 &  5&  19\\\hline
Daisy  &  1,024,000  &  0.08 &  5&  29\\\hline
Daisy2  &  1,024,000 &  0.08 &  5&  59 \\\hline
\end{tabular}
}
\end{table}

All experiments were conducted on a PC with Intel Core i7-4770K CPU@3.5GHz
and 16-GB RAM. CRQopt \eqref{eq:CRQopt} is 
solved via solving QEPmin \eqref{eq:QEPmin}. In our tests, we choose the maximum Lanczos steps
${\tt maxit}=300$ and use $\delta_k^{\QEPmin} < 8\times 10^{-5}$ as the stopping criterion. Besides, we choose the minimum Lanczos steps ${\tt minit}=120$ and check the stopping conditions every $5$ Lanczos steps to reduce the cost of checking the stopping conditions.

%\subsection{Numerical results}\label{sec-app-result}
\paragraph{Quality of the model.}
We apply the model \eqref{eq:ncutcon2} and Lanczos algorithm
for CRQopt \eqref{eq:CRQopt} on different kinds of images and show the
results for segmentation and the computed eigenvector
in Figure \ref{fig:imagecut}.
We can see that the image cut results of the  model \eqref{eq:ncutcon2}
indeed agree with our natural visual separation of the object and the background.
Daisy and Daisy2 are the same image but with two different ways of prior partial labeling.
For both ways of prior partial labelling, the computed image cuts look equally well.
%For two different set of labels, the image cut results are both consistent with our intuition such that the objects are separated from the background.
%\paragraph{Timing.}
%We show
Table \ref{Table:runtimeimages2} displays the wall-clock runtime and the numbers of Lanczos steps used for the images.

\begin{figure}
\begin{center}
\includegraphics[width=1.2in]{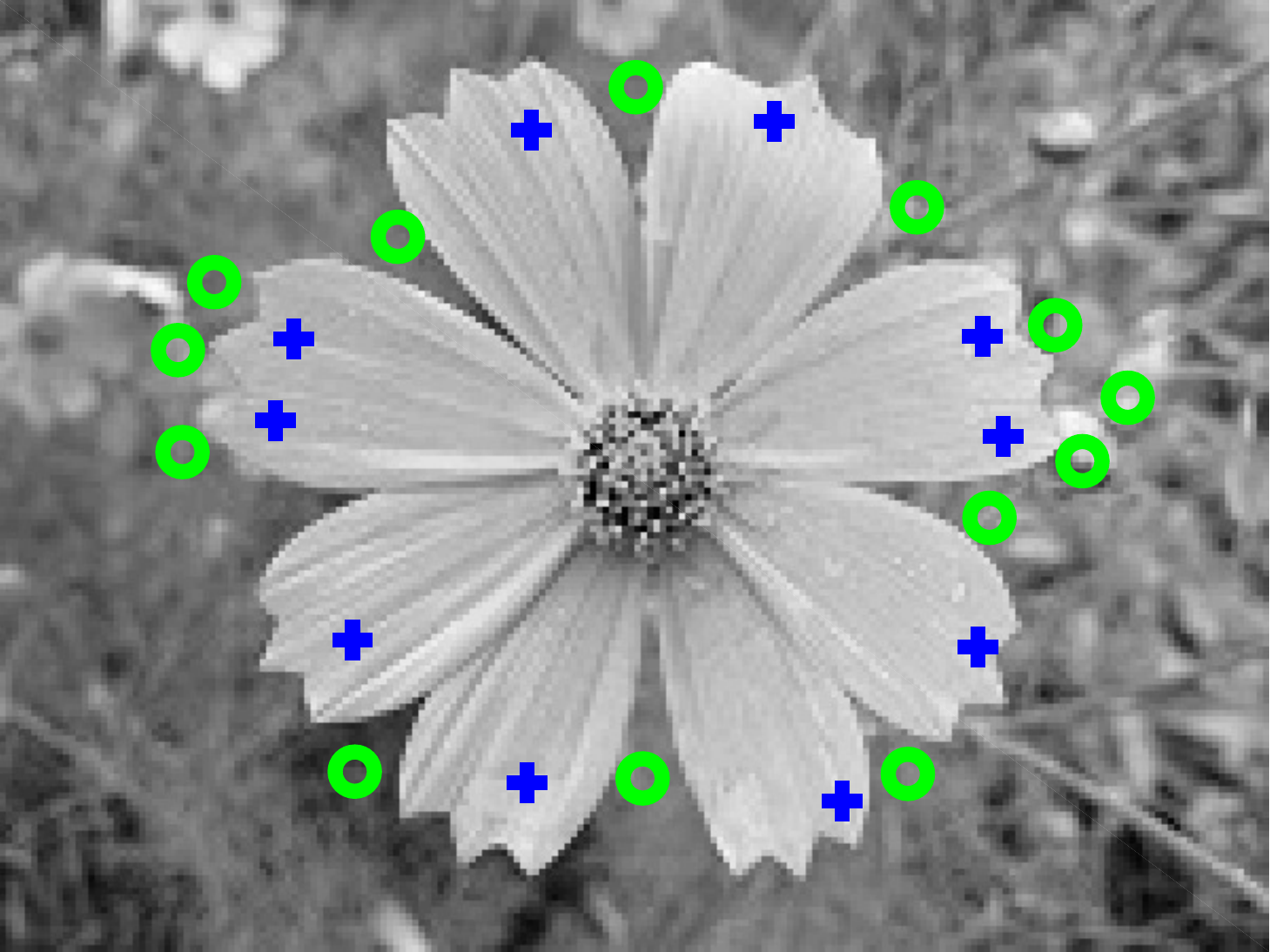}
\includegraphics[width=1.2in]{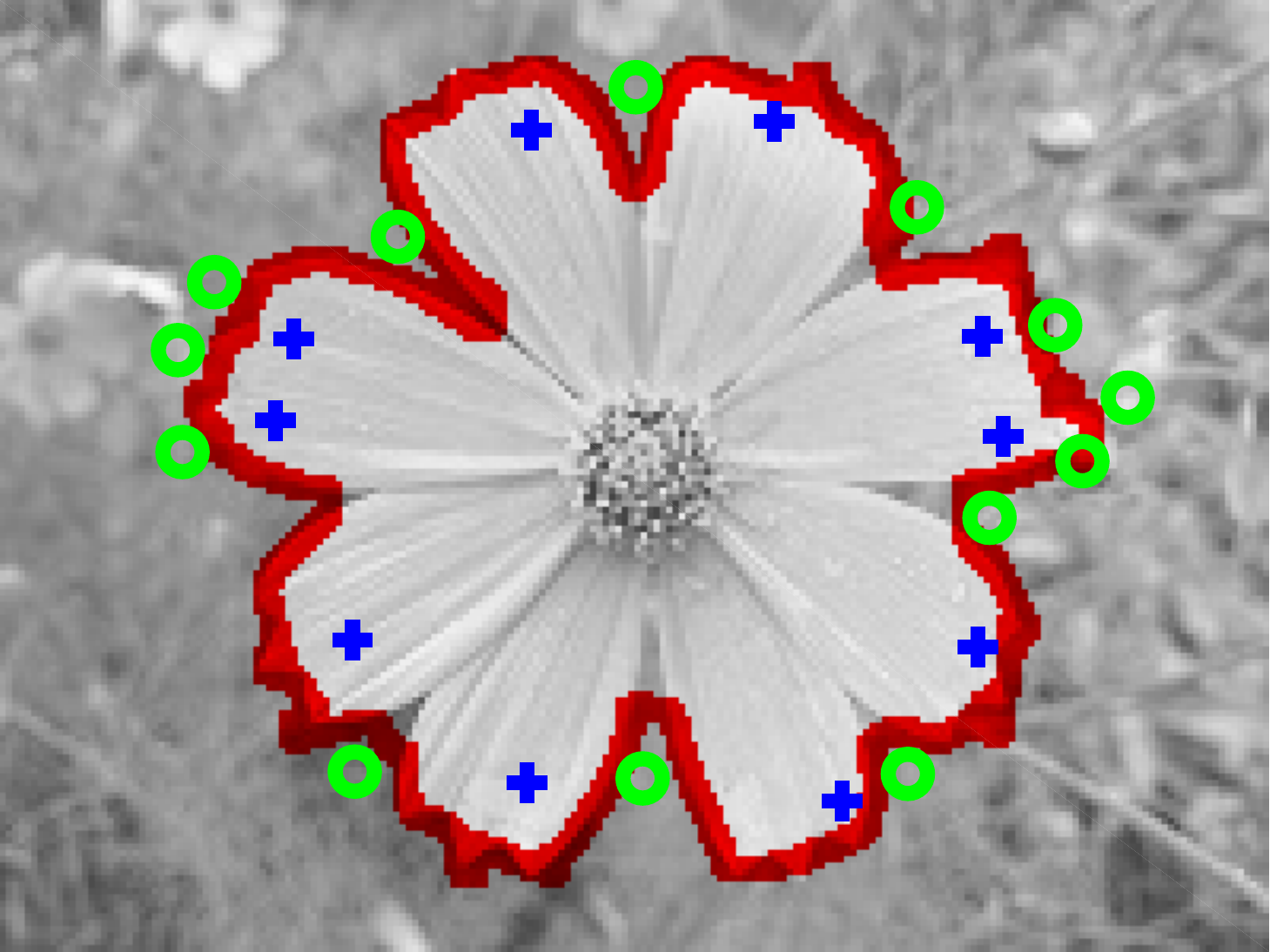}
\includegraphics[width=1.2in]{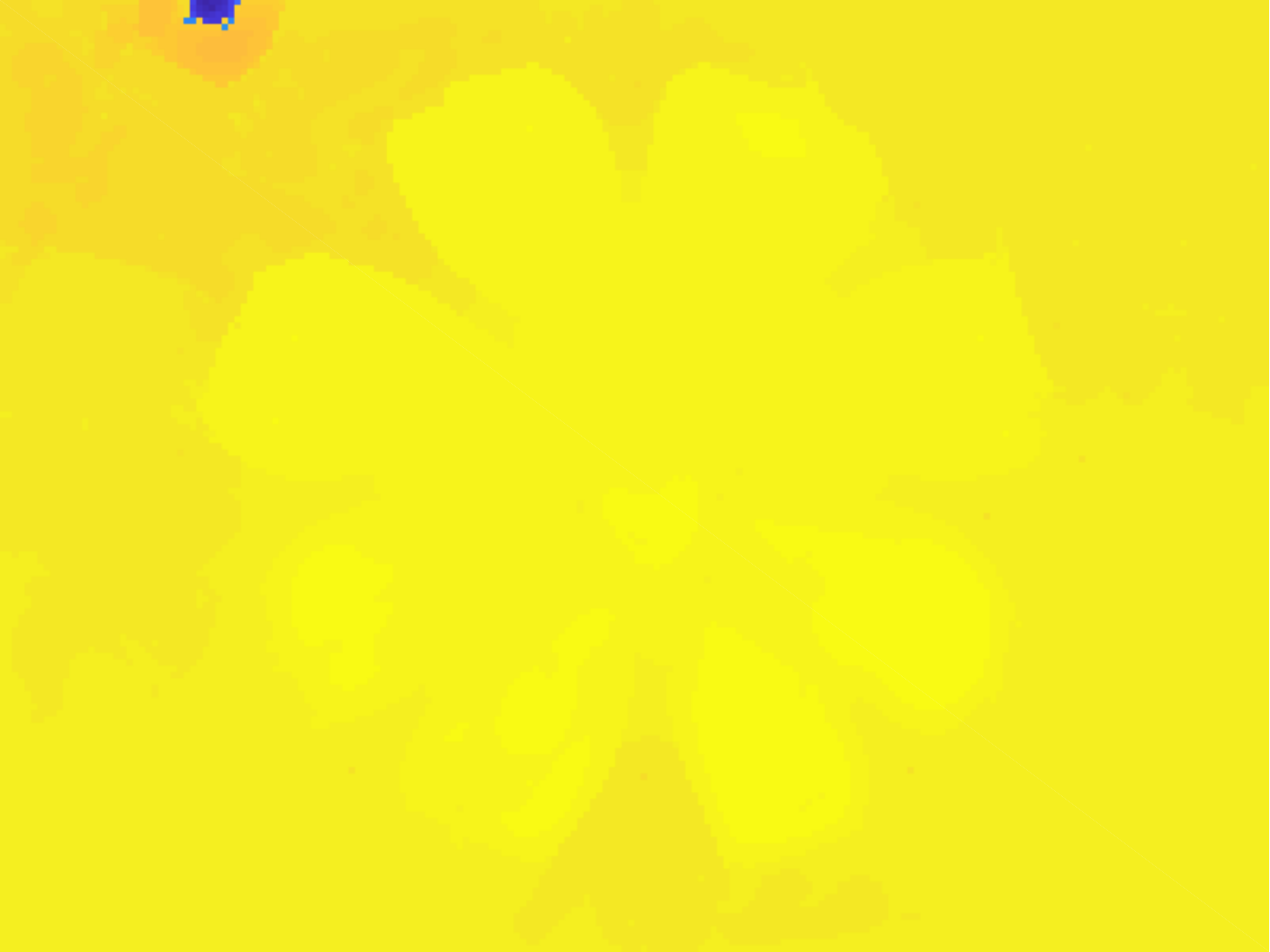}\\
\includegraphics[width=1.2in]{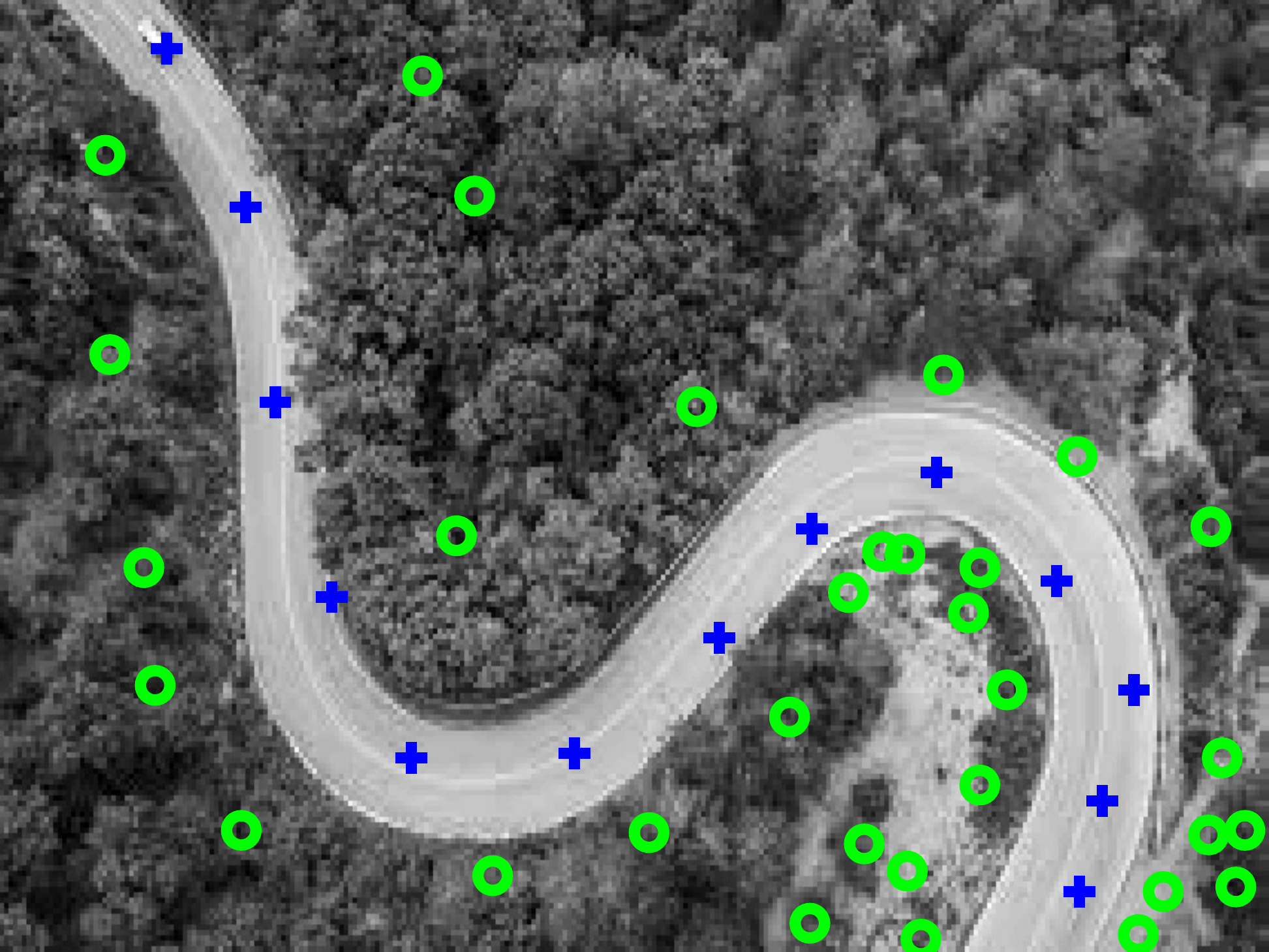}
\includegraphics[width=1.2in]{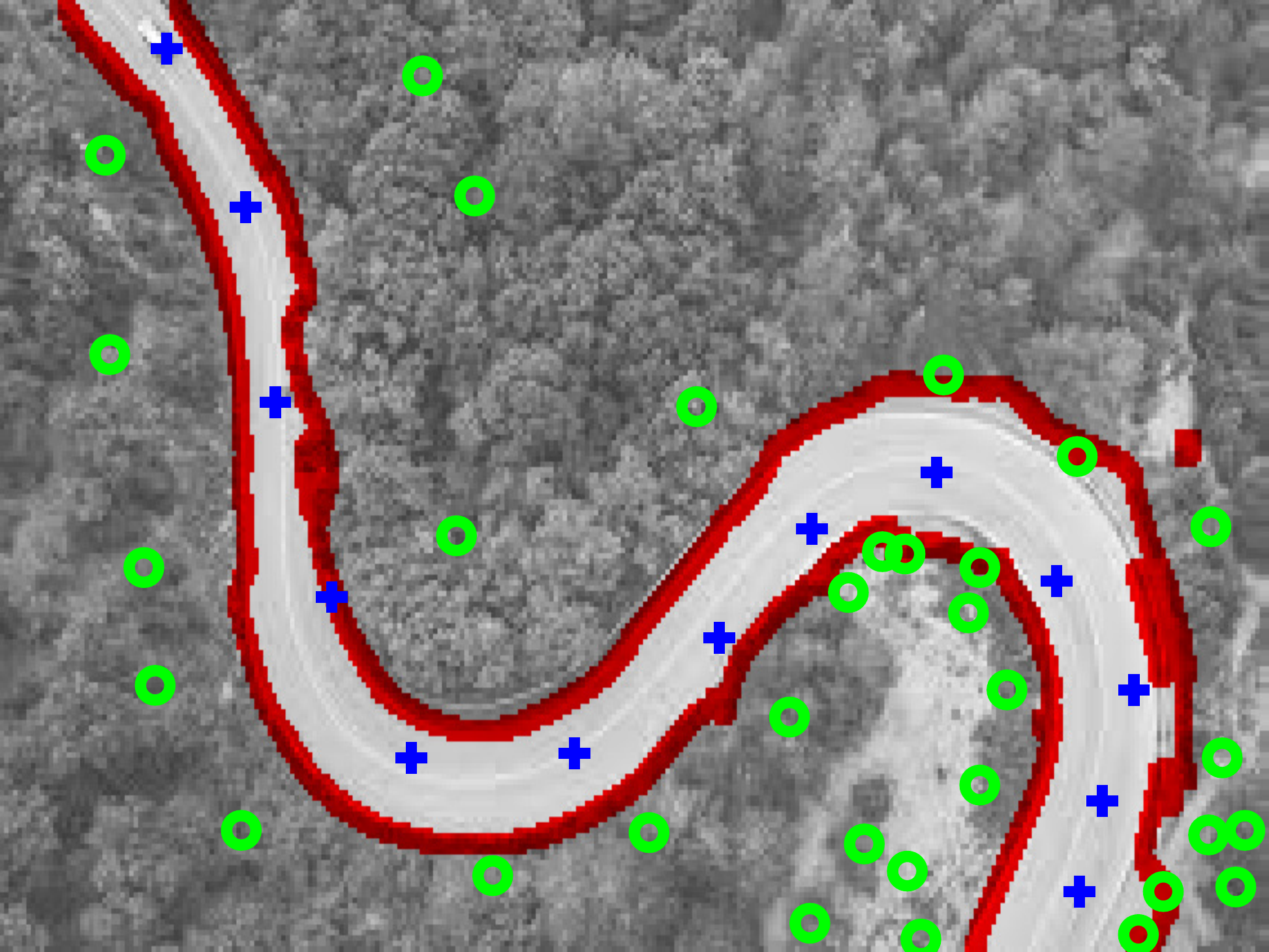}
\includegraphics[width=1.2in]{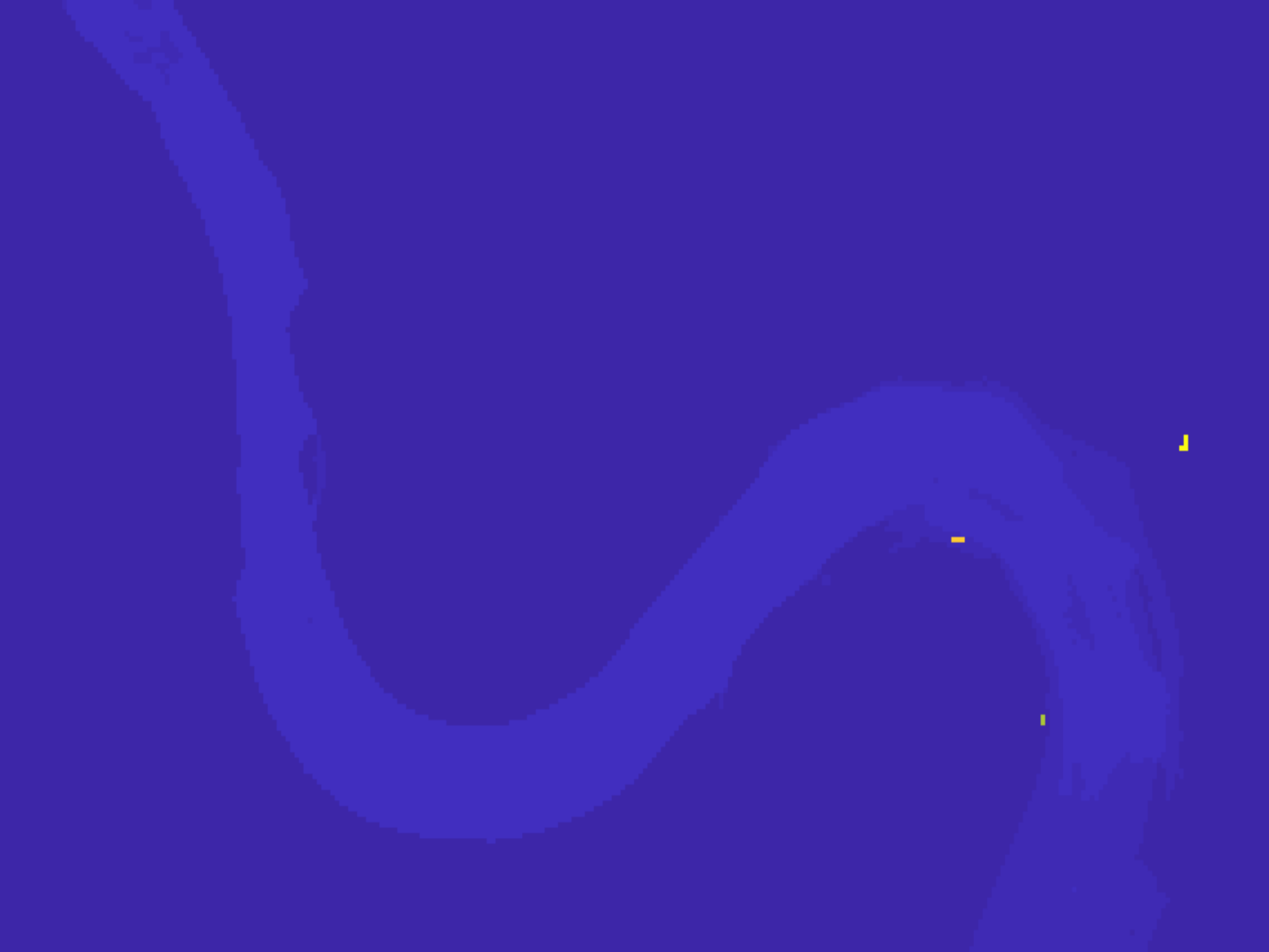}\\
\includegraphics[width=1.2in]{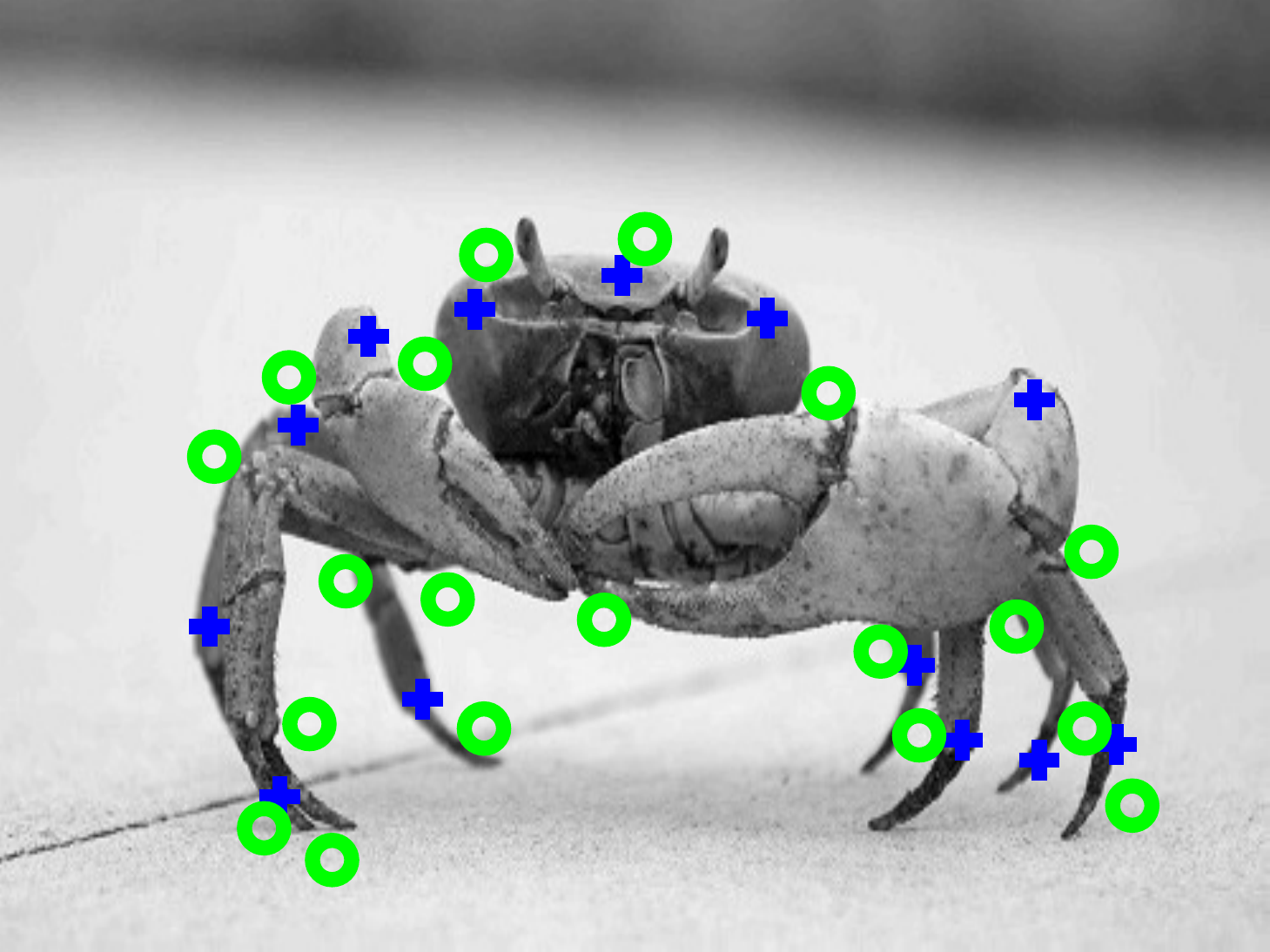}
\includegraphics[width=1.2in]{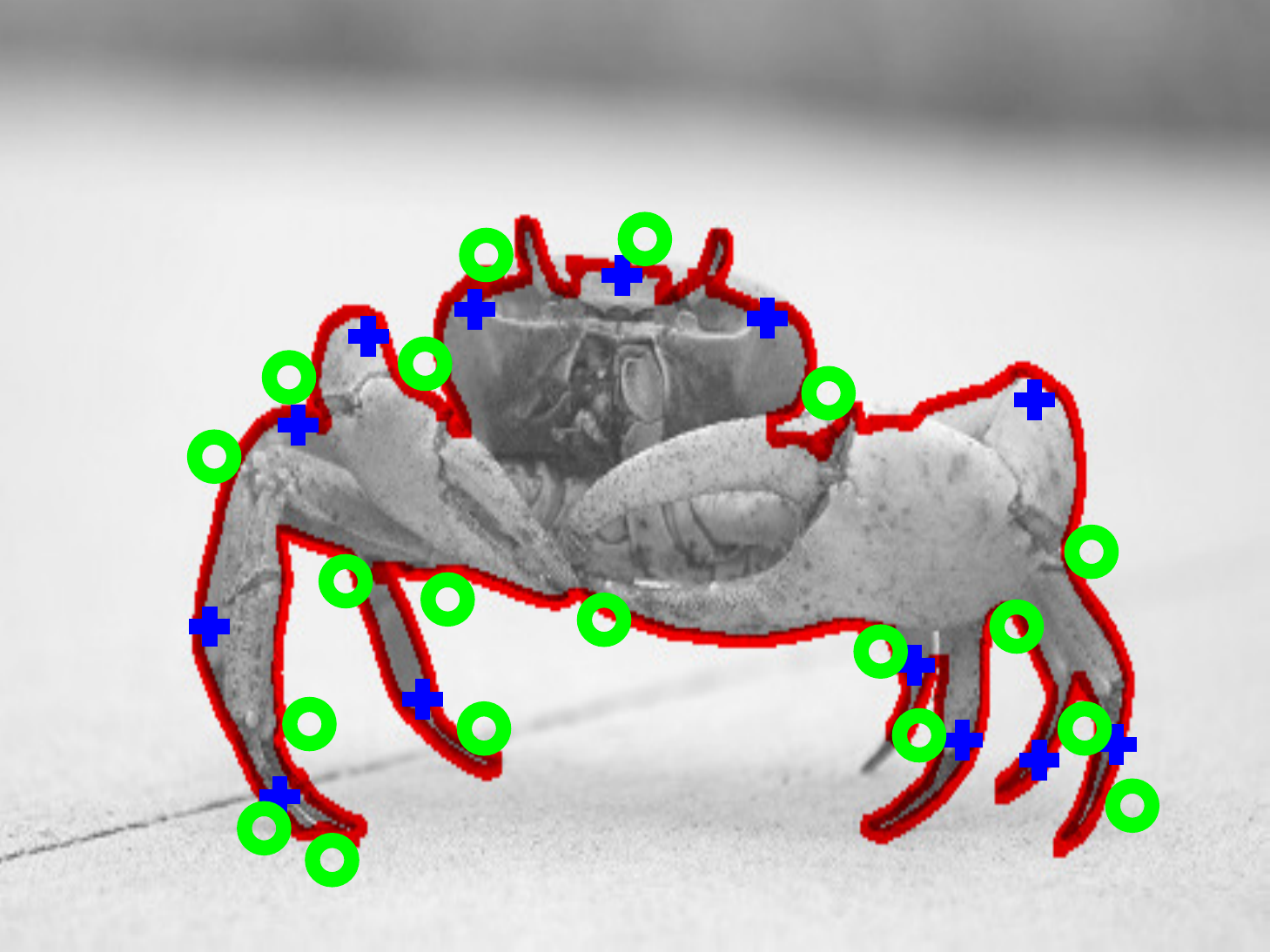}
\includegraphics[width=1.2in]{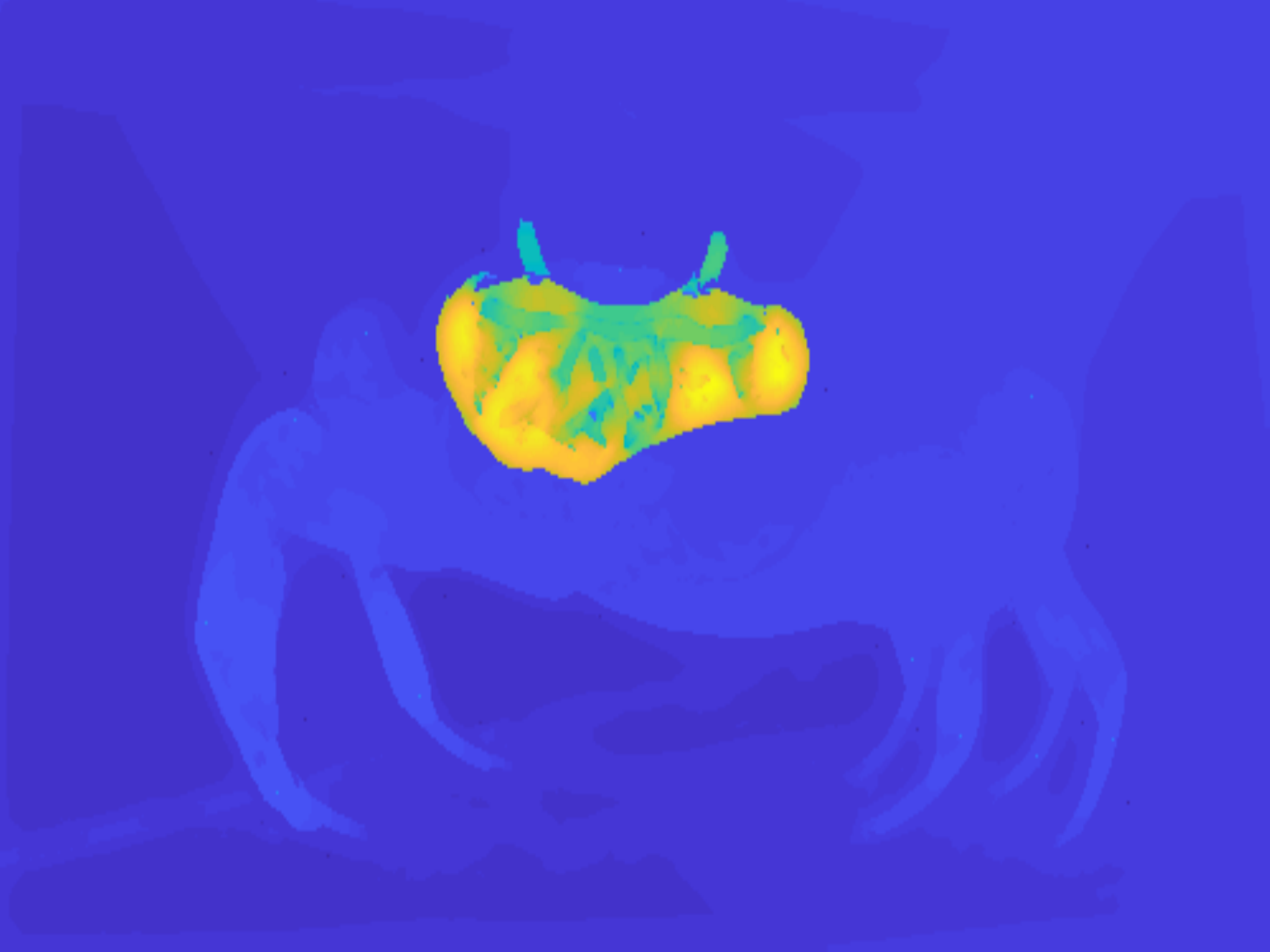}\\
\includegraphics[width=1.2in]{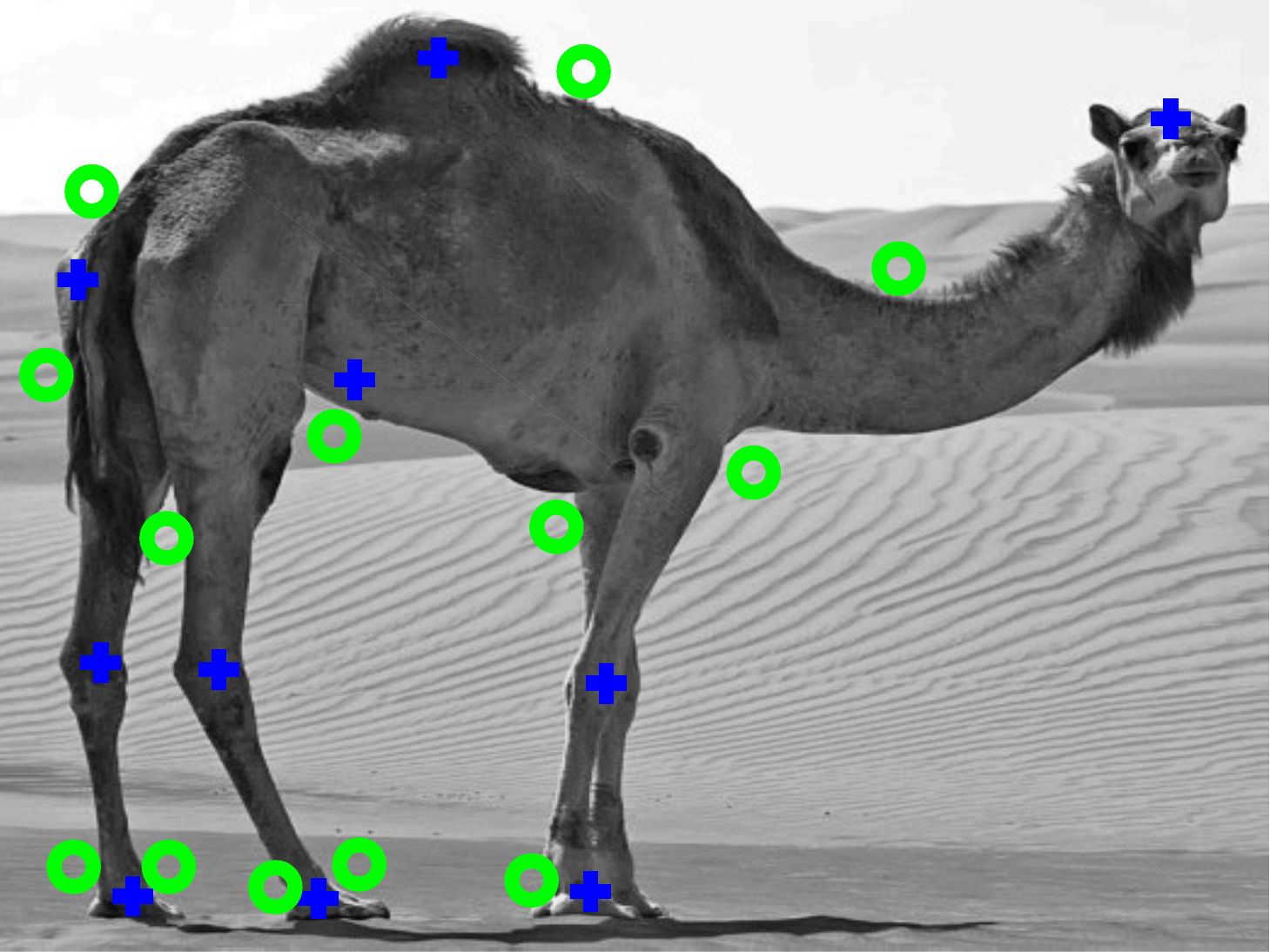}
\includegraphics[width=1.2in]{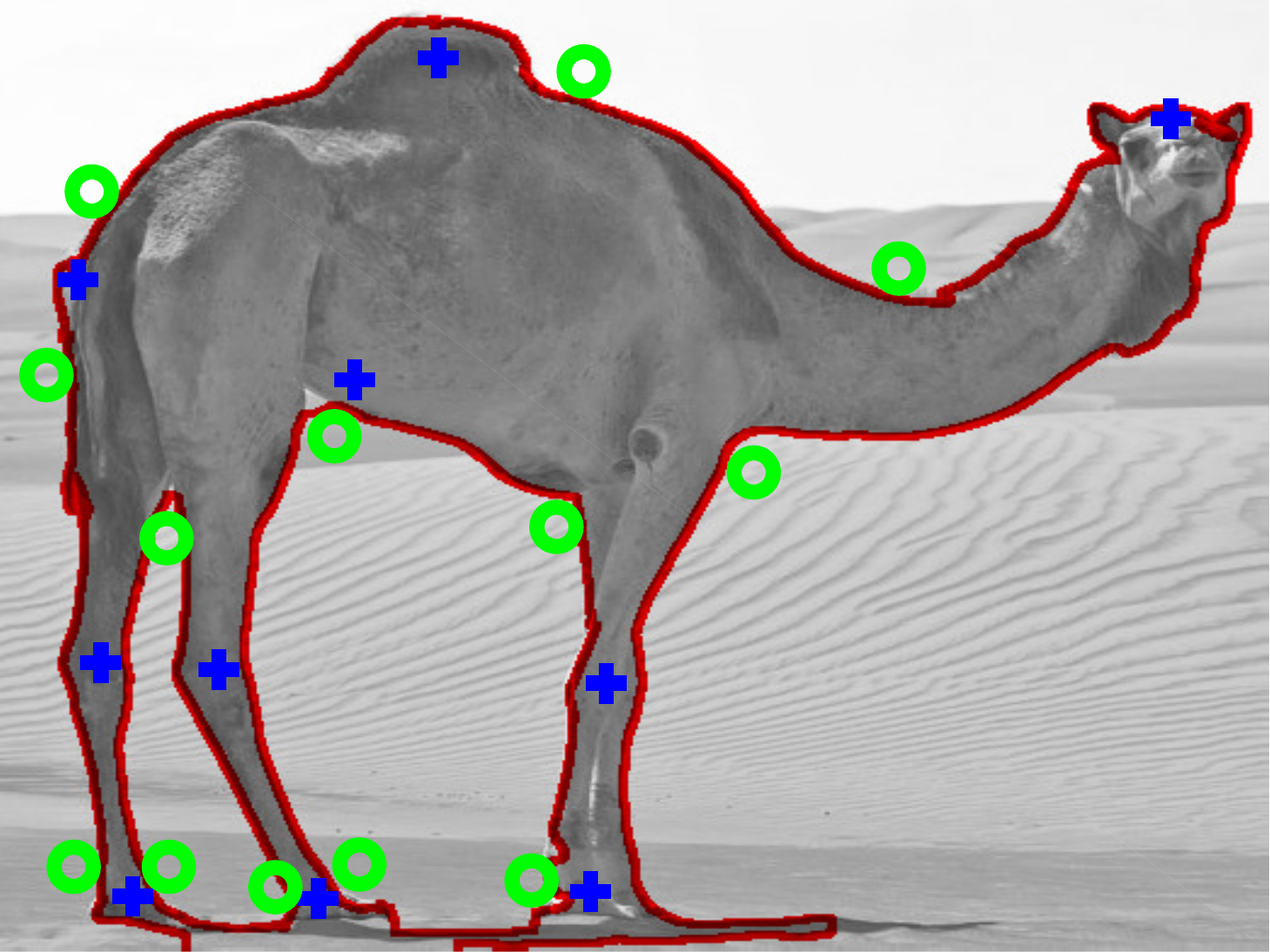}
\includegraphics[width=1.2in]{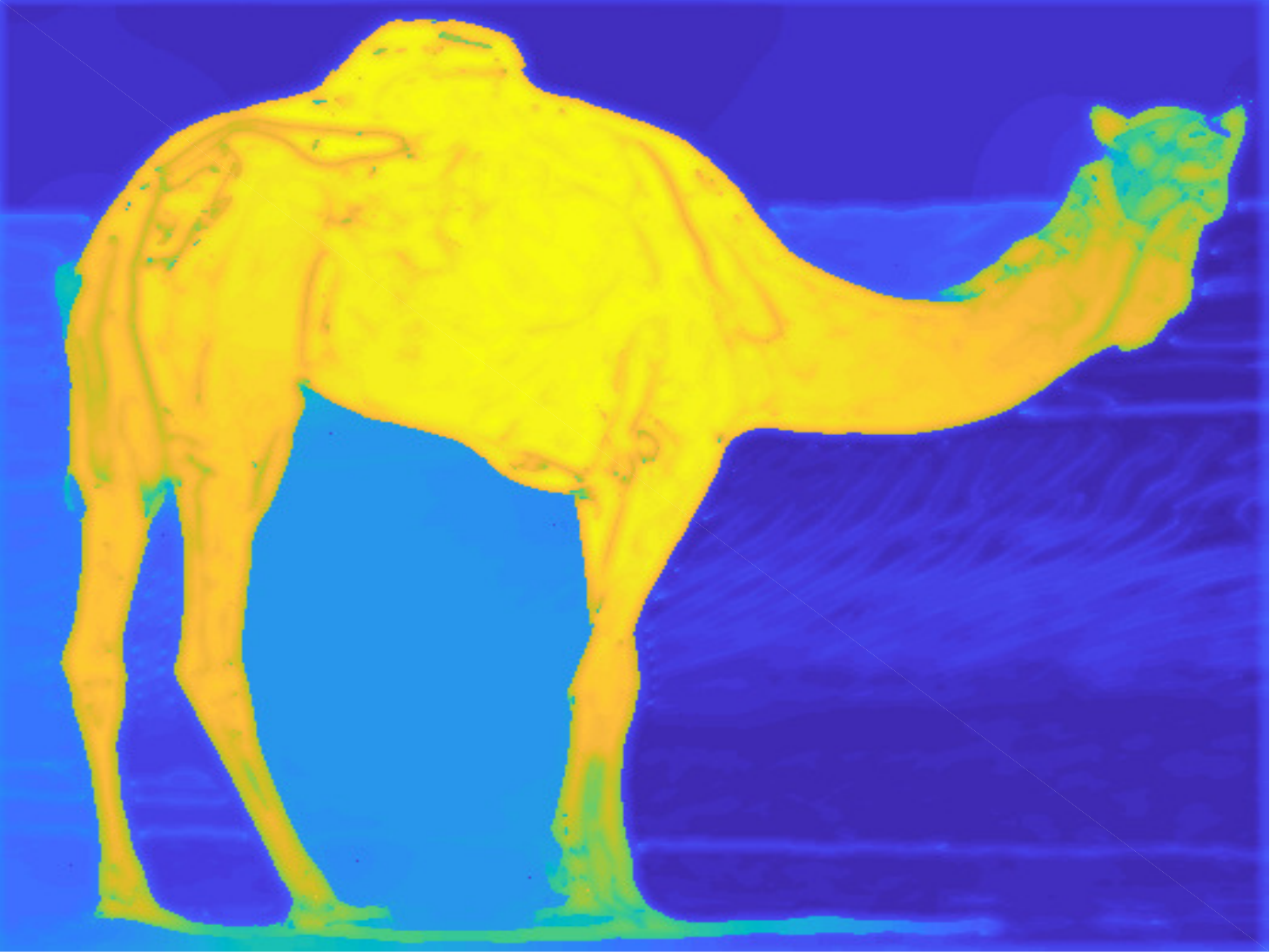}\\
\includegraphics[width=1.2in]{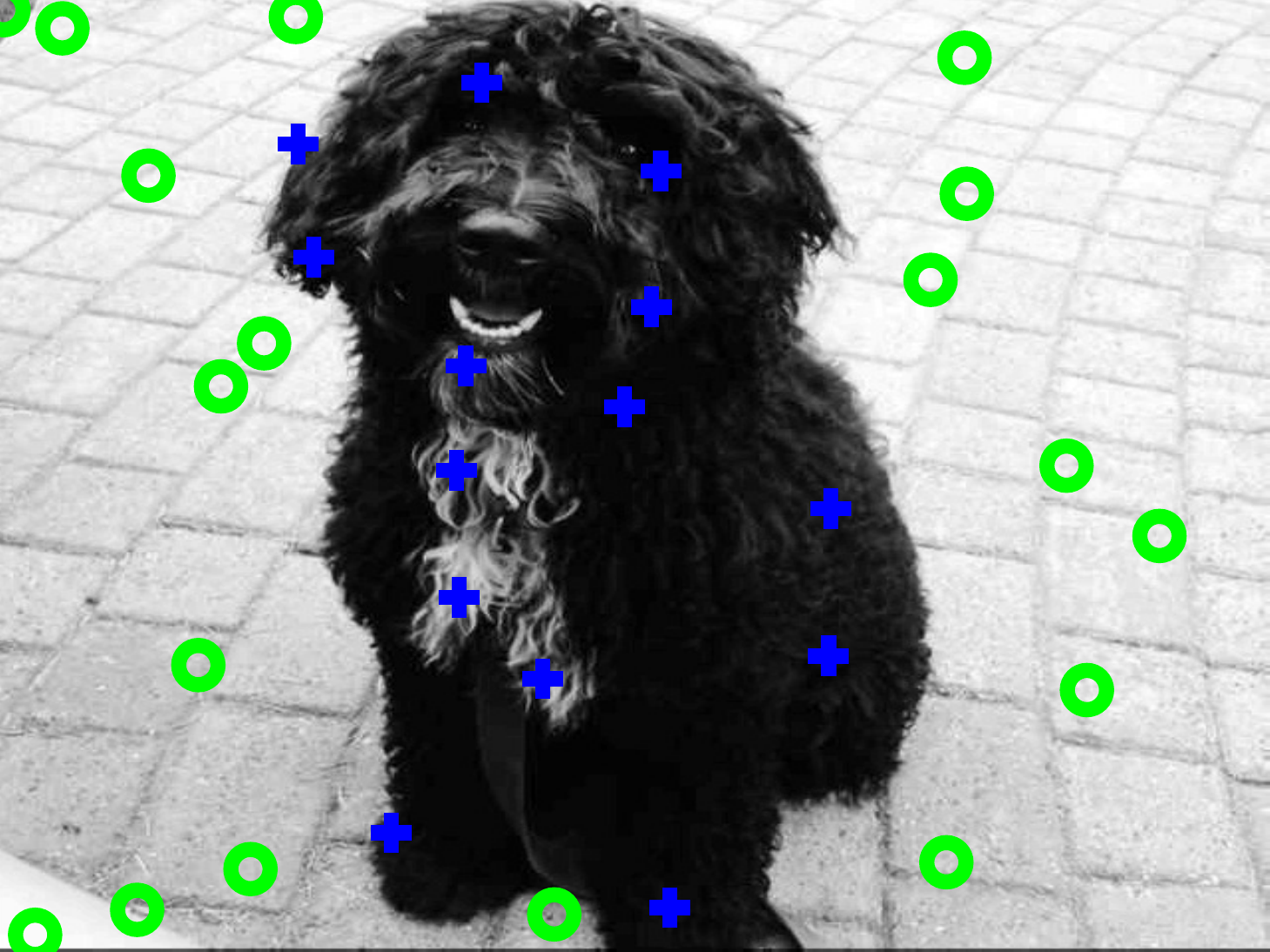}
\includegraphics[width=1.2in]{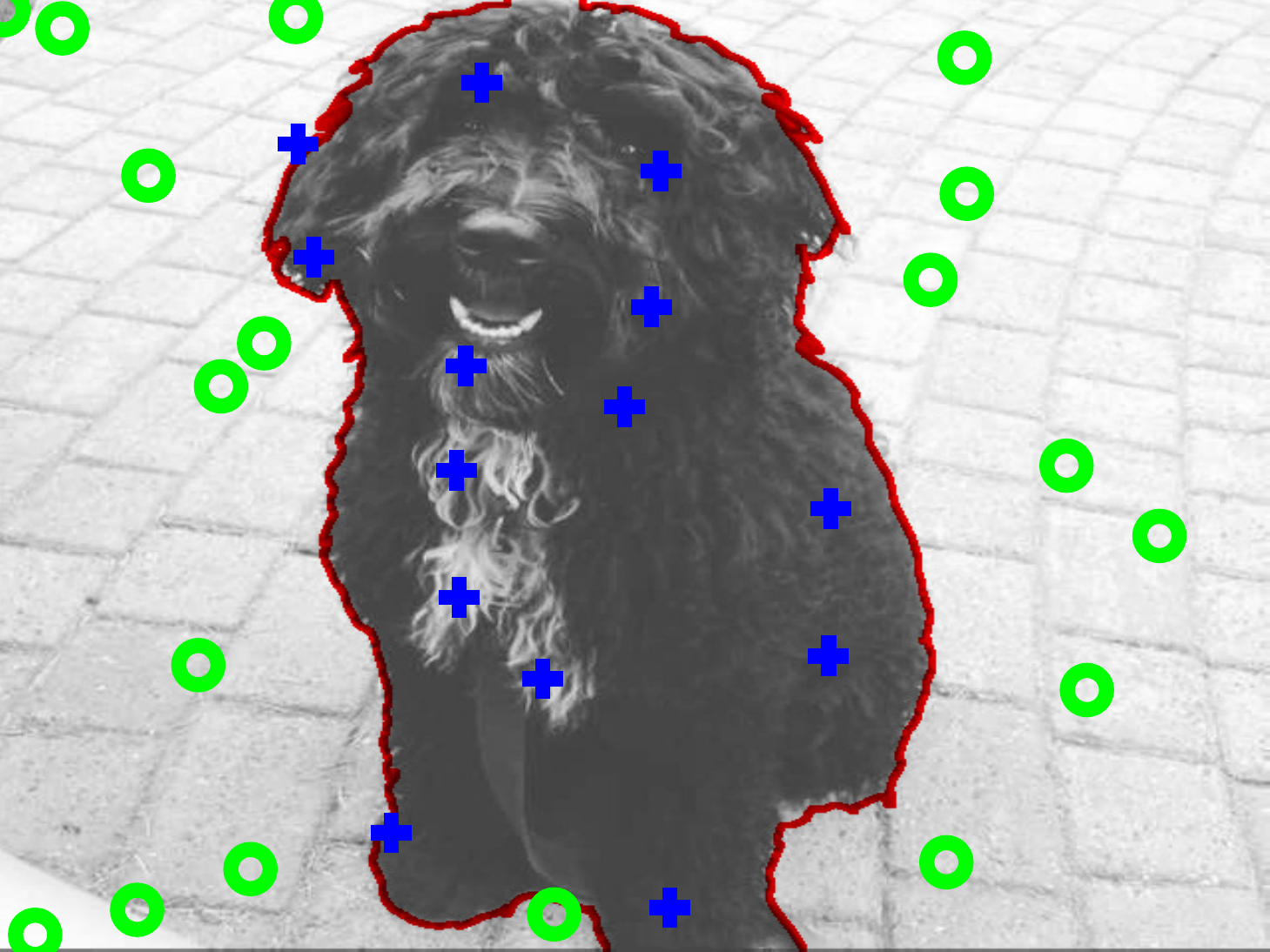}
\includegraphics[width=1.2in]{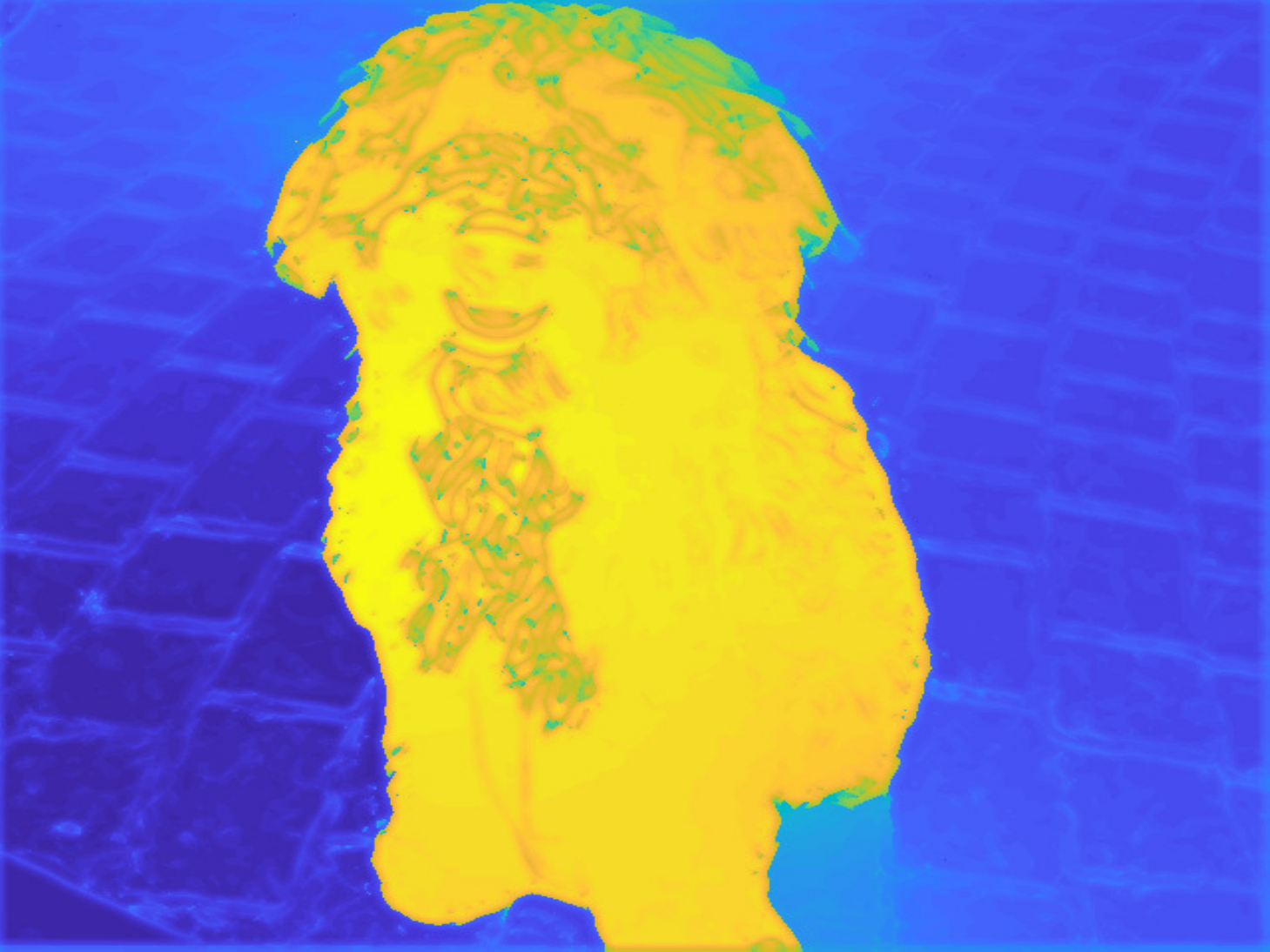}\\
\includegraphics[width=1.2in]{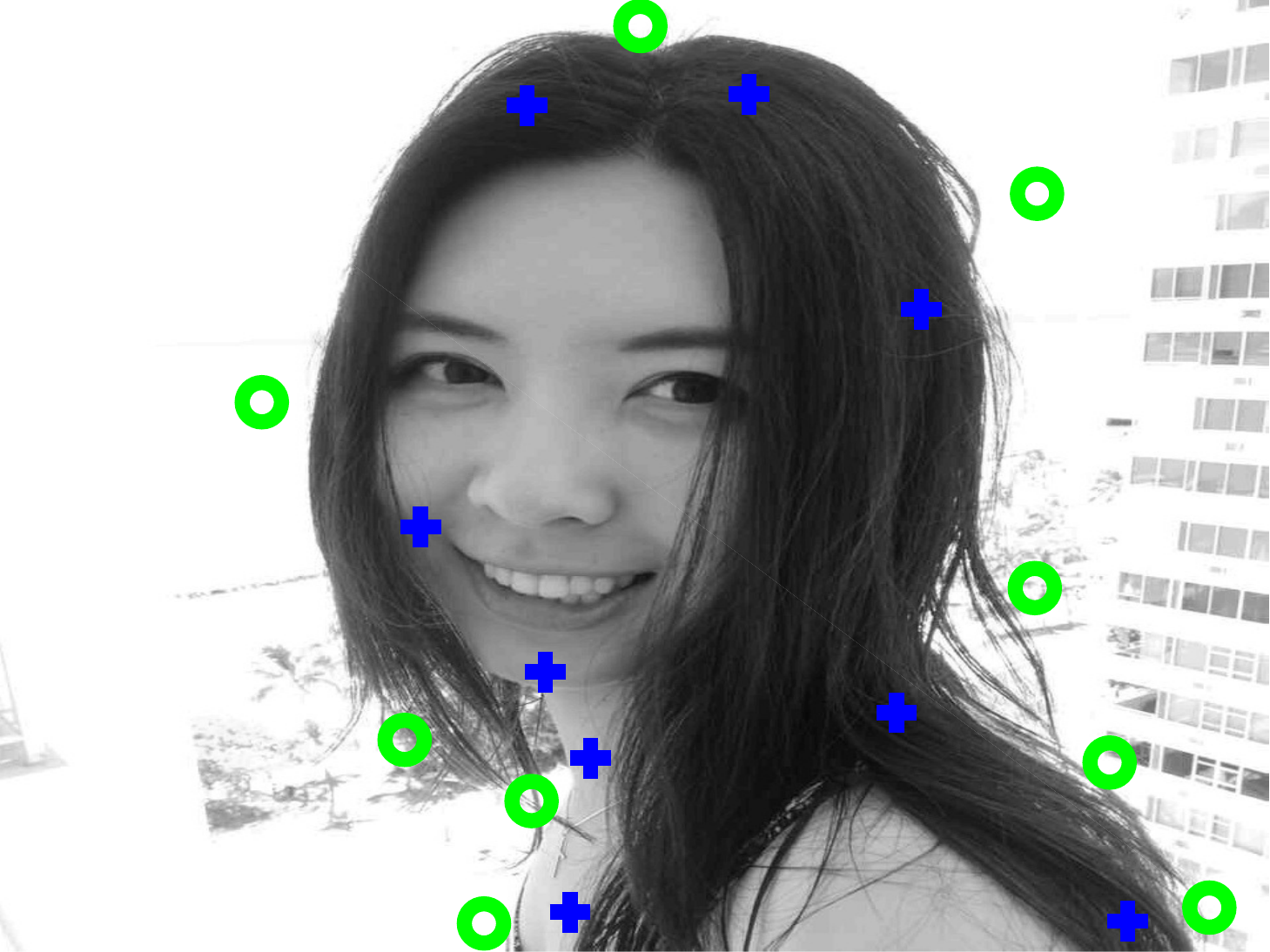}
\includegraphics[width=1.2in]{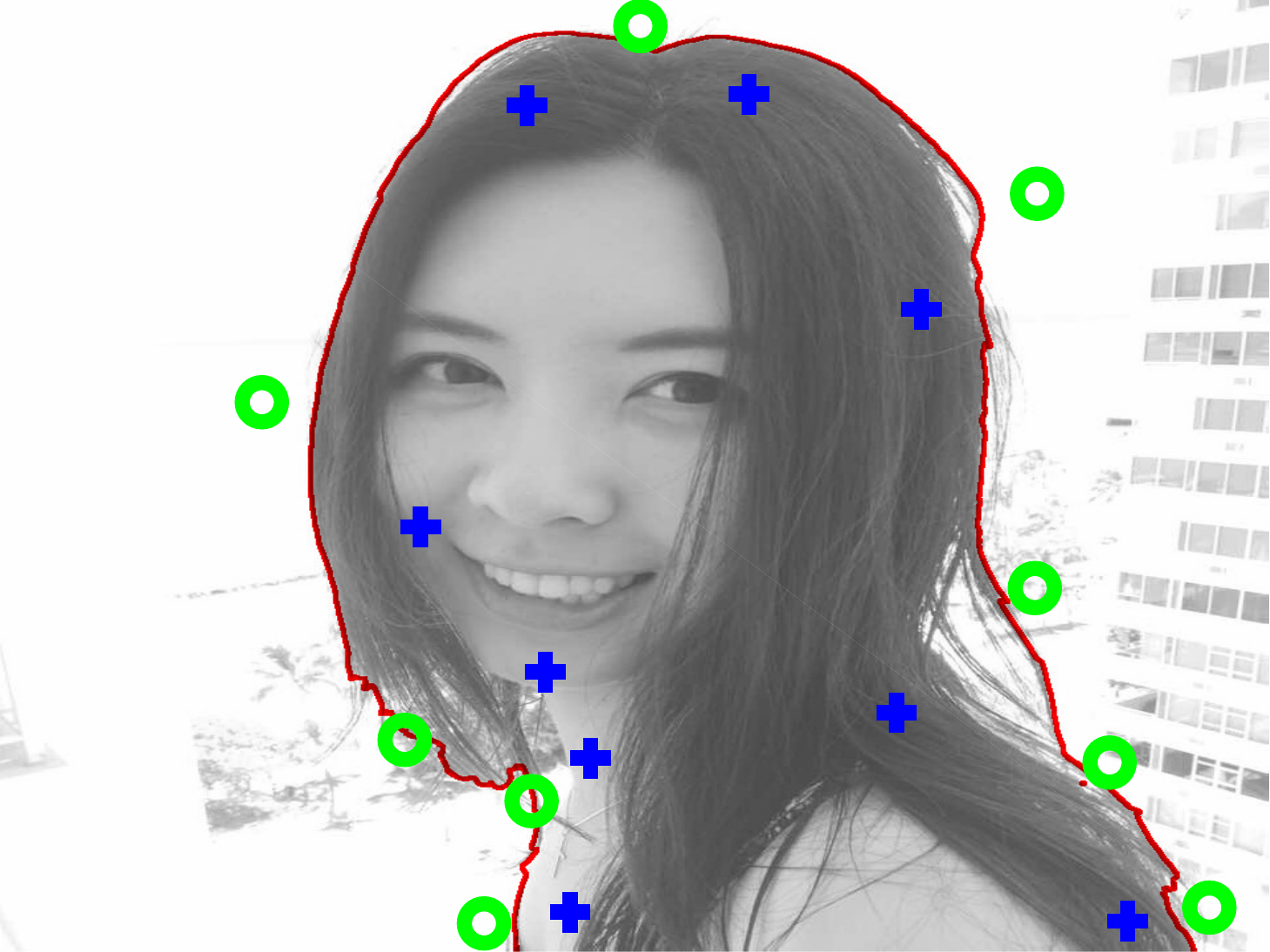}
\includegraphics[width=1.2in]{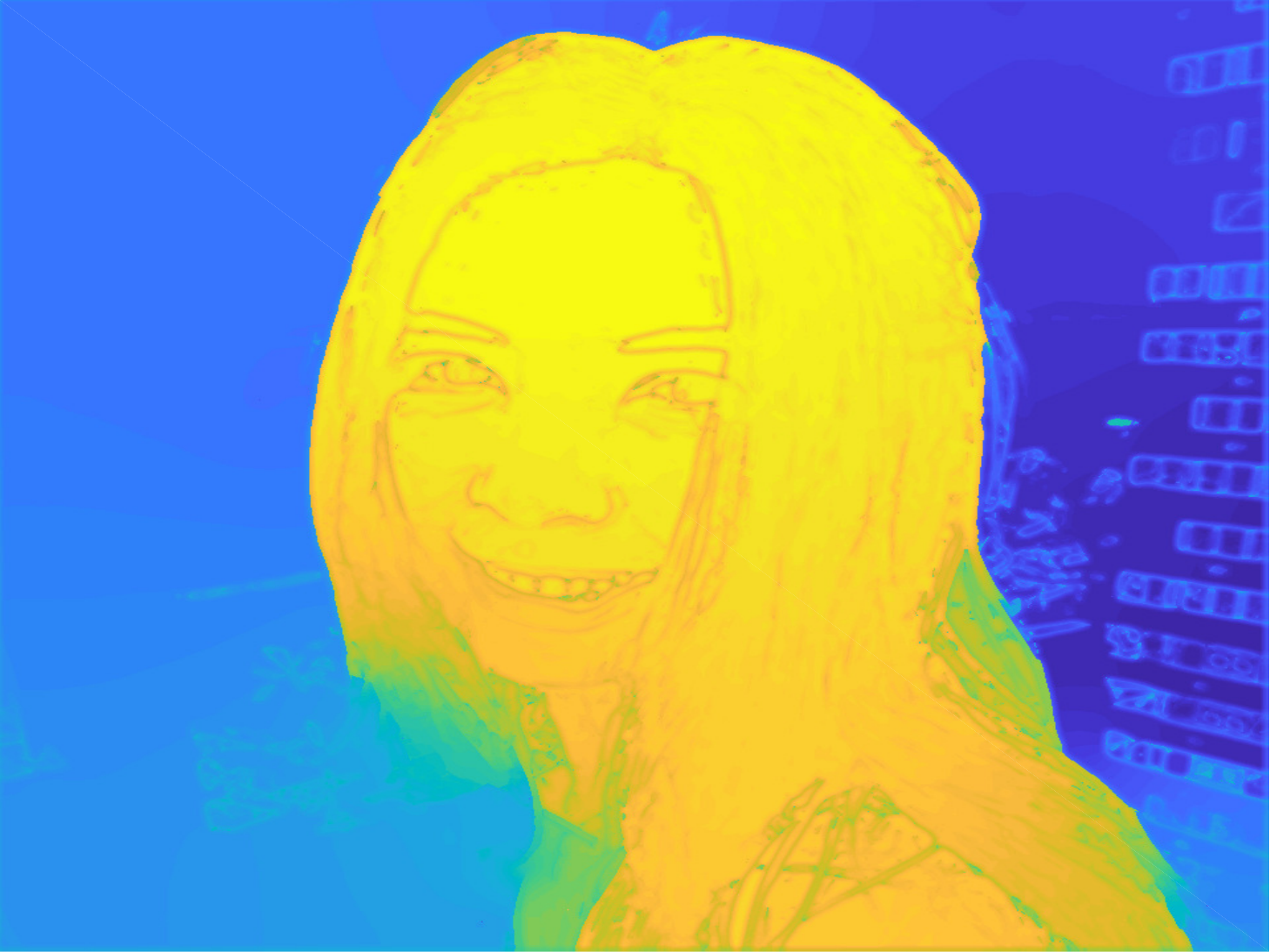}\\
\includegraphics[width=1.2in]{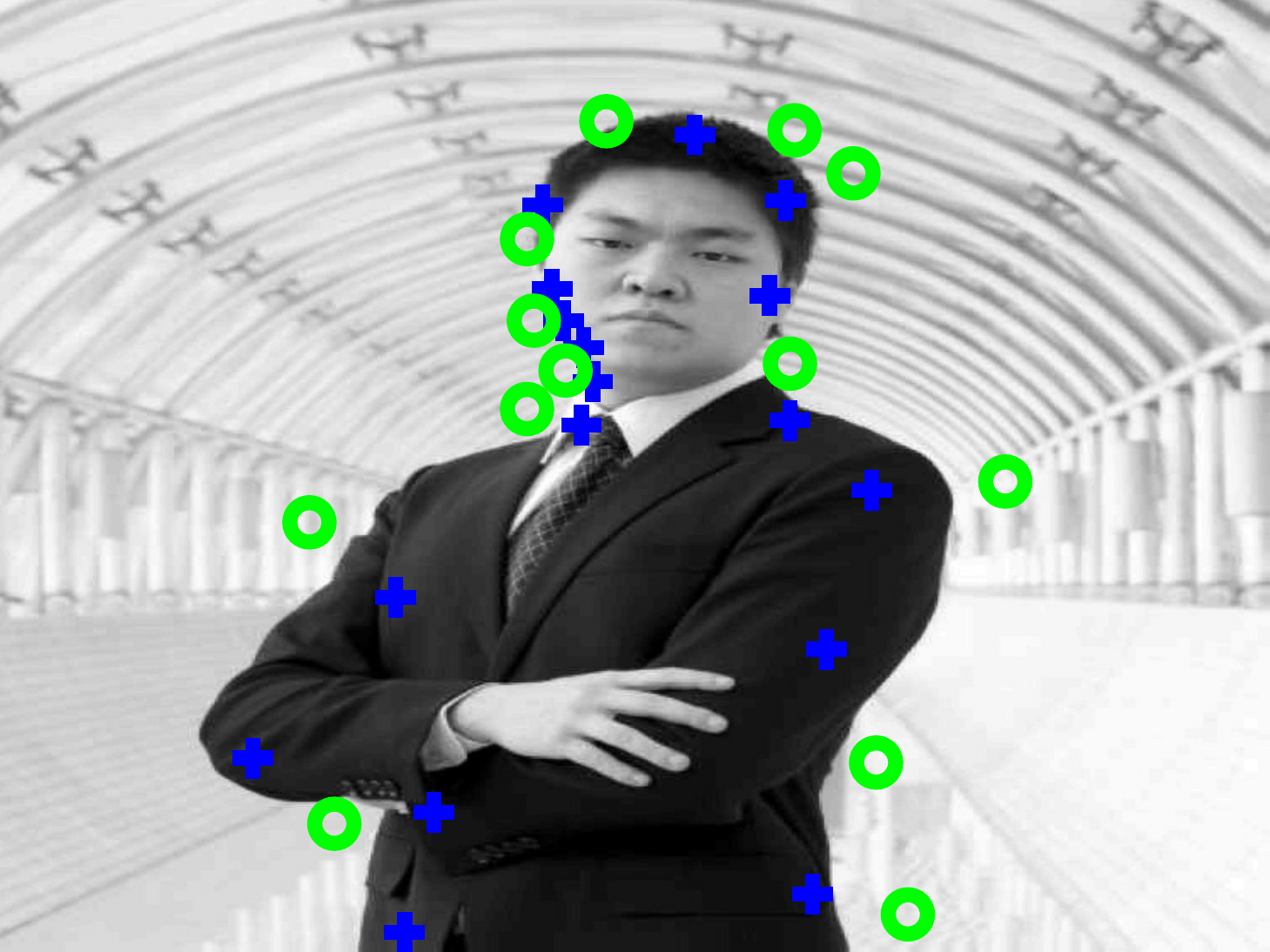}
\includegraphics[width=1.2in]{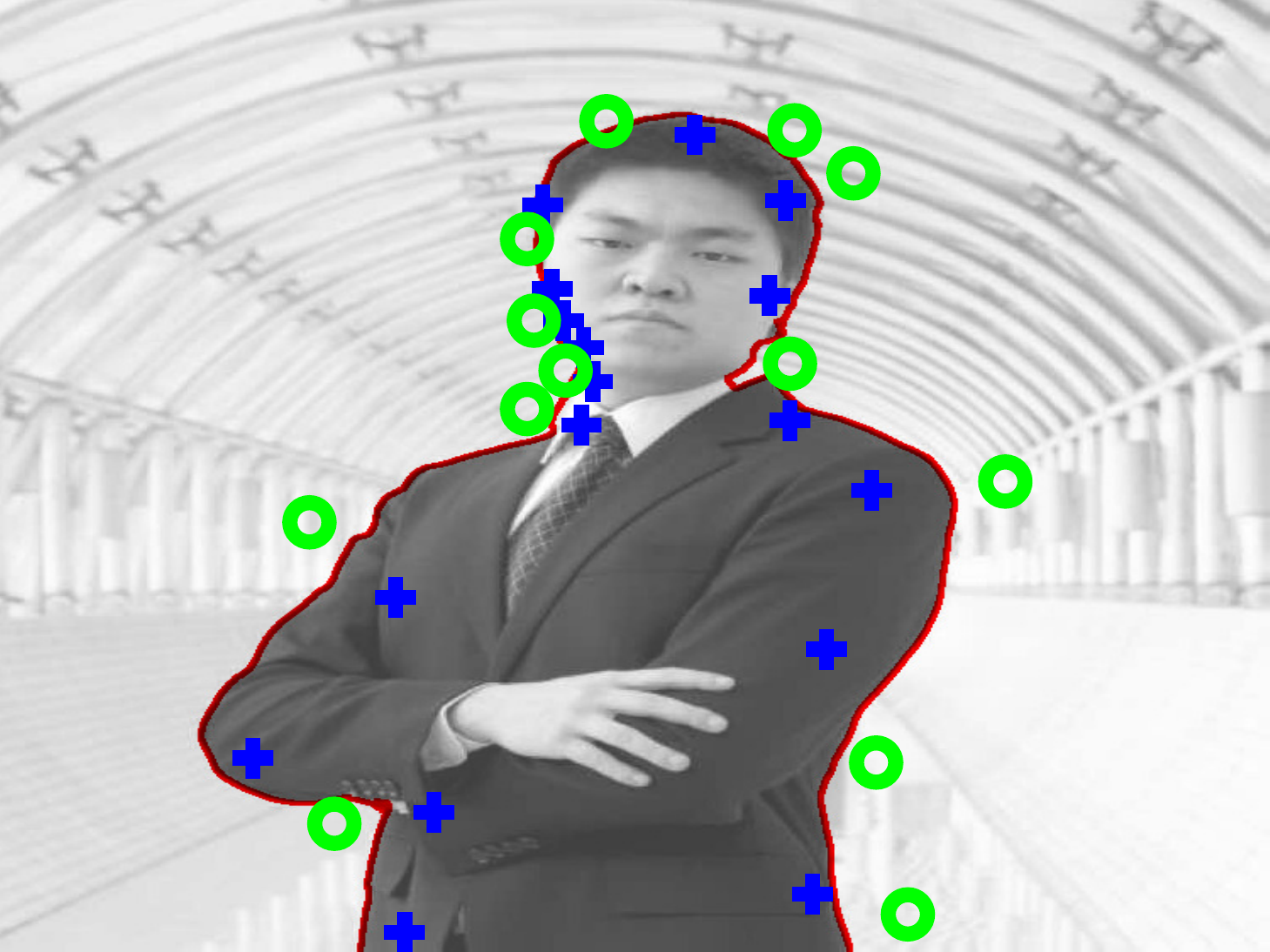}
\includegraphics[width=1.2in]{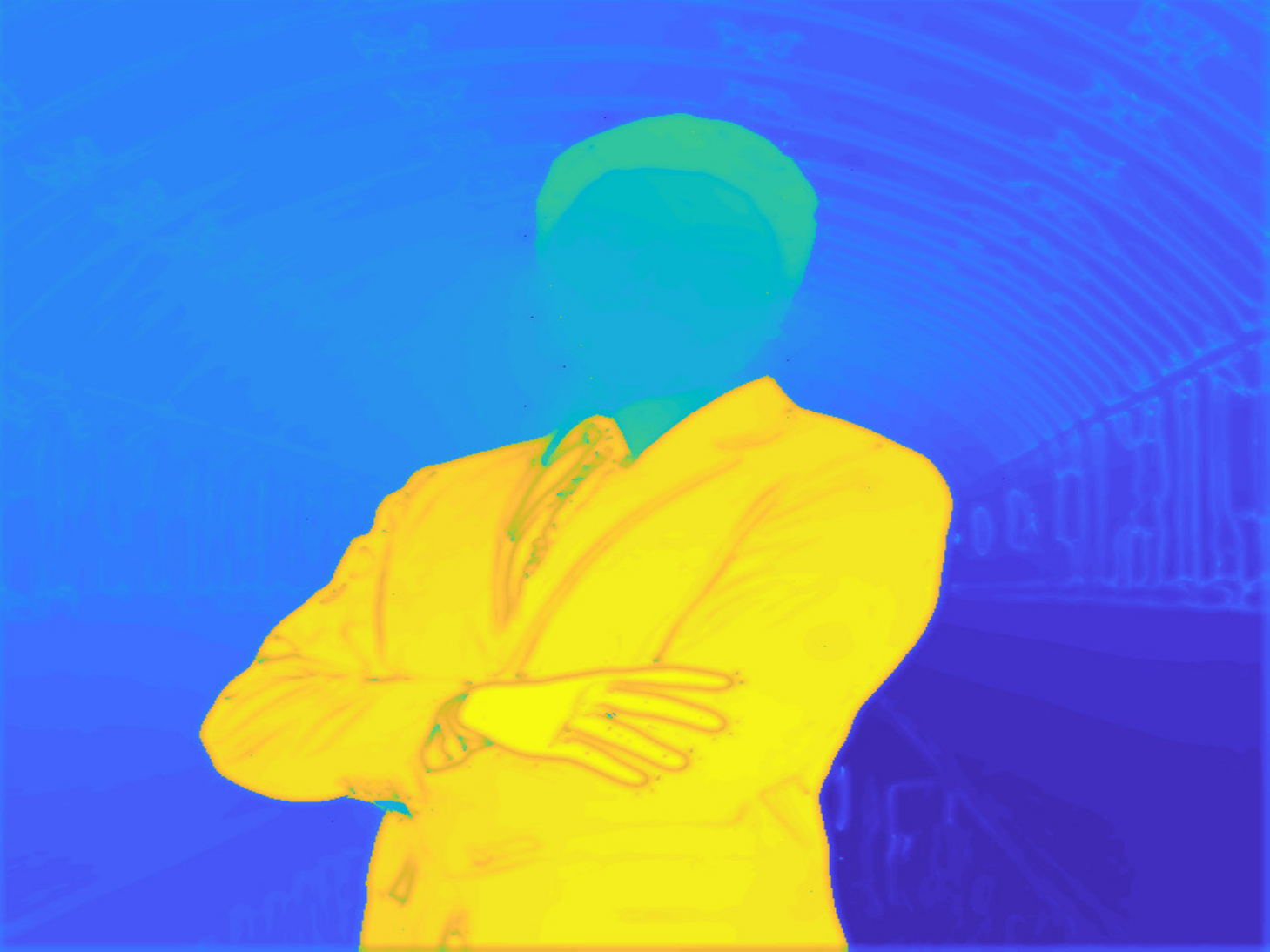}\\
\includegraphics[width=1.2in]{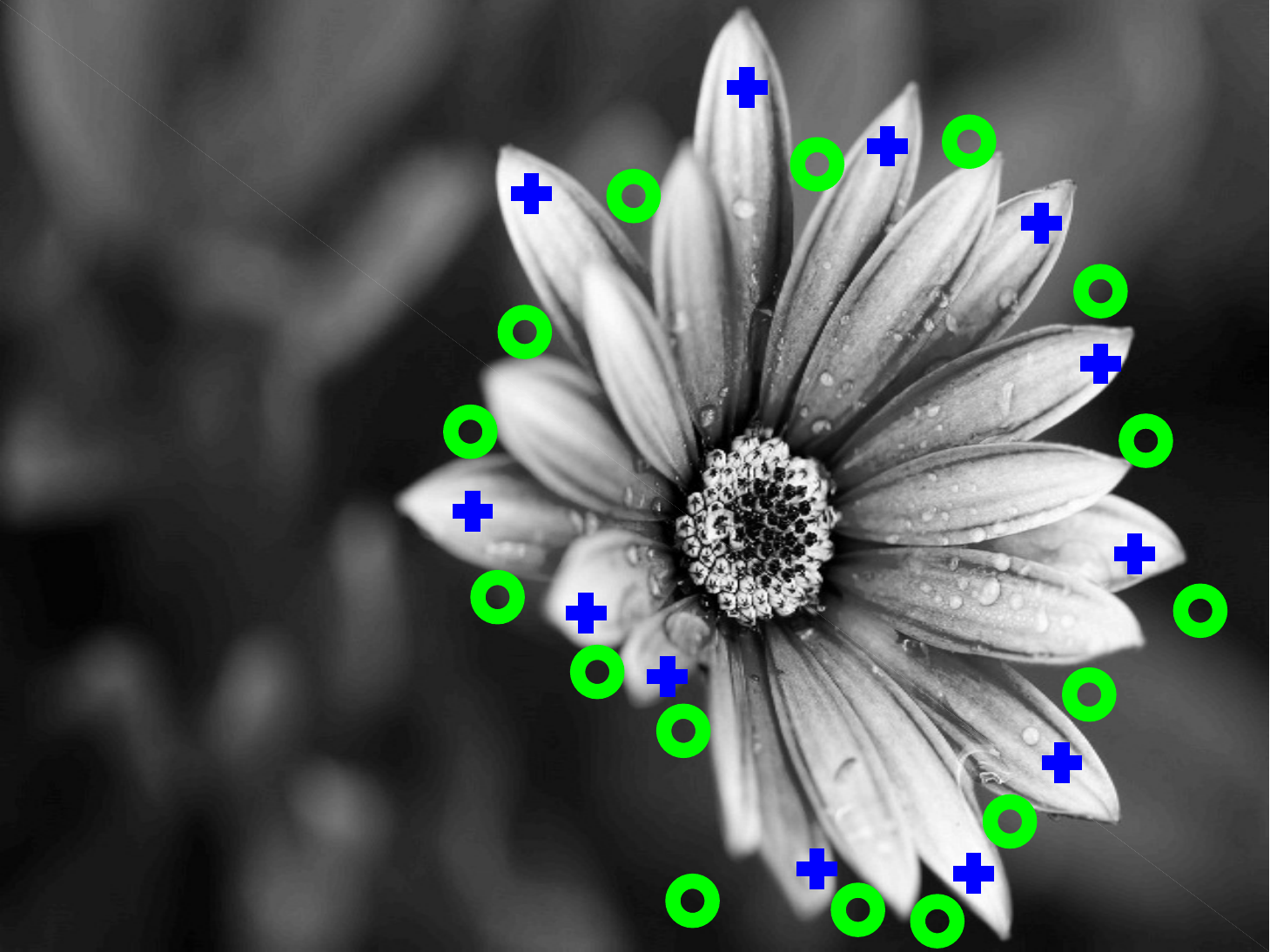}
\includegraphics[width=1.2in]{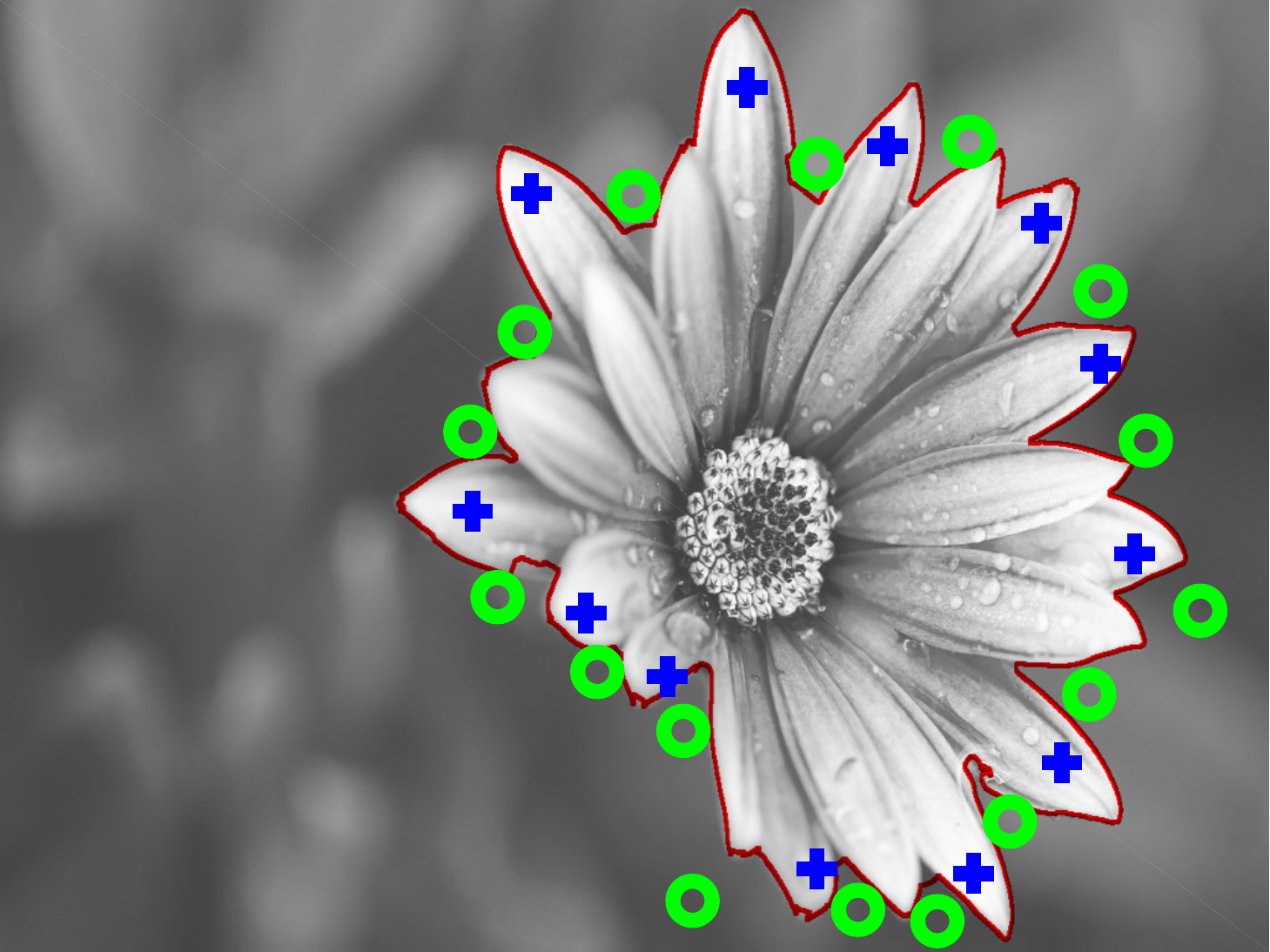}
\includegraphics[width=1.2in]{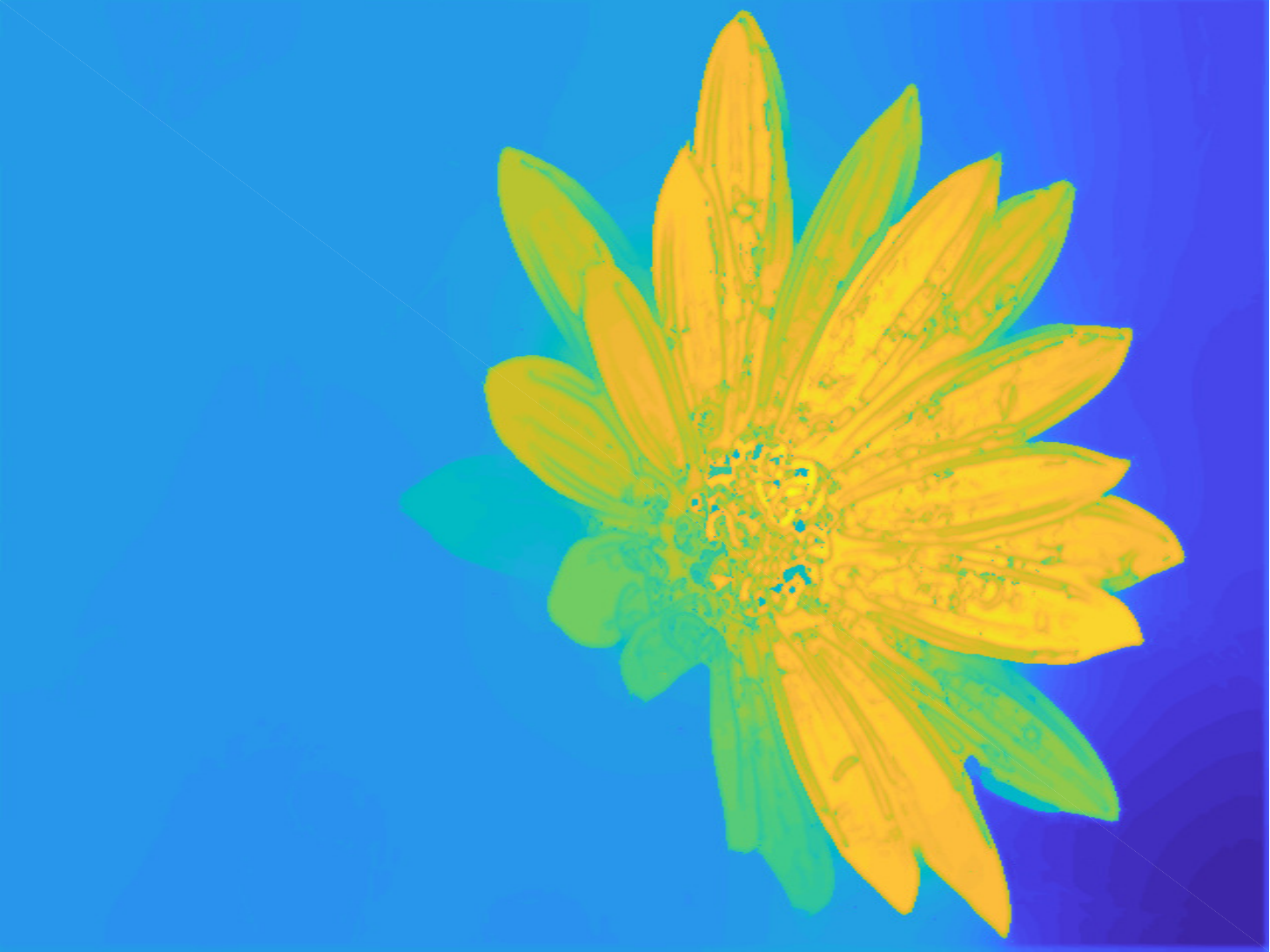}\\
\includegraphics[width=1.2in]{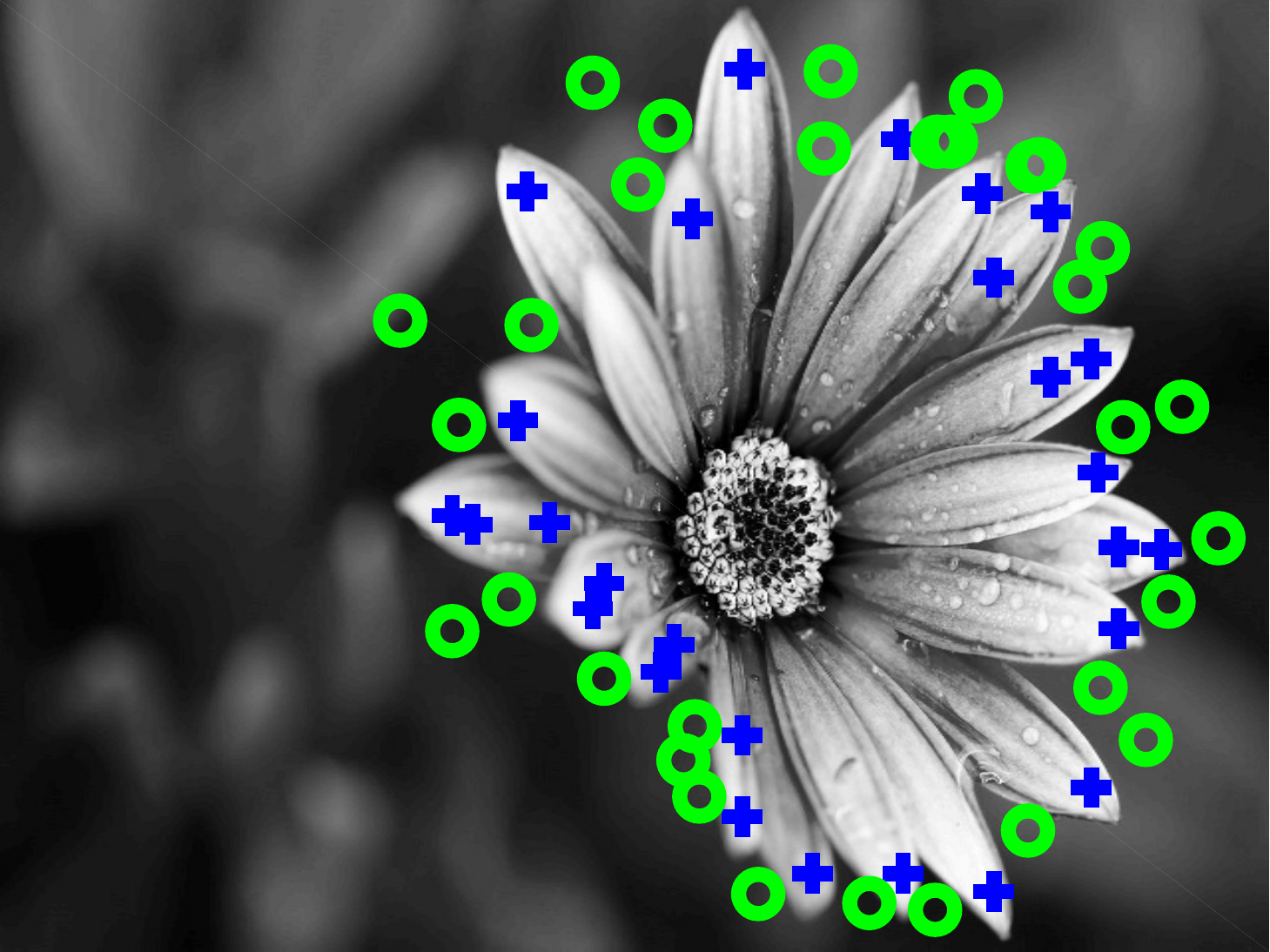}
\includegraphics[width=1.2in]{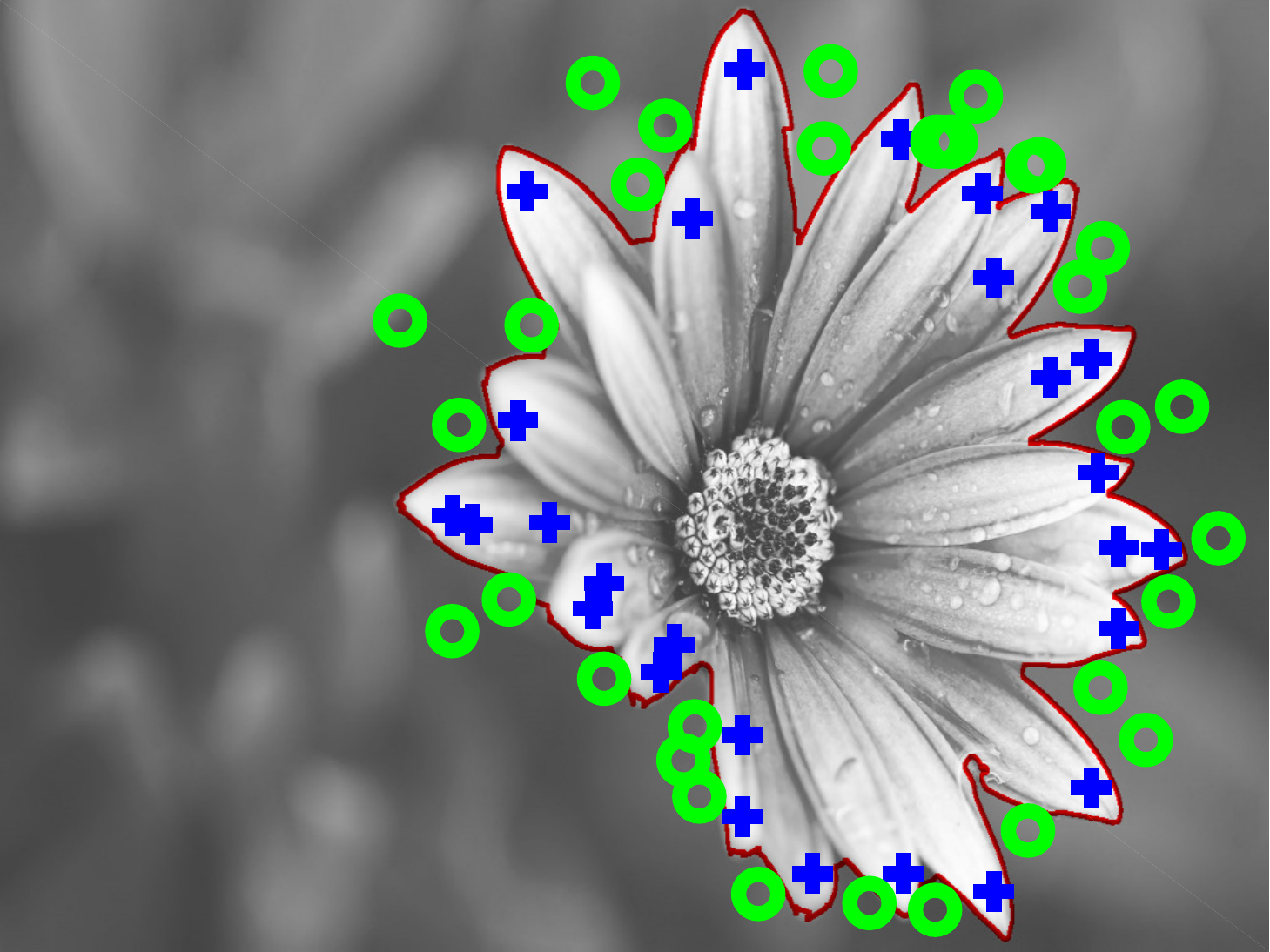}
\includegraphics[width=1.2in]{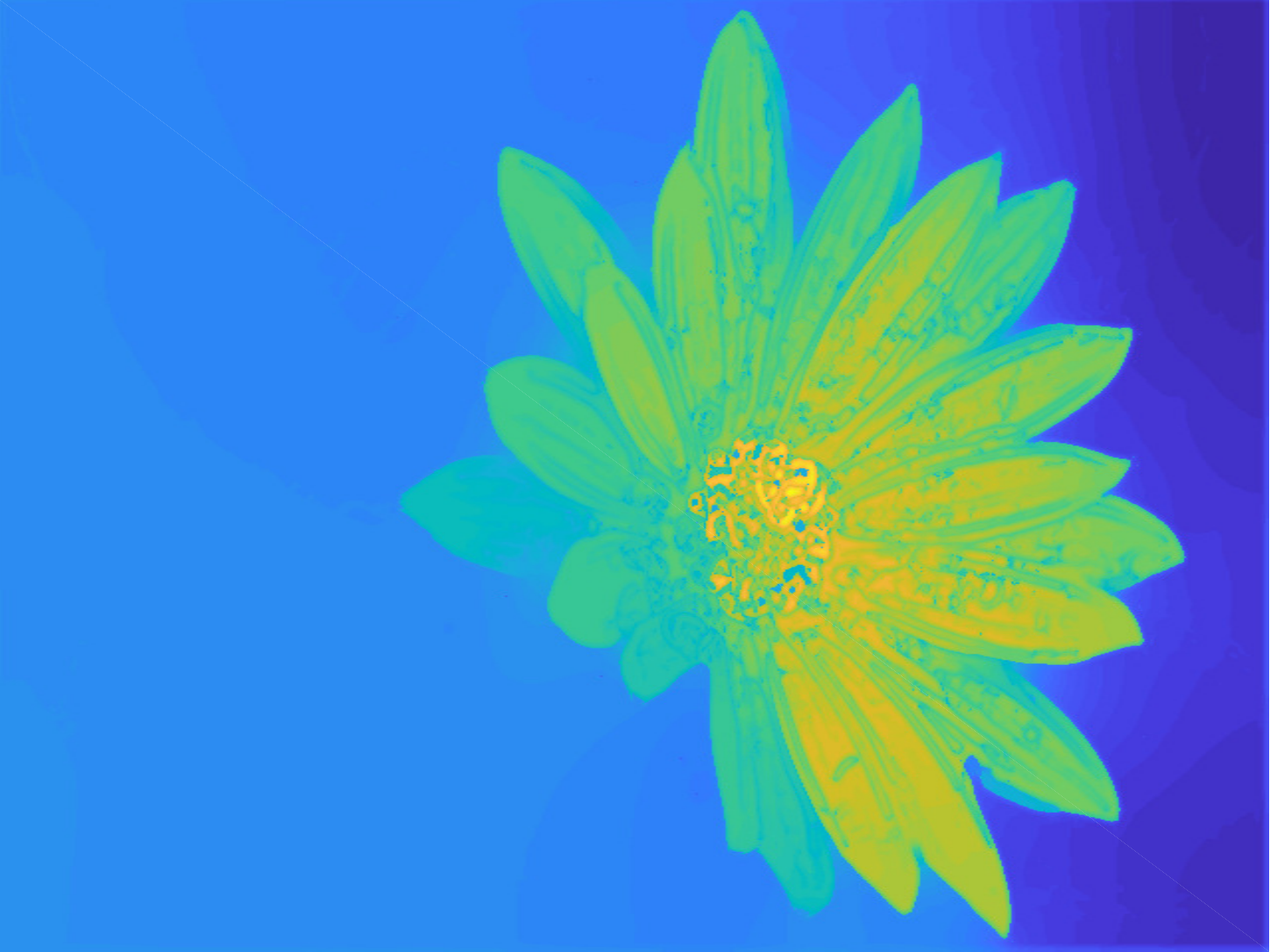}\\
\end{center}
\caption{The left, middle and right columns are labels,
results of image cut and the heat maps of the solutions
by the Lanczos algorithm for CRQopt, respectively.
Images from top to bottom are
Flower, Road, Crab, Camel, Dog, Face1, Face2,
Daisy and Daisy2, respectively.}
\label{fig:imagecut}
\end{figure}

\begin{table}[H]
\caption{Runtime (in seconds) and number of Lanczos steps}
\label{Table:runtimeimages2}
\centerline{
\begin{tabular}[b]{|@{\hspace{5pt}}c@{\hspace{5pt}}||c|c|} \hline
Image  &   Run Time    &   Lanczos steps     \\ \hline\hline
Flower  & 4.61& 210 \\\hline
Road  &14.92& 200 \\\hline
%Shell  & 6.15& 205 \\\hline
Crab & 21.58& 135 \\\hline
Camel  & 31.12& 300 \\\hline
Dog  & 22.33& 135 \\\hline
Face1  & 67.46& 215  \\\hline
Face2  & 35.54& 165  \\\hline
Daisy  & 84.09 &  235 \\\hline
Daisy2  & 105.80 &  245 \\\hline
\end{tabular}
}
\end{table}

The purpose of experiments on Daisy and Daisy2 which are 
the same images but with two different ways of labeling
is to observe how the size $m$ of the linear constraints may affect
running time. Daisy has $29$ linear constraints while Daisy2 has $59$.
As shown in Table~\ref{Table:runtimeimages2}, the Lanczos algorithm took  $84.09$ seconds for Daisy
and $105.80$ seconds for Daisy2, suggesting the larger $m$ is, the more times
the Lanczos algorithm needs, as expected, to solve the associated CQRopt.
This is because matrix-vector product $Px$ does more work as $m$ increases.

In Table \ref{Table:runtimeimages3}, we show the running time for Fast-GE-2.0~\cite{jixb:2017}, projected power method \cite{xuls:2009}, and the Lanczos algorithm for a few examples. For comparable segmentation quality,
the runtime of the Lanczos algorithm for CRQopt \eqref{eq:CRQopt}
is significantly less than the existing methods,
including Fast-GE-2.0 and
the projected power method.
 For example,
with the same prior labeling on the image Crab,
Fast-GE-2.0 and the projected power method
take $47.13$ seconds and $446.76$ seconds, respectively, while our Lanczos algorithm only takes $21.18$ seconds.
Again, with the same labeling on Daisy and Daisy2,
Fast-GE-2.0 takes $1572.81$ and seconds $1319.58$ seconds, respectively,
the projected power method fails to converge in three hours, while the Lanczos algorithm
only takes $84.09$ seconds and $105.80$ seconds, respectively.

\begin{table}[H]
\caption{Runtime  for Fast-GE-2.0, projected power method and the Lanczos algorithm}
\label{Table:runtimeimages3}
\centerline{
\begin{tabular}[b]{|@{\hspace{5pt}}c@{\hspace{5pt}}||c|c|c|} \hline
Image  &  Fast-GE-2.0  & Projected Power Method &  Lanczos algorithm     \\ \hline\hline
Crab & 47.13 s&446.76 s&21.58 s \\\hline
Daisy  &1572.81 s& 3+ hours & 84.09 s \\\hline
Daisy2  &1319.58 s& 3+ hours& 105.80 s  \\\hline
\end{tabular}
}
\end{table}

\section{Conclusions}\label{sec-concluding}
Although the constrained Rayleigh quotient  optimization problem (CRQopt) \eqref{eq:CRQopt},
also known as the linear constrained eigenvalue problem, has been around since 1970s, 
%and was studied somewhat systematically in 1989 in \cite{gagv:1989}, 
some of the mathematical claims
were not rigorously justified. There are not many numerical methods that are suitable 
for large scale CRQopt \eqref{eq:CRQopt}, such as those arising from constrained image segmentation.  
The projected power method \cite{xuls:2009} converges too slow while
the method in \cite{gozz:2000} is for the homogeneous constraints only.
Eigenvalue optimization method \cite{erof:2011} could be too 
expensive. 
In this paper, we launched a systematical and rigorous theoretical study of the problem
and, as a result, devised an efficient Lanczos algorithm for large scale CRQopt \eqref{eq:CRQopt}.
We perform a detailed convergence analysis. As an application, we apply our Lanczos algorithm
to the image cut problem with partial prior labeling. Numerical experiments on several images
demonstrate the effectiveness of the algorithm in terms of accuracy and superior efficiency
compared to Fast-GE-2.0~\cite{jixb:2017} and
the projected power method \cite{xuls:2009}. 
For future work, our goal is to solve rLGopt \eqref{eq:rLGopt} for nearly hard case 
and applications of our algorithms on more machine learning problems such as 
outlier removal \cite{lihs:2014}, semi-supervised kernel PCA \cite{paog:2013}, 
and transductive learning \cite{joac:2003}.

%We discussed an algorithm for CRQopt based on the reduction of Lagrange equations and QEP.
%Numerical examples and applications on constrained image segmentation problems showed the correctness and efficiency of our algorithm.
%For future work, our goal is to solve rLGopt \eqref{eq:rLGopt} for the nearly hard case and applications of our algorithms on more machine
%learning problems such as outlier removal \cite{lihs:2014}, semi-supervised kernel PCA \cite{paog:2013},
%and transductive learning \cite{joac:2003}.

Although our developments in this article have been restricted to the real numbers, 
their extensions to the complex version of CRQopt~\eqref{eq:CRQopt}
$$
\min_{v \in \mathbb{C}^n} v^{\HH}Av\quad\mbox{s.t.}\,\, v^{\HH}v=1\,\,\mbox{and}\,\,C^{\HH}v=b
$$
is rather straightforward, where $A\in\bbC^{n\times n}$ is Hermitian, 
i.e., $A=A^{\HH}$, $C\in\bbC^{n\times m}$.
Essentially, all we need to do is to replace all transposes ${\,\cdot\,}^{\T}$ 
by complex conjugate transposes ${\,\cdot\,}^{\HH}$.

%%\newpage
\appendix
\section{Solve secular equation}\label{sec:secularEq}

We are interested in computing the smallest zero $\lambda_{\ast}$ of the secular function
\begin{equation}\label{eq:sec-fun:H'}
\chi(\lambda):=\sum_{i=1}^{n}\frac {\xi_i^2}{(\lambda-\theta_i)^2}-\gamma^2,
\end{equation}
where  it is assumed
\begin{align*}
&\gamma>0,\,\, \theta_1\le\theta_2\le\cdots\le\theta_n,\,\,\mbox{and}, \\
&\mbox{either $\xi_1\ne 0$, or $\xi_1=0$ but}\,\,\lim_{\lambda\to\theta_1^-}\chi(\lambda)>0.
\end{align*}
Those assumptions guarantee that $\chi(\lambda)$ has a unique zero $\lambda_{\ast}$ in $(-\infty,\theta_1)$. This is because
$$
\lim_{\lambda\to-\infty}\chi(\lambda)=-\gamma^2<0,\,\,
\lim_{\lambda\to\theta_1^-}\chi(\lambda)>0,\,\,
\mbox{and}\,\,
\chi'(\lambda)=-2\sum_{i=1}^{n}\frac {\xi_i^2}{(\lambda-\theta_i)^3}>0\,\,
\mbox{for $\lambda<\theta_1$}.
$$
First, we find an initial lower bound $\alpha^{(0)}$ of $\lambda_\ast$, i.e., $\alpha^{(0)}<\theta_1$ such that $\chi(\alpha^{(0)})<0$. Note
$$
\chi(\lambda)\le\sum_{i=1}^{n}\frac {\xi_i^2}{(\lambda-\theta_1)^2}-\gamma^2
\quad\mbox{for $\lambda<\theta_1$}.
$$
One such $\alpha^{(0)}$ can be found by solving
$$
\sum_{i=1}^{n}\frac {\xi_i^2}{(\alpha^{(0)}-\theta_1)^2}-\gamma^2=0
\quad\Rightarrow\quad
\alpha^{(0)}=\theta_1-\delta_0
\,\,
\mbox{with}\,\,
\delta_0=\frac 1{\gamma}\sqrt{\sum_{i=1}^{n}\xi_i^2}\,.
$$
We conclude that $\lambda_\ast\in[\alpha^{(0)},\beta^{(0)}]$, where $\beta^{(0)}=\theta_1$.
Quantities $\alpha^{(k)}$ and $\beta^{(k)}$ will be determined during our
iterative process to be described  such that $\lambda_\ast\in[\alpha^{(k)},\beta^{(k)}]$.

Without loss of generality, we may assume that
$$
\mbox{if $\theta_1=\cdots=\theta_d<\theta_{d+1}$, then $\xi_2=\cdots=\xi_d=0$.}
$$
Let
\begin{equation}\label{eq:j0}
j_0=\min\{i\,:\,\xi_i\ne 0\}.
\end{equation}
To find the initial guess of the root, we solve
$$
\frac {\xi_{j_0}^2}{(\lambda-\theta_{j_0})^2}+\underbrace{\sum_{i=j_0+1}^{n}\frac {\xi_i^2}{([\theta_{j_0}-\delta_0]-\theta_i)^2}-\gamma^2}_{=:-\eta}=0
$$
for $\lambda$ to get
$$
\lambda^{(0)}=\begin{cases}
          \theta_{j_0}-|\xi_{j_0}|/\sqrt{\eta},&\quad\mbox{if $\eta>0$}, \\
          \theta_{j_0}-\delta_0/2,  &\quad\mbox{if $\eta\le 0$},
          \end{cases}
$$
where the second case is based on bisection.

For the iterative scheme, suppose we have an approximation $\lambda^{(k)}\approx\lambda_\ast$. First, the interval $(\alpha^{(k)},\beta^{(k)})$ will be updated as
$$
\mbox{$\alpha^{(k+1)}\leftarrow\lambda^{(k)}$ and $\beta^{(k+1)}\leftarrow\beta^{(k)}$ if $\chi(\lambda^{(k)})<0$}$$
$$\mbox{$\beta^{(k+1)}\leftarrow\lambda^{(k)}$ and $\alpha^{(k+1)}\leftarrow\alpha^{(k)}$ if $\chi(\lambda^{(k)})> 0$.}
$$
Then we find the next approximation $\lambda^{(k+1)}$. For that purpose,
we seek to approximate $\chi$, in the neighborhood of $\lambda^{(k)}$, by
$$
g(\lambda):=-b+\frac a{(\lambda-\theta_{j_0})^2}\approx\chi(\lambda),
$$
such that
\begin{alignat*}{2}
g(\lambda^{(k)})&\equiv -b+\hphantom{2}\frac a{(\lambda^{(k)}-\theta_{j_0})^2}&&=\chi(\lambda^{(k)})=\hphantom{-2}\sum_{i=1}^{n}\frac {\xi_i^2}{(\lambda^{(k)}-\theta_i)^2}-\gamma^2,\\
g'(\lambda^{(k)})&\equiv\hphantom{-b} -2\frac a{(\lambda^{(k)}-\theta_{j_0})^3}&&=\chi'(\lambda^{(k)})=-2\sum_{i=1}^{n}\frac {\xi_i^2}{(\lambda^{(k)}-\theta_i)^3},
\end{alignat*}
yielding
\begin{align*}
a&=-\frac 12(\lambda^{(k)}-\theta_{j_0})^3\chi'(\lambda^{(k)})=(\lambda^{(k)}-\theta_{j_0})^3\sum_{i=1}^{n}\frac {\xi_i^2}{(\lambda^{(k)}-\theta_i)^3}>0, \\
b&=\frac a{(\lambda^{(k)}-\theta_{j_0})^2}-\chi(\lambda^{(k)})=(\lambda^{(k)}-\theta_{j_0})\sum_{i=1}^{n}\frac {\xi_i^2}{(\lambda^{(k)}-\theta_i)^3}-\chi(\lambda^{(k)}).
\end{align*}
Ideally, $b>0$ so that $g(\lambda)=0$ has a solution in $(-\infty,\theta_{j_0})$. Assuming $b>0$, we find the next approximation
$\lambda^{(k+1)}\approx\lambda_\ast$ is given by
\begin{equation}\label{eq:tau-new}
\lambda^{(k+1)}=\theta_1-\sqrt{a/b}.
\end{equation}

Now if $b\le 0$ (then $\lambda^{(k+1)}$ as in \eqref{eq:tau-new} is undefined) or if $\lambda^{(k+1)}\not\in(\alpha,\beta)$, we let
$\lambda^{(k+1)}$ be $(\alpha^{(k+1)}+\beta^{(k+1)})/2$ according to bisection method.

%------------------------------------------------------------

\section{Proof of the equivalence between CRQopt and the 
eigenvalue optimization problem}\label{sec:dual}
Suppose 
$U\in\bbR^{n\times (n-m)}$ has full column rank and that $\cR(U)=\cN(C^{\T})$ and let $u\in\bbR^n$
satisfies $C^{\T}u=\sqrt n\,b$. Define
\begin{equation}\label{eq:defN}
\widehat{C}=[C^{\T},\ -\sqrt{n}b], \,\,
 N=\kbordermatrix{&\sss n-m &\sss 1 \\
   \sss n & U & u\\
   \sss 1 & 0&1}.
\end{equation}
and 
$$
L=N^{\T}\begin{bmatrix}
A &0\\0& 0
\end{bmatrix}N,\,\,
E=N^{\T}\begin{bmatrix}
-\frac{I}{n+1} &0\\0& 1-\frac{1}{n+1}
\end{bmatrix}N,\,\,
M=N^{\T}\begin{bmatrix}
I_n &0\\0& 0
\end{bmatrix}N.
$$
Note that it is easy to see that $\cR(N)=\mathcal{N}(\widehat{C})$. 

In this appendix we prove that CRQopt \eqref{eq:CRQopt} is equivalent to 
the following eigenvalue optimization problem
%\marginpar{\tiny Is the proof here is the same as that in \cite{erof:2011}?}
\begin{equation}\label{eq:res-erof:2011}
\max_{t\in\bbR}\lambda_{\min}(L+tE,M),
\end{equation}
where $\lambda_{\min}(L+tE,M)$ is the smallest eigenvalue of
$(L+tE)x=\lambda Mx$.
This equivalency was initiately established by 
Eriksson, Olsson and Kahl \cite{erof:2011}. 
However, the statements presented here are stronger than 
the related ones in \cite{erof:2011}. For examples, 
we will prove $M$ is positive definite, and we can use 
'$\max$' in \eqref{eq:res-erof:2011} instead of 
'$\sup$' in \cite{erof:2011}. 

\medskip

 Let $\widetilde{v}=\sqrt{n}v$, 
 $\widehat{v}=\begin{bmatrix}
             \widetilde{v}\\1
             \end{bmatrix}$, $\widehat{A}=\begin{bmatrix}
A &0\\0& 0
\end{bmatrix}$, $\widehat{B}=\begin{bmatrix}
I_n &0\\0& 0
\end{bmatrix}$.
Then $v$ is a minimizer of CRQopt \eqref{eq:CRQopt} if and only if $\widehat{v}$ is a minimizer of
%\begin{equation}\label{eq:CRQopttrans1}
% \min \frac{\widehat{v}^{\T} \widehat{A}\widehat{v}}
%             {\widehat{v}^{\T}\widehat{B}\widehat{v}},
% \quad\mbox{s.t.}\,\,  \widetilde{v}^{\T}\widetilde{v}=n,\,\, \widehat{C}\widehat{v}=0.
%\end{equation}
%\begin{subequations}\label{eq:CRQopttrans1}
%\begin{empheq}[left={\empheqlbrace}]{alignat=2}
% \min     ~& \frac{\widehat{v}^{\T} \widehat{A}\widehat{v}}
%             {\widehat{v}^{\T}\widehat{B}
%         \widehat{v}},     \\
%\mbox{s.t.}  ~& \widetilde{v}^{\T}\widetilde{v}=n,  \\
%             ~& \widehat{C}\widehat{v}=0.
%\end{empheq}
%\end{subequations}
%Clearly, \eqref{eq:CRQopttrans1} is equivalent to
\begin{equation}\label{eq:CRQopttrans2}
\min\frac{\widehat{v}^{\T} \widehat{A}\widehat{v}}
             {\widehat{v}^{\T}\widehat{B}\widehat{v}},
 \quad\mbox{s.t.}\,\,   \widehat{v}_{(n+1)}^2=1,\,\,
 \widehat{v}^{\T}\widehat{v}=n+1,  \,\,
  \widehat{C}\widehat{v}=0.
\end{equation}
%            \begin{subequations}\label{eq:CRQopttrans2}
%\begin{empheq}[left={\empheqlbrace}]{alignat=2}
%\min         ~& \frac{\widehat{v}^{\T} \widehat{A}\widehat{v}}
%             {\widehat{v}^{\T}\widehat{B}
%         \widehat{v}},     \\
%\mbox{s.t.}  ~& \widehat{v}_{(n+1)}^2=1,  \\
%& \widehat{v}^{\T}\widehat{v}=n+1,  \\
%             ~& \widehat{C}\widehat{v}=0.\label{eq:homolinear}
%\end{empheq}
%\end{subequations}
%To remove the constraint $\widehat{C}\widehat{v}=0$, let $N$ be defined in \eqref{eq:defN}.
Since $\cR(N)=\mathcal{N}(\widehat{C})$,  
for any $\widehat{v}$ satisfying $\widehat{C}\widehat{v}=0$, there exists 
$\widehat{y}\in\bbR^{n-m+1}$
such that $\widehat{v}=N\widehat{y}$,  $N$ is defined in \eqref{eq:defN}.
By the matrix structure in \eqref{eq:defN}, we know that $\widehat{v}_{(n+1)}^2=1$
if and only if  $\widehat{y}_{(n-m+1)}^2=1$. Therefore, solving \eqref{eq:CRQopttrans2} is equivalent to solving
\begin{equation}\label{eq:CRQopttrans}
\min  \frac{\widehat{y}^{\T}L\widehat{y}}{\widehat{y}M\widehat{y}},
 \quad\mbox{s.t.}\,\,   \widehat{y}_{(n-m+1)}^2-1=0,\,\, \widehat{y}^{\T}N^{\T}N\widehat{y}=n+1.
\end{equation}
%\begin{subequations}\label{eq:CRQopttrans}
%\begin{empheq}[left={\empheqlbrace}]{alignat=2}
%\min         ~& \frac{\widehat{y}^{\T}L\widehat{y}}{\widehat{y}M\widehat{y}},     \\
%\mbox{s.t.}  ~& \widehat{y}_{(n-m+1)}^2-1=0,  \\
%             ~& \widehat{y}^{\T}N^{\T}N\widehat{y}=n+1.
%\end{empheq}
%\end{subequations}

To prove \eqref{eq:CRQopttrans} is equivalent to its dual problem, 
we use the following result on the duality of the 
quadratic constrained optimization problems.

\begin{lemma}[{\cite[Corollary 1]{erof:2011}}]\label{lm:p=d-quad}
Let $y^{\T}A_2y+2b^{\T}_2y+c_2$ be a positive semidefinite quadratic form.
If there exists $y$ such that $y^{\T}A_3y+2b^{\T}_3y+c_3<0$ and if $A_3$ is positive semidefinite, then the primal problem
$$
\inf_y \frac{y^{\T}A_1y+2b_1^{\T}y+c_1}{y^{\T}A_2y+2b_2^{\T}y+c_2},\quad\mbox{\rm s.t.}\,\,y^{\T}A_3y+2b^{\T}_3y+c_3= 0
$$
and the dual problem
$$
\sup_{\lambda }\inf_{y}\frac{y^{\T}(A_1+\lambda A_3)y+2(b_1+\lambda b_3)^{\T}y+(c_1+\lambda c_3)}{y^{\T}A_2y+2b^{\T}_2y+c_2}
$$
has no duality gap.
\end{lemma}
\begin{proof}
See {\cite[Corollary 1]{erof:2011}}.
\end{proof}

With the help of Lemma~\ref{lm:p=d-quad},
we have the following theorem to show that
that there is no duality gap between the optimization problem 
\eqref{eq:CRQopttrans} and its dual problem.

\begin{theorem}[{\cite[Theorem 1]{erof:2011}}]\label{thm:dual}
Let
$\widehat{A}_i=\begin{bmatrix}
A_i &b_i\\b_i^{\T} &c_i
\end{bmatrix}$ for $i=1,2,3$. If $\widehat{A}_2$ and $A_3$ are  positive semidefinite and if there exists $\widehat{y}$ such that $\widehat{y}^{\T}\widehat{A}_3\widehat{y}<n+1$
and $\widehat{y}_{n+1}^2=1$, then the primal problem
\begin{equation}\label{eq:primal}
\inf_{y^{\T}A_3y+2b^{\T}_3y+c_3= n+1} \frac{y^{\T}A_1y+2b_1^{\T}y+c_1}{y^{\T}A_2y+2b_2^{\T}y+c_2}
  =\inf_{\widehat{y}^{\T}\widehat{A}_3\widehat{y}= n+1, \widehat{y}_{n+1}^2=1}\frac{\widehat{y}^{\T}\widehat{A}_1\widehat{y}}
                                                                                   {\widehat{y}^{\T}\widehat{A}_2\widehat{y}}
\end{equation}
and its dual
$$\sup_t\inf_{\widehat{y}^{\T}\widehat{A}_3\widehat{y}=n+1}\frac{\widehat{y}^{\T}\widehat{A}_1\widehat{y}-t\widehat{y}_{n+1}^2-t}{\widehat{y}^{\T}\widehat{A}_2\widehat{y}}$$
has no duality gap.
\end{theorem}
\begin{proof}
Let $\gamma_\ast$ be the optimal value of \eqref{eq:primal}, then
\begin{align}
\gamma_\ast&=\inf_{\widehat{y}^{\T}\widehat{A}_3\widehat{y}= n+1, \widehat{y}_{n+1}^2=1}\frac{\widehat{y}^{\T}\widehat{A}_1\widehat{y}}
                                                                                   {\widehat{y}^{\T}\widehat{A}_2\widehat{y}} \notag
\\
&=\sup_t\inf_{\widehat{y}^{\T}\widehat{A}_3\widehat{y}= n+1, \widehat{y}_{n+1}^2=1}\frac{\widehat{y}^{\T}\widehat{A}_1\widehat{y}+t\widehat{y}_{n+1}^2-t}
                                                                                   {\widehat{y}^{\T}\widehat{A}_2\widehat{y}} \notag\\
                                                                                 &\geq\sup_t\inf_{\widehat{y}^{\T}\widehat{A}_3\widehat{y}= n+1}\frac{\widehat{y}^{\T}\widehat{A}_1\widehat{y}+t\widehat{y}_{n+1}^2-t}
                                                                                   {\widehat{y}^{\T}\widehat{A}_2\widehat{y}} \notag\\
                                                                                   &\geq\sup_{t,\lambda}\inf_{\widehat{y}} \frac{\widehat{y}^{\T}\widehat{A}_1\widehat{y}+t\widehat{y}_{n+1}^2-t+\lambda(\widehat{y}^{\T}\widehat{A}_3\widehat{y}-( n+1))}
                                                                                   {\widehat{y}^{\T}\widehat{A}_2\widehat{y}} \notag\\
                                                                                   &=\sup_{t,\lambda}\inf_{\widehat{y}} \frac{{y}^{\T}{A}_1{y}+2b_1^{\T}y+c_1+t\widehat{y}_{n+1}^2-t+\lambda({y}^{\T}{A}_3{y}+2b_3^{\T}y+c_3-( n+1))}
                                                                                   {{y}^{\T}{A}_2{y}+2b_2^{\T}y+c_2} \notag\\
                                                                                   &=\sup_{t,\lambda}\inf_{\widehat{y}_{n+1}^2=1} \frac{{y}^{\T}{A}_1{y}+2b_1^{\T}y+c_1+\lambda({y}^{\T}{A}_3{y}+2b_3^{\T}y+c_3-( n+1))}
                                                                                   {{y}^{\T}{A}_2{y}+2b_2^{\T}y+c_2} \label{eq:pr6}\\
                                                                                   &=\inf_{y^{\T}A_3y+2b^{\T}_3y+c_3= n+1} \frac{y^{\T}A_1y+2b_1^{\T}y+c_1}{y^{\T}A_2y+2b_2^{\T}y+c_2}=\gamma_\ast, \label{eq:pr8}
\end{align}
where \eqref{eq:pr6} and \eqref{eq:pr8} apply Lemma \ref{lm:p=d-quad}.
\end{proof}

\begin{remark}
One of the conditions in {\cite[Theorem 1]{erof:2011}} is 
``$\widehat{A}_3$ is positive semidefinite''. However, the proof of Theorem \ref{thm:dual} 
applies Lemma \ref{lm:p=d-quad}, which requires $A_3$ to be positive semidefinite and 
there exists $\widehat{y}$ such that $\widehat{y}^{\T}\widehat{A}_3\widehat{y}<n+1$
and $\widehat{y}_{n+1}^2=1$. Therefore, the condition ``$\widehat{A}_3$ is positive 
semidefinite'' is not necessary.
In addition, in the statement of {\cite[Theorem 1]{erof:2011}}, one of the constraints is 
$ {y}_{n+1}^2=1$. However, in \eqref{eq:primal}, the size of the 
matrix $A_i$ and $\widehat{A_i}$ is $n\times n$ and $(n+1)\times (n+1)$ for $i=1,2,3$, 
respectively. Therefore, we consider $y\in\mathbb{R}^n$ and $\widehat{y}\in\mathbb{R}^{n+1}$.
Therefore, we change the constraint $ {y}_{n+1}^2=1$ to $ \widehat{y}_{n+1}^2=1$.
\end{remark}

We now prove that the conditions of Theorem~\ref{thm:dual} 
are staisfied for the constrained Rayleigh quotient optimization problem \eqref{eq:CRQopttrans}.

\begin{lemma}\label{lem:interior}
Suppose $\|v_0\|<1$, where $v_0=(C^{\T})^\dag b$.
% \eqref{eq:CRQopt-1} and \eqref{eq:CRQopt-2}, i.e., $v^{\T}v=1$ and $C^{\T}v=b$,
Then there exists $\widehat{y}$ such that $\|\widehat{y}\|_N^2=\widehat{y}^{\T}N^{\T}N\widehat{y}<n+1$ and $\widehat{y}_{(n-m+1)}=1$.
\end{lemma}

\begin{proof}
Note that $v_0=(C^{\T})^\dag b$ is the minimum norm solution of $C^{\T}v=b$.
Let $\widehat{v}=[\sqrt{n}v_0^{\T}, 1]^{\T}$. Then $\widehat{v}\in\mathcal{N}(\widehat{C})$ and thus
there exists $\widehat{y}$ such that $\widehat{v}=N\widehat{y}$ for which we have
$\|\widehat{y}\|_N=\|\widehat{v}\|_2<\sqrt{n+1}$, and, at the same time, $\widehat{y}_{(n-m+1)}=\widehat{v}_{(n+1)}=1$.
\end{proof}

By Lemma \ref{lem:interior} and Theorem \ref{thm:dual}, the optimization problem \eqref{eq:CRQopttrans} is equivalent to its dual problem
\begin{equation}\label{eq:dual}
\sup_t\inf_{\widehat{y}^{\T}N^{\T}N\widehat{y}=n+1} \frac{\widehat{y}^{\T}L\widehat{y}+t\widehat{y}_{n-m+1}^2-t}{\widehat{y}^{\T}M\widehat{y}}.
\end{equation}
Since
$$
t\widehat{y}_{n-m+1}^2-t=t\widehat{y}_{n-m+1}^2-t\frac{\widehat{y}^{\T}N^{\T}N\widehat{y}}{n+1}=\widehat{y}^{\T}E\widehat{y},
$$
\eqref{eq:dual} is equivalent to
\begin{equation}\label{eq:dual2}
\sup_t\inf_{\widehat{y}^{\T}N^{\T}N\widehat{y}=n+1} \frac{\widehat{y}^{\T}(L+tE)\widehat{y}}{\widehat{y}^{\T}M\widehat{y}}.
\end{equation}

To transform the dual problem \eqref{eq:dual2} to an eigenvalue problem, 
we first prove that $M$ is positive definite.

\begin{lemma}\label{lem:pd}
Let $b$ be as defined in \eqref{eq:CRQopt-2} and $b\neq 0$.  $N$ has full column rank, then $M$ is positive definite.
\end{lemma}
\begin{proof}
It is clear that $M$ is positive semi-definite. We claim that $M$ is nonsingular.
Suppose, to the contrary, that $M$ is singular. Then there exists a nonzero $x$ such that $Mx=0$.

We claim that $x_{(n-m+1)}\neq 0$; otherwise suppose $x_{(n-m+1)}= 0$ and write
$x=\begin{bmatrix}
     x_1 \\
     0
   \end{bmatrix}$. It follows from $Mx=0$ that $U^{\T}Ux=0$, implying $x_1=0$ because $U$ has full column rank.
Thus $x=0$, a contradiction.

Without loss of generality, we may normalize $x_{(n-m+1)}$ to $1$, i.e.,  $x=\begin{bmatrix}
x_1\\1
\end{bmatrix} $. Note that $M=N^{\T}N-e_{n-m+1}e_{n-m+1}^{\T}$. $Mx=0$ implies $N^{\T}Nx=\begin{bmatrix}
0\\1
\end{bmatrix}$.
$N^{\T}N$ is invertible. We now express $(N^{\T}N)^{-1}_{(n-m+1,n-m+1)}$ in two different ways.
$N^{\T}Nx=\begin{bmatrix}
0\\1
\end{bmatrix}$ yields $x=(N^{\T}N)^{-1}\begin{bmatrix}
                                             0\\ 1
                                             \end{bmatrix}$ and thus
$$
1=\begin{bmatrix}
    0\\ 1
    \end{bmatrix}^{\T}x=\begin{bmatrix}
    0\\ 1
    \end{bmatrix}^{\T}(N^{\T}N)^{-1}\begin{bmatrix}
                                             0\\ 1
                                             \end{bmatrix}=(N^{\T}N)^{-1}_{(n-m+1,n-m+1)}.
$$
{On the other hand,
$$
N^{\T}N=\begin{bmatrix}
U^{\T}U &U^{\T}u\\u^{\T}U &u^{\T}u+1
\end{bmatrix}.
$$
By the assumption that $U$ has full column rank, $U^{\T}U$ is invertible. 
With help of a formula
\footnote{$\det\left(\begin{bmatrix}
A&B\\C&D
\end{bmatrix}\right)=\det(A)\det(D-CA^{-1}B)$ when $A$ is invertable.} 
of the determinant of block matrices, we have
$$\det(N^{\T}N)=\det(U^{\T}U)\det[(1+u^{\T}u-u^{\T}U(U^{\T}U)^{-1}U^{\T}u].$$
According to the relationship between the inverse and the adjoint of a matrix, we find}
\begin{align*}
(N^{\T}N)^{-1}_{(n-m+1,n-m+1)}&=(-1)^{n-m+1+n-m+1}\frac{\det(U^{\T}U)}{\det(N^{\T}N)}\\
&=\frac{\det(U^{\T}U)}{\det(U^{\T}U)\det[(1+u^{\T}u-u^{\T}U(U^{\T}U)^{-1}U^{\T}u]}\\
   &=\frac{\det(U^{\T}U)}{\det(U^{\T}U)[1+u^{\T}(I-P_U)u]},
\end{align*}
where $P_U$ is the orthogonal projection onto $\cR(U)$. Therefore, $(N^{\T}N)^{-1}_{(n-m+1,n-m+1)}=1$
if and only if $u^{\T}(I-P_U)u=0$ implying that $u$ is in the column space of $U$. Without loss of generality,
we may assume the first column of $U$ is $u$. Now
subtract the first column of $N$ from its last column to conclude that  $e_{n+1}$ is in the null space of $\widehat{C}$,
which contradicts that $b\neq 0$.
\end{proof}

By Lemma \ref{lem:pd} and Courant-Fisher minimax theorem \cite[Theorem 8.1.2]{govl:2013}, finding
$$
\inf_{\widehat{y}^{\T}N^{\T}N\widehat{y}=n+1} \frac{\widehat{y}^{\T}(L+tE)\widehat{y}}{\widehat{y}^{\T}M\widehat{y}}
$$
is equivalent to finding the smallest eigenvalue of $K^{-1}(L+tE)K^{-\T}x=\lambda x$,
where $M=KK^{\T}$ is the Cholesky factorization of $M$. Therefore, \eqref{eq:dual2} is equivalent to
\begin{equation}\label{eq:eigt}
\sup_t\lambda_{\min}(L+tE,M).
\end{equation}

Finally, we prove that the maximum value can be obtained, i.e., 
'$\sup$' in \eqref{eq:eigt} can be replaced by '$\max$'.

\begin{lemma}
Let $f(t)=\lambda_{\min}(L+tE,M)$. There exits $t_0\in\bbR$ such that $f(t_0)=\sup_{t\in\bbR}f(t)$.
\end{lemma}
\begin{proof}
 We prove the claim by showing that
$$
\lim_{t\rightarrow+\infty}f(t)=\lim_{t\rightarrow-\infty}f(t)=-\infty.
$$
First, let $v_1\in\cR(N)$ with the last component being zero, and set $y_1=N^{\T}v_1$.
We have $y_1^{\T}Ey_1=-\frac{\|v_1\|^2_2}{n+1}<0$ and $y_1^{\T}My_1>0$ 
since $M$ is positive definite. Hence
$$
\lim_{t\rightarrow+\infty}f(t)
  =\lim_{t\rightarrow+\infty}\inf_{\widehat{y}} \frac{\widehat{y}^{\T}(L+tE)\widehat{y}}{\widehat{y}^{\T}M\widehat{y}}
  \le  \lim_{t\rightarrow+\infty}\frac{y_1^{\T}(L+tE)y_1}{y_1^{\T}My_1}
  \le \lim_{t\rightarrow+\infty}t \frac{y_1^{\T}Ey_1}{y_1^{\T}My_1}+\lambda_{\max}(L,M)
  = -\infty.
$$
Recall $v_0=(C^{\T})^{\dag}b$ and the assumption that $\|v_0\|<1$.
Let $v_2=[\sqrt{n}v_0^{\T},\ 1]^{\T}$. Clearly $v_2\in\cR(N)$ and let $y_2=N^{\T}v_2$.
We have $y_2^{\T}Ey_2=-\frac{\|v_0\|^2_2}{n+1}+1-\frac{1}{n+1}>0$ since $\|v_0\|<1$ and $y_2^{\T}My_2>0$
since $M$ is positive definite. Hence
$$
\lim_{t\rightarrow-\infty}f(t)
  =\lim_{t\rightarrow-\infty}\inf_{\widehat{y}} \frac{\widehat{y}^{\T}(L+tE)\widehat{y}}{\widehat{y}^{\T}M\widehat{y}}
  \le  \lim_{t\rightarrow-\infty}\frac{y_2^{\T}(L+tE)y_2}{y_2^{\T}My_2}
  \le \lim_{t\rightarrow-\infty}t \frac{y_2^{\T}Ey_2}{y_2^{\T}My_2}+\lambda_{\max}(L,M)
  = -\infty.
$$

Therefore, there exits $t_1<0$ such that $f(t)<f(0)$ for $t<t_1$  and there exits $t_2>0$ such that $f(t)<f(0)$
for when $t>t_2$. Therefore
$$
\sup_{t\in\bbR} f(t)=\sup_{t\in[t_1,t_2]} f(t).
$$
Because $f(t)=\lambda_{\min}(L+tE,M)$ is a continuous function \cite{stsu:1990},
there exists $t_0\in[t_1,t_2]$ such that $f(t_0)=\sup_{t\in\bbR} f(t)$.
\end{proof}

In conclusion, we have shown that CRQopt \eqref{eq:CRQopt} is equivalent to 
the eigenvalue optimization problem \eqref{eq:res-erof:2011}.

\section{CRQPACK}\label{sec-alg-software}
The Lanczos algorithm for solving CRQopt \eqref{eq:CRQopt} described
in this paper has been implemented in MATLAB.  
In the spirit of reproducible research, MATLAB scripts of 
the implementation of the Lanczos algorithm and the data that used to 
generate numerical results presented in this paper are packed
in a software called package called CRQPACK. 
CRQPACK can be obtained from 
\begin{center}
\url{https://www.math.ucdavis.edu/~yszhou/CRQPACK.zip}.
\end{center}
CRQPACK consists of three folders:
\begin{itemize}
\item \verb|src|:  the source code for solving CRQopt \eqref{eq:CRQopt}.

It consists of four functions \verb|CRQ_Lanczos|, \verb|QEPmin|, \verb|LGopt| and \verb|rLGopt|.
\verb|CRQ_Lanczos| is the driver and calls
\verb|QEPmin| and \verb|LGopt|. \verb|LGopt| is dependent on \verb|rLGopt|.

In addition, we also provide two other drivers
for solving CRQopt \eqref{eq:CRQopt}, namely 
\verb|CRQ_explicit| for the direct method \cite{gagv:1989} and 
\verb|CRQ_ppm| for the projected power method \cite{xuls:2009}. 

\item \verb|synthetic|:  the drivers for numerical examples
      in section \ref{sec-crq-ex}.

 \verb|correct.m| and \verb|QEPres.m| are for the examples in 
            Sections \ref{sec-crq-ex-correct} and \ref{sec-crq-ex-res},
 respectively. \verb|CRQsharp.m| is used to generate the plots for Example~\ref{ex:sharp} 
            on error bounds in \eqref{eq:rLGopt-UBs-1} and \eqref{eq:rLGopt-UBs-2}, while
 \verb|CRQnotsharp.m| on the error bounds \eqref{eq:rLGopt-UBs-1} and \eqref{eq:rLGopt-UBs-2}.
% one of which showswhere the bound is sharp and the bound is pessimistic, respectively, in section \ref{sec-crq-ex-sharp}.

\item \verb|imagecut|: the code for constrained image segmentation. 

It has three subfolders: \verb|examples| contains the drivers,
\verb|data| contains image data including prior labeling information,
and \verb|auxiliary| contains program to generate the 
            matrices $A$, $C$, and vector $b$ of CRQopt \eqref{eq:CRQopt}.
\end{itemize}
The syntax  of calling the driver \verb|CRQ_Lanczos| is as follows:
\begin{center}
\verb|[v,info] = CRQopt(A,C,b,opts)|
\end{center}
where 
\begin{itemize}
\item \verb|A|:  the    matrix $A$ in CRQopt \eqref{eq:CRQopt}
\item \verb|C|:  the   matrix $C$ in CRQopt \eqref{eq:CRQopt}
\item \verb|b|:  the vector $b$ in CRQopt \eqref{eq:CRQopt}
\item \verb|opts|: option parameters:
\begin{itemize}
\item \verb|opts.maxit|: maximum number of Lanczos iteraions
\item \verb|opts.minit|: minimum number of Lanczos iteraions
\item \verb|opts.tol|: tolerance of relative residual
\item \verb|opts.method|: method to solve the optimization problem

      1: solve CRQopt via LGopt (default) 

        2: solve CRQopt via QEPmin

\item \verb|opts.checkstep|: the number of Lanczos steps between solving two 
      rLGopt or two rQEPmin and checking the residuals

\item \verb|opts.resopt|: option for computing the residual (only valid when \verb|opts.method=2|)

      0 : using residual bound \eqref{eq:NRes-QEPmin-2} to estimate residual (default)

        1: using  residual \eqref{eq:NRes-QEPmin-1}

  \item \verb|opts.returnQ|: indicator that whether the algorithm returns $Q_k$ 
        in structure \verb|info|
\end{itemize}

\item \verb|v|:  computed solution of CRQopt \eqref{eq:CRQopt}
\item \verb|info|: information for some internal data: 
\begin{itemize}
\item \verb|info.n0|: vector $n_0$
\item \verb|info.b0|: vector $b_0$
\item \verb|info.gamma2|: the square of parameter $\gamma$
\item \verb|info.k|: the number of Lanczos steps
\item \verb|info.T|: tridiagonal matrix $T_k$
\item \verb|info.mu|: computed eigenvalue or Lagrange multipliers in each iteration
\item \verb|info.res|: norms of relative residual of Lagrange equations/QEP in each iteration
\item \verb|info.Q|: the matrix $Q_k$. This field is valid only when  \verb|opts.returnQ=1|
 \item \verb|info.x|: a cell, whose elements are the solutions of all rLGopt \eqref{eq:rLGopt} solved.
 This field is valid only when\verb|opts.method=1|.
 \item \verb|info.s|: a cell, whose element are the eigenvectors of all LEP \eqref{eq:QEPproj'ed-lin} corresponding
 to the desired eigenvalue. This field is valid only when  \verb|opts.method=2|.
\end{itemize}
\end{itemize}

\section*{Acknowledgment}
YZ would like to thank Mr.~Ning Wan for sharing his study notes of
theory and algorithm of CRQopt,
Mr.~Yanwen Luo for his help in the proof for Lemma \ref{lm:noeig},
Dr.~Chengming Jiang for providing his FAST-GE2.0 implementation
of the constrained image segmentation, and 
Mr.~Michael Ragone for his comments on an early version of 
this manuscript.

% reset the page style
%\pagestyle{plain}
%\renewcommand\bibname{References}
\bibliographystyle{plain}
\bibliography{references}
\end{document}